\PassOptionsToPackage{table}{xcolor}
\documentclass[10pt, a4paper, reqno]{amsart}
\usepackage[T1]{fontenc}
\usepackage[utf8]{inputenc}

\usepackage{amsmath,amsthm,amssymb}
\usepackage{amscd}
\numberwithin{equation}{section}
\usepackage{times}
\usepackage{bbm}
\usepackage{cite}
\usepackage{enumitem}
\usepackage[left=1.12in,right=1.12in,top=1.12in,bottom=1.12in]{geometry} 
\usepackage{calc}
\usepackage{ifthen,tabularx,graphicx,multirow}
\graphicspath{ %
	{Figures/} %
	{Figures/Zoom} %
}
\usepackage{rotating}
\usepackage{hyperref}
\usepackage{tikz}
\usepackage{csquotes}
\usepackage{mathrsfs}
\usepackage{dsfont}
\usepackage[utf8]{inputenc}

\usepackage{tikz-cd}
\usetikzlibrary{tikzmark}
\usepackage{ etoolbox}
\usepackage{tcolorbox}
\usepackage{enumitem}
\usepackage{scrextend}

\usepackage{algorithm2e}
\newcolumntype{Y}{>{\centering\arraybackslash}X}
\newtheorem{lemma}{Lemma}
\newtheorem{proposition}[lemma]{Proposition}
\newtheorem{theorem}[lemma]{Theorem}
\numberwithin{lemma}{section}
\newtheorem{corollary}[lemma]{Corollary}

\theoremstyle{definition}
\newtheorem{definition}[lemma]{Definition}

\newcommand{\nh}{\text{NH}}

\newcommand{\strbdepsph}{\epsilon_{\mathbb{P}h\mathrm{B}}^{\mathrm{s}}(\mathrm{p})}
\newcommand{\strbdepsp}{\epsilon_{\mathbb{P}\mathrm{B}}^{\mathrm{s}}(\mathrm{p})}

\newcommand{\gprob}{\Gamma^{\mathrm{ran}}}

\newcommand{\pr}{\mathbb{P}}
\theoremstyle{definition}
\newtheorem{remark}[lemma]{Remark}

\theoremstyle{definition}

\theoremstyle{definition}

\theoremstyle{theorem}

\theoremstyle{theorem}
\newtheorem{problem}{Problem}

\theoremstyle{theorem}

\newcount\colveccount
\newcommand*\colvec[1]{
	\global\colveccount#1
	\begin{pmatrix}
		\colvecnext
	}
	\def\colvecnext#1{
		#1
		\global\advance\colveccount-1
		\ifnum\colveccount>0
		\\
		\expandafter\colvecnext
		\else
	\end{pmatrix}
	\fi
}
\makeatletter
\newcounter{framedeqn}
\AtBeginEnvironment{subequations}{\let\c@equation\c@framedeqn}
\makeatother

\newcommand{\SetNine}{\mathcal{S}^9}

\newcommand{\opBall}[2]{\mathcal{B}_{#1}{\left(#2\right) }}
\newcommand{\clBall}[2]{\overline{\mathcal{B}}_{#1}{\left(#2\right) }}

\newcommand{\xilp}{\Xi_{\mathrm{LP}}}

\newcommand{\xibp}{\Xi_{\mathrm{BP}}}
\newcommand{\xibpdn}{\Xi_{\mathrm{BPDN}}}

\newcommand{\xicl}{\Xi_{\mathrm{CL}}}
\newcommand{\xiul}{\Xi_{\mathrm{UL}}}

\newcommand{\xibptv}{\Xi_{\mathrm{BPTV}}}

\newcommand{\xiultv}{\Xi_{\mathrm{ULTV}}}

\newcommand{\complex}{\mathbb{C}}
\newcommand{\real}{\mathbb{R}}
\newcommand{\indic}{\chi}
\newcommand{\sgn}{\text{sgn}}
\newcommand{\argmin}{\mathop{\mathrm{arg\, min}}}

\newcommand{\supp}{\text{supp}}

\newcommand{\tv}[1]{\|#1\|_{\mathrm{TV}}}

\newcommand{\cond}[1]{\mathrm{cond}(#1)}

\newcommand{\probab}{\mathbb{P}}

\newcommand{\dalgo}{d}

\newcommand{\vecYCL}{y^{\mathrm{CL}}}
\newcommand{\vecYLP}{y^{\mathrm{A}}}
\newcommand{\vecYl}{y^{\mathrm{L}}}
\newcommand{\vecYTV}{y^{\mathrm{TV}}}
\newcommand{\matLP}{A}
\newcommand{\matl}{L}
\newcommand{\matTV}{T}

\newcommand{\orvec}{\mathscr{O}_{\mathrm{vec}}}
\newcommand{\ormat}{\mathscr{O}_{\mathrm{mat}}}
\newcommand{\orsol}{\mathscr{O}_{\mathrm{sol}}}

\newcommand{\flipOp}{P_{\text{flip}}}

\newcommand{\albetSet}{\mathcal{L}}
\newcommand{\albetSetBPTV}[1]{\mathcal{L}^{\text{BP,TV},#1}}
\newcommand{\albetSetULTV}[1]{\mathcal{L}^{\text{UL,TV},#1}}
\newcommand{\albetSetEFBPTV}[1]{\mathcal{L}^{\text{BP,TV,E},#1}}
\newcommand{\albetSetEFULTV}[1]{\mathcal{L}^{\text{UL,TV,E},#1}}
\newcommand{\dist}{\mathrm{dist}}
\newcommand{\disM}{\mathrm{dist}_{\mathcal{M}}}

\newcommand{\actv}[1]{{\left(#1\right)}^{\text{act}}}

\newcommand{\omlp}{\Omega^{\mathrm{LP}}}

\newcommand{\ombpl}{\Omega^{\mathrm{BP},\ell^1}}
\newcommand{\omcll}{\Omega^{\mathrm{CL},\ell^1}}
\newcommand{\omull}{\Omega^{\mathrm{UL},\ell^1}}

\newcommand{\ombptv}{\Omega^{\mathrm{BP},\mathrm{TV}}}
\newcommand{\omultv}{\Omega^{\mathrm{UL},\mathrm{TV}}}

\newcommand{\omlpef}{\Omega^{\mathrm{LP},\mathrm{E}}}

\newcommand{\ombplef}{\Omega^{\mathrm{BP},\ell^1,\mathrm{E}}}
\newcommand{\omcllef}{\Omega^{\mathrm{CL},\ell^1,\mathrm{E}}}
\newcommand{\omullef}{\Omega^{\mathrm{UL},\ell^1,\mathrm{E}}}

\newcommand{\ombptvef}{\Omega^{\mathrm{BP},\mathrm{TV},\mathrm{E}}}
\newcommand{\omultvef}{\Omega^{\mathrm{UL},\mathrm{TV},\mathrm{E}}}

\newcommand{\omlpb}[1][K]{\Omega^{\mathrm{LP}}_{m,N,#1}}
\newcommand{\ombplb}[1][K]{\Omega^{\mathrm{BP},\ell^1}_{m,N,#1}}
\newcommand{\omcllb}[1][K]{\Omega^{\mathrm{CL},\ell^1}_{m,N,#1}}
\newcommand{\omullb}[1][K]{\Omega^{\mathrm{UL},\ell^1}_{m,N,#1}}
\newcommand{\ombptvb}[1][K]{\Omega^{\mathrm{BP},\mathrm{TV}}_{m,N,#1}}
\newcommand{\omultvb}[1][K]{\Omega^{\mathrm{UL},\mathrm{TV}}_{m,N,#1}}

\newcommand{\omlps}[1][K]{\Omega^{\mathrm{LP},\mathrm{s}}_{m,N,#1}}

\newcommand{\ombpls}[1][K]{\Omega^{\mathrm{BP},\ell^1,\mathrm{s}}_{m,N,#1}}
\newcommand{\omclls}[1][K]{\Omega^{\mathrm{CL},\ell^1,\mathrm{s}}_{m,N,#1}}
\newcommand{\omulls}[1][K]{\Omega^{\mathrm{UL},\ell^1,\mathrm{s}}_{m,N,#1}}

\newcommand{\ombptvs}[1][K]{\Omega^{\mathrm{BP},\mathrm{TV},\mathrm{s}}_{m,N,#1}}
\newcommand{\omultvs}[1][K]{\Omega^{\mathrm{UL},\mathrm{TV},\mathrm{s}}_{m,N,#1}}
\newcommand{\omlpw}[1][K]{\Omega^{\mathrm{LP},\mathrm{w}}_{m,N,#1}}

\newcommand{\ombplw}[1][K]{\Omega^{\mathrm{BP},\ell^1,\mathrm{w}}_{m,N,#1}}
\newcommand{\omcllw}[1][K]{\Omega^{\mathrm{CL},\ell^1,\mathrm{w}}_{m,N,#1}}
\newcommand{\omullw}[1][K]{\Omega^{\mathrm{UL},\ell^1,\mathrm{w}}_{m,N,#1}}

\newcommand{\ombptvw}[1][K]{\Omega^{\mathrm{BP},\mathrm{TV},\mathrm{w}}_{m,N,#1}}
\newcommand{\omultvw}[1][K]{\Omega^{\mathrm{UL},\mathrm{TV},\mathrm{w}}_{m,N,#1}}

\newcommand{\omlpos}{\Omega^{\mathrm{LP},\mathrm{s}}_{m,N}}
\newcommand{\omlpow}{\Omega^{\mathrm{LP},\mathrm{w}}_{m,N}}
\newcommand{\vecYMainLP}[1]{y^{\mathrm{LP},#1}}
\newcommand{\vecYMainBP}[1]{y^{\mathrm{BP},\ell^1,#1}}
\newcommand{\vecYMainUL}[1]{y^{\mathrm{UL},\ell^1,#1}}
\newcommand{\vecYMainCL}[1]{y^{\mathrm{CL},#1}}
\newcommand{\vecYMainBPTV}[1]{y^{\mathrm{BP},\mathrm{TV},#1}}
\newcommand{\vecYMainULTV}[1]{y^{\mathrm{UL},\mathrm{TV},#1}}

\newcommand{\matLPBPExit}{A^{\mathrm{E}}}
\newcommand{\matLPObj}{A^{\mathrm{LP}, \mathrm{D}}}
\newcommand{\oh}{\mathcal{O}}

\newcommand{\dyadic}{\mathbb{D}}
\newcommand{\length}{\mathrm{Len}}
\newcommand{\ones}{\mathbf{1}}
\newcommand{\CompactK}{\mathcal{K}}
\newcommand{\DataTur}{\mathrm{Data}^{\mathrm{Tur}}}
\newcommand{\DataBSS}{\mathrm{Data}^{\mathrm{BSS}}}
\newcommand{\Alphabet}{\mathcal{A}}
\newcommand{\InputVecsBSS}{\mathcal{V}}
\newcommand{\InfeasibleInputs}{\Sigma^{\mathrm{FP}}}
\newcommand{\MultiMinInputs}{\Sigma^{\mathrm{RCC}}}
\usepackage{mathtools}
\DeclarePairedDelimiter\floor{\lfloor}{\rfloor}
\DeclarePairedDelimiter\ceil{\lceil}{\rceil}

\begin{document}
\title[Computational barriers in estimation, regularisation and learning]{The extended Smale's 9th problem \\
---\\ On computational barriers and paradoxes in estimation,  regularisation, computer-assisted proofs and learning}

\author{A. Bastounis} 
\address{School of Mathematics, University of Edinburgh}
\email{abastoun@ed.ac.uk}

\author{A. C. Hansen} 
\address{Department of Applied Mathematics and Theoretical Physics, University of Cambridge}
\email{ach70@cam.ac.uk}

\author{V. Vla\v{c}i\'{c}} 
\address{D-ITET, ETH Zürich}
\email{vlacicv@mins.ee.ethz.ch}

\maketitle

\section{Introduction}
Linear and semidefinite programming (LP, SDP), regularisation through basis pursuit (BP) and Lasso have seen great success in mathematics, statistics, data science, computer-assisted proofs and learning. The success and performance of LP is traditionally attributed to the fact that it is polynomially solvable (colloquially, ``in P'') for rational inputs. On the other hand, in his list of problems for the 21st century \cite{21century_Smale} S. Smale calls for ``[Computational] models which process approximate inputs and which permit round-off computations''.  Indeed, since e.g. $\sqrt{\cdot}$ and $\exp(\cdot)$ do not have exact representations, inaccurate data input is a daily encounter. The relevance of such a model is further emphasised by the fact that even a rational number such as $1/3$ is stored approximately in base-2 when using floating-point arithmetic, a situation faced by most software. This model allowing inaccurate input of arbitrary precision, which we call \emph{the extended model}, leads to extended versions of fundamental problems such as: ``Are LP and other aforementioned problems in P?'' The same question can be asked of an extended version of Smale's 9th problem \cite{21century_Smale} on the list of mathematical problems for the 21st century.  Recall, Smale's 9th problem reads:
\begin{displayquote}
	\normalsize
	{\it Is there a polynomial time algorithm over the real numbers which decides the feasibility of the linear system of inequalities $Ax \geq y$, and if so, outputs such an $x$?}
\end{displayquote}
One can thus pose this problem in the extended model where $A$ and $y$ are given as inexact inputs with arbitrary precision (see for example L. Lov\'{a}sz \cite[p. 34]{lovasz1987algorithmic}).
Similarly, the optimisation problems BP, SDP and Lasso, where the task is to output a solution to a specified precision, can likewise be posed in the extended model.
Given the widespread use of randomised algorithms, one can then ask questions on the existence of randomised algorithms providing an approximate solution with a certain probability. We will collectively refer to these problems as the \emph{extended Smale's 9th problem} (see Problem \ref{Extended_Smale} for the precise formulation), which we will consider in both the Turing \cite{Turing_Machine} and the Blum-Shub-Smale (BSS) \cite{BSS_machine} model for real arithmetic.

We settle this problem in both the negative and the positive, revealing two surprises: (1) In mathematics, sparse regularisation, statistics, and learning, one successfully computes with non-computable functions. The same happens also in computer-assisted proofs, for example in the proof of Kepler's conjecture (Hilbert's 18th problem) \cite{Hales1, Hales2}. (2) In order to mathematically characterise this phenomenon, one needs an intricate complexity theory for, seemingly paradoxically, non-computable functions. 

\vspace{2mm}

{\bf Main results (The extended Smale's 9th problem).} Short summary of Theorems \ref{Cor:main}, \ref{thm:ExitFlag}, \ref{thm:Smales9}, \ref{th:smale_comp_sens}: {\it Consider the task of computing a minimiser of  LP, BP, or Lasso in the extended model, and choose any $\ell^p$-norm to measure the error. Then, for any integer $K>2$, there exists a class of feasible inputs $\Omega$ such that, simultaneously, we have the following.
\begin{itemize}[leftmargin=8mm]
	\item[(i)]
	No algorithm, even randomised, can produce $K$ correct digits of the true solution for all inputs in $\Omega$ (with probability exceeding $\mathrm{p} > 1/2$ in the randomised case).
	\item[(ii)]
	If we allow randomised algorithms with non-zero probability of not halting (i.e., not producing an output), then no such algorithm can produce $K$ correct digits for all inputs in $\Omega$ with probability exceeding $\mathrm{p} > 2/3$. However, there does exist such an algorithm that produces $K$ correct digits for all inputs in $\Omega$ with probability $2/3$.
	\item[(iii)]
	One cannot decide if a given algorithm taking inputs from $\Omega$ fails to produce $K$ correct digits on a given input (with probability exceeding $\mathrm{p} > 1/2$ in the randomised case).  This is in fact strictly harder than solving the original problem in the following sense. Even if given an oracle for solving, e.g., LP accurately, one cannot decide if the algorithm for solving LP successfully produces $K$ correct digits of the true solution. 
	\item[(iv)]
	There does exist an algorithm that provides $K-1$ correct digits for all inputs in $\Omega$. However, any such algorithm will need an arbitrarily long time to achieve this. Specifically, there is an $\Omega^{\prime} \subset \Omega$ with inputs of fixed dimensions, such that, for any $T > 0$ and any algorithm $\Gamma$, there exists an input $\iota \in \Omega^{\prime}$ so that either $\Gamma(\iota)$ does not approximate the true solution with $K-1$ correct digits or the runtime of $\Gamma$ on $\iota$ exceeds  $T$. Moreover, for any randomised algorithm $\Gamma^{\mathrm{ran}}$ and $\mathrm{p} \in (0,1/2)$, there exists an input $\iota \in \Omega^{\prime}$ such that     
	\begin{equation*}
		\begin{split}
			&\mathbb{P}\big(\gprob(\iota) \text{ does not approximate the true solution with } K-1 \text{ correct digits } \\
			& \hspace{4cm} \text{ or the runtime of } \Gamma \text{ on } \iota \text{ exceeds }  T \big) > \mathrm{p}.
		\end{split}
	\end{equation*}
	\item[(v)]
	The problem of producing $K-2$ correct digits for all inputs in $\Omega$ is in P, i.e., can be solved in polynomial time in the number of variables $n$.
	\item[(vi)] If one only considers (i) - (iv), $\Omega$ can be chosen with any fixed dimensions $m < N$ with $m \geq 4$ (see \S \ref{sec:comp_prob} for the precise formulation of the problems). Moreover, if one only considers (i) - (iii), then $K$ can be chosen to be 1.
	\item[(vii)]
	For BP problems with $\delta$ as in \eqref{problems3}, satisfying the robust nullspace property (which is satisfied with high probability in many cases in the sciences) there is an approximation threshold $\epsilon_0(\delta) > 0$. That is, the problem of computing an $\epsilon$-approximation of a minimiser is in P for $\epsilon > \epsilon_0 $ and $\notin$ P for $\epsilon < \epsilon_0$, regardless of P vs NP and
\begin{equation}\label{eq:approx_trhold}
\frac{1}{2}\delta \leq \epsilon_0(\delta) \lesssim \delta.
\end{equation} 
\item[(viii)] Similar results to the above hold for the extended LP feasibility decision problem, see Theorem \ref{thm:Smales9}. 
\end{itemize}
}
\vspace{2mm}

In the above, we use the unqualified term \emph{algorithm} to mean an algorithm that always provides an output, i.e., always halts.
More precise and elaborate versions of these claims will follow in Theorem \ref{Cor:main}, Theorem \ref{thm:ExitFlag}, Theorem \ref{thm:Smales9}, and Theorem \ref{th:smale_comp_sens}.

\begin{remark}[{\bf Condition and failure of modern algorithms}]
	It may come as a surprise that the above results are independent of the many condition numbers in the literature (see \S \ref{condition}). Indeed, bounded condition numbers generally do not imply the existence of successful algorithms, and, on the other hand, there are problem classes in P with infinite condition numbers, see Theorem \ref{th:smale_comp_sens} and Remark \ref{rem:inf_cond}. Moreover, the input can be bounded from above and below. The reader is invited to consult \S \ref{sec:failure} for a demonstration of the failure of modern software consistent with the above results. 
\end{remark}

\noindent The extended Smale's 9th problem has many implications that can be summarised briefly as follows.

{\bf \emph{Computer-assisted proofs and non-computable problems:}} The recent computer-assisted proof \cite{Hales1, Hales2} of Kepler's conjecture/Hilbert's 18th problem (see also \cite{Lagarias2011, Lagarias2011_intro}) was possible despite relying on computing with non-computable problems. This is a consequence of the extended Smale's 9th problem for deciding feasibility as shown in Theorem \ref{thm:Smales9} in \S \ref{sec:Smale_Kepler}.  A crucial part of the proof of  Kepler's conjecture consists of the numerical computation of LPs with inexact input, as in the extended model described above. In view of (i) - (iii) above, this may seem paradoxical. However, as shown in \S \ref{sec:Smale_Kepler}, non-computable problems can indeed be used in computer-assisted proofs, the Dirac-Schwinger conjecture \cite{fefferman1990, fefferman1992, fefferman1993aperiodicity,  fefferman1994, fefferman1994_2, fefferman1995, fefferman1996interval, fefferman1996, fefferman1997} serving as another example.  Our result in Theorem \ref{thm:ExitFlag} -- described in (iii) above on the limits of computing the exit flag -- can be interpreted as limitations on proof checkers.   

{\bf \emph{New phase-transitions in optimisation (regardless of P vs NP):}} The problem of computing $\epsilon$-approximations (see \S \ref{sec:hardness}) to the objective function of NP-hard optimisation problems often leads to phase transitions at the approximation threshold $\epsilon_{\mathrm{A}} > 0$ (see \eqref{eq:appr_thr} and note that in some cases $\epsilon_{\mathrm{A}}$ depends on the size of the problem). Indeed, assuming that P $\neq$ NP we often have the following:
\vspace{1mm}
\begin{equation}\label{eq:phase-transition_hardness}
\begin{tabular}{c}
			{\bf \emph{Classical phase}} \\
			{\bf \emph{transition in hardness}} \\
			{\bf \emph{of approximation}}
	\end{tabular} 
	\quad
\begin{tikzcd}[column sep = 2.5cm]
		 \fbox{ \begin{tabular}{c}
			Computing 
			\\ $\epsilon$-approx $\in$ P
		\end{tabular} }
 \arrow[r, shift left, "\epsilon_A > \epsilon "{name=U, above}]
		\arrow[leftarrow,r, shift right, "\epsilon_A < \epsilon  "{name=D,below}]
		&  \fbox{ \begin{tabular}{c}
			Computing $\epsilon$-approx\\ is NP-Hard (thus $\notin$ P)
		\end{tabular} }
\end{tikzcd}
\end{equation}
\vspace{1mm}
The fact that $\epsilon_{\mathrm{A}} > 0$ often follows from the PCP theorem \cite{Sudan, Arora_JACM_98_2, Lovasz_JACM_96}, for overviews see \cite{Arora2007} and \cite{Sudan_Overview_2009} and references therein.
The extended Smale's 9th problem leads to similar -- yet more complex -- phase transitions for the problem of computing $\epsilon$-approximations to minimisers in the extended model for classical combinatorial optimisation problems such as LP and problems in continuous optimisation such as BP. This phenomenon is characterised by the \emph{strong breakdown-epsilon} $\epsilon_{\mathrm{B}}^{\mathrm{s}}$ and the \emph{weak breakdown-epsilon} $\epsilon_{\mathrm{B}}^{\mathrm{w}}$ (Definition \ref{def:Breakdown-epsilons_strong} and Definition \ref{def:Breakdown-epsilons}), yielding phase transitions in several directions for LP (the computational cost is measured as a function of the number of variables): 
\begin{equation}\label{eq:phase}
\begin{tikzcd}[column sep=0mm, row sep=0.9cm]
\begin{tabular}{c}
			{\bf \emph{New phase}}  \\
			{\bf \emph{transitions due to the}}\\
			 {\bf \emph{extended Smale's 9th}}\\
	\end{tabular}
	&\fbox{\begin{tabular}{c}
			Computing $\epsilon$-approx $\notin$ k-EXPTIME $\forall \, k$\\
			 but $\in$ R (computable)
			 \end{tabular} }
		 \arrow[dl, shift left, "\epsilon_{\mathrm{B}}^{\mathrm{w}} < \epsilon "{name=U, below,sloped}]
\arrow[leftarrow,dl, shift right, "\epsilon_{\mathrm{B}}^{\mathrm{w}}> \epsilon > \epsilon_{\mathrm{B}}^{\mathrm{s}}"{name=D,anchor=center,sloped,above}] 
\arrow[dr, shift left, "\epsilon_{\mathrm{B}}^{\mathrm{s}} > \epsilon "{name=U, above,sloped}]
\arrow[leftarrow,dr, shift right, "\epsilon_{\mathrm{B}}^{\mathrm{s}}< \epsilon < \epsilon_{\mathrm{B}}^{\mathrm{w}}"{name=D,below,sloped}] & \\
	 \fbox{ \begin{tabular}{c}
			Computing 
			\\ $\epsilon$-approx $\in$ P
		\end{tabular} }
\arrow[rr, shift left, "\epsilon_{\mathrm{B}}^{\mathrm{s}}=\epsilon_{\mathrm{B}}^{\mathrm{w}} > \epsilon "{name=U, above}]
\arrow[leftarrow,rr, shift right, "\epsilon_{\mathrm{B}}^{\mathrm{s}} = \epsilon_{\mathrm{B}}^{\mathrm{w}}< \epsilon"{name=D,below}] & & 
\fbox{ \begin{tabular}{c}
			Computing\\ $\epsilon$-approx $\notin$ R
			\\  (non-computable)
		\end{tabular} }
\end{tikzcd}
\end{equation}

The extended Smale's 9th problem is related to many areas of optimisation such as complexity theory, approximation algorithms for hard problems, fundamental limits in continuous optimisation and robust optimisation \cite{Nesterov1, Nesterov_Nemirovski_Acta, Nemirovski1995, Nemirovski07, Nesterov2, Nesterov2018, Nemirovski_robust, Nemirovski_robust2}. It leads to phase transitions as in \eqref{eq:phase} for the problem of computing $\epsilon$-approximations
to minimisers in continuous optimisation and in combinatorial optimisation. The phase transitions are independent of the P vs NP question. 

Our statements above using integers ($K$, $K-1$, $K-2$) can be viewed as `quantised' phase transition thresholds. In particular, we consider the integers $\lceil|\log(\epsilon_{\mathrm{B}}^{\mathrm{w}})|\rceil$ and $\lceil|\log(\epsilon_{\mathrm{B}}^{\mathrm{s}})|\rceil$, but one can easily state our main results with the actual breakdown-epsilons describing the 'unquantised' phase transition threshold as in \eqref{eq:phase}. 
The phase transition between P and non-computable problems (the lower part of \eqref{eq:phase}) is not a direct consequence of our stated theorem, but can easily be established from our constructions used in the proofs with simple modifications. The new results suggest a classification program on which problems in continuous and combinatorial optimisation will have such phase transitions in the extended model. 

To address the extended Smale's 9th problem in the probabilistic setting we develop a theory for randomised algorithms through the \emph{probabilistic strong breakdown epsilon} $\epsilon_{\mathbb{P}\mathrm{B}}^{\mathrm{s}}$ and the \emph{probabilistic weak breakdown epsilon} $\epsilon_{\mathbb{P}\mathrm{B}}^{\mathrm{w}}$ (Definition \ref{prob_strong_break} and Definition \ref{prob_weak_break}). This yields phase transitions similar to \eqref{eq:phase} in the probabilistic sense for both continuous and combinatorial optimisation problems. See \eqref{eq:prob_phase} for an example of a phase transition diagram in the probabilistic setting.  

{\bf \emph{The extended Smale's 9th in practice -- Non-computability is not rare:}} 
The phenomenon illustrated in our summary of the main results above is not due to exotic situations that will never occur in practice. Indeed, in \S \ref{sec:not_rare} we demonstrate how standard algorithms will fail on basic problems drawn from standard probability distributions. In fact, our proof techniques reveal much more than the statements of our theorems, and can be used to explain exactly why we get the failures documented in \S \ref{sec:not_rare}. Moreover, in Theorem \ref{th:smale_comp_sens} we 
characterise the approximation threshold \eqref{eq:approx_trhold} given the standard assumptions required in the sciences, and it is not zero. However, the approximation threshold \eqref{eq:approx_trhold} matches the classical recovery error bounds -- up to a constant -- explaining the success of the use of convex optimisation in the sciences despite the non-computability phenomenon. 
The first results leading to upper bounds on the approximation threshold \eqref{eq:approx_trhold} were done in \cite{NemirovskiLMCO}.

\tableofcontents

\subsection{The computational problems -- Finding minimisers}\label{sec:comp_prob}
Finding minimisers for linear and semidefinite programming, regularisation techniques such as basis pursuit, Lasso etc. has become a main focus over the last decades. These approaches have in many areas of mathematics, statistics, learning and data science changed the state of the art from linear to non-linear approaches, typically via obtaining minimisers of convex problems \cite{Adcock2016, Juditsky_2011, Juditsky_2012, Nesterov_Nemirovski_Acta, candesCSMag, CandesRombergTao, CohenDahmenDeVore, Chambolle_Alg, Chambolle_Lions, Chambolle_2011, Chan, donohoCS, Donoho_BP, Gabriel1, Gabriel2, TibshiraniLasso, Tibshirani_Book, Osher_ROF, Osher_JAMS, Becker_2011, Mario}. The list of areas using these techniques is far reaching and their influence has been extensive. The key problems to compute are:
\begin{itemize}
	\item[(i)] Linear Programming (LP)
	\vspace{-1mm}
	\begin{equation}\label{problems}
		z \in \mathop{\mathrm{arg min}}_{x} \langle x , c \rangle \text{ subject to } Ax=y, \quad x \geq 0,
	\end{equation}
	\vspace{-5mm}
	\item[(ii)] Basis Pursuit (BP)
	\vspace{-1mm}
	\begin{equation}\label{problems3}
		z \in \mathop{\mathrm{arg min}}_{x} \mathcal{J}(x) \text{ subject to } \|Ax - y\|_2 \leq \delta, \qquad \delta \in [ 0,1],
	\end{equation}
	\vspace{-5mm}
	\item[(iii)] Unconstrained Lasso (UL)
	\vspace{-1mm}
	\begin{equation}\label{problems4}
		z \in  \mathop{\mathrm{arg min}}_{x} \|Ax-y\|^2_2 + \lambda\,  \mathcal{J}(x) , \qquad \qquad \quad \, \, \, \, \lambda \in (0,1],
	\end{equation}
	\vspace{-5mm}
	\item[(iv)]  Constrained Lasso (CL)
	\vspace{-1mm}
	\begin{equation}\label{problems5}
		z \in  \mathop{\mathrm{arg min}}_{x} \|Ax-y\|_2 \text{ subject to } \|x\|_1 \leq \tau, \quad  \, \, \, \tau > 0,
	\end{equation}
	\vspace{-5mm}
	\item[(v)] Semidefinite Programming (SDP)
	\vspace{-1mm}
	\begin{equation}\label{problems2}
		Z \in \mathop{\mathrm{arg min}} _{X \in \mathbb{S}^n}\langle C,X\rangle_{\mathbb{S}^n} \text{ subject to } \langle A_{k},X\rangle_{\mathbb  {S}^n} = b_{k}, \, X \succeq 0, \, k =  1, \hdots, m.
	\end{equation}
	\vspace{-5mm}
\end{itemize}
In the above notation we have
\[
A \in \mathbb{R}^{m \times N}, y \in \mathbb{R}^m, c \in \mathbb{R}^N, \quad \mathcal{J}(x) = \|x\|_1 \text{ or } \mathcal{J}(x) = \|x\|_{\mathrm{TV}},
\]
where the TV semi-norm is defined as $\|x\|_{\mathrm{TV}}=\sum_{j=1}^{N-1}|x_j-x_{j+1}|$.
For SDP, the notation is
\[
C, A_k \in \mathbb{S}^n \text{ (real $n \times n$ symmetric matrices)}, \quad b_k \in \mathbb{R},  \quad \langle C,X\rangle_{\mathbb{S}^n} = \mathrm{trace}(C^TX).
\]

Note that all of the problems above may have multi-valued solutions in certain cases. Whenever this occurs, the computational problem of interest is to compute any of these solutions. 
We will throughout the paper use the notation 
\begin{equation}\label{eq:the_Xi}
	\Xi: \Omega \rightrightarrows  \mathcal{M},
\end{equation}
to denote the multivalued solution map, mapping an input $\iota \in \Omega$ to a metric space $(\mathcal{M},d_{\mathcal{M}})$, allowing measurement of error.  The metric space is typically $\mathbb{R}^N$ or $\mathbb{C}^N$ equipped with the $\|\cdot\|_2$ norm, however, any metric can be considered.  Even though the solution map $\Xi$ may be multivalued, in our theory the output of an algorithm will always be single-valued.
	Thus, if $\Gamma: \Omega \rightarrow \mathcal{M}$ is an algorithm we measure the approximation error by 
	\[
	\disM(\Gamma(\iota), \Xi(\iota)) = \inf_{\xi \in \Xi(\iota)} d_{\mathcal{M}}(\Gamma(\iota),\xi).
	\]

\begin{remark}[{\bf Objective function vs minimisers}]
	In this paper we are primarily concerned with the problem of obtaining minimisers that are vectors and not the real-valued minimum value of the objective function. There is a very rich literature \cite{Boyd, Condition, Nesterov1, Nesterov2, Nemirovski_Book2001, Wright2, Lagarias_Wright, Wrig97, NoceWrig06} on how to compute the objective function, and, in particular, the minimum value
	$f(x^*) = \min \{f(x)\, \vert \, x \in \mathcal{X}\},$ for some convex function $f:\mathbb{R}^d \rightarrow \mathbb{R}$, convex set $\mathcal{X} \subset \mathbb{R}^d$, and minimiser $x^* \in \mathcal{X}$. The traditional problem of interest is as follows. Given $\epsilon > 0$, compute an $x_{\epsilon} \in \mathbb{R}^d$ such that $f(x_{\epsilon}) - f(x^*) \leq \epsilon$. Note that  $f(x_{\epsilon}) - f(x^*) \leq \epsilon$ does not necessarily mean that  
	\begin{equation}\label{eq:the_bound}
		\|x_{\epsilon} - x^*\| \leq \epsilon.
	\end{equation}
	In this paper, however, the problem of computing $x_{\epsilon}$ satisfying \eqref{eq:the_bound} is the main focus. 
	The motivation behind this is self-evident as there are vast areas of mathematics of information, regularisation, estimation, learning, compressed sensing and data sciences where the object of interest is the minimiser and not the minimum value. 
	\end{remark}

\section{The extended model and the extended Smale's 9th problem}
The question: ``is LP in P?'' \cite{khachiyan1980polynomial, gacs1981khachiyan, lawler1980great} was a fundamental problem whose solution, proven by L. Khachiyan -- based on work by N. Shor, D. Yudin,  A. Nemirovski -- reached the front page of The New York Times \cite{Lovasz_book}. The affirmative answer has been refined several times and is now typically stated in the following form. One can solve LPs with rational inputs in runtime is bounded by
\begin{equation}\label{eq:Karm_bound}
	\mathcal{O}(n^{{3.5}}L^{2}\cdot \log L\cdot \log \log L),
\end{equation}
where $n$ denotes the number of variables and $L$ is the number of bits or digits required in the representation of the inputs \cite{karmarkar1984new, renegar1988polynomial}. The problem, however, is that in an overwhelming number of problems in computational mathematics and scientific computing the input contains irrational numbers. This leads to the following basic question:
\begin{displayquote}
	\normalsize
	{\it Given a class of LPs that contain irrational numbers which can be computed in polynomial time, what is the computational cost of computing a $K$-digit accurate approximate minimiser? Is that problem in P (solvable in polynomial time in the number of variables $n$)?}
\end{displayquote} 
Note that the estimate \eqref{eq:Karm_bound} will not answer this question as $L = \infty$ for an irrational number. 

\subsection{Inexact input and the extended model} 
An example of an LP where the matrix $A$ contains irrational numbers is when its rows derive from the discrete cosine transform or a discrete wavelet transform. Note that these are not contrived examples. In fact, in the fields of inverse problems, medical imaging, compressed sensing, etc. this is a common occurrence.  Since $A$ contains numbers that cannot be represented exactly as binary numbers, the bound \eqref{eq:Karm_bound} does not apply.
Therefore, the classical model for asking ``is LP in P?'' does not address the common case of irrepresentable input. On the practical side, an overwhelming amount of the modern software used is based on floating-point arithmetic, and hence if the input is rational, there will be inexactness due to the floating-point representation. For example, $1/3$ can only be approximated in base $2$, giving rise to round-off approximation. Indeed, the following quote explains the situation succinctly:
\begin{displayquote}
	\normalsize
	{\it ``But real number computations and algorithms which work only in exact arithmetic can offer only limited understanding. Models which process approximate inputs and which permit round-off computations are called for.''}
	\\[5pt]
	\rightline{ --- S. Smale (from the list of mathematical problems for the 21st century \cite{21century_Smale}) \hspace{15mm}}
\end{displayquote}

This issue illustrates the classical dichotomy in mathematics between the discrete and continuous. Indeed, as mentioned above, classical complexity analysis for LP is done in a discrete model. Yet, LP can also be naturally considered in the continuous world, which is the standard way in which Smale's 9th problem is stated. The following quote draws attention to this dichotomy:
\begin{displayquote}
	\normalsize
	{\it `` Perhaps the most successful tool in economics and
		operations research is linear programming, which lives on the boundary of discrete and continuous.''}
	\\[5pt]
	\rightline{ --- L. Lov\'{a}sz (from``Discrete and Continuous:
		Two sides of the same?'' \hspace{20mm}}  \rightline{in ``Visions in Mathematics'', essays on mathematics entering the 21st century \cite{Lovasz2010}) \hspace{20mm}}
\end{displayquote}
These issues call for an extension of the continuous model that accommodates inexact input, discussed next.

\subsubsection{The extended model - inexact input provided by an oracle}\label{rem:extended_model}
Suppose that we are given an algorithm (a Turing or Blum-Shub-Smale (BSS) machine) intended to solve LP (or any of the other problems in \S \ref{sec:comp_prob}), and furthermore assume that the algorithm is equipped with an oracle $\mathscr{O}$ that can acquire the true input to any accuracy $\epsilon$. A natural assumption in this scenario is that the oracle completes its task in time polynomial in $|\log(\epsilon)|$ (see for example 
Lov\'{a}sz \cite[p. 36]{lovasz1987algorithmic}). More concretely, given a domain $\Omega \subset \mathbb{C}^n$ of inputs, the algorithm cannot access $\iota \in \Omega$, but rather, for any $k \in \mathbb{N}$, it can call the oracle $\mathscr{O}$ to obtain $\tilde \iota = \mathscr{O}(\iota,k) \in \mathbb{C}^n$ satisfying
\begin{equation}\label{eq:oracle22}
	\|\mathscr{O}(\iota,k)-\iota\|_{\infty} \leq 2^{-k}, \qquad \forall\, \iota \in \Omega, \, \forall k\in\mathbb{N},
\end{equation} 
and the time cost of accessing $\mathscr{O}(\iota,k)$ is polynomial in $k$. Another key assumption when discussing the success of the algorithm is that it must be ``oracle agnostic'', i.e., it must work with any choice of the oracle $\mathscr{O}$ satisfying \eqref{eq:oracle22}. In the Turing model the Turing machine accesses the oracle via an oracle tape and in the BSS model the BSS machine accesses the oracle through an oracle node. Note that the extended computational model of having inexact input can be found in many areas of the mathematical literature, and we mention only a small subset here \cite{bishop1967foundations, bravermancook2006computing, Cucker_Smale97, Fefferman_Klartag, Fefferman_Klartag2, Ko1991ComplexityTO, lovasz1987algorithmic}.

\subsubsection{The extended Smale's 9th problem}

Given that the input is inexact, the output of an algorithm will come with an error as well.
The model, both in the Turing and the BSS case, where one measures the computational cost of running the algorithm in terms of the number of variables $n$ and the error (or the number of correct digits $K = |\log(\epsilon)|$, where $\epsilon$ is the error) is well established. See, for example \cite[p. 29]{BCSS}, \cite[p. 34]{ Lovasz_book}  and \cite[p. 131]{Valiant_book}). 
We thus arrive at the following extension of Smale's 9th problem.

\begin{problem}[{\bf The extended Smale's 9th problem}]\label{Extended_Smale}
	Given any of the problems in \eqref{problems} - \eqref{problems5}, represented by the solution map $\Xi$ mapping a class of inputs $\Omega$ into a metric space $(\mathcal{M},d_{\mathcal{M}})$, is there an algorithm which decides the feasibility of the problem, and if so, produces an output that is correct up to $K$ digits (where the error is measured via $\disM$) and whose computational cost is bounded by a polynomial in $K$ and the number of variables $n$?
\end{problem}

This question can be asked both in the Turing model, where the computational cost can be expressed either in terms of the number of steps performed by the Turing machine, or alternatively in terms of the total number of arithmetic operations and comparisons as well as the space complexity. In the BSS model, the computational cost is given by the total number of arithmetic operations and comparisons executed by the BSS machine. We will consider all these cases.

\begin{remark}[Weaker and randomised versions of the extended Smale's 9th problem]
	The question in Problem \ref{Extended_Smale} can be weakened by asking if, for a fixed $K$, there is an algorithm -- with polynomial runtime in the number of variables -- that produces an output that is correct up to $K$ digits. One can also weaken the statement by allowing randomised algorithms and asking whether an algorithm succeeds with probability $\mathrm{p} \in [0,1]$. In the case of feasibility questions one can weaken the question by relaxing the constraints to only be satisfied up to a certain accuracy. For example, given $K \in \mathbb{N} \cup \{\infty\}$ and $ M \in \real$, one may ask to decide whether there is an $x \in \mathbb{R}^N$ such that
	\begin{equation*}
			\langle x , c \rangle_K \leq M \text{ subject to } Ax = y, \quad x \geq 0,
	\end{equation*}
	where 
	$
	\langle x , c \rangle_K = \lfloor  10^{K} \langle x , c \rangle \rfloor  10^{-K}.
	$
	We will discuss this particular problem later in \S \ref{sec:Kepler} in connection with the computer-assisted proof of Kepler's conjecture. 
\end{remark}

\section{Main Theorem I (Part a): The extended Smale's 9th -- Computing solutions}
The main results on the existence of successful polynomial cost algorithms for the extended Smale's 9th problem are summarised in Theorem \ref{Cor:main}, Theorem \ref{thm:Smales9}, and Theorem \ref{th:smale_comp_sens}, whereas Theorem \ref{thm:ExitFlag} deals with the decision problem of certifying the correctness of an algorithm (the ``exit flag'' problem). We now present each of these theorems.

\subsection{Universality of the results}\label{sec:complexity-for-noncomp}
The statements in the theorems below are well-defined up to the definition of an algorithm, randomised algorithm and runtime. There are a myriad of different types of machines that can be used to model an algorithm: the Turing machine \cite{Turing_Machine} (and its cousins including the Markov model \cite{MarkovModel}), the BSS machine \cite{BSS_machine}, the von Neumann architecture  \cite{von_Neumann}, the real RAM \cite{realRAM}, etc. as well as their randomised versions (see Remark \ref{rem:model_of_comp}). The different models are not equivalent when it comes to computability and runtime. Thus, to create universal impossibility results we use general algorithms (defined in \S \ref{sec:SCI_hierarchy}) and a randomised general algorithms (defined in \S \ref{sec:randomised}) that encompass any reasonable definition of a computational model in the way that they are more powerful than any standard machine, therefore making the impossibility results stronger.  Some of the stated results are slightly weaker than what we actually prove. The fully formal statements of the theorems can be found in the propositions in \S\ref{Sec:background} further below, and references to these follow each theorem.   

The first paradoxical result demonstrates the following intricate phenomenon. Key problems used in statistical estimation, sparse regularisation, compressed sensing, learning, and modern data science require a complexity theory for non-computable functions. The term ``non-computable'' here refers to the classical definition given by Turing in \cite{Turing_Machine} (there is an algorithm that for any $\epsilon > 0$ produces an $\epsilon$-approximation).
In the following $\Xi$ will denote the solution map (as in \eqref{eq:the_Xi}) to any of the problems \eqref{problems} - \eqref{problems5} with set of inputs $\Omega$, so that 
\begin{equation}\label{eq:the_Omega}
	\Omega = \bigcup_{N>m\geq 4}^{\infty} \Omega_{m,N},  \quad \Xi:  \Omega_{m,N} \rightrightarrows \mathcal{M}_{N},
\end{equation}
where $\Omega_{m,N}$ is a nonempty set of inputs for $\Xi$ of fixed dimensions $m$ and $N$, and $\mathcal{M}_N$ is $\mathbb{R}^{N}$ equipped with the $\|\cdot\|_p$ norm for some $p \in [1,\infty]$. The standard doubled arrow notation $\Xi:  \Omega_{m,N} \rightrightarrows \mathcal{M}_{N}$ indicates that the mapping $\Xi$ is multi-valued. The fact that the dimensions blow up is crucial in order to make sense of ``in P''-type statements. 

\begin{remark}[Computing $K$ correct digits]
	Having specified the $p$-norm $\|\cdot\|_p$ for measuring the error in $\mathcal{M}_N$, we frequently discuss whether an algorithm can ``produce (or compute) $K$ correct digits for $\Xi$''. For an algorithm having access to the dimensions $m$, $N$ and an oracle representation $\tilde\iota$ (according to \ref{eq:oracle22}) of an input $\iota\in \Omega_{m,N}$, this will mean the assertion that
	\begin{equation}\label{eq:getting_K}
		\disM(\Gamma(m,N,\tilde{\iota}), \Xi(\iota)) = \inf_{\xi \in \Xi(\iota)} \|\Gamma(m,N,\tilde\iota)- \xi\|_p \leq 10^{-K},
	\end{equation}
	for all $m$, $N$ and all possible oracle representations $\tilde\iota$ of all $\iota\in \Omega_{m,N}$. Moreover, in the randomised case, \eqref{eq:getting_K} should happen with a certain probability that will always be made explicit. 
\end{remark}

\begin{remark}[Condition numbers]
In the following we refer to several condition numbers common in the literature, namely the condition of a matrix, the feasibility-primal condition number $C_{\mathrm{FP}}$, and the condition of a solution map. The precise definitions of these can be found in \S\ref{condition}.
\end{remark}

\begin{theorem}[The extended Smale's 9th problem - computing solutions]\label{Cor:main}
	Let $\Xi$ denote the solution map to any of the problems \eqref{problems} - \eqref{problems5} with the regularisation parameters satisfying $\delta\in[0,1]$, $\lambda\in(0,1/3]$, and $\tau\in[1/2,2]$ (and additionally being rational in the Turing case) and consider the $\|\cdot\|_p$-norm for measuring the error, for an arbitrary $p\in[1,\infty]$. 
	Let $K > 2$ be an integer.
	There exists a class $\Omega$ of feasible inputs as in \eqref{eq:the_Omega} so that, simultaneously, we have the following. 
	\begin{itemize}[leftmargin=8mm]
		\item[(i)]
		No algorithm can produce $K$ correct digits on each input in $\Omega$ as in \eqref{eq:getting_K}. Moreover, for any $\mathrm{p} > \frac{1}{2}$, no randomised algorithm can produce $K$ correct digits with probability greater than or equal to $\mathrm{p}$ on each input in $\Omega$.  
		\item[(ii)] If we allow randomised algorithms with a non-zero probability of not halting (not producing an output), then, for any $\mathrm{p} > \frac{2}{3}$, no such algorithm can produce $K$ correct digits with probability greater than or equal to $\mathrm{p}$ on each input in $\Omega$. However, there does exist such an algorithm that can produce $K$ correct digits on each input in $\Omega$ with probability $2/3$.
		\item[(iii)]
		There does exist an algorithm (a Turing or a BSS machine) that produces $K-1$ correct digits for all inputs in $\Omega$. However, any such algorithm will need an arbitrarily long time to achieve this. In particular, for any fixed dimensions $m$, $N$, any $T > 0$, and any algorithm $\Gamma$, there exists an input $\iota \in \Omega_{m,N}$ such that either $\Gamma$ on input $\iota$ does not produce $K-1$ correct digits for $\Xi(\iota)$ or the runtime of $\Gamma$ on $\iota$ exceeds $T$. Moreover, for any randomised algorithm $\Gamma^{\mathrm{ran}}$ and $\mathrm{p} <1/2$ there exists an input $\iota \in \Omega_{m,N}$ such that     
		\begin{equation*}
			\begin{split}
				&\mathbb{P}\big(\gprob(\iota) \text{ does not produce $K-1$ correct digits for } \Xi(\iota)  \\
				& \hspace{5cm} \text{ or the runtime of } \Gamma \text{ on } \iota \text{ exceeds } T \big) > \mathrm{p}.
			\end{split}
		\end{equation*}
		\item[(iv)]
		There exists a polynomial $\mathrm{pol}: \mathbb{R} \rightarrow \mathbb{R}$, as well as a Turing machine and a BSS machine that both produce $K-2$ correct digits for all inputs in $\Omega$, so that the number of arithmetic operations for both machines is bounded by $\mathrm{pol}(n)$, where $n=m+mN$ is the number of variables, and the number of digits required from the oracle \eqref{eq:oracle22} is bounded by $\mathrm{pol}(\log(n))$. Moreover, the space complexity of the Turing machine is bounded by $\mathrm{pol}(n)$. 
		\item[(v)] If one only considers (i) - (iii), $\Omega$ can be chosen with any fixed dimensions $m$ and $N$ provided that $m\geq 4$ and $N > m$. Moreover, if one only considers (i) then $K$ can be chosen to be 1.
	\end{itemize}
	The statements (i) - (iii) above are true even when we require the input in each $\Omega_{m,N}$ to be well-conditioned and bounded from above. In particular, for any input $\iota = (y,A) \in \Omega_{m,N}$ ($\iota = (y,A,c)$ in the case of LP) we have 
	$\mathrm{Cond}(AA^*)\leq 3.2$, $C_{\mathrm{FP}}(\iota)\leq 4$, $\mathrm{Cond}(\Xi) \leq 179$, $\|y\|_\infty\leq 2$, and $\|A\|_{\max} = 1$. 
\end{theorem}

The precise and slightly stronger version of Theorem \ref{Cor:main} is summarised in Proposition \ref{Cor:main_SCI} and Proposition \ref{Cor:main_SCI_cont}.  
\begin{remark}[The proof techniques of Theorem \ref{Cor:main}]
	The techniques developed in this paper to prove Theorem \ref{Cor:main} can be used and extended to produce computability and complexity results in other fields. In particular, they form the basis of some of the developments in \cite{NNInstab} and \cite{NNInstab2} on the limitations of AI and Smale's 18th problem.
	
\end{remark}
\begin{remark}[Models of computation -- Universal lower bounds]\label{rem:model_of_comp}
	As previously mentioned, our results and proof techniques are universal across all reasonable models of computation. In particular, the impossibility results in Theorem \ref{Cor:main} hold in the Markov model based on the Markov algorithm \cite{MarkovModel} -- i.e. when the inexact input is required to be computable. The technical justification of this claim is given in Remark \ref{rem:TurVsMarkov} and Remark \ref{remark:CompDelta1Markov}. 
\end{remark}

\begin{remark}[Lower bound on the vector component of the inputs]
Note that the matrix component of every input $\iota=(y,A)\in\Omega$ in Theorem \ref{Cor:main} is bounded from both below and above with universal bounds, specifically, $\|A\|_{\max}:=\max_{i,j}|A_{i,j}|=1$. The vector component $y$ admits an upper bound but does not admit a lower bound uniform across $\Omega$. This minor deficiency is an artefact of our construction of $\Omega$ and could be easily remedied by carrying out a more involved construction. This, however, would significantly increase the technical detail of the proof while not contributing to the theory nor any of the techniques developed, so we refrain from doing so.
\end{remark}

\begin{remark}[Theorem \ref{Cor:main} and semidefinite programming (SDP)]
The results of Theorem \ref{Cor:main} also hold for the SDP problem \eqref{problems2}, which can be shown easily by employing the standard argument to recast the LP problem (together with the associated class $\Omega$) as an SDP. With this embedding the impossibility results hold immediately, however the parts referring to the existence of algorithms would need to be proven separately. 
\end{remark}

\subsection{Consequence of Theorem \ref{Cor:main} -- New phase transitions}\label{sec:hardness}
A crucial problem in combinatorial optimisation is the question of hardness of approximation \cite{Arora_JACM_98, Arora_JACM_98_2, Lovasz_JACM_96, Arora2007}. In particular, we follow \cite{Papadimitriou1994} and suppose we are given an optimisation problem depending on $n$ variables. Thus, for each instance $z \in \Omega \subset \mathbb{R}^n$, where $\Omega$ is some domain of inputs, we have a set of feasible solutions, call it $F(z)$. Moreover, the goal is to minimise an objective function/cost function $f_z :\mathbb{R}^d \rightarrow \mathbb{R}$ for some $d \in \mathbb{N}$ and 
\[
\mathrm{OPT}(z) := \min_{x\in F(z)} f_z(x).
\]
Many optimisation problems may be NP-hard. However, for any $\epsilon > 0$ one can ask if there exists an algorithm $\Gamma$ such that for any instance $z \in \Omega$, $\Gamma(z) \in F(z)$ and 
\begin{equation}\label{eq:eps_alg}
f_z(\Gamma(z)) \leq (1+\epsilon) \mathrm{OPT}(z), \qquad \mathrm{Runtime}(\Gamma(z)) \leq \mathrm{pol}(n),
\end{equation} 
for some polynomial $\mathrm{pol}: \mathbb{R} \rightarrow \mathbb{R}$.  If the optimisation problem is a maximisation problem, the inequality will be reversed as well as the sign in front of $\epsilon$. 
The output $\Gamma(z)$ satisfying \eqref{eq:eps_alg} is referred to as an \emph{$\epsilon$-approximate solution}. Hardness of approximation is traditionally investigated in the Turing model, thus the word algorithm here means a Turing machine, and $\mathrm{Runtime}(\Gamma(z))$ is the runtime of the Turing machine $\Gamma$ given input $z$. One can also ask if there is a randomised algorithm $\gprob$ such that 
\begin{equation}\label{eq:eps_alg_rand}
\mathbb{P}\Big(f_z(\gprob(z)) \leq (1+\epsilon) \mathrm{OPT}(z)\Big) \geq 2/3, \qquad \mathrm{Runtime}(\gprob(z)) \leq \mathrm{pol}(n).
\end{equation} 

 An algorithm $\Gamma$, as described above -- such that \eqref{eq:eps_alg} is satisfied -- is called an \emph{$\epsilon$-approximation algorithm}, and a randomised algorithm $\gprob$ such that \eqref{eq:eps_alg_rand} is satisfied is called a \emph{randomised $\epsilon$-approximation algorithm}. 
Given an optimisation problem one defines the \emph{approximation threshold}, as defined in \cite{Papadimitriou1994}
\begin{equation}\label{eq:appr_thr}
		\epsilon_{\mathrm{A}} := \inf\{\epsilon \geq 0 \, \vert \, \text{there exists an $\epsilon$-approximation algorithm}\}, 	
	\end{equation}
and the \emph{probabilistic approximation threshold}
\[
		\epsilon_{\mathbb{P}\mathrm{A}} := \inf\{\epsilon \geq 0 \, \vert \, \text{there exists a randomised $\epsilon$-approximation algorithm}\}.
\]
For several NP-hard optimisation problems it follows that $\epsilon_{\mathbb{P}\mathrm{A}} =0$ and $\epsilon_{\mathrm{A}} = 0$. Indeed, in several cases there are \emph{polynomial-time approximation schemes} (PTAS) or \emph{polynomial-time randomized approximation schemes} (PRAS) \cite{Arora2007, Papadimitriou1994}. 
However, if P $\neq$ NP, there are also many examples where $\epsilon_{\mathrm{A}} > 0$, which implies a sharp \cite{Hastad_JACM, Hastad_Acta} phase transition, where the problem of computing an $\epsilon$-approximation is in P for $\epsilon > \epsilon_{\mathrm{A}} $ but $\notin P$ for $\epsilon < \epsilon_{\mathrm{A}}$ as visualised in the phase transition diagram \eqref{eq:phase-transition_hardness}. 

The techniques built to prove the probabilistic statements in Theorem \ref{Cor:main} can be used to establish phase transition diagrams for LP, BP, Lasso, etc. as in \eqref{eq:prob_phase} in the probabilistic case. The key are the
\emph{probabilistic strong breakdown epsilon} $\epsilon_{\mathbb{P}\mathrm{B}}^{\mathrm{s}}$ and the \emph{probabilistic weak breakdown epsilon} $\epsilon_{\mathbb{P}\mathrm{B}}^{\mathrm{w}}$ (Definition \ref{prob_strong_break} and Definition \ref{prob_weak_break}):
\begin{equation}\label{eq:prob_phase}
\begin{tikzcd}[column sep=0mm, row sep=1.3cm]
\begin{tabular}{c}
			{\bf \emph{New phase}}  \\
			{\bf \emph{transitions in the}}\\
			{\bf \emph{probabilistic cases}}\\
\end{tabular}
	&\fbox{\begin{tabular}{c}
			Computing $\epsilon$-approx $\notin$ BPP\\
			 but $\in$ R (computable)
			 \end{tabular} }
		 \arrow[dl, shift left, "\epsilon_{\mathbb{P}\mathrm{B}}^{\mathrm{w}} < \epsilon "{name=U, below,sloped}]
\arrow[leftarrow,dl, shift right, "\epsilon_{\mathbb{P}\mathrm{B}}^{\mathrm{w}}> \epsilon > \epsilon_{\mathbb{P}\mathrm{B}}^{\mathrm{s}}"{name=D,anchor=center,sloped,above}] 
\arrow[dr, shift left, "\epsilon_{\mathbb{P}\mathrm{B}}^{\mathrm{s}} > \epsilon "{name=U, above,sloped}]
\arrow[leftarrow,dr, shift right, "\epsilon_{\mathbb{P}\mathrm{B}}^{\mathrm{s}} < \epsilon < \epsilon_{\mathbb{P}\mathrm{B}}^{\mathrm{w}}"{name=D,below,sloped}] & \\
	 \fbox{ \begin{tabular}{c}
			Computing 
			\\ $\epsilon$-approx $\in$ P
		\end{tabular} }
\arrow[rr, shift left, "\epsilon_{\mathbb{P}\mathrm{B}}^{\mathrm{s}}=\epsilon_{\mathbb{P}\mathrm{B}}^{\mathrm{w}} > \epsilon "{name=U, above}]
\arrow[leftarrow,rr, shift right, "\epsilon_{\mathbb{P}\mathrm{B}}^{\mathrm{s}} = \epsilon_{\mathbb{P}\mathrm{B}}^{\mathrm{w}}< \epsilon"{name=D,below}] & & 
\fbox{ \begin{tabular}{c}
Computing $\epsilon$-approx $\notin$ R,\\
		$\forall \, \gprob, c > 1/2 \, \exists \, \iota \text{ s.t. }$\\
		$\mathbb{P}(\gprob(\iota) \text{ is not an $\epsilon$-approx}) < c$
		\end{tabular} }
\end{tikzcd}
\end{equation}

\section{Failure of modern algorithms and computing the exit flag}\label{sec:failure} 

\subsection{Computing the exit flag -- Can correctness of algorithms be certified?}
A crucial topic in computational mathematics is the reliability of algorithms and certification of their correctness.  It is therefore natural to test whether the built-in algorithms in, for example MATLAB are reliable. We consider two concrete examples: the linear program
\begin{equation}\label{eq:lin_prog43}
	\min_{x \in \mathbb{R}^2} x_1 + x_2\,\,\, \mbox{subject to \, $ x_1 + (1-\delta)x_2 = 1$, $\qquad x_1,x_2 \geq 0$,}
\end{equation}
where $\delta > 0$ is a parameter, and the centred and standardised (so that the columns of the design matrix are normalised) Lasso problem 
\begin{equation}\label{eq:lasso43}
	\min_{x \in \real^N} \frac{1}{m}\|A_{\delta}D_{\delta}x - y\|_2^2 + \lambda \|x\|_1,
\end{equation}
where 
$m =3, N=2$, $\lambda \in (0,1/\sqrt{3}]$, 
\begin{equation}\label{eq:the_matrix}
	A_\delta = \begin{pmatrix} \frac{1}{\sqrt{2}} - \delta & \frac{1}{\sqrt{2}} \\ -\frac{1}{\sqrt{2}} - \delta & -\frac{1}{\sqrt{2}}\\ 2\delta & 0 \end{pmatrix} \in \real^{3 \times 2}, \quad y = \begin{pmatrix} 1/\sqrt{2} & -1/\sqrt{2} & 0 \end{pmatrix}^T \in \real^3,
\end{equation}
and $D_{\delta}$ is the unique diagonal matrix such that each column of $A_{\delta}D_{\delta}$ has norm $\sqrt{m}$.

In order to compute a solution to \eqref{eq:lin_prog43}, we consider MATLAB's \texttt{linprog} command; a well-established optimisation solver for linear programs. This is a general purpose solver, which offers three different algorithms: `dual-simplex' (the default), `interior-point', and `interior-point-legacy'. Besides a minimiser, \texttt{linprog} also computes an additional output -- \texttt{EXITFLAG} -- which is an integer value corresponding to the reason for why the algorithm halted. Note that $+1$ indicates convergence to a minimiser, all other values indicate some form of failure. In Table \ref{tab:test_linprog} we apply  the three \texttt{linprog} algorithms (with default settings) to the problem \eqref{eq:lin_prog43} with different values of $\delta$. The results are fascinating. Not only does \texttt{linprog} completely fail to compute a minimiser accurately, it also fails to recognise that the computed minimiser is incorrect: in all cases, the \texttt{EXITFLAG} returns the value $+1$ indicating a successful termination. 

\begin{table}
\centering
	\begin{tabular}{@{}c| cc|cc|cc|@{} } 
		\cline{2-7}
		&
		\multicolumn{2}{c|}{`dual-simplex'} &
		\multicolumn{2}{c|}{`interior-point'} &
		\multicolumn{2}{c|}{`interior-point-legacy'} 
		\\
		$\delta$ &  Error\ & \texttt{EXITFLAG} & Error \ & \texttt{EXITFLAG} & Error\ & \texttt{EXITFLAG} \\
		\hline
		$2^{-1}$   & $0$ & $1$ & $0$ & $1$ & $6.0 \cdot 10^{-12}$ & $1$ \\
		$2^{-15}$  & $0$ & $1$ & $0$ & $1$ & $3.0 \cdot 10^{-5}$ & $1$ \\
		$2^{-20}$  & $0$ & $1$ & $0$ & $1$ & $7.0 \cdot 10^{-7}$ & $1$ \\
		$2^{-24}$  & $0$ & $1$ & $0$ & $1$ & $7.1 \cdot 10^{-8}$ & $1$ \\
		$2^{-26}$  & $1.4$ & $1$ & $1.4$ & $1$ & $1.2\cdot 10^{-1}$ & $1$ \\
		$2^{-28}$  & $1.4$ & $1$ & $1.4$ & $1$ & $4.6\cdot 10^{-1}$ & $1$ \\
		$2^{-30}$  & $1.4$ & $1$ & $1.4$ & $1$ & $7.1\cdot 10^{-1}$ & $1$ \\
		\hline
	\end{tabular}
	\vspace{2mm}
	\caption{Testing the output of \texttt{linprog} applied to the problem in \eqref{eq:lin_prog43} for the algorithms `dual-simplex', `interior-point' and `interior-point-legacy'. The table shows the error $\|\hat{x} - \tilde{x}\|_{\ell^2}$ and the value of \texttt{EXITFLAG} ($1$ means successful output), where $\hat{x}$ is the true minimiser of \eqref{eq:lin_prog43} and $\tilde{x}$ is the computed approximate minimiser. Note that machine epsilon is $\epsilon_{\mathrm{mach}} = 2^{-52}$.}
	\label{tab:test_linprog}
\end{table}

\begin{table}
	\begin{tabular}{@{}c |ccc|ccc|ccc|@{} } 
		\cline{2-10}
		&
		\multicolumn{3}{c|}{\qquad Default settings}  &
		\multicolumn{3}{c|}{`\texttt{RelTol}' $= \epsilon_{\mathrm{mach}}$} &
		\multicolumn{3}{c|}{`\texttt{RelTol}' $= \epsilon_{\mathrm{mach}}$} 
		\\
		&
		\multicolumn{3}{c|}{} &
		\multicolumn{3}{c|}{} &
		\multicolumn{3}{c|}{`\texttt{MaxIter}' $= \epsilon^{-1}_{\mathrm{mach}}$} 
		\\
		\hline
	$\delta$ & Error & Runtime & \texttt{Warn} & Error & Runtime & \texttt{Warn} & Error & Runtime & \texttt{Warn} \\
		\hline
		$2^{-1}$  & $1 \cdot 10^{-16}$  & $< 0.01s$ & $0$ & $1 \cdot 10^{-16}$  & $< 0.01s$ & $0$ & $1 \cdot 10^{-16}$ & $<0.01s$ & $0$ \\
		$2^{-7}$   & $0.68$ & $< 0.01s$ & $0$ & $2\cdot 10^{-16}$ & $0.02s$    & $0$ & $2\cdot 10^{-16}$ & $0.02s$ & $0$\\
		$2^{-15}$  &$1.17$ & $< 0.01s$ & $0$ & $1.17$ & $0.33s$    & $1$ & $1\cdot 10^{-11}$ & $1381.5s$ & $0$ \\
		$2^{-20}$  &$1.17$ & $< 0.01s$ & $0$ & $1.17$ & $0.33s$    & $1$ & $\text{no output}$ & $>12h$ & $0$ \\
		$2^{-24}$  &$1.17$ & $< 0.01s$ & $0$ & $1.17$ & $0.34s$    & $1$ & $\text{no output}$ & $>12h$ & $0$ \\
		$2^{-26}$  &$1.17$ & $< 0.01s$ & $0$ & $1.17$ & $0.34s$    & $1$ & $\text{no output}$ & $>12h$ & $0$ \\
		$2^{-28}$  &$1.17$ & $< 0.01s$ & $0$ & $1.17$ & $< 0.01s$    & $0$ & $1.17$ & $<0.01s$ & $0$ \\
		$2^{-30}$  &$1.17$ & $< 0.01s$ & $0$ & $1.17$ & $< 0.01s$    & $0$ & $1.17$ & $<0.01s$ & $0$ \\
		\hline
	\end{tabular}
	\vspace{2mm}
	\caption{The output of \texttt{lasso} applied to \eqref{eq:lasso43} with inputs as in \eqref{eq:the_matrix} and $\lambda = 0.1$. The table shows the error $\|\hat{x} - \tilde{x}\|_{\ell^2}$ (where $\hat{x}$ is the true minimiser and $\tilde{x}$ is the computed minimiser), the CPU runtime, and a boolean value indicating whether a \texttt{Warning} was issued.}
	\label{tab:test_lasso}
\end{table}

To compute a solution to \eqref{eq:lasso43}, we consider Matlab's \texttt{lasso} command. We test it with default settings as well as the tolerance parameter set to machine epsilon $\epsilon_{\mathrm{mach}} = 2^{-52}$ and also the maximum number of iterations to $\epsilon_{\mathrm{mach}}^{-1}$. The \texttt{lasso} routine does not have an `exit flag', however, it provides a \texttt{Warning} if it considers the output to be untrustworthy. The results of this experiment are summarised in Table \ref{tab:test_lasso}, where we display $1$ under the \texttt{Warn} column if a \texttt{Warning} was issued, or $0$ if no warning was issued. As is evident, the failure of \texttt{lasso} is similar to the failure of \texttt{linprog}, however, an interesting observation is that the \texttt{Warning} parameter is occasionally able to verify the wrong solution, yet, most of the time, no warning is issued despite completely inaccurate outputs.  

\begin{figure}
\begin{center}
    \setlength{\tabcolsep}{2pt}
    \resizebox{1\textwidth}{!}{
    \begin{tabular}{@{}>{\centering}m{0.5\textwidth}>{\centering\arraybackslash}m{0.5\textwidth}@{}}   
\texttt{spgl1} on basis pursuit with $\delta = 0$ & MATLAB's \texttt{lasso} on Lasso with $\lambda = 10^{-2}$\\
\includegraphics[width=1\linewidth]{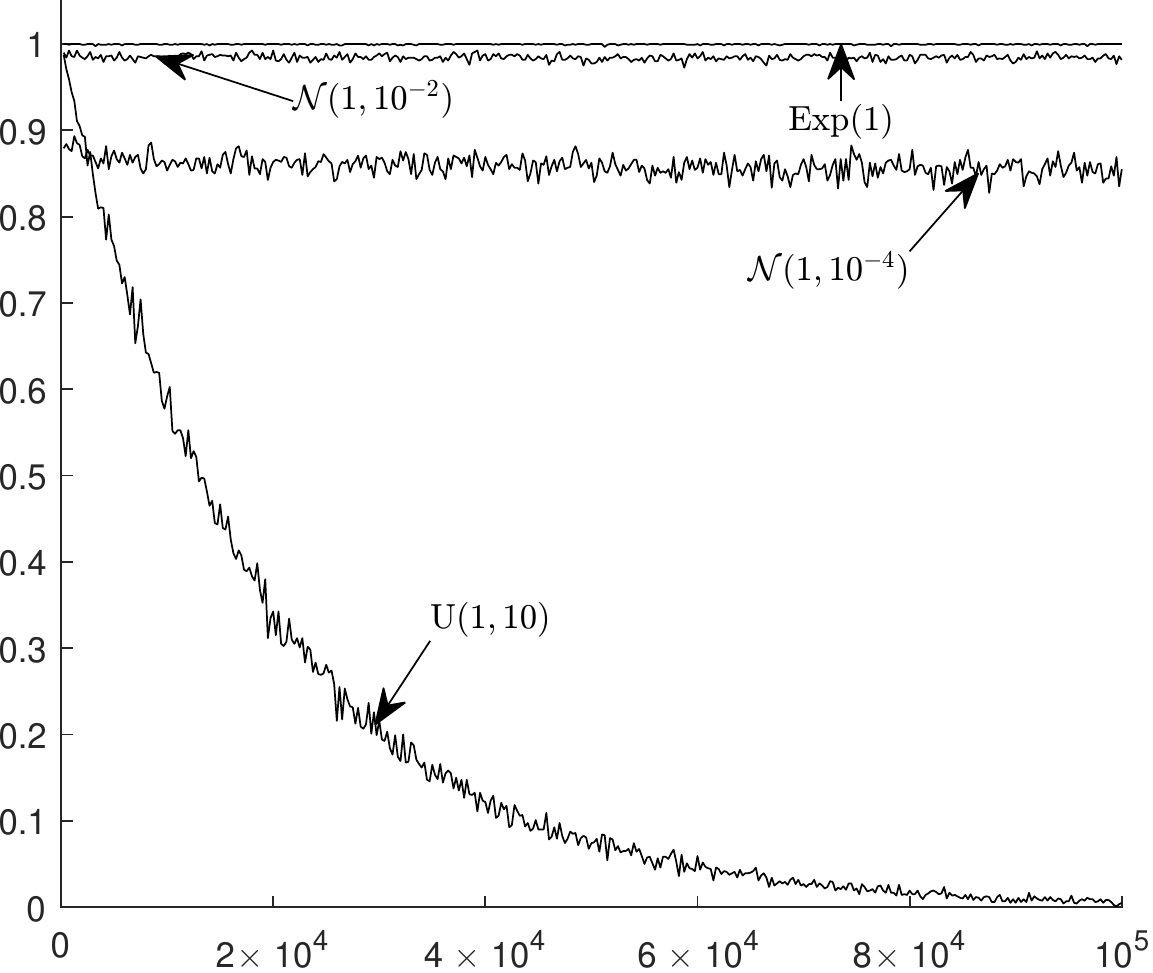} &
\includegraphics[width=1\linewidth]{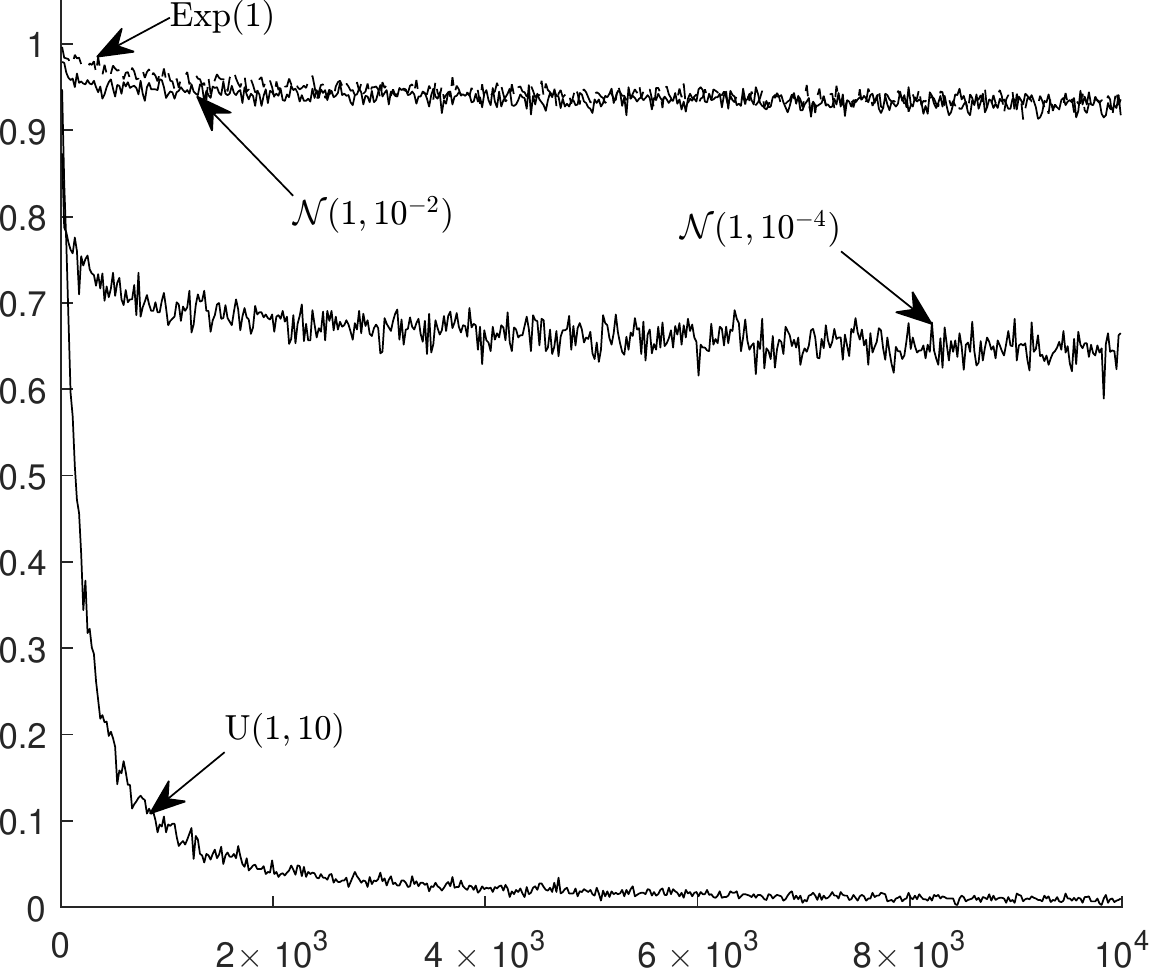}
\end{tabular}
}
\end{center}
\caption{({\bf Random matrices -- Non-computability is not rare}). The vertical axis represents the success rate $\frac{\# \text{ of successes }}{\# \text{ of trials }}$, where $\# \text{ of trials } = 1200$. Success $\Leftrightarrow$ computed solution is accurate to at least $K = 2$ digits ($\|\cdot\|_{\infty}$ norm). The horizontal axis shows the dimension $N$. In all cases, $A \in \mathbb{R}^{m \times N}$ in \eqref{problems3} and  \eqref{problems4} is iid -- as described in \S \ref{sec:not_rare} -- according to the distributions $\mathrm{U}(a,b)$, $\mathrm{Exp}(\nu)$ and $\mathcal{N}(\mu, \sigma^2)$, being the uniform distribution on $[a,b]$, the exponential distribution with parameter $\nu$ and the normal distribution with mean $\mu$ and variance $\sigma$. }  
\label{fig:rand_matrix}
\end{figure}

\subsection{Non-computability is not rare -- The proof techniques reveal much more}\label{sec:not_rare}
Theorem \ref{Cor:main} demonstrates that for any integer $K$, there are -- for all problems \eqref{problems} - \eqref{problems2} --  classes of inputs for which no algorithm can compute a correct $K$ digit approximate solution. This statement, as is typical for a result regarding non-computability, describes a worst case scenario. However, the proof techniques of our theorems reveal much more. Indeed, for random matrix ensembles, one can characterise the probability of failure of algorithms. This is because our proof of Theorem \ref{Cor:main} is constructive.

To be more precise, Figure \ref{fig:rand_matrix} displays experiments with well-established algorithms such as \texttt{spgl1} \cite{spgl1} and MATLAB's \texttt{lasso}. We have tested these algorithms on BP \eqref{problems3} with $\delta = 0$ and Lasso \eqref{problems4} with $\lambda = 10^{-2}$. In both cases, all accuracy parameters in the algorithms were set to machine precision $\epsilon_{\mathrm{mach}}$ in MATLAB, and the number of iterations in \texttt{spgl1} and \texttt{lasso} were set to $1000$ and the default parameter respectively. We executed these algorithms on inputs $A \in \mathbb{R}^{m \times N}$ and $y \in \mathbb{R}^{m}$, where the entries of $A$ are iid according to a distribution $\mathcal{D}$ and $y = Ae_i$ where $i \in \{1,\hdots, N\}$ is chosen uniformly at random. In particular, we examine the cases where $\mathcal{D}$ is a normal distribution $\mathcal{N}(\mu, \sigma^2)$ with mean $\mu$ and variance $\sigma$, $\mathcal{D}$ is a uniform distribution $\mathrm{U}(a,b)$ on the interval $[a,b]$, or $\mathcal{D}$ is an exponential distribution $\mathrm{Exp}(\nu)$ with parameter $\nu$.  Figure \ref{fig:rand_matrix} displays the results for $m = 1$ and varying $N$s, where we plot the `success rate' given by
\[
\text{ Success rate } = \frac{\# \text{ of successes }}{\# \text{ of trials }} \in [0,1],
\]
 as a function of the dimension $N$. In all cases $\# \text{ of trials } = 1200$. 
 Given a distribution $\mathcal{D}$, let $\iota_N = (y,A)$ be such that $y \in \real^{1}$ and $A \in \real^{1 \times N}$ are randomly chosen as described above. For any choice of algorithm $\Gamma$ that solve Basis Pursuit or Lasso implemented in floating point arithmetic and any $K \in \mathbb{N}$ define (when the limit exists)
\[
\lim_{N \rightarrow \infty}\mathbb{P}\big (\Gamma(\iota_N) \text{ provides $K$ correct digits}\big) =: P_{\infty}^{\Gamma}(\mathcal{D}).
\]

Our proof techniques allow one to show that, for example, when $\Gamma$ represents the \texttt{spgl1} algorithm we have 
\[
P_{\infty}^{\Gamma}(\mathcal{D}) = 0 \text{ for } \mathcal{D} = \mathrm{U}(a,b). 
\]
Note that this asymptotic behaviour is already visible in Figure \ref{fig:rand_matrix}.  
For the Gaussian and Exponential distributions, the issue is more complicated and there may be transient behaviour.  Indeed, let $\Xi$ denote the solution map to the Lasso problem in Figure \ref{fig:rand_matrix} (so that $\Xi$ outputs a solution to \eqref{problems4}) and assume that $\mathcal{D}$ is $\mathcal{N}(\mu,\sigma^2)$ or that $\mathcal{D}$ is $\mathrm{Exp}(\nu)$. Then $\mathbb{P}(\|\Xi(\iota_N)\|_{\infty} < 10^{-k}) \rightarrow 1$, as $N \rightarrow \infty$ for all $k \in \mathbb{N}$. Thus, an algorithm that always outputs zero will eventually become correct with high probability.   
However, in Figure \ref{fig:rand_matrix}, there is behaviour of the following form: there exists an $M$ so that 
\[
\mathbb{P}\big (\Gamma(\iota_N) \text{ provides $K$ correct digits} \, \big) > \mathbb{P}\big (\Gamma(\iota_{N+1}) \text{ provides $K$ correct digits} \, \big) \quad N \leq M,
\]
for some large $M \in \mathbb{N}$, yet for any algorithm $\Gamma$ -- such that the objective function applied to $\Gamma(\iota_N)$ is $\varepsilon$ (in this example $\varepsilon = 10^{-5}$ suffices) away from the true minimum -- we have $P^{\Gamma}_{\infty}(\mathcal{D}) = 1$. The latter is typically true for both \texttt{spgl1} and \texttt{lasso}, thus Figure \ref{fig:rand_matrix} demonstrates a transient behaviour for both the normal and exponential distributions. 
These phenomena can be mathematically analysed and proven by using the specific techniques used in the proof of Theorem \ref{Cor:main}, however, as this paper is already substantial in length, the proofs of these results will appear elsewhere.

\section{Main Theorem II: The extended Smale's 9th -- Computing the exit flag} 

The failures of the \texttt{EXITFLAG} and the \texttt{Warning} -- that were supposed to detect the incorrect solutions in the experiment in \S \ref{sec:failure} -- suggest the following question. 

\begin{problem}[Can the `exit flag' be computed?]
	Consider an algorithm designed to compute any of the problems \eqref{problems} - \eqref{problems5}. Suppose that the algorithm should produce $K$ correct digits.
	Can we compute the `exit flag' for this algorithm, i.e., the function taking on the value $1$ if the algorithm succeeds in producing $K$ correct digits, and $0$ else?
\end{problem}

\subsection{The exit flag cannot be computed - even when given an exact solution oracle} 
Let $\Xi$ denote the solution map (as in \eqref{eq:the_Xi}) to any of the problems 
\eqref{problems} - \eqref{problems5}, let $K \in \mathbb{N}$, and let $\Gamma$ be an algorithm approximating $\Xi$. Suppose that $\Omega$ is a collection of inputs such that no algorithm can produce $K$ correct digits for all inputs in $\Omega$ (recall that the existence of such an $\Omega$ is guaranteed by Theorem \ref{Cor:main}).
The exit flag problem is as follows. Is it is possible to decide whether $\Gamma$ fails or succeeds on a particular input? More precisely, can we find another algorithm that can compute the exit flag, meaning the function of the input taking on the value $1$ when $\Gamma$ succeeds in producing $K$ correct digits for a particular input and $0$ when it fails?

\subsubsection{Exact solution oracle}
To illustrate the difficulty of computing the exit flag, we add the assumption that we can also access the true solution, and show that even with this extra information the task is impossible.  We thus assume for a fixed (but arbitrarily small) parameter $\omega>0$, the algorithm can access a $\rho\in\opBall{\omega}{\Xi(\iota)}$, i.e., an element of $\mathcal{M}$ that is strictly within $\omega$ of a true solution. We will consider an algorithm with access to such an oracle to be successful only if it works for any such $\rho$. Throughout the paper we will use the wording ``with access to an exact solution oracle of precision $\omega$'' to mean the situation just described, and this concept will be formally treated in \S \ref{sec:exit-flag-formal}. 
To avoid trivial examples we need another assumption. For instance, if $\Gamma$ outputs values that do not belong to the range of the solution map $\Xi$, then one can easily construct examples where it becomes trivial to check if $\Gamma$ fails. Concretely, if $\Gamma$ produces correct output when $\Gamma( \iota) \in \Xi( \Omega)$ and wrong output when $\Gamma( \iota)=\xi \notin \Xi(  \Omega)$, the exit flag is easy to compute: it is simply 0 when the output is $\xi$ and $1$ otherwise. A natural assumption precluding this situation could thus be that $\Gamma( \iota) \in \Xi( \Omega)$, for all $ \iota \in \Omega$. This, however, would again create issues in the Turing model if $\Xi(\Omega)$ contained elements with irrational entries, and hence we introduce another precision parameter $\alpha>0$ and make the following assumption:
\begin{equation}\label{eq:introEFassum}
	\disM(\Gamma( \iota), \Xi( \Omega)) < \alpha\;\text{ for all  }  \iota \in \Omega.
\end{equation}
We are now ready to state the following result which shows that computing the exit flag is typically harder than the original problem.

\begin{theorem}[Impossibility of computing the exit flag]\label{thm:ExitFlag}
	Let $\Xi$ denote the solution map to any of the problems \eqref{problems} - \eqref{problems5} with the regularisation parameters satisfying $\delta\in[0,1]$, $\lambda\in(0,1/3]$, and $\tau\in[1/2,2]$ (and additionally being rational in the Turing case) and consider the $\|\cdot\|_p$-norm for measuring the error, for an arbitrary $p\in[1,\infty]$. Let $K \in \mathbb{N}$ and fix real $ \alpha$ and  $\omega$ so that $0<\alpha\leq \omega<10^{-K}$. Then, for any fixed dimensions $ N>m\geq 4$, there exists a class of inputs $\Omega$ for $\Xi$ such that, if $\Gamma$ is an algorithm satisfying \eqref{eq:introEFassum} with parameter $\alpha$ for the computational problem of approximating $\Xi$ with $K$ correct digits, then we have the following.
	\begin{itemize}[leftmargin=8mm]
		\item[(i)]
		No algorithm, even randomised with access to an exact solution oracle of precision $\omega$, can compute the exit flag of $\Gamma$ (with probability exceeding $\mathrm{p} > 1/2$ in the randomised case).
		\item[(ii)] If we allow randomised algorithms with non-zero probability of not halting (producing an output), then no such algorithm, even with access to an exact solution oracle of precision $\omega$, can compute the exit flag of $\Gamma$ with probability exceeding $\mathrm{p} > 1/2$. 
		\item[(iii)]
		The problem of computing the exit flag of $\Gamma$ is strictly harder than computing $K$ correct digits of $\Xi$ in the following sense: if one is given the exit flag as an oracle then it is possible to construct an algorithm that computes $K$ correct digits of $\Xi$. However, if one is instead given an oracle providing a $K$-digit approximation to $\Xi$, then it is still not possible to compute the exit flag of $\Gamma$.
		\item[(iv)] 
		For linear programming and basis pursuit, however, there exists a class of inputs $\Omega^{\sharp} \neq \Omega$ such that no algorithm, even randomised with non-zero probability of not halting, can compute the exit flag of $\Gamma$ (with probability exceeding $\mathrm{p} > 1/2$ in the randomised case), yet one can compute the exit flag with a deterministic algorithm with access to an exact solution oracle of precision $\omega$. 
	\end{itemize}
	The statements (i) and (ii) above are true even when we require the input in each $\Omega$ to be well-conditioned and bounded from above. In particular, for any input $\iota = (y,A) \in \Omega_{m,N}$ ($\iota = (y,A,c)$ in the case of LP) we have 
	$\mathrm{Cond}(AA^*)\leq 3.2$, $C_{\mathrm{FP}}(\iota)\leq 4$, $\mathrm{Cond}(\Xi) \leq 179$, $\|y\|_\infty\leq 2$, and $\|A\|_{\max} = 1$. 
\end{theorem}

\noindent The precise version of Theorem \ref{thm:ExitFlag} can be found in Proposition \ref{prop:ExitFlag}.

\begin{remark}[Halting vs non-halting algorithms]\label{remark:EFPartiWeakerThanPartii}
	Note that (ii) in Theorem \ref{thm:ExitFlag} directly implies (i) and is hence a stronger statement. However, we deliberately put both (i) and (ii) in the statement of the theorem to allow comparison with Theorem \ref{Cor:main}. Indeed, in Theorem \ref{Cor:main} there is a difference in the lower bound for the probability between the case of a halting algorithm vs. a non-halting algorithm, which is not the case in Theorem \ref{thm:ExitFlag}.
\end{remark}

\section{Main Theorem III: The extended Smale's 9th -- Feasibility and Kepler's conjecture}\label{sec:Smale_Kepler}

\subsection{Smale's 9th problem and the computational proof of Kepler's conjecture}\label{sec:Kepler}
There is an important link between Smale's 9th problem and the recent proof of Kepler's conjecture \cite{Hales1, Hales2}. Kepler's conjecture, originally stated in 1611\cite{kepler1966strena}, hypothesised an optimality result for sphere packing in three dimensions. Continuing for nearly two decades, the program led by T. Hales \cite{Hales1, Hales2}  concluded with a computer-assisted proof (the Flyspeck program) verifying Kepler's conjecture. The link to Smale's 9th problem is that Hales' program hinges on the numerical verification of more than fifty thousand decision problems of the following form: Given $K \in \mathbb{N}$ and $M \in \real$, does there exist an $x \in \mathbb{R}^N$  such that 
\begin{equation}\label{Dual_LP_extended}
	\langle x , c \rangle_k \leq M \text{ subject to } Ax = y, \quad x \geq 0,
\end{equation}
where 
$
\langle x , c \rangle_K = \lfloor  10^{K} \langle x , c \rangle \rfloor  10^{-K}
$? What makes this numerical verification delicate is that the inputs $(y,A,c)$ are also approximated through numerical computation, and hence we find ourselves in the extended model. 
Informally, we could think of $\langle x , c \rangle_K$ as $\langle x , c \rangle$ computed to $K$ correct digits (in the Flyspeck program \eqref{Dual_LP_extended} is computed with $K = 6$). 
The key question is now, is this problem in general decidable? 

We will consider input sets of the form
\begin{equation}\label{eq:def_Omega_decision}
		\Omega = \bigcup_{m<N } \Omega_{m,N}\text{ so that }  \quad \Xi_K:  \Omega_{m,N} \rightarrow \{0,1\}= \{\mathrm{no}, \mathrm{yes}\},
\end{equation}
where, for each $m$ and $N$, $\Omega_{m,N}$ is a non-empty set of inputs of fixed dimensions $m$ and $N$ for the solution map $\Xi_K$ of \eqref{Dual_LP_extended}. Again, as in \S\ref{sec:complexity-for-noncomp}, we need to have inputs of arbitrarily large dimensions in order to be able to state ``in P'' results.
We make the dependence of the solution map on $K$ explicit, as we will consider \eqref{Dual_LP_extended} on the same input set and with the same $M$, but with different values of $K$.

Before discussing decidability, we note the following intricacy behind the choice of $K$. Suppose that the $K$-th digit behind the decimal point is 9. Then $\{x\in\real^N \,\vert\, \langle x , c \rangle_K\leq M\} = \{x\in\real^N \,\vert\, \langle x , c \rangle_{K-1}\leq M\}$, and hence $\Xi_{K}=\Xi_{K-1}$, i.e., the decision problem \eqref{Dual_LP_extended} is indistinguishable for $K$ and $K-1$. Now, writing  $\SetNine_k$ for the set of positive real numbers whose $k$-th digit after the decimal point is $9$, we have the following theorem demonstrating the intricacy of the Smale's 9th problem in the extended model. This theorem shows that Smale's 9th problem in the extended model actually becomes a classification theory and testifies to the intricacy of the computational proof of Kepler's conjecture.

\begin{theorem}[The extended Smale's 9th problem - deciding feasibility]\label{thm:Smales9}
Let  $M\geq 0$ be real and $K  > 2$ an integer, and consider the decision problem \eqref{Dual_LP_extended} for this value of $M$ and the corresponding solution maps $\Xi_K$, $\Xi_{K-1}$, and $\Xi_{K-2}$. Then there is a class of inputs $\Omega$ as in \eqref{eq:def_Omega_decision} such that we have the following:
	\begin{itemize}[leftmargin=8mm]
		\item[(i)] There does \emph{not} exist any sequence of algorithms $\{\Gamma_j\}_{j\in \mathbb{N}}$ with output in $\{0,1\}$ so that, for all $\iota \in \Omega$, we have $\lim_{j\to\infty}\Gamma_j(\iota) = \Xi_K(\iota)$ and, if $\Gamma_j(\iota) = 1$ for some $j \in \mathbb{N}$, then $\Xi_K(\iota)=\Gamma_j(\iota) =1$. 
		\item[(ii)] There does \emph{not} exist a randomised algorithm $\Gamma^{\mathrm{ran}}$, even with non-zero probability of not halting, such that, for all $\iota \in \Omega$, we have $\Gamma^{\mathrm{ran}}(\iota) = \Xi_K(\iota)$ with probability exceeding $\mathrm{p} > 1/2$. 
\end{itemize}
Additionally, depending on whether $M$ is in one (or both) of $\SetNine_{K}$ or $\SetNine_{K-1}$, one or both of the following hold as well:
\begin{itemize}[leftmargin=8mm]
		\item[(iii)] There does exist an algorithm $\Gamma$ (Turing or BSS machine) such that, for all $ \iota \in \Omega$, we have $\Gamma(\iota) = \Xi_{K-1}(\iota)$. However, any algorithm will need arbitrarily long time to compute $\Xi_{K-1}$. In particular, for any fixed dimensions $m$ and $N$, any $T > 0$, and any algorithm $\Gamma$, there exists an input $\iota \in \Omega_{m,N}$ such that either $\Gamma(\iota) \neq \Xi_{K-1}(\iota)$ or the runtime of $\Gamma$ on $\iota$ exceeds $T$. Moreover, for any randomised algorithm $\Gamma^{\mathrm{ran}}$ and $\mathrm{p} \in (0,1/2)$, there exists an input $\iota \in \Omega_{m,N}$ such that     
		\[
		\mathbb{P}\left(\gprob(\iota) \text{ is wrong }  \text{or the runtime exceeds } T \right) >\mathrm{p}.
		\]
		\item[(iv)] 
		There exists a polynomial $\mathrm{pol}: \mathbb{R} \rightarrow \mathbb{R}$ as well as a Turing machine and a BSS machine that both compute $\Xi_{K-2}$ for all inputs in $\Omega$, so that the number of arithmetic operations for both machines is bounded by $\mathrm{pol}(n)$, where $n=m+mN$ is the number of variables, and the number of digits required from the oracle \eqref{eq:oracle22} is bounded by $\mathrm{pol}(\log(n))$. Moreover, the space complexity of the Turing machine is bounded by $\mathrm{pol}(n)$.
	\end{itemize}
	Concretely, (iii) holds whenever $M\notin\SetNine_{K}$, whereas (iv) holds whenever $M\notin\SetNine_{K-1}$.
	Finally, if one only considers (i) - (iii), $\Omega$ can be chosen with any fixed dimensions $m < N$ with $N \geq 4$. Moreover, if one only considers (i) then $K$ can be chosen to be 1.

\end{theorem}

\noindent The precise version of Theorem \ref{thm:Smales9} can be found in Proposition \ref{prop:SmalesNinthProposition}.

\subsection{Undecidable problems in computer-assisted proofs: The proof of Kepler's conjecture} 
The Flyspeck program is a stunning example of a successful computer-assisted proof of one of the few open questions spanning several centuries. However, Theorem \ref{thm:Smales9} reveals that the computational proof is actually even more intricate than one might think. In fact, part (ii) of Theorem \ref{thm:Smales9} demonstrates that the general decision problem is actually undecidable even for $K = 1$, and even in a randomised setting. Moreover, part (i) reveals that not only is it undecidable, but no sequence of algorithms can certify a positive answer to the decision problem. In the language of the Solvability Complexity Index (SCI) hierarchy (see \S \ref{sec:SCI_hierarchy}), this means that the problem is not in $\Sigma_1$. Given the results of Theorem \ref{thm:Smales9}, it may seem paradoxical that one could prove Kepler's conjecture through the Flyspeck program by trying to decide the decision problem $\Xi_K$ with $K = 6$, which is done in the proof.  There are other examples too of conjectures that are proven by means of computer-assisted proofs relying on problems that are non-computable in the general case, for instance the Dirac-Schwinger conjecture proved by C. Fefferman and L. Seco in \cite{fefferman1990, fefferman1992, fefferman1993aperiodicity,  fefferman1994, fefferman1994_2, fefferman1995, fefferman1996interval, fefferman1996, fefferman1997} (see further discussion in \cite{SCI}).

	Theorem \ref{thm:Smales9} shows that the Flyspeck program may have been impossible in several ways. Specifically, any of the following scenarios could have happened.
	\begin{itemize}
		\item[(1)] Conditions of part (i) of Theorem \ref{thm:Smales9} apply, i.e., the answer to all 50000 decision problems involved in the proof would have been \emph{yes}, and hence Kepler's conjecture would have been true, yet the computations in the Flyspeck program would continue forever, regardless of the computing power, never producing the 50000 affirmative answers needed. Thus, it would never confirm Kepler's conjecture.
		\item[(2)] Conditions of part (ii) of Theorem \ref{thm:Smales9} apply instead of part (i). In particular, one might design a sequence of algorithms $\{\Gamma_j\}_{j\in \mathbb{N}}$ so that for any linear program input $\iota$ we have $\Gamma_j(\iota) \rightarrow \Xi_K(\iota)$ as $j \rightarrow \infty$ and, if $\Gamma_j(\iota) = 1$ for some $j \in \mathbb{N}$, then $\Gamma_j(\iota) = \Xi_K(\iota)$. Still, because of (ii), if one of the decision problems had a negative answer, the Flyspeck program would run forever and never conclude with an answer. Moreover, one could not even use randomised algorithms and conclude with a probability better than coin flipping. 
		\item[(3)] Conditions of part (iii) of Theorem \ref{thm:Smales9} apply instead of parts (i) and (ii), i.e., there is an algorithm $\Gamma$ such that for any linear program input $\iota$ we have $\Gamma(\iota) = \Xi_K(\iota)$. However, because of (iii) the following could occur. All 50000 decision problems would have answer {\it yes}, yet it would take $10^{80}$ years for the Flyspeck program to conclude that Kepler's conjecture was true.
	\end{itemize}
The scenario (1) can, in fact, always be avoided by considering different decision problems where one replaces $\langle x , c \rangle_K \leq M$ in \eqref{Dual_LP_extended} by $\langle x , c \rangle_K < M$. For this problem, a statement analogous to part (i) of Theorem \ref{thm:Smales9} is no longer true. In fact, for every $K \in \mathbb{N}$, there does exist a sequence of algorithms $\{\Gamma_j\}_{j\in \mathbb{N}}$ such that for any linear program input $\iota$ we have $\Gamma_j(\iota) \rightarrow \Xi_K(\iota)$ as $j \rightarrow \infty$ and, if $\Gamma_j(\iota) = 1$ for some $j \in \mathbb{N}$, then $\Gamma_j(\iota) = \Xi_K(\iota)$. In the proof of Kepler's conjecture, all decision problems have the property that $\langle x , c \rangle_K < M$ for some feasible $x$ (and the appropriate $M$). Moreover, these decision problems can be solved quickly, and one is in fact in conditions analogous to part (iv) of Theorem \ref{thm:Smales9}. 

\section{Main Theorem I (Part b): The extended Smale's 9th in the sciences}

\subsection{The extended Smale's 9th problem and why things often work in practice}\label{sec:Smale9_comp_sens}

In view of Theorem \ref{Cor:main}, a natural question is the following: Under which conditions on the set of inputs may \eqref{problems3} in the extended model be in P? Obvious candidates for such conditions would be various standard assumptions on problems in the sciences such as sparse regularisation. Thus, Theorem \ref{th:smale_comp_sens} below is a continuation of Theorem \ref{Cor:main} demonstrating and quantifying the phase transitions in concrete examples.  A fundamental concept is sparsity, where we say that a vector $x$ is $s$-sparse if it has at most $s$ non-zero entries.  Furthermore, a standard requirement is that the matrix $A$ satisfy either the restricted isometry property or the robust nullspace property (RNP) \cite{AdcockHansenBook, foucartBook, CohenDahmenDeVore}. We will only consider the latter. Concretely, a matrix $A \in \mathbb{C}^{m \times N}$ is said to satisfy the \emph{$\ell^2$-robust nullspace property of order $s$} with parameters $\rho \in (0,1)$ and $\tau>0$ if
\begin{equation}\label{eq:RNP}
	\|v_S\|_2 \leq \frac{\rho}{\sqrt{s}} \|v_{S^c}\|_1 + \tau \|Av\|_2,
\end{equation}
for all index sets $S \subseteq \{1,\hdots, N\}$ with $|S| \leq s$ (where $|\cdot|$ denotes the cardinality) and vectors $v \in \mathbb{C}^N$. The notation $v_S$ means that the entries of $v$ with indices in $S^c$ are set to zero and the remaining entries are left unchanged.
We can now make our question more specific.
\begin{problem}\label{qP_CS}
	Given $K \in \mathbb{N}$, is the problem of computing $K$ correct digits of a solution to the  BP problem \eqref{problems3} in the extended model in P, assuming $y = Ax$, for some $x$ guaranteed to be $s$-sparse, and $A$ satisfies the $\ell^2$-robust nullspace property of order $s$ with parameters $\rho$ and $\tau$?
\end{problem}

We consider classes of inputs specified as follows. Fix real constants $\rho\in(1/3,1)$, $\tau>10$, $b_1>3$, and $b_2 > 6$, and, for  $\epsilon\in[0,1]$ and $s,m,N\in\mathbb{N}$ such that $m\leq N$ and $s\leq N$  define 
\begin{equation}\label{eq:Omega_smN}
	\Omega_{s,m,N}^\epsilon= \{(y,A) \in \mathbb{R}^m\times  \mathbb{R}^{m\times N} \, \vert \, (y,A) \text{ satisfies \eqref{eq:conditions_on_Ay}}\},
\end{equation}
where 
\begin{equation}\label{eq:conditions_on_Ay}
	\begin{split}
		&A \text{ satisfies the RNP } \eqref{eq:RNP} \text{ of order $s$ with parameters $\rho$ and $\tau$,}\\
		&\|y-Ax\|_2 \leq \epsilon  \text{ for some $s$-sparse $x$, } \|y\|_2 \leq b_1, \,\,  \|A\|_2 \leq b_2\sqrt{N/m} .
	\end{split}
\end{equation}
Finally, set  
\begin{equation}\label{eq:the _Omega}
	\Omega^\epsilon =\bigcup_{\substack{s,m,N\in\mathbb{N}\\ m\leq N}}\,\Omega_{s,m,N}^\epsilon.
\end{equation}
Then, defining $\log_{10}(0) = -\infty$ and recalling the definition of the $C_{\mathrm{RCC}}$ condition number (see \S\ref{condition}), we have the following theorem.

\begin{theorem} \emph{(The extended Smale's 9th problem in the sciences)}{\bf.}\label{th:smale_comp_sens}
	Let $\xibp$ denote the solution map to the $\ell^1$-BP problem \eqref{problems3} and consider the $\|\cdot\|_{2}$-norm for measuring the error. If $\Omega_{s,m,N}^\epsilon$ and $\Omega^\epsilon$ are as in \eqref{eq:Omega_smN} and \eqref{eq:the _Omega} the following holds.
	\begin{itemize}[leftmargin=8mm]
		\item[(i)] There exists a constant $C > 0$, independent of $\rho,\tau,b_1$ and $b_2$, such that if we fix $s = 2^k$ for some $k \in \mathbb{N}$ and $m, N$ satisfying $N > m\geq  7s$ and $m\geq C s\log^2(2s)\log(N)$, then we have the following. For $\delta \in (0,1]$ and  
		\[
		K \geq \left\lceil \log_{10}\left(2\delta^{-1}\right)\right \rceil,
		\]
		there does not exist any algorithm (even randomised) that can produce $K$ correct digits for $\xibp$
		(with probability exceeding $\mathrm{p} > 1/2$ in the randomised case) for all inputs in $\Omega_{s,m,N}^0$.
		\item[(ii)] There exists a polynomial $\mathrm{pol}:\real^2\to\real$ and, for every $\delta\in[0, (1-\rho)/(16\tau)]$, there exists an algorithm $\Gamma_\delta$ (either a Turing machine or a BSS machine) such that, for $K \in \mathbb{N}$ satisfying 
		\begin{equation}\label{eq:theK}
			K \leq \left  \lfloor \log_{10}\big((1-\rho)(16\tau)^{-1}\delta^{-1}\big) \right \rfloor,
		\end{equation} 
		$\Gamma_\delta$ produces $K$ correct digits for $\xibp$ for all inputs in $\Omega^\delta$. In the case that $\delta = 0$, \eqref{eq:theK} is to be interpreted as $K<\infty$.
		Moreover, the runtime of $\Gamma_\delta$ (steps performed by the Turing machine, arithmetic operations performed by the BSS machine) is bounded by $\mathrm{pol}(n,K)$, where $n = m+mN$ is the number of variables. In particular, for $\delta = 0$, the runtime is bounded by $\mathrm{pol}(n,K)$ for all $K \in \mathbb{N}$. 
		\item[(iii)]
		Consider the $\ell^1$-BP problem \eqref{problems3} with $\delta=0$.
		For any fixed $s\geq 3$, there are infinitely many pairs $(m,N)$ and inputs $\iota = (Ax,A) \in \Omega^0$, where $x\in\real^m$ is $s$-sparse and $A \in \real^{m \times N}$ is a subsampled Hadamard, Bernoulli, or Hadamard-to-Haar matrix such that
		\[
		C_{\mathrm{RCC}}(\iota) = \infty.
		\] 
		In particular, appreciating (ii) and (iii), there exist inputs in $\Omega$ with infinite RCC condition number, yet the problem is in P.
	\end{itemize} 
\end{theorem}
\noindent The precise version of Theorem \ref{th:smale_comp_sens} can be found in Proposition \ref{prop:CSResult}.
The proof of Theorem \ref{th:smale_comp_sens} reveals upper and lower bounds on the probabilistic strong breakdown epsilon $\epsilon_{\mathbb{P}\mathrm{B}}^{s}$ (Definition \ref{prob_strong_break}) for $\ell^1$-Basis Pursuit given that $A$ satisfies the robust nullspace property. Indeed, we have 
\[
\frac{1}{2}\delta \leq \epsilon_{\mathbb{P}\mathrm{B}}^{s} \leq \frac{16\tau}{1-\rho}\delta. 
\]

	Theorem \ref{th:smale_comp_sens} can be viewed as the quantified cousin of Theorem \ref{Cor:main}. Indeed, Theorem \ref{th:smale_comp_sens} demonstrates the facets of Theorem \ref{Cor:main} in areas in the sciences  where the phenomenon described in Theorem \ref{Cor:main} occur under usual conditions. However, it is important to emphasise that, while Theorem \ref{th:smale_comp_sens} and Theorem \ref{Cor:main} explain why things fail, they also explain why things often work. Specifically, the negative part
	(i) of Theorem \ref{th:smale_comp_sens} can be interpreted in conjunction with the positive parts (ii) and (iii) as follows. Error bounds in sparse regularisation are typically linear in $\delta$, hence not being able to compute a minimiser to better accuracy than $\delta$ is perfectly acceptable.

\begin{remark}[{\bf In P, yet the condition number is $\infty$}]\label{rem:inf_cond}	
Part (iii) of Theorem \ref{th:smale_comp_sens} demonstrates how one can have tractable problems that are nonetheless ill-conditioned according to the RCC condition number. 	Moreover, our examples are not contrived, but key examples from applications. Indeed, subsampled Hadamard matrices (where the rows are randomly sampled in a typical compressed sensing fashion) and Bernoulli matrices form the foundations of compressive imaging in microscopy, lensless camera, compressive video, etc. \cite{Breaking}.
\end{remark}

\section{The SCI hierarchy -- Mathematical preliminaries for the proofs}\label{Sec:background}
In this section we present outlines of the proofs of the main theorems and how to navigate through the manuscript. In addition we present the mathematical preliminaries needed. We begin by giving a brief account of the new concepts introduced in the paper in order to prove the theorems.  The techniques needed in our proofs substantially extend the recent program on the Solvability Complexity Index (SCI) hierarchy that was initiated in \cite{Hansen_JAMS} and has been developed in a series of papers \cite{SCI, CRAS, Colbrook_2021, Colbrook_2019, Hansen_JAMS, ben2021computing, colbrook2019foundations} in order to solve longstanding computational problems.
The SCI hierarchy generalises the arithmetical hierarchy \cite{odifreddi1992classical} to arbitrary computational problems in any computational model. It is motivated by Smale's program on foundations of computational mathematics and some of his fundamental problems \cite{Smale81, Smale85} on the existence of algorithms for polynomial root finding -- solved by C. McMullen \cite{McMullen1, McMullen2, Smale_McMullen} and P. Doyle \& C. McMullen \cite{Doyle_McMullen}.

\subsection{Condition - precise definitions}\label{condition}
We recall the standard definitions of condition used in optimisation 
\cite{Renegar1, Renegar2, BCSS, Wright, Condition}. The classical 
condition number of a matrix $A$ is given by 
$
\mathrm{Cond}(A) = \|A\|_2\|A^{-1}\|_2.
$
For different types of condition numbers related to a mapping $\Xi: \Omega \subset \mathbb{C}^n \rightarrow \mathbb{C}^m$ we need to establish what types of perturbations we are interested in. For example, if $\Omega$ denotes the set of diagonal matrices, we may not be interested in perturbations in the off-diagonal elements as they will always be zero. In particular, we may only be interested in perturbations in the coordinates that are varying in the set $\Omega$. Thus, given $\Omega \subset  \mathbb{C}^n$ we define the active coordinates of $\Omega$ to be 
$
\mathcal{A}(\Omega) = \{j \in \{1,2,\dotsc,n\} \, \vert \, \exists \, x,y \in \Omega, x_j \neq y_j\},
$
which we refer to as the active set as well as  
\begin{equation}\label{eq:def-active-set}
\Omega^{\text{act}} = \{x \in \real^{n} \, \vert \, \exists \, y \in \Omega \text{ such that } x_{\mathcal{A}^c} = y_{\mathcal{A}^c}\}, \quad \mathcal{A} = \mathcal{A}(\Omega). 
\end{equation}
In particular, $\Omega^{\text{act}}$ describes the perturbations allowed as dictated by the active set of coordinates, namely the perturbations along the non-constant coordinates of elements in $\Omega$. We can now recall some of the classical condition numbers from the literature \cite{Condition}.

We begin with {\it condition of a mapping}.
Let
$\Xi: \Omega \subset \real^n \rightarrow \real^m$ be a linear or non-linear mapping, and suppose that $\Xi$ is also defined on some set $\hat \Omega$ with $\Omega \subset \hat \Omega$. Then,  
\begin{equation}\label{eq:cond1}
	\mathrm{Cond}(\Xi)= \sup_{x \in \Omega} \lim_{\epsilon \rightarrow 0^+} \sup_{\substack{x+z \in \Omega^{\text{act} }\cap \hat \Omega\\ 0 < \| z \|_\infty  \leq \epsilon}} \left \{ \frac{\mathrm{dist}_\infty(\Xi(x+z),\Xi(x))}{\| z \|_\infty} \right \},
\end{equation}
where we allow for multivalued functions by defining 
$
\mathrm{dist}_\infty(\Xi(x),\Xi(z)) = \inf_{ x' \in \Xi(x),  z' \in \Xi(z)}\| x' -  z'\|_\infty
$ (here, the supremum of the empty set is understood to be $0$). We typically choose $\hat \Omega$ to be the largest set on which $\Xi$ can be defined in a natural fashion.

Next we have {\it distance to infeasibility} and the \emph{Feasibility Primal} condition number, defined for both Linear Programming \eqref{problems} and Basis Pursuit \eqref{problems3}. Specifically, we define the set of \emph{infeasible inputs} (which we denote by $\InfeasibleInputs$) for a linear program to be the set of $(y,A)$ such that no $x$ exists with $x \geq 0$ and $Ax=y$. Similarly, for basis pursuit denoising with parameter $\delta \geq 0$, we define the set of infeasible inputs (which we also denote, with a slight abuse of notation, by $\InfeasibleInputs$) as the set of inputs for which no $x$ exists with $\|Ax-y\|_2 \leq \delta$.

We then define the {\it Feasibility Primal} (FP) condition number for either basis pursuit or linear programming with inputs in $\Omega$ according to
\begin{equation}\label{eq:cond2Semi}
	C_{\mathrm{FP}}(y,A):= \begin{cases} \frac{\|y\|_2 \vee \|A\|_2}{\dist_2\left[(y,A),\InfeasibleInputs \cap \Omega^{\text{act}}\right]} & \text{ if } \dist_2\left[(y,A),\InfeasibleInputs \cap \Omega^{\text{act}}\right] \neq 0\\
		\infty & \text{ otherwise} 
	\end{cases} 
\end{equation}
where we use the standard $\vee, \wedge$ as notation for max/min respectively and where, for a pair $(y,A) \in \real^m \times \real^{m \times N}$ and a set of pairs $S \subseteq \real^m \times \real^{m \times N}$, $\dist_2$ is defined in the following way:
\begin{equation*}
	\dist_2[(y,A),S] = \inf\{ \|y'-y\|_2 \vee \|A'-A\|_2\,\vert\,  {(y',A') \in S} \}.
\end{equation*}
Here, the infimum of the empty-set is understood to be $\infty$.
Note that UL \eqref{problems4} and CL \eqref{problems5} are always feasible, and hence the Feasibility Primal condition number is not defined for these problems. 

Finally, we have the {\it distance to inputs with several minimisers} and the \emph{RCC condition number}: 
For any of problems \eqref{problems} to \eqref{problems5}, we define the set of inputs with several minimisers (which we denote by $\MultiMinInputs$) to be the set of pairs $(y,A)$ such that the problem with input $(y,A)$ has at least two distinct solutions. We then define the {\it RCC condition number} according to
\begin{equation}\label{eq:cond3Semi}
	C_{\mathrm{RCC}}(y,A):= \begin{cases} \frac{\|y\|_2 \vee \|A\|_2}{\dist_2\left[(y,A),\MultiMinInputs \cap \Omega^{\text{act}}\right]} & \text{ if } \dist_2\left[(y,A),\MultiMinInputs \cap \Omega^{\text{act}}\right] \neq 0\\
		\infty & \text{ otherwise}
	\end{cases} .
\end{equation}

\begin{remark}[RCC condition number]
	Note that the definition of the RCC condition number here is slightly weaker than the one in \cite{Condition}, that is to say, our RCC condition number is less than or equal to the one defined in \cite{Condition}. However, our weaker definition simply makes our results stronger in that we show that $C_{\mathrm{RCC}}(y,A) = \infty$ yet the problem of computing minimisers is still in P. 
	We have also split the definition in \cite{Condition} into two parts - one for feasibility, and one for two minimisers. This allows for a more granular analysis of the influence of condition.
\end{remark}

\subsection{SCI hierarchy}\label{sec:SCI_hierarchy}
To formalise our results and the underlying theory we recall the definitions from \cite{SCI} and start by defining a \emph{computational problem}.

\begin{definition}[Computational problem]\label{definition:ComputationalProblem}
	Let $\Omega$ be some set, which we call the \emph{input} set,
	and $\Lambda$ be a set of complex valued functions on $\Omega$ such that for $\iota_1, \iota_2 \in \Omega$, then $\iota_1 = \iota_2$ if and only if $f(\iota_1) = f(\iota_2)$ for all $f \in \Lambda$, called an \emph{evaluation} set. Let $(\mathcal{M},d)$ be a metric space, and finally let $\Xi:\Omega\to \mathcal{M}$ be a function which we call the \emph{solution map}.
	We call the collection $\{\Xi,\Omega,\mathcal{M},\Lambda\}$ a \emph{computational problem}. When it is clear what $\mathcal{M}$ and $\Lambda$ are we will sometimes write $\{\Xi,\Omega\}$ for brevity. 
\end{definition}

The set $\Omega$ is essentially the set of objects that give rise to the various instances of our computational problem. It can be a family of matrices (infinite or finite) and vectors, a collection of polynomials, a family of Schr\"odinger operators with a certain potential etc. The solution map $\Xi : \Omega\to\mathcal{M}$ is what we are interested in computing. It could be one of the problems in \eqref{problems}, the set of eigenvalues of an $n\times n$ matrix, the spectrum of a Hilbert (or Banach) space operator, root(s) of a polynomial, etc. 
Finally, the set $\Lambda$ is the collection of functions that provide us with the information we are allowed to read, say matrix elements and vector coefficients, polynomial coefficients or pointwise values of the potential of a Schr\"odinger operator, for example.

\subsubsection{Algorithms and tower of algorithms}

The cornerstone of the SCI framework is the definition of a general algorithm, introduced next.

\begin{definition}[General Algorithm]\label{definition:Algorithm}
	Given a  computational problem $\{\Xi,\Omega,\mathcal{M},\Lambda\}$, a \emph{general algorithm} is a mapping $\Gamma:\Omega\to\mathcal{M}\cup \{\nh\}$ such that, for every $\iota\in\Omega$, the following conditions hold:
	\begin{enumerate}[label=(\roman*)]
		\item there exists a nonempty subset of evaluations $\Lambda_\Gamma(\iota) \subset\Lambda $, and, whenever $\Gamma(\iota) \neq  \nh$, we have $|\Lambda_\Gamma(\iota)|<\infty$\label{property:AlgorithmFiniteInput},
		\item  the action of $\,\Gamma$ on $\iota$ is uniquely determined by $\{f(\iota)\}_{f \in \Lambda_\Gamma(\iota)}$, \label{property:AlgorithmDependenceOnInput}
		\item for every $\iota^{\prime} \in\Omega$ such that $f(\iota^\prime)=f(\iota)$ for all $f\in\Lambda_\Gamma(\iota)$, it holds that $\Lambda_\Gamma(\iota^{\prime})=\Lambda_\Gamma(\iota)$.\label{property:AlgorithmSameInputSameInputTaken}
	\end{enumerate}
\end{definition}

\begin{remark}[The purpose of a general algorithm: universal impossibility results]
	The purpose of a general algorithm is to have a definition that will encompass any model of computation, and that will allow impossibility results to become universal. Given that there are several non-equivalent models of computation, impossibility results will be shown with this general definition of an algorithm.
\end{remark}

\begin{remark}[The non-halting output $\nh$]
	The non-halting ``output'' $\nh$ of a general algorithm may seem like an unnecessary distraction given that a general algorithm is just a mapping, which is strictly more powerful than a Turing or a BSS machine. However, the $\nh$ output is needed when the concept of a general algorithm is extended to a randomised general algorithm (as done in \S \ref{sec:randomised}) needed to prove Theorem \ref{Cor:main}. Potentially surprisingly, as shown in Theorem \ref{Cor:main}, allowing a randomised algorithm to not halt with positive probability makes it more powerful. Note, however, that adding the $\nh$ as a possible output of a general algorithm does not make any results weaker.
	A technical remark about $\nh$ is also appropriate, namely that $\Lambda_{\Gamma}(\iota)$ is allowed to be infinite in the case when $\Gamma(\iota) = \nh$. This is to allow general algorithms to capture the behaviour of a Turing or a BSS machine not halting by virtue of requiring an infinite amount of input information.
\end{remark}

Owing to the presence of the special non-halting ``output'' $\nh$, we have to extend the metric $d_{\mathcal{M}}$ on $\mathcal{M}\times \mathcal{M}$ to $d_{\mathcal{M}}:\mathcal{M} \cup \{\nh\} \times \mathcal{M} \cup \{\nh\}\to \real_{\geq 0}$ in the following way:
\begin{equation}\label{eq:extended-metric}
	d_{\mathcal{M}}(x,y) = \begin{cases} d_{\mathcal{M}}(x,y) & \text{ if } x,y \in \mathcal{M} \\
		0 & \text{ if } x = y = \nh\\
		\infty & \text{ otherwise.} \end{cases}
\end{equation}

Definition \ref{definition:Algorithm} is sufficient for defining a randomised general algorithm, which is the only tool from the SCI theory needed in order to prove Theorem \ref{Cor:main}. However, for several other results, the full definition of the SCI hierarchy will be useful. The first is the concept of a tower of algorithms, which is important in the general SCI theory in order to allow for problems that require several limits, such as spectral problems for operators on Hilbert spaces. 

\begin{definition}[Tower of algorithms]\label{tower_funct}
	Given a computational problem $\{\Xi,\Omega,\mathcal{M},\Lambda\}$, a \emph{general tower of algorithms of height $k$
		for $\{\Xi,\Omega,\mathcal{M},\Lambda\}$} is a family of sequences of functions 
	\begin{equation*}
		\Gamma_{n_k}:\Omega
		\rightarrow \mathcal{M}, \ \Gamma_{n_k, n_{k-1}}:\Omega \rightarrow \mathcal{M}, \hdots 
		\Gamma_{n_k, \hdots, n_1}:\Omega \rightarrow \mathcal{M},
	\end{equation*}
	where $n_k,\hdots,n_1 \in \mathbb{N}$ and the functions $\Gamma_{n_k, \hdots, n_1}$ are general algorithms in the sense of Definition \ref{definition:Algorithm}. Moreover, for every $\iota \in \Omega$,
	\begin{equation}\label{conv}
		\Xi(\iota)= \lim_{n_k \rightarrow \infty} \Gamma_{n_k}(\iota), \quad \Gamma_{n_k, \hdots, n_{j+1}}(\iota)= \lim_{n_j \rightarrow \infty} \Gamma_{n_k, \hdots, n_j}(\iota) \quad j=k-1,\dots,1.
	\end{equation}
\end{definition}

\begin{remark}[Multivalued functions] When dealing with optimisation problems one needs a framework that can handle multiple solutions. As the setup above does not allow $\Xi$ to be multi-valued we need some slight changes.  We allow $\Xi$ to be multivalued, even though towers of algorithms are not. Hence, the only difference to the standard SCI hierarchy is that the first limit in \eqref{conv}
	is replaced by
	\[
	\disM(\Xi(\iota),\Gamma_{n_k}(\iota)) \longrightarrow 0, \qquad n_k \rightarrow \infty, 
	\]
	where 
	$
	\disM(\Xi(\iota),\Gamma_{n_k}(\iota)) := \inf_{x \in \Xi(\iota)}d_{\mathcal{M}}(x,\Gamma_{n_k}(\iota)).
	$ 
\end{remark}

As already mentioned above, the purpose of a general algorithm is to obtain universal impossibility results. Conversely, for a more granular analysis of the complementary positive results, i.e., the analysis of problems where it is possible to construct an algorithm, we define \emph{arithmetic towers} depending on the underlying computational model.

\begin{definition}[Arithmetic towers]\label{def:arithmetic_tower}
	Given a computational problem $\{\Xi,\Omega,\mathcal{M},\Lambda\}$, where $\Lambda$ is countable, an \emph{arithmetic tower of algorithms of height $k$} for $\{\Xi,\Omega,\mathcal{M},\Lambda\}$ is defined as a tower of algorithms whose lowest-level functions $\Gamma = \Gamma_{n_k, \hdots, n_1} :\Omega \rightarrow \mathcal{M}$ satisfy the following:
	For each $\iota\in\Omega$ the mapping $(n_k, \hdots, n_1) \mapsto \Gamma_{n_k, \hdots, n_1}(\iota) = \Gamma_{n_k, \hdots, n_1}(\{f(\iota)\}_{f \in \Lambda})$ is recursive, and $\Gamma_{n_k, \hdots, n_1}(\iota)$ is a finite string of complex numbers that can be identified with an element in $\mathcal{M}$.
\end{definition}

\begin{remark}[Recursiveness]\label{rem:recursiveness} For the positive results we will always construct \emph{recursive} algorithms, the meaning of which depends on whether one is working with the Turing or the BSS model. In the Turing case, recursive means that $f(\iota) \in \mathbb{Q}$, for all $f \in \Lambda$ and $\iota \in \Omega$, $\Lambda$ is countable, and $\Gamma_{n_k, \hdots, n_1}(\{f(\iota)\}_{f \in \Lambda})$ can be executed by a Turing machine \cite{Turing_Machine}, that takes $(n_k, \hdots, n_1)$ as input, and that has an oracle tape providing $f(\iota)$ for every $f \in \Lambda$, see for example \cite{Ko1991ComplexityTO}. In the BSS model, recursive means that $f(\iota) \in \mathbb{R}$ (or $\mathbb{C}$) for all $f \in \Lambda$, and $\Gamma_{n_k, \hdots, n_1}(\{f(\iota)\}_{f \in \Lambda})$ can be executed by a Blum-Shub-Smale (BSS) machine \cite{BCSS} that takes $(n_k, \hdots, n_1)$, as input, and that has an oracle node that provides $f(\iota)$, for every $f \in \Lambda$. See also Remark \ref{rem:oracles}. 
\end{remark}

\subsubsection{Solvability Complexity Index and Hierarchy}

With the concept of towers of algorithms established we can finally define the Solvability Complexity Index (SCI).

\begin{definition}[Solvability Complexity Index]\label{complex_ind}
	Given a computational problem $\{\Xi,\Omega,\mathcal{M},\Lambda\}$, it is said to have \emph{Solvability Complexity Index $\mathrm{SCI}(\Xi,\Omega,\mathcal{M},\Lambda)_{\alpha} = k$} of type $\alpha$, where $\alpha=G$ (for `general') or $\alpha=A$ (for `arithmetic'), if $k$ is the smallest integer for which there exists a tower of 
	algorithms of height $k$ and type $\alpha$. If no such tower exists then $\mathrm{SCI}(\Xi,\Omega,\mathcal{M},\Lambda)_{\alpha} = \infty.$ If 
	there exists such a tower $\{\Gamma_n\}_{n\in\mathbb{N}}$ of height 1, and additionally $\Xi = \Gamma_{n_1}$ for some $n_1 < \infty$, then we define $\mathrm{SCI}(\Xi,\Omega,\mathcal{M},\Lambda)_{\alpha} = 0$.
\end{definition}

The Solvability Complexity Index induces the SCI hierarchy. The key is that this hierarchy does not collapse regardless of the computational model \cite{SCI}. Thus, it provides a universal framework for classifying problems in scientific computing. 

\begin{definition}[The Solvability Complexity Index Hierarchy]
	\label{definition:sciHierarchy}
	Consider a collection $\mathcal{C}$ of computational problems and a type $\alpha=A$ or $\alpha=G$ for the computational problems in $\mathcal{C}$.
	Define 
	\begin{equation*}
		\begin{split}
			\Delta^{\alpha}_0 &:= \{\{\Xi,\Omega\} \in \mathcal{C} \ \vert \   \mathrm{SCI}(\Xi,\Omega)_{\alpha} = 0\}\\
			\Delta^{\alpha}_{m+1} &:= \{\{\Xi,\Omega\}  \in \mathcal{C} \ \vert \   \mathrm{SCI}(\Xi,\Omega)_{\alpha} \leq m\}, \qquad \quad \forall m \in \mathbb{N},
		\end{split}
	\end{equation*}
	as well as
	\[
	\Delta^{\alpha}_{1} := \{\{\Xi,\Omega\}  \in \mathcal{C}   \  \vert \ \exists \ \{\Gamma^n\}_{n\in \mathbb{N}} \text{ tower of type $\alpha$ s.t. } \forall n \ \disM(\Gamma^n(\iota),\Xi(\iota)) \leq 2^{-n}\}. 
	\]
\end{definition}

\subsubsection{The computational problems}\label{sec:the_setup}
We are now ready to formalise the problems of interest \eqref{problems}-\eqref{problems5}  as computational problems in the sense of Definition \ref{definition:ComputationalProblem} for fixed dimensions $m$ and $N$. In all the problems under consideration we consider $(\mathcal{M},d)=(\mathbb{R}^{m} \cup \{\infty\},\|\cdot\|)$, where the norm depends on the different problems. In all cases except for neural networks, the set $\Omega$ will always contain vectors and matrices, and thus the set of evaluations will naturally be the collection of the coordinate functions. Concretely, we will often consider input sets $\Omega$ whose elements are of the form $\iota = (y,A)$ with $y \in \mathbb{R}^m$ and $A \in \mathbb{R}^{m\times N}$, in which case we let 
\begin{equation*}
\Lambda_{m,N} = \{f_j^{\mathrm{vec}}\}_{j=1}^m\,\cup\, \{f_{j,k}^{\mathrm{mat}}\}_{j=1, k=1}^{j=m, k=N},
\end{equation*}
where $f_j^{\mathrm{vec}}(\iota)=y_j$ and $f_{j,k}^{\mathrm{mat}}(\iota)=A_{j,k}$.
In the following we write $\mathcal{J}$  for $\|\cdot\|_1$ or $\|\cdot\|_{\mathrm{TV}}$ depending on the context. 

\begin{itemize}[leftmargin=8mm]
	\item[(i)]  {\it Linear Programming}: $\iota \in \Omega$ is of the form $\iota = (y,A)$, where $y \in \mathbb{R}^m$ and $A \in \mathbb{R}^{m \times N}$. For fixed $c\in\real^N$ the solution map is given by
	\[
	\Xi(\iota) = \mathop{\mathrm{arg min}}_{x \in \mathbb{R}^N} \langle x , c \rangle \text{ such that } Ax=y, \quad x \geq 0.
	\]
	\item[(ii)] Basis Pursuit: $\iota \in \Omega$ is of the form $\iota = (y,A)$ where $y \in \mathbb{R}^m$ and $A \in \mathbb{R}^{m \times N}$. 
	For a fixed parameter $\delta \geq 0$ the solution map is given by
	\begin{equation*}
		\Xi(\iota) = 
		\begin{cases}
			\mathop{\mathrm{arg min}}_{x \in \real^N} \mathcal{J}(x) \text{ such that } \|Ax - y\|_2 \leq \delta, & \text{ if } \iota \text{ is feasible}\\
			\varnothing & \text{else}
		\end{cases}\;.
	\end{equation*}
	
	\item[(iii)] {\it Constrained Lasso}: $\Omega$ consists of inputs of the same form as  for basis pursuit. For a fixed parameter $\tau \geq 0$, the solution map is
	\[
	\Xi(\iota) =  \mathop{\mathrm{arg min}}_{x \in \real^N} \|Ax-y\|_2 \text{ such that } \|x\|_1 \leq \tau.
	\]
	\item[(iv)] {\it Unconstrained Lasso}: $\Omega$ consists of inputs of the same form as  for basis pursuit. For a fixed parameter $\lambda > 0$, the solution map is
	\[
	\Xi(\iota) =  \mathop{\mathrm{arg min}}_{x \in \real^N} \|Ax-y\|^2_2 + \lambda\, \mathcal{J}(x) .
	\]
	
\end{itemize}

\subsection{Inexact input and perturbations} 
Suppose we are given a computational problem $\{\Xi, \Omega, \mathcal{M}, \Lambda\}$, and that 
$
\Lambda = \{f_j\}_{j \in \beta},
$
 where $\beta$ is some index set that can be finite or infinite.  However, obtaining $f_j$ may be a computational task on its own, which is exactly the problem in most areas of 
computational mathematics. In particular, for $\iota \in \Omega$, $f_j(\iota)$ could be the number $e^{\frac{\pi}{j} i }$ for example. Hence, we cannot access $f_j(\iota)$, but rather $f_{j,n}(\iota)$ where $f_{j,n}(\iota) \rightarrow f_{j}(\iota)$ as $n \rightarrow \infty$. In this paper we will be interested in the case when this can be done with error control. In particular, we consider $f_{j,n}: \Omega \to \dyadic_n + i \dyadic_n$, where $\dyadic_n:=\{k\,2^{-n}\, \vert \,k\in\mathbb{Z}\}$, such that
\begin{equation}\label{Lambda_limits2}
	\|\{f_{j,n}(\iota)\}_{j\in\beta} - \{f_j(\iota)\}_{j\in\beta}\|_{\infty} \leq 2^{-n}, \quad \forall \iota \in \Omega.
\end{equation}
By analogy with $\Delta_1$ classification, we will call a collection of such functions $\Delta_1$-information for the computational problem. Formally, we have the following.

\begin{definition}[$\Delta_{1}$-information]\label{definition:Lambda_limits}
	Let $\{\Xi, \Omega, \mathcal{M}, \Lambda\}$ be a computational problem with $\Lambda = \{f_j\}_{j \in \beta}$. Suppose that, for each $j\in\beta$ and $n\in\mathbb{N}$, there exists an  $f_{j,n}: \Omega \to \dyadic_n + i \dyadic_n$ such that \eqref{Lambda_limits2} holds. We then say that the set $\hat\Lambda=\{f_{j,n} \, \vert \, j\in\beta ,n\in\mathbb{N}\}$ provides $\Delta_1$-information for $\{\Xi, \Omega, \mathcal{M}, \Lambda\}$. Moreover, we denote the family of all such $\hat \Lambda$ by $\mathcal{L}^1(\Lambda)$. 
\end{definition}
\begin{remark}[Turing vs Markov]\label{rem:TurVsMarkov}
	 It is possible to consider a restriction of $\Delta_1$ information wherein for each $j$, the $f_{j,n}$ forms a computable sequence. This restriction (known as the Markov model \cite{MarkovModel}) strengthens the negative results and weakens the positive results - our techniques apply to this model as will be seen in the proof of Proposition \ref{prop:DrivingNegativeProposition} (see Remark \ref{remark:CompDelta1Markov} for further details).
\end{remark}
\begin{remark}
One can take the analogy with $\Delta_1$ classification in the SCI hierarchy even further by considering $f_j$ that are higher up in the hierarchy, and analogously define $\Delta_m$-information for $m>1$ (see \cite{SCI}). However, this is beyond the scope of the present paper.
\end{remark}

Note that we want to have algorithms that can the computational problems $\{\Xi,\Omega,\mathcal{M},\hat \Lambda\}$ for all possible choices of  $\hat \Lambda \in  \mathcal{L}^1(\Lambda)$. In order to formalise this we define what we mean by a computational problem with $\Delta_1$-information.

\begin{definition}[Computational problem with $\Delta_1$-information]\label{definition:Omega_tilde_Delta_m}
	Given $\{\Xi,\Omega,\mathcal{M},\Lambda\}$ with  $\Lambda=\{f_j\}_{j \in \beta}$, the corresponding \emph{computational problem with $\Delta_1$-information} is defined as 
	$
	\{\Xi,\Omega,\mathcal{M},\Lambda\}^{\Delta_1} := \{\tilde \Xi,\tilde \Omega,\mathcal{M},\tilde \Lambda\},
	$
	where 
	\begin{equation}\label{eq:omega_tilde}
	\tilde \Omega = \left\{ \tilde \iota = \big\{(f_{j,1}(\iota), f_{j,2}(\iota), f_{j,3}(\iota), \dots) \big\}_{j \in \beta} \, \vert \, \iota \in \Omega,\; f_{j,n}: \Omega \to \dyadic_n + i \dyadic_n \text{ satisfy \eqref{Lambda_limits2}} \right\},
	\end{equation}
 	$\tilde \Xi(\tilde \iota) = \Xi(\iota)$, and $\tilde \Lambda = \{\tilde f_{j,n}\}_{j,n \in \beta \times \mathbb{N}}$, where $\tilde f_{j,n}(\tilde \iota) = f_{j,n}(\iota)$. Given an $\tilde\iota\in\tilde\Omega$, there is a unique $\iota\in\Omega$ for which $\tilde\iota=\big\{(f_{j,1}(\iota), f_{j,2}(\iota), f_{j,3}(\iota), \dots) \big\}_{j \in \beta}$ (by Definition \ref{definition:ComputationalProblem}). We say that this $\iota\in\Omega$ \emph{corresponds} to $\tilde\iota\in\tilde\Omega$.
\end{definition}
\begin{remark}
Note that the correspondence of a unique $\iota$ to each $\tilde\iota$ in Definition \ref{definition:Omega_tilde_Delta_m} ensures that $\tilde\Xi$ and the elements of $\tilde\Lambda$ are well-defined.
\end{remark}

One may interpret the computational problem $\{\Xi,\Omega,\mathcal{M},\Lambda\}^{\Delta_1} = \{\tilde \Xi,\tilde \Omega,\mathcal{M},\tilde \Lambda\}$ as follows. The collection $\tilde \Omega$ is the family of all sequences approximating the inputs in $\Omega$. For an algorithm to be successful for $\{\Xi,\Omega,\mathcal{M},\Lambda\}^{\Delta_1}$ it must work for all $\tilde \iota \in \tilde \Omega$, that is, for any sequence approximating $\iota$, as opposed to a particular choice of $\Delta_1$-information for $\{\Xi, \Omega, \mathcal{M}, \Lambda\}$ according to Definition \ref{definition:Lambda_limits}. The relation between these two closely related concepts will be further elucidated in \S\ref{sec:different-impossibility-results} below.

\begin{remark}[Oracle tape/node providing $\Delta_1$-information]\label{rem:oracles}
For impossibility results we use general algorithms and randomised general algorithms (as defined below), and thus, due to their generality, we do not need to specify how the algorithms read the information. However, all positive results are done for problems $\{\Xi,\Omega,\mathcal{M},\Lambda\}^{\Delta_1} := \{\tilde \Xi,\tilde \Omega,\mathcal{M},\tilde \Lambda\}$ with either a Turing or a BSS machine, and we thus need to specify how  $\tilde \iota \in \tilde \Omega$ in \eqref{eq:omega_tilde}  is passed  to the algorithm as an input. Suppose that $\beta$, the index set for $\Lambda$, is countable. In the Turing case we follow the standard convention in the literature (see for example \cite{Ko1991ComplexityTO}), in particular, $\tilde \iota$ is represented by an oracle tape that, on input $(j,n)\in\beta \times \mathbb{N}$ (where $j$ is an integer or otherwise encoded in a finite alphabet), prints the (unique) finite binary string representing the dyadic number $\tilde f_{j,n}(\tilde \iota)$. Similarly, for a BSS machine, we assume the standard setup (see \cite{BCSS}) with an oracle node that on input $(j,n)\in \beta \times \mathbb{N}$ returns $\tilde f_{j,n}(\tilde \iota)$. 
\end{remark}

\subsection{Breakdown epsilons -- the key to proving $K$, $K-1$, $K-2$ -type theorems}\label{sec:BDepsilons}
The purpose of breakdown epsilons is to characterise the lower bounds on the accuracy that can be achieved by algorithms. As the name suggests, the breakdown epsilons are the fundamental barriers on the best possible accuracy that any algorithm can achieve. There are two types of (deterministic) breakdown epsilons, namely the \emph{strong} and the \emph{weak} one, as well as their probabilistic versions. We define and discuss each of them in detail below. We begin with the strong breakdown epsilon.
\begin{definition}[Strong breakdown epsilon]\label{def:Breakdown-epsilons_strong}
	Given a computational problem $\{\Xi,\Omega,\mathcal{M},\Lambda\}$, we define its \emph{strong breakdown epsilon} as follows:
	\begin{equation*}
		\epsilon_{\mathrm{B}}^{\mathrm{s}} := \sup\{\epsilon \geq 0 \, \vert \, \forall \, \text{general algorithms } \Gamma,  \, \exists \, \iota \in \Omega  \text{ such that } \disM(\Gamma(\iota),\Xi(\iota)) > \epsilon\}.
	\end{equation*}
\end{definition}	
Hence, the strong breakdown epsilon is the largest number $\epsilon \geq 0$ such that no algorithm can provide accuracy exceeding $\epsilon$.

\begin{remark}[The breakdown epsilons with respect to a specific computational model]\label{rem:breakdown-eps}
	The purpose of the strong and other breakdown epsilons is to establish universal impossibility results in the form of lower bounds on the achievable accuracy, and hence we define them in terms of general algorithms (and randomised general algorithms to be introduced below). However, occasionally, it may be convenient to work with the concept of a breakdown epsilon specific to the computational model. This can easily be done by replacing the words `general algorithm' in Definition \ref{def:Breakdown-epsilons_strong} by either `Turing machine' or `BSS machine', for example. In this case we will use the superscript `A' (for arithmetic): 
	\[
	\epsilon_{\mathrm{B}}^{\mathrm{s,A}} \text{ is the strong breakdown epsilon in the Turing (or BSS) model},
	\]
	where, as in Remark \ref{rem:recursiveness}, the Turing model is understood if the inputs are rationals, and the BSS model otherwise. 
\end{remark}

The weak breakdown epsilon, introduced next, is the largest $\epsilon \geq 0$ such that all algorithms will need to use an arbitrarily large amount of input information to reach $\epsilon$ accuracy. 
To make this precise we first need to define the \emph{minimum amount of input information}, which, in turn, we define in terms of an enumeration of the elements in $\Lambda$. We thus assume that $\Lambda$ is countable and enumerated according to
\begin{equation}\label{eq:the_Lambd}
\Lambda = \{f_k \, \vert \, k \in \mathbb{N}, \, k\leq |\Lambda| \},
\end{equation}
where $|\Lambda|$ denotes the cardinality of $\Lambda$.
As we will see below, although we assume that $\Lambda$ has a specific enumeration, many of the key concepts defined below, including the weak breakdown epsilon, are independent of the specific enumeration, and so the only assumption needed is that $\Lambda$ be countable.

\begin{definition}[Minimum amount of input information]\label{def:minimum_runtime}
	Given the computational problem $\{\Xi,\Omega,\mathcal{M},\allowbreak\Lambda\}$, where $\Lambda =\{f_k \, \vert \, k \in \mathbb{N}, \, k\leq |\Lambda| \}$ and a general algorithm $\Gamma$, we define the \emph{minimum amount of input information} $T_{\Gamma}(\iota)$ for $\Gamma$ and $\iota \in \Omega$ as 
	\[
	T_{\Gamma}(\iota) :=\sup\lbrace m \in \mathbb{N} \, \vert \, f_{m} \in \Lambda_{\Gamma}(\iota) \rbrace.
	\]
	Note that, for $\iota$ such that $\Gamma(\iota) = \nh$, the set $\Lambda_{\Gamma}(\iota)$ may be infinite (see Definition \ref{definition:Algorithm}), in which case $T_{\Gamma}(\iota)=\infty$.
	
\end{definition}
We are now ready to define the weak breakdown epsilon. 
\begin{definition}[Weak breakdown epsilon]\label{def:Breakdown-epsilons}
	Given the computational problem $\{\Xi,\Omega,\mathcal{M},\Lambda\}$, where $\Lambda =\{f_k \, \vert \, k \in \mathbb{N}, \, k\leq |\Lambda| \}$, we define the
	\emph{weak breakdown epsilon} by
	\begin{equation*}
		\begin{split}
			\epsilon_{\mathrm{B}}^{\mathrm{w}} &= \sup\{\epsilon \geq 0 \, \vert \,  \forall  \text{ general algorithms } \Gamma  \text{ and }  M \in \mathbb{N} \, \, \exists \, \iota \in \Omega  \text{ such that } \\
			& \hspace{4cm} \disM(\Gamma(\iota),\Xi(\iota)) > \epsilon \text{ or } T_{\Gamma}(\iota) > M \}.
		\end{split}
	\end{equation*}
\end{definition}
\noindent As mentioned above, it is easy to see that the \emph{weak breakdown epsilon} is independent of the ordering of $\Lambda$. 
In words, if the breakdown epsilons are greater than $\epsilon$, we have the following:
\begin{itemize}
	\item[(i)] (Strong Breakdown Epsilon)
	For any algorithm there is an input such that the algorithm fails to produce an $\epsilon$-accurate solution on that input.
	\item[(ii)] (Weak Breakdown Epsilon)
	One can choose an arbitrary large integer $M$ and find an input such that the algorithm will need to  have used at least the $M$-th input and still not have reached $\epsilon$ accuracy. In other words, the amount of information needed is unbounded to reach $\epsilon$-accuracy. 
\end{itemize}  

\begin{remark}[Independence of ordering]
Although the minimum amount of input information is dependent of the enumeration of $\Lambda$ in \eqref{eq:the_Lambd}, it is easy to see that the weak breakdown epsilon and the probabilistic weak breakdown epsilon (to be defined below) are independent of the choice of the enumeration.  
\end{remark}

\subsubsection{The weak breakdown epsilon for problems with $\Delta_1$-information}

In the case of computational problems with $\Delta_1$-information, the minimum amount of input information is  related to the `accuracy needed' on the input to the algorithm.

\begin{definition}[Number of correct `digits' on the input]\label{def:no-of-input-bits}
Suppose that $\hat \Lambda = \{ f_{k,m} \, \vert \, (k,m) \in \beta \times \mathbb{N}\}$ provides $\Delta_1$-information for $\{\Xi,\Omega,\mathcal{M},\Lambda\}$.  Given a general algorithm $\Gamma$ for the problem $\{\Xi,\Omega,\mathcal{M},\hat\Lambda\}$, we define the \emph{`number of  digits' required on the input} according to
\[
D_{\Gamma}( \iota) = \sup\{m\in\mathbb{N} \, \vert \, \exists\,  k \in\beta \text{ s.t. } f_{k,m} \in  \hat\Lambda_{\Gamma}( \iota) \}.
\]
\end{definition}

\begin{remark}[Interpretation of $T_{\Gamma}$ and $D_{\Gamma}$: accuracy on input and lower bound on runtime\label{rem:TGammaAndDGamma}]
Consider the computational problem $\{\Xi,\Omega,\mathcal{M},\hat\Lambda\}$ above. Then, independently of the enumeration of $\hat\Lambda$ used to define $T_{\Gamma}$, we have that
\begin{equation}\label{eq:Dexplodes->Texplodes}
D_{\Gamma}( \iota_n)\to \infty\quad \text{as }n\to \infty \qquad \implies \qquad T_{\Gamma}( \iota_n)\to \infty\quad \text{as }n\to \infty,
\end{equation} 
for every general algorithm $\Gamma$ and every sequence $\{\iota_n\}_{n=0}^{\infty}$ in $ \Omega$. Moreover, in the case when $\Lambda$ is finite, the implication \eqref{eq:Dexplodes->Texplodes} also holds in reverse.
Therefore, the weak breakdown epsilon for problems with $\Delta_1$-information derived from a finite set of evaluations is equivalently given by
\begin{equation*}
		\begin{split}
			\epsilon_{\mathrm{B}}^{\mathrm{w}} &= \sup\{\epsilon \geq 0 \, \vert \,  \forall  \text{ general algorithms } \Gamma  \text{ and }  M \in \mathbb{N} \, \, \exists \, \iota \in \Omega  \text{ such that } \\
			& \hspace{4cm} \disM(\Gamma(\iota),\Xi(\iota)) > \epsilon \text{ or } D_{\Gamma}(\iota) > M \}.
		\end{split}
	\end{equation*}
Note that any reasonable complexity model would have a definition of runtime of an algorithm $\Gamma$ such that the runtime is at least the number of digits it acquires from the input. In particular, any reasonable definition of the runtime of $\Gamma$ (such as the standard definition to be specified in \S \ref{sec:comp_runtime}) will have 
	\begin{equation}\label{eq:runtimeGeqD}
	\mathbf{Runtime}_{\Gamma}( \iota) \geq D_{\Gamma}( \iota).
	\end{equation}
	Hence, for such computational problems, $D_{\Gamma}$ can be used to show that if the weak breakdown epsilon is greater than $\epsilon$ then any algorithm will have arbitrarily large runtime when attempting to achieve $\epsilon$-accuracy. 
\end{remark}

\subsection{Randomised algorithms}\label{sec:randomised}

In many contemporary fields of mathematics of information such as deep learning, the use of randomised algorithms is widespread. We therefore need to extend the concept of a general algorithm to a \emph{randomised random algorithm}.

\begin{definition}[Randomised General Algorithm]\label{definition:ProbablisticAlgorithm}	Given a computational problem $\{\Xi,\Omega,\mathcal{M},\Lambda\}$, where $\Lambda = \{f_k \, \vert \, k \in \mathbb{N}, \, k\leq |\Lambda| \}$, a \emph{randomised general algorithm} (RGA) is a collection $X$ of general algorithms $\Gamma:\Omega\to\mathcal{M}\cup \{\nh\}$, a sigma-algebra $\mathcal{F}$ on $X$, and a family of probability measures $\{\mathbb{P}_{\iota}\}_{\iota \in \Omega}$ on $\mathcal{F}$ such that the following conditions hold:
	\begin{enumerate}[label=(P\roman*)]
		\item For each $\iota \in \Omega$, the mapping $\gprob_{\iota}:(X,\mathcal{F}) \to (\mathcal{M}\cup \{\nh\}, \mathcal{B})$ defined by $\gprob_{\iota}(\Gamma) = \Gamma(\iota)$ is a random variable, where $\mathcal{B}$ is the Borel sigma-algebra on $\mathcal{M}\cup \{\nh\}$. \label{property:PAlgorithmMeasurable}
		\item For each $n \in \mathbb{N}$ and $\iota \in \Omega$, we have $\lbrace \Gamma \in X \, \vert \, T_{\Gamma}(\iota) \leq n \rbrace \in \mathcal{F}$.
		\label{property:PAlgorithmRTMeasurable}
		
		\item For all $\iota_1,\iota_2 \in \Omega$ and $E \in \mathcal{F}$ so that, for every $\Gamma \in E$ and every $f \in \Lambda_{\Gamma}(\iota_1)$, we have $f(\iota_1) = f(\iota_2)$, it holds that $\mathbb{P}_{\iota_1}(E) = \mathbb{P}_{\iota_2}(E)$.
		\label{property:PAlgorithmConsistent}
	\end{enumerate}
	It is not immediately clear whether condition \ref{property:PAlgorithmRTMeasurable} for a given RGA $(X,\mathcal{F},\{\pr_\iota\}_{\iota\in\Omega})$ holds independently of the choice of the enumeration of $\Lambda$. This is indeed the case and will be established in Lemma \ref{lemma:RGAindepOfOrdering} further below.
\end{definition}

\begin{remark}[Assumption \ref{property:PAlgorithmRTMeasurable}]
	Note that \ref{property:PAlgorithmRTMeasurable} in Definition \ref{definition:ProbablisticAlgorithm} is needed in order to ensure that the minimum amount of input information also becomes a valid random variable. More specifically, for each $\iota \in \Omega$, we define the random variable
	\[
	T_{\gprob}(\iota): X\to\mathbb{N}\cup\{\infty\} \text{ according to } \Gamma \mapsto T_{\Gamma}(\iota).
	\]
	As the minimum amount of input information is typically related to the minimum runtime, one would be dealing with a rather exotic probabilistic model if $T_{\gprob}(\iota)$ were not a random variable. Indeed, note that the standard models of randomised algorithms (see \cite{Arora2007}) can be considered as RGAs (in particular, satisfying  \ref{property:PAlgorithmRTMeasurable}).
\end{remark}	

\begin{remark}[The purpose of a randomised general algorithm: universal lower bounds]
	As for a general algorithm, the purpose of a randomised general algorithm is to have a definition that will encompass every model of computation, which will allow lower bounds and impossibility results to be universal. Indeed, randomised Turing and BSS machines can be viewed as randomised general algorithms.
\end{remark}

We will, with a slight abuse of notation, also write $\mathrm{RGA}$ for the family of all randomised general algorithms for a given a computational problem and refer to the algorithms in $\mathrm{RGA}$ by  $\gprob$. 		
With the definitions above we can now make probabilistic version of the strong breakdown epsilon as follows.

\begin{definition}[Probabilistic strong breakdown epsilon]\label{prob_strong_break} Given a computational problem $\{\Xi,\Omega,\mathcal{M},\Lambda\}$, where $\Lambda = \{f_k \, \vert \, k \in \mathbb{N}, \, k\leq |\Lambda| \}$, 
	we define the \emph{probabilistic strong breakdown epsilon} $\epsilon_{\mathbb{P}\mathrm{B}}^{\mathrm{s}}: [0,1) \to \mathbb{R}$ according to
	\begin{align*}
		\epsilon_{\mathbb{P}\mathrm{B}}^{\mathrm{s}}(\mathrm{p}) = \sup\{&\epsilon \geq 0, \, \vert \, \forall \, \gprob \in \mathrm{RGA} \,\,\exists \, \iota \in \Omega \text{ such that } \mathbb{P}_{\iota}(\disM(\gprob_{\iota},\Xi(\iota)) > \epsilon) > \mathrm{p}\},
	\end{align*}
	where $\gprob_{\iota}$ is defined in \ref{property:PAlgorithmMeasurable} in Definition \ref{definition:ProbablisticAlgorithm}. 
\end{definition}

As already mentioned in previous sections, impossibility results for randomised algorithms can differ if one considers only those algorithms that halt on every input, leading to the following two definitions.

\begin{definition}[Halting randomised general algorithms]\label{def:halting_randomised}
	A  randomised general algorithm $\gprob$ for a computational problem $\{\Xi,\Omega,\mathcal{M},\Lambda\}$ is called a \emph{halting randomised general algorithm} (hRGA) if $\mathbb{P}_{\iota}(\gprob_{\iota} = \nh) = 0$, for all $\iota\in\Omega$.
\end{definition}

We denote the class of all halting randomised general algorithms by $\mathrm{hRGA}$.
\begin{definition}[Probabilistic strong halting breakdown epsilon]\label{prob_strong_break_halt}
Given the computational problem $\{\Xi,\Omega,\mathcal{M},\Lambda\}$, where $\Lambda =\{f_k \, \vert \, k \in \mathbb{N}, \, k\leq |\Lambda| \}$, we define the \emph{halting probabilistic strong Breakdown-epsilon} $\epsilon_{\mathbb{P}h\mathrm{B}}^{\mathrm{s}}: [0,1) \to \mathbb{R}$ according to
	\begin{align*}
		\epsilon_{\mathbb{P}h\mathrm{B}}^{\mathrm{s}}(\mathrm{p}) = \sup\{&\epsilon \geq 0, \, \vert \, \forall \, \gprob \in \mathrm{hRGA} \,\,\exists \, \iota \in \Omega \text{ such that } \mathbb{P}_{\iota}(\disM(\gprob_{\iota},\Xi(\iota)) > \epsilon) > \mathrm{p}\},
	\end{align*}
	where $\gprob_{\iota}$ is defined in \ref{property:PAlgorithmMeasurable} in Definition \ref{definition:ProbablisticAlgorithm}. 
\end{definition}

Note that the probabilistic strong Breakdown-epsilon is not a single number but a function of $\mathrm{p}$. Specifically, it is the largest $\epsilon$ so that the probability of failure with at least $\epsilon$-error is greater than $\mathrm{p}$. 
Similarly, there is a probabilistic version of the weak breakdown epsilon.

\begin{definition}[Probabilistic weak breakdown epsilon]\label{prob_weak_break}
Given the computational problem $\{\Xi,\Omega,\mathcal{M},\allowbreak\Lambda\}$, where $\Lambda = \{f_k \, \vert \, k \in \mathbb{N}, \, k\leq |\Lambda| \}$, we define the \emph{probabilistic weak breakdown epsilon} $\epsilon^{\mathrm{w}}_{\mathbb{P}\mathrm{B}}: [0,1) \to \mathbb{R}$ according to
	\begin{equation*}
		\begin{split}
			\epsilon_{\mathbb{P}\mathrm{B}}^{\mathrm{w}}(\mathrm{p}) &= \sup\{\epsilon \geq 0 \, \vert \,  \forall \, \gprob \in  \mathrm{RGA} \text{ and }  M \in \mathbb{N} \, \, \exists \, \iota \in \Omega \text{ such that } \\
			& \qquad \mathbb{P}_{\iota}(\disM(\gprob_{\iota},\Xi(\iota)) > \epsilon \text{ or } T_{\gprob}(\iota) > M) > \mathrm{p} \},
		\end{split}
	\end{equation*}	
	where $\gprob_{\iota}$ is defined in \ref{property:PAlgorithmMeasurable} in Definition \ref{definition:ProbablisticAlgorithm}.
\end{definition}
As for the (deterministic)  weak breakdown epsilon discussed above, it is easy to see that the probabilistic weak breakdown epsilon is independent of the enumeration of $\Lambda$.
The probabilistic weak breakdown epsilon describes a weaker form of failure than the probabilistic strong breakdown epsilon. In particular,  the weak breakdown epsilon of $\mathrm{p}$ is the largest $\epsilon$ so that for, any randomised algorithm and $M \in \mathbb{N}$, the probability of either getting an error of size at least $\epsilon$ or having spent runtime longer than $M$, is greater than $\mathrm{p}$. 
The connection between the different breakdown epsilons will be summarised in Proposition \ref{lemma:BreakdownLinks} further below.

\subsection{Different types of impossibility results\label{sec:different-impossibility-results}}
There are several non-equivalent statements about the non-existence of algorithms for a computational problem $\{\Xi,\Omega,\mathcal{M},\Lambda\}$ with $\Delta_1$-information that we will discuss. For a fixed $\varepsilon > 0$, the statements are as follows:
\begin{itemize}
\item[(i)] $\exists \,\hat \Lambda \in \mathcal{L}^1(\Lambda)$ such that, for the computational problem $\{\Xi,\Omega,\mathcal{M},\hat \Lambda\}$, we have $\epsilon_{\mathrm{B}}^{\mathrm{s}} \geq \varepsilon$.
\item[(ii)] When considering the computational problem $\{\Xi,\Omega,\mathcal{M},\Lambda\}^{\Delta_1}$, we have $\epsilon_{\mathrm{B}}^{\mathrm{s}} \geq \varepsilon$. 
\item[(iii)] $\{\Xi,\Omega,\mathcal{M},\Lambda\}^{\Delta_1} \notin \Delta_1^G.$
\end{itemize}
Note that $\text{ (i) } \Rightarrow \text{ (ii) } \Rightarrow \text{ (iii) }.$ Observe also that (i) says that there is a particular   
choice $\hat \Lambda$ of collections of approximations that provide $\Delta_1$-information for $\Lambda$ such that no algorithm, given this specific $\Delta_1$-information, will be able to secure $\varepsilon$ accuracy. Note that the strong breakdown epsilon in (i) and (ii) could be replaced with any of the other breakdown epsilons (at a fixed $\mathrm{p}$, in the case of the probabilistic breakdown epsilons), and the implication $\text{ (i) } \Rightarrow \text{ (ii) }$ would still hold. However, the implication $\text{ (ii) } \Rightarrow \text{ (iii) }$ holds only for the various strong breakdown epsilons.  

\begin{remark}[Turing's classical definition of non-computability]
Note that in the SCI hierarchy framework, the classical Turing definition of non-computability \cite{Turing_Machine,Ko1991ComplexityTO} is equivalent to  $\{\Xi,\Omega,\mathcal{M},\Lambda\}^{\Delta_1} \notin \Delta_1^A.$ However, as the main results in this paper demonstrate, in order to characterise the many facets of scientific computing one needs the richness of the SCI hierarchy framework. 
\end{remark}

\subsection{Theorem \ref{Cor:main} -- Part (i) and (ii) in the SCI language}\label{sec:Thrm_in_SCI}
Let $\Xi$ denote the solution map (as in \eqref{eq:the_Xi}) to any of the problems \eqref{problems} - \eqref{problems5}.
The statements in Theorem \ref{Cor:main} are well-defined up to the definition of (random) algorithms and their runtime. Now that we have defined these concepts and the breakdown epsilons, we state (i) and (ii) in Theorem \ref{Cor:main} in the precise SCI language. 
For $m,N\in\mathbb{N}$ such that $m\leq N$ and a class of inputs $\Omega$ for any of the problems \eqref{problems} - \eqref{problems5}, define 
\begin{equation}\label{eq:Omega_mN}
	\Omega_{m,N} = \{(y,A)\in\Omega  \, \vert \,(y,A) \in \mathbb{R}^m\times  \mathbb{R}^{m\times N}  \}, \quad \mathcal{M}_N = \mathbb{R}^N.
\end{equation}
In the case of linear programming, we set $c = (1,\hdots, 1) \in \mathbb{R}^N$ for each fixed dimension $N$. We then have the following.

\begin{proposition}\label{Cor:main_SCI}
	Let $\Xi$ denote the solution map to any of the problems \eqref{problems} - \eqref{problems5} with the regularisation parameters satisfying $\delta\in[0,1]$, $\lambda\in(0,1/3]$, and $\tau\in[1/2,2]$ (and additionally being rational in the Turing case) and let the metric on $\mathcal{M}_N$ be induced by the $\|\cdot\|_p$-norm, for an arbitrary $p\in[1,\infty]$. Let $K \geq 1$ be an integer.
	There exist a set of inputs 
	\begin{equation}\label{eq:Omega_set}
		\Omega=\bigcup_{\substack{m,N\in\mathbb{N}\\ 4 \leq m < N}}\,\Omega_{m,N}, \text{ so that } \,\Xi:  \Omega_{m,N}\rightrightarrows \mathcal{M}_N,
	\end{equation}
	where $\Omega_{m,N}$ and $\mathcal{M}_N$ are defined as in \eqref{eq:Omega_mN}, and $\mathcal{M}_N$ is equipped with the $\|\cdot\|_p$ norm for some $p \in [1,\infty]$, as well as sets of $\Delta_1$-information $\hat \Lambda_{m,N} \in \mathcal{L}^1(\Lambda_{m,N})$ such that we have the following. 
	\begin{itemize}[leftmargin=8mm]
		\item[(i)]
		For the computational problem $\{\Xi,\Omega_{m,N},\mathcal{M}_N, \hat \Lambda_{m,N}\}$ with any $m,N\in\mathbb{N}$, $4 \leq m<N$, we have 
 $\epsilon_{\mathbb{P}h\mathrm{B}}^{\mathrm{s}}(\mathrm{p}) > 10^{-K}$, for all $\mathrm{p} > \frac{1}{2}$. Hence $\{\Xi,\Omega_{m,N},\mathcal{M}_N,\Lambda_{m,N}\}^{\Delta_1} \notin \Delta_1^G$.
		\item[(ii)] 
		For the same $\{\Xi,\Omega_{m,N},\mathcal{M}_N, \hat \Lambda_{m,N}\}$ as in (i) we have $\epsilon_{\mathbb{P}\mathrm{B}}^{\mathrm{s}}(\mathrm{p}) > 10^{-K}$, for all $\mathrm{p} > \frac{2}{3}$.
		However, when considering the problems $\{\Xi,\Omega_{m,N},\mathcal{M}_N,\Lambda_{m,N}\}^{\Delta_1} = \{\tilde \Xi,\tilde \Omega_{m,N},\mathcal{M}_N, \tilde \Lambda_{m,N}\}$, there exists a randomised algorithm $\Gamma^{\mathrm{ran}}$ (a randomised Turing machine or a randomised BSS machine), with a non-zero probability of not halting, that takes the dimensions $m$, $N$ and any $\tilde \iota \in \tilde \Omega_{m,N}$ as input (see Remark \ref{rem:oracles}) and satisfies
		\begin{equation*}
			\begin{split}
				\mathbb{P}_{\tilde{\iota}}\big(\disM(\Gamma^{\mathrm{ran}}(m, N, \tilde \iota), \tilde\Xi(\tilde\iota)) \leq 10^{-K}\big) \geq \frac{2}{3}.
			\end{split}
		\end{equation*}
	\end{itemize}
	The statements above are true even when we require the inputs in $\Omega_{m,N}$ to be well-conditioned for all $m$, $N$ and bounded from above and below. In particular, for any input $\iota = (y,A) \in \Omega_{m,N}$ we have 
	$\mathrm{Cond}(AA^*)\leq 3.2$, $C_{\mathrm{FP}}(\iota)\leq 4$, $\mathrm{Cond}(\Xi) \leq 179$, $\|y\|_\infty\leq 2$, and $\|A\|_{\max} = 1$.
\end{proposition}

\subsection{Runtime of algorithms}\label{sec:comp_runtime}
In order to provide exact statements of the rest of Theorem \ref{Cor:main} we need the precise definition of the runtime of an algorithm, 
for which one has to specify a computational model. Recall that we want to prove results about algorithms for problems  $\{\Xi,\Omega,\mathcal{M},\Lambda\}^{\Delta_1}$ with $\Delta_1$-information, as such results are stronger than, say, the corresponding results for $\{\Xi,\Omega_{m,N},\mathcal{M}_N, \hat \Lambda\}$ with a particular choice of $\hat \Lambda \in \mathcal{L}^1(\Lambda)$. We have already discussed specific computational models in Definition \ref{def:arithmetic_tower} and Remark \ref{rem:recursiveness}, however, in order to make precise statements about complexity we need to specify such models more closely. Concretely, we will consider the following models.
\begin{itemize}
	\item[(i)] ({\bf Turing model}).
	We follow the classical setup in \cite{Ko1991ComplexityTO}, wherein the components of the inputs $\tilde \iota = \{f_{j,n}(\iota)\}_{j,n \in \beta \times \mathbb{N}}\in\tilde\Omega$ are provided by an oracle tape that, at various times during the computation, contains a binary string representing one of the dyadic numbers $f_{j,n}(\iota)$. Specifically, the Turing machine can make a query to the \emph{input oracle} $\mathscr{O}$, which then writes $f_{j,n}(\iota)$ onto the oracle tape at cost $j+n$, always choosing the unique finite binary string representing $f_{j,n}(\iota)$. Now, given a Turing machine $\Gamma$ for the problem $\{\Xi,\Omega,\mathcal{M},\Lambda\}^{\Delta_1}$, we define its runtime on $\tilde \iota$ as follows:
	\begin{equation*}
		\begin{split}
			\mathbf{Runtime}_{\Gamma}(\tilde \iota) &= \text{ the number of steps performed by the Turing machine $\Gamma$ before halting}\\
			& \qquad \text{ + the sum of  costs of all the calls to the oracle for $\tilde \iota$ (with the cost as above).}
		\end{split}
	\end{equation*}
	Note that this model accounts implicitly for the cost of all the calls to the oracle, as the output of the oracle is printed on the work tape, and hence read by the machine.   
	\item[(ii)] ({\bf Arithmetic Turing model}). In the case where one only wishes to count the number of arithmetic operations and comparisons that an algorithm performs, there is no canonical account of the cost of calling the oracle, as opposed to the classical Turing model above. Hence, in this model (as discussed in \S \ref{rem:extended_model} and (see for example 
	Lov\'{a}sz \cite[p. 36]{lovasz1987algorithmic})), it is natural to adopt the convention that the cost of calling the oracle $\mathscr{O}$ (as specified in item (i) above) is given by $\mathrm{pol}(n)$, for all $j \in \beta$, $n\in\mathbb{N}$, and $\iota \in \Omega$, for some non-zero polynomial $\mathrm{pol}$ fixed beforehand.  
	Thus, we define the arithmetic runtime of an algorithm $\Gamma$ in the Turing model  on $\tilde \iota$  as follows:
	\begin{equation*}
		\begin{split}
			\mathbf{Runtime}_{\Gamma}(\tilde \iota) &= \text{ the number of arithmetic operations and comparisons} \\
			& \qquad \qquad\text{ performed by $\Gamma$ before halting }\\ 
			& \qquad \text{ + the sum of  costs of all the calls to the oracle for $\tilde \iota$ (with the cost as above).}
		\end{split}
	\end{equation*}
	\item[(iii)] ({\bf The BSS model}). In this model, the BSS machine is equipped with an extra oracle node $\mathscr{O}$ for $\tilde \iota = \{f_{j,n}(\iota)\}_{j,n \in \beta \times \mathbb{N}}\in\tilde\Omega$ that outputs $f_{j,n}$ on input $(j,n)$ (see \cite{BCSS} for details). As discussed in \S \ref{rem:extended_model} and in item (ii) above, it is customary to assume that the cost of calling the oracle is polynomial in $n$, as in the arithmetic Turing model. Thus, given a BSS machine $\Gamma$ for the problem $\{\Xi,\Omega,\mathcal{M},\Lambda\}^{\Delta_1}$, we define its runtime on $\tilde \iota$ as follows:
	\begin{equation*}
		\begin{split}
			\mathbf{Runtime}_{\Gamma}(\tilde \iota) &= \text{ the number of arithmetic operations and comparisons} \\
			& \qquad \qquad\text{ performed by the BSS machine $\Gamma$ before halting }\\ 
			& \qquad \text{ + the sum of  costs of all the calls to the oracle for $\tilde \iota$ (with the cost as above).}
		\end{split}
	\end{equation*}
\end{itemize}

\subsection{Theorem \ref{Cor:main} -- Part (iii) - (v) in the SCI language}
We follow the setup in \S \ref{sec:Thrm_in_SCI}. Note that with the definitions of runtime in \S \ref{sec:comp_runtime} all satisfy \eqref{eq:runtimeGeqD}, and hence we have that, for a given problem $\{\Xi,\Omega,\mathcal{M},\Lambda\}^{\Delta_1}$ with $\Lambda$ finite and a probabilistic weak breakdown epsilon $\epsilon_{\mathbb{P}\mathrm{B}}^{\mathrm{w}}(\mathrm{p}) > r$, for some $\mathrm{p}, r > 0$, any randomised algorithm will require an arbitrarily long runtime in order to achieve accuracy of at least $r$ with probability at least $\mathrm{p}$.   

We are now ready to formalise the rest of Theorem \ref{Cor:main}.

\begin{proposition}\label{Cor:main_SCI_cont}
There exist $\Omega$ and $\hat \Lambda_{m,N} \in \mathcal{L}^1(\Lambda_{m,N})$ as described in \eqref{eq:Omega_set} for which the conclusion of Proposition \ref{Cor:main_SCI} holds with $K \geq 2$, and we additionally have the following.
	\begin{itemize}[leftmargin=8mm]
		\item[(iii)] 
		For the problem $\{\Xi,\Omega_{m,N},\mathcal{M}_N, \hat \Lambda_{m,N}\}$ with any $m,N\in\mathbb{N}$, $4 \leq m<N$, we have that $\epsilon_{\mathbb{P}\mathrm{B}}^{\mathrm{w}}(\mathrm{p}) > 10^{-(K-1)}$, for all $\mathrm{p} > \frac{1}{2}$. However, when considering the problems  $\{\Xi,\Omega_{m,N},\mathcal{M}_N,\allowbreak\Lambda_{m,N}\}^{\Delta_1} = \{\tilde \Xi,\tilde \Omega_{m,N},\mathcal{M}_N, \tilde \Lambda_{m,N}\}$, there exists an algorithm $\Gamma$ (a Turing machine or a BSS machine) that takes the dimensions $m$, $N$ and any $\tilde\iota \in \tilde\Omega_{m,N}$ as input (see Remark \ref{rem:oracles}) and satisfies
		\[
		\disM(\Gamma(m, N, \tilde \iota), \tilde\Xi(\tilde\iota)) \leq 10^{-(K-1)}.
		\]
		\item[(iv)]
		There exists a polynomial $\mathrm{pol}: \mathbb{R} \rightarrow \mathbb{R}$  and an algorithm $\Gamma$ (a Turing machine or a BSS machine) that takes the dimensions $m$, $N$ and any $\tilde\iota \in \tilde\Omega_{m,N}$ as input (see Remark \ref{rem:oracles}) and satisfies
		\[
		\disM(\Gamma(m, N, \tilde \iota), \tilde\Xi(\tilde\iota)) \leq 10^{-(K-2)}.
		\]
		Furthermore, in the Turing case, the arithmetic runtime and the space complexity of the Turing machine are bounded by $\mathrm{pol}(n_{\mathrm{var}})$, where $n_{\mathrm{var}}=mN+m$ is the number of variables, and the number of digits read from the oracle tape is bounded by $\mathrm{pol}(\log(n_{\mathrm{var}}))$. In particular, the runtime in the Turing model is polynomial in $n_{\mathrm{var}}$. In the BSS case, the runtime is likewise bounded by $\mathrm{pol}(n_{\mathrm{var}})$.
	\end{itemize}	
\end{proposition}

In the proof of Proposition \ref{Cor:main_SCI_cont} we will additionally need the following concepts of the size of a dyadic rational and the encoding length of rational vectors.

	\begin{definition}[Size of dyadic rationals]
		For a dyadic rational $d$ given by its binary representation $d=\pm s_ns_{n-1}\cdots s_0.t_1t_2\cdots t_m$, we call $m$ its precision, and we define its bit size as $m+n+3$. The bit size of a dyadic rational vector or matrix is defined as the sum of the bit sizes of its entries. We define the bit size of a vector of dyadic rationals as the sum of bit sizes of its entries.
	\end{definition}
	
	\begin{definition}[Encoding length of rational numbers]
		We define the encoding length of an integer $n\in\mathbb{Z}$ as $\length(n)=1+\lceil{\log_2(|n|+1)}\rceil$. For a rational $p/q\in\mathbb{Q}$, where $p$ and $q>0$ are coprime integers, we define the encoding length of $p/q$ as $\length(p/q)=\length(p)+\length(q)$. Similarly, if $w$ is a vector of rationals, we define the encoding length $\length(w)$ as the sum of the encoding lengths of the components of $w$.
	\end{definition}

\subsection{Theorem \ref{th:smale_comp_sens} in the SCI language}
We recall the setup for Theorem \ref{th:smale_comp_sens}. 
Fix real constants $\rho\in(1/3,1)$, $\tau>10$, $b_1>3$, and $b_2 > 6$. For $s,m,N\in\mathbb{N}$ such that $m\leq N$ and $\epsilon \in [0,1]$, define 
\begin{equation}\label{eq:Omega_smN2}
	\Omega^{\epsilon}_{s,m,N} = \{(y,A) \in \mathbb{R}^m\times  \mathbb{R}^{m\times N} \, \vert \, (y,A) \text{ satisfies \eqref{eq:conditions_on_Ay2}}\},
\end{equation}
where 
\begin{equation}\label{eq:conditions_on_Ay2}
	\begin{split}
		&A \text{ satisfies the RNP } \eqref{eq:RNP} \text{ of order $s$ with parameters $\rho$ and $\tau$,}\\
		\|y-Ax\|_2 &\leq \epsilon \text{ for some $x$ that is $s$-sparse, and } \|y\|_2 \leq b_1, \,\, \|A\|_2 \leq b_2\sqrt{N/m} .
	\end{split}
\end{equation}
Finally, define  
\begin{equation}\label{eq:the _Omega2}
	\Omega^{\epsilon}=\bigcup_{\substack{s,m,N\in\mathbb{N}\\ m\leq N}}\,\Omega^{\epsilon}_{s,m,N}.
\end{equation}

\begin{proposition}\label{prop:CSResult}
	Let $\xibp$ denote the solution map to the $\ell^1$-BP problem \eqref{problems3} and consider the $\|\cdot\|_{2}$-norm for measuring the error. If $\Omega_{s,m,N}^\epsilon$ and $\Omega^\epsilon$ are as in \eqref{eq:Omega_smN2} and \eqref{eq:the _Omega2} the following holds.	
	\begin{itemize}[leftmargin=8mm]
		\item[(i)] There exists a constant $C > 0$, independent of $\rho,\tau,b_1$ and $b_2$, such that if we fix $s = 2^k$ for some $k \in \mathbb{N}$ and any $m, N$ such that $N > m$ with $m\geq C s\log^2(2s)\log(N)$, then we have the following. For an arbitrary $\delta \in (0,1]$ consider the computational problem $\{\xibp,\Omega^{0}_{s,m,N},\mathcal{M},\Lambda\}^{\Delta_1}$, where $\mathcal{M} = \mathbb{R}^N$. Let
		\[
		K \geq \left\lceil\log_{10}\left(2\delta^{-1}\right)\right \rceil.
		\]
		Then, for any $\mathrm{p} > \frac{1}{2}$, we have $\epsilon_{\mathbb{P}h\mathrm{B}}^{\mathrm{s}}(\mathrm{p}) > 10^{-K}$. Hence 
		$\{\xibp,\Omega^{0}_{s,m,N},\mathcal{M},\Lambda\}^{\Delta_1} \notin \Delta_1^G$.
		
		\item[(ii)] There exists a polynomial $\mathrm{pol}: \mathbb{R}^2 \rightarrow \mathbb{R}$ and, for every $\delta\in[0, (1-\rho)/(16\tau)]$, there exists an algorithm $\Gamma_\delta$ (a Turing machine or a BSS machine) that takes any dimensions $m$, $N$, the accuracy parameter $K$, and any $\tilde \iota \in \tilde \Omega^{\delta}_{s,m,N}$ as input (as described in Remark \ref{rem:oracles}) and satisfies
		\[
		\disM(\Gamma_\delta(m, N, K, \tilde \iota), \tilde{\Xi}_{\mathrm{BP}}(\tilde \iota)) \leq 10^{-K}
		\]
		for all $K \in \mathbb{N}$ satisfying
		\begin{equation}\label{eq:BoundOnKForPositive}
			K \leq \left  \lfloor \log_{10}\big((1-\rho)(16\tau)^{-1}\delta^{-1}\big) \right \rfloor.
		\end{equation}
		In the case that $\delta = 0$, \eqref{eq:BoundOnKForPositive} is to be interpreted as $K<\infty$.
		Moreover, the runtime of $\Gamma_\delta$ (steps performed by the Turing machine, arithmetic operations performed by the BSS machine) is bounded by $\mathrm{pol}(n,K)$, where $n = m+ mN$ is the number of variables. In particular, for $\delta = 0$, the runtime is bounded by $\mathrm{pol}(n,K)$ for all $K \in \mathbb{N}$.
		\item[(iii)]
		Consider the $\ell^1$-BP problem \eqref{problems3} with $\delta=0$.
		For any fixed $s \geq 3$, there are infinitely many pairs $(m,N)$ and inputs $\iota = (Ax,A) \in \Omega^0$, where $x\in\real^m$ is $s$-sparse and $A \in \real^{m \times N}$ is a subsampled Hadamard, Bernoulli, or Hadamard-to-Haar matrix such that
		\[
		C_{\mathrm{RCC}}(\iota) = \infty.
		\] 
		In particular, appreciating (iii) and (iv), there exist inputs in $\Omega$ with infinite RCC condition number, yet the problem is in P.
		
	\end{itemize} 
\end{proposition}
\begin{remark}
	The independence of the runtime on $\delta$ in part (ii) is a consequence of \eqref{eq:BoundOnKForPositive}.
\end{remark}
\begin{remark}
	One can strengthen the result in part (ii) in the following way: it is possible to use Nemirovski's surface-following method (concretely, the result in Section 4.2 in \cite{Nemirovski1995}) to devise an algorithm that solves the basis pursuit problem with $\ell^1$ regularisation up to precision $10^{-K}>16\tau\delta/(1-\rho)$ to obtain the explicit complexity bound $\mathcal{O}(1)\cdot N^{3.5}\cdot K\cdot \log\left(2+\frac{N\tau}{1-\rho}\right) $ in the arithmetic model of computation (more precisely, the BSS model with an oracle for the square root). We choose not to do so in this paper and instead base our proof of part (ii) on the ellipsoid algorithm to obtain simple polynomial time complexity bounds in both the BSS and Turing models of computation.
\end{remark}

\subsection{Computing the exit flag -- Theorem \ref{thm:ExitFlag} in the SCI language}\label{sec:exit-flag-formal}
Let $\{\Xi,\Omega,\mathcal{M},\Lambda\}^{\Delta_1}$ be a computational problem and recall the definition of $\{\Xi,\Omega,\mathcal{M},\Lambda\}^{\Delta_1} = \{\tilde \Xi,\tilde \Omega,\mathcal{M},\tilde \Lambda\}$ from Definition \ref{definition:Omega_tilde_Delta_m}. Suppose that we are interested in finding an algorithm $\Gamma$ that, for every $\tilde \iota \in \tilde \Omega$, computes an approximation to $\tilde \Xi(\tilde \iota)$ with $\kappa$ accuracy (i.e. so that $\disM(\Gamma(\tilde \iota),\Xi(\tilde \iota)) \leq \kappa$). If the breakdown epsilon 
$\epsilon_{\mathrm{B}}^{\mathrm{s}}$ is greater than $\kappa$ then this task is impossible -- there will be at least one input for which the algorithm fails to get $\kappa$ accuracy. However, that does not mean that a given algorithm will fail on \textit{all} possible inputs. There may be a set of inputs for which the algorithm succeeds in producing an output that is $\kappa$ away from a true solution. This leads to the following exit flag question: 
\begin{displayquote}
	\normalsize
	{\it Can we design an algorithm that identifies whether or not $\Gamma$ fails on a given input (meaning that $\Gamma$ fails to produce $\kappa$ accuracy on its input)?}
\end{displayquote}

Phrased more precisely in the language of computational problems, given a recursive (either in the Turing or the BSS model) algorithm $\Gamma$ for $ \{\tilde \Xi,\tilde \Omega,\mathcal{M},\tilde \Lambda\}$, we define 
\begin{equation}\label{eq:Omega_exit}
	 \tilde\Omega_{\Gamma}:= \{\tilde \iota \in \tilde \Omega \, \vert \, \disM(\Gamma(\tilde \iota),\tilde \Xi(\tilde \iota)) \leq \kappa\}\subset \tilde \Omega,
\end{equation}
then $\epsilon_{\mathrm{B}}^{\mathrm{s}} > \kappa$ implies that $\tilde \Omega_\Gamma$ is a strict subset of $\tilde \Omega$. Now, defining
\begin{equation}\label{eq:Xi_exit}
	\Xi^E: \tilde \Omega \ni \tilde \iota \mapsto
	\begin{cases}
		1 & \text{ if } \tilde \iota \in  \tilde\Omega_\Gamma \\
		0 & \text{ if } \tilde \iota \in \tilde \Omega\setminus  \tilde\Omega_\Gamma,
	\end{cases}
\end{equation}
the exit flag problem is to design an algorithm $\Gamma^E$ to solve the computational problem 
\begin{equation}\label{eq:comp_prob_exit}
	\{\Xi^E,\Omega,\{0,1\},\Lambda\}^{\Delta_1} = \{\Xi^E,\tilde \Omega,\{0,1\},\tilde \Lambda\},
\end{equation}
where the metric on the space $\{0,1\}$ is inherited from $\real$.

\begin{remark}[Key assumption]\label{rem:EFAlgorithmAssumption}
	Of course, this problem is trivial if the algorithm produces outputs that are far away from the set $\Xi(\Omega)$. This is a somewhat contrived scenario as such algorithms would not be sensible candidates for attempting to solve the problems defined in equations \eqref{problems}-\eqref{problems5}.  We thus need to make a technical assumption on the type of algorithms we will analyse for the exit flag problem. Concretely, we fix an $\alpha < \kappa$ and assume that our algorithm $\Gamma$ defined on $\Omega$ satisfies
	\begin{equation}\label{assumption:AlgorithmCloseToTheRange}
		\disM(\Gamma(\tilde \iota),\tilde \Xi(\tilde \Omega)) < \alpha\text{ for all  } \tilde \iota \in \tilde \Omega.
	\end{equation}
\end{remark}

Given the computational problem $\{\Xi^E,\tilde \Omega,\{0,1\},\tilde \Lambda\}$, one may ask where it lies in the SCI hierarchy. We will provide results on this, however, in some circumstances we will actually prove stronger results by considering the problem of computing the exit flag with additional access to an oracle that provides an exact solution to the problem that $\Gamma$ is trying to solve. In particular, one can ask:
\begin{displayquote}
	\normalsize
	{\it Given an oracle for the true solution, can we design an algorithm that identifies whether or not $\Gamma$ fails on a given input (meaning that $\Gamma$ fails to produce $K$ correct digits when given the input)?}
\end{displayquote}
Concretely, we will assume that we are allowed to design the exit flag algorithm $\Gamma^E$ so that $\Gamma^E$ has access to both and element $\rho$ of $ \tilde \Xi(\tilde \iota)$ as well as the input $\tilde \iota$. We will consider $\Gamma^E$ to be successful only if it correctly computes the exit flag for every such $\rho$. Indeed, it is reasonable to preclude the exit algorithm from being able to access a ''convenient'' element of $ \tilde \Xi(\tilde \iota)$, as this would make it too powerful for any nontrivial statements to be made.

\subsubsection{Formalising the oracle computational problem}
In order to make the concept of an oracle for the true solution precise we need to formalise the definition of an oracle computational problem. 
Consider two computational problems  $\{\Xi_1,\Omega,\mathcal{M}_1,\Lambda_1\}$ and $\{\Xi_2,\Omega,\mathcal{M}_2,\Lambda_2\}$  with the same input set $\Omega$, and assume that $\mathcal{M}_2 \subset \mathbb{C}^M$. 
Given $\omega \geq 0$, suppose that  functions $\{g_k\}_{k=1}^M$, $g_k : \Omega_1 \rightarrow \mathbb{D}$, are such that 
\begin{equation}\label{eq:the_gk}
	 \{g_k( \iota)\}_{k=1}^M \in \mathcal{B}^{\infty}_{\omega}( \Xi_2( \iota)).
\end{equation}
Then $\{g_k( \iota)\}_{k=1}^M$ is an $\omega$-accurate approximation of an element of $ \Xi_2( \iota)$. We can thus define
\begin{equation}\label{def:mathcalLwOracle}
\mathcal{L}^{\mathcal{O}, \omega,  \Xi_2}( \Lambda_1)=\{ \Lambda_1 \cup \{g_k\}_{k=1}^M \,\vert\, g_k :  \Omega \rightarrow \mathbb{D}, k\in\{1,\dots,M\},\text{ satisfy \eqref{eq:the_gk}}\}
\end{equation} 
Note that it is crucial to allow for the approximation parameter $\omega$ since a Turing machine would be unable to access the true solution if it has irrational entries. 

Now, for a fixed $\Lambda_1^+\in \mathcal{L}^{\mathcal{O}, \omega,  \Xi_2}(\Lambda_1)$, the computational problem $\{ \Xi_1, \Omega,\mathcal{M}_1, \Lambda_1^+\}$ would allow the algorithm to access information about $ \Xi_2$.
However, as discussed above, a sensible algorithm relying on ``oracle information'' ought to work with every choice of such an oracle. We thus define a new oracle computational problem $\{ \Xi_1^{\mathcal{O}}, \Omega^{\mathcal{O}},\mathcal{M}_1, \Lambda_1^{\mathcal{O}}\}$ by analogy to Definition \ref{definition:Omega_tilde_Delta_m} according to
\begin{equation}\label{eq:exit-flag-general-Omega}
 \Omega^{\mathcal{O}} =  \Omega^{\mathcal{O}}(\omega) = 
\left\{  \iota^{\mathcal{O}} =  \iota  \oplus \{g_k( \iota)\}_{k=1}^M
\, \vert \,  \iota \in  \Omega, \, \{g_k\}_{k=1}^M \text{ satisfy } \eqref{eq:the_gk} \right\},
\end{equation}
$ \Xi_1^{\mathcal{O}}( \iota^{\mathcal{O}}) =  \Xi_1( \iota)$ and  
$ \Lambda_1^{\mathcal{O}} = \{ f^{\mathcal{O}}\, \vert \,f\in\Lambda_1\}  \cup \{g^{\mathcal{O}}_k\}_{k=1}^M$, 
where  the $f^{\mathcal{O}}$ and $g^{\mathcal{O}}_k$ are defined as
\[
 f^{\mathcal{O}}( \iota^\mathcal{O}) =  f( \iota), \quad 
g^{\mathcal{O}}_k( \iota^\mathcal{O}) = g_k( \iota).
\]
The dependency on  $\omega$ in the definition of $ \Omega^{\mathcal{O}}(\omega)$ will usually be suppressed to lighten the notation. We can now make the following formal definition.

\begin{definition}[Oracle computational problem]\label{def:oracle}
	Given two computational problems  $\{\Xi_1,\Omega,\mathcal{M}_1,\Lambda_1\}$ and $\{\Xi_2,\Omega,\mathcal{M}_2,\Lambda_2\}$ as specified above, we say that 
	\[
	\{  \Xi_1, \Omega,\mathcal{M}_1, \Lambda_1\}^{\mathcal{O},\omega, \{ \Xi_2, \Omega,\mathcal{M}_2, \Lambda_2\}}:=\{ \Xi_1^{\mathcal{O}}, \Omega^{\mathcal{O}},\mathcal{M}_1, \Lambda_1^{\mathcal{O}}\}
	\]
	where $\Omega^{\mathcal{O}}$, $\Xi_1^{\mathcal{O}}$, and $\Lambda_1^{\mathcal{O}} $ are as specified above,
	is the oracle computational problem with respect to $\{\Xi_2,\Omega,\mathcal{M}_2,\Lambda_2\}$. Whenever it is clear which computational problem the oracle problem is relative to, we will simply write $\{  \Xi_1, \Omega,\mathcal{M}_1 ,\Lambda_1\}^{\mathcal{O},\omega}$
	for $\{  \Xi_1, \Omega,\mathcal{M},_1 \Lambda_1\}^{\mathcal{O}, \omega, \{ \Xi_2, \Omega,\mathcal{M}_2, \Lambda_2\}}$ .
\end{definition}

\subsubsection{Computing the exit flag with an oracle}

Returning to the exit flag problem, one could also ask the opposite question to what we asked above: 
\begin{displayquote}
	\normalsize
	{\it Given an oracle for the exit flag can we design an algorithm that produces $\omega$ accurate solution to the original problem?}
\end{displayquote}
A positive answer to this question and a negative answer to the original question can be seen as evidence that the exit flag problem is strictly harder than the original problem, which is the topic of the following proposition formalising Theorem \ref{thm:ExitFlag}.

\begin{proposition}[Impossibility of computing the exit flag]\label{prop:ExitFlag}
	Let $\Xi$ denote the solution map to any of the problems \eqref{problems} - \eqref{problems5} with the regularisation parameters satisfying $\delta\in[0,1]$, $\lambda\in(0,1/3]$, and $\tau\in[1/2,2]$ (and additionally being rational in the Turing case)  and let the metric on $\mathcal{M}$ be induced by the $\|\cdot\|_p$ norm. Let $K \in \mathbb{N}$ and fix $\alpha$ and  $\omega$ so that $0<\alpha\leq \omega<10^{-K}$. Then, for any fixed dimensions $m < N$ with $m \geq 4$, there exists a class of inputs $\Omega$ such that, if $\Gamma$ is a general algorithm for the computational problem $\{\Xi,\Omega,\mathcal{M},\Lambda\}^{\Delta_1} = \{\tilde \Xi,\tilde \Omega,\mathcal{M},\tilde \Lambda\}$ satisfying \eqref{assumption:AlgorithmCloseToTheRange} (where we write $\Lambda $ for $\Lambda_{m,N}$ to lighten the notation) and $\{\Xi^E,\Omega,\{0,1\},\Lambda\}^{\Delta_1} = \{\Xi^E,\tilde \Omega,\{0,1\},\tilde \Lambda\}$ is the exit flag computational problem defined in \eqref{eq:comp_prob_exit}, we then have the following. 
	\begin{itemize}[leftmargin=8mm]
		\item[(i)]
		There exists a $\tilde \Lambda^+  \in \mathcal{L}^{\mathcal{O}, \omega, \tilde \Xi}(\tilde \Lambda)$ such that, for  $\{ \Xi^E,\tilde \Omega,\{0,1\},\tilde \Lambda^+\}$ we have $\epsilon_{\mathbb{P}\mathrm{B}}^{\mathrm{s}}(\mathrm{p}) \geq 1/2$, for all $\mathrm{p} > \frac{1}{2}$.
		\item[(ii)]
		The problem of computing the exit flag of $\Gamma$ is strictly harder than computing a $K$ digit approximation to $\Xi$ in the following sense: For the oracle problem $\{ \tilde \Xi,\tilde \Omega,\mathcal{M},\tilde \Lambda\}^{\mathcal{O},\omega}$ with respect to $\{ \Xi^E,\tilde \Omega,\{0,1\},\tilde \Lambda\}$ we have $\epsilon_{\mathrm{B}}^{\mathrm{s,A}} \leq 10^{-K}$ (see Remark \ref{rem:breakdown-eps}). However, when considering  the oracle problem $\{ \Xi^E,\tilde \Omega,\{0,1\},\tilde \Lambda\}^{\mathcal{O},\omega}$ with respect to $\{\tilde \Xi,\tilde \Omega,\mathcal{M},\tilde \Lambda\}$, we have $\epsilon_{\mathrm{B}}^{\mathrm{s}}\geq 1/2$ (which follows from (i)).
		\item[(iii)] 
		If $\Xi$ is the solution map to either linear programming or basis pursuit, there exists a class of inputs $\Omega^{\sharp} \neq \Omega$ such that, if $\Gamma$ is a general algorithm for the computational problem $\{\Xi,\Omega^{\sharp} ,\mathcal{M},\allowbreak\Lambda\}^{\Delta_1}$ satisfying \eqref{assumption:AlgorithmCloseToTheRange} with $\alpha$, then there is a $\hat \Lambda \in \mathcal{L}^1(\Lambda)$ so that for $\{\Xi^E,\Omega^{\sharp} ,\{0,1\},\hat \Lambda\}$ we have $\epsilon_{\mathbb{P}\mathrm{B}}^{\mathrm{s}}(\mathrm{p}) \geq 1/2$, for all $\mathrm{p} > \frac{1}{2}$. However, 
		if we consider the oracle problem $\{ \Xi^E,\tilde \Omega^{\sharp},\{0,1\},\tilde \Lambda\}^{\mathcal{O},\omega}$ with respect to $\{\Xi,\Omega^{\sharp},\mathcal{M},\Lambda\}^{\Delta_1}$, then $\{ \Xi^E,\tilde \Omega^{\sharp},\{0,1\},\tilde \Lambda\}^{\mathcal{O},\omega} \in \Delta_1^A$.  		
	\end{itemize}
	The statements above are true even when we require the inputs in $\Omega_{m,N}$ to be well-conditioned for all $m$, $N$ and bounded from above and below. In particular, for any input $\iota = (y,A) \in \Omega_{m,N}$ we have 
	$\mathrm{Cond}(AA^*)\leq 3.2$, $C_{\mathrm{FP}}(\iota)\leq 4$, $\mathrm{Cond}(\Xi) \leq 179$, $\|y\|_\infty\leq 2$, and $\|A\|_{\max} = 1$.
\end{proposition}

\subsection{The full SCI hierarchy and the feasibility part of the extended Smale's 9th problem}
In order to prove Theorem \ref{thm:Smales9} we need the full SCI hierarchy. Building on the SCI hierarchy introduced in Definition \ref{definition:sciHierarchy}, one may assume that the metric space $\mathcal{M}$, say $\mathcal{M} = \mathbb{R}$ or $\mathcal{M} = \{0,1\}$ with the standard metric, is equipped with extra structure, say a total order, that allows one to define convergence from above or below of functions valued in $\mathcal{M}$. By analogy to using the metric $d_{\mathcal{M}}$ in the approximation property \eqref{conv} of towers of algorithms, the extra structure on $\mathcal{M}$ allows us to define for a new type of approximation. As we argue below, this is important, for example, in computer-assisted proofs and scientific computing. 

For simplicity of exposition, in this subsection we consider only computational problems whose solution map $\Xi$ is single-valued.

\begin{definition}[The SCI Hierarchy for a totally ordered set]\label{def:tot_ord}
	Given the setup in Definition \ref{definition:sciHierarchy} and additionally assuming that $\mathcal{M}$ is totally ordered, we define the following sets: 
	\begin{equation*}
		\begin{split}
			\Sigma^{\alpha}_0 &= \Pi^{\alpha}_0 = \Delta^{\alpha}_0,\\
			\Sigma^{\alpha}_{m} &= \{\{\Xi,\Omega,\mathcal{M},\Lambda\} \in \Delta_{m+1}^{\alpha} \ \vert \  \exists \text{ tower }  \{\Gamma_{n_{m}, \hdots, n_1}\} \text{ of height $m$ s.t. }\Gamma_{n_{m}}(\iota) \nearrow \Xi(\iota) \ \, \forall \iota \in \Omega\}, \\
			\Pi^{\alpha}_{m} &= \{\{\Xi,\Omega,\mathcal{M},\Lambda\} \in \Delta_{m+1}^{\alpha} \ \vert \  \exists  \text{ tower }\{\Gamma_{n_{m}, \hdots, n_1}\}  \text{ of height $m$ s.t. }\Gamma_{n_{m}}(\iota) \searrow \Xi(\iota) \ \, \forall \iota \in \Omega\},		\end{split}
	\end{equation*}
	for $m\in\mathbb{N}$, where $\nearrow$ and $\searrow$ denotes convergence from below and above respectively, and the towers of algorithms are assumed to consist of halting algorithms.
\end{definition}

In the special case when $\mathcal{M} = \{0,1\}$, Definition \ref{def:tot_ord} yields the full SCI hierarchy for arbitrary decision problems, where $1$ is interpreted as \emph{true} and $0$ as \emph{false}. 
One can then naturally consider a computational problem with $\Delta_1$-information $\{\Xi,\Omega,\mathcal{M},\Lambda\}^{\Delta_1}$ and ask in which of the sets $\Delta_k^{\alpha}$, $\Sigma_k^{\alpha}$, or $\Pi_k^{\alpha}$ it lies.

\subsection{Theorem \ref{thm:Smales9} in the SCI language}

Recall that, for $k \in \mathbb{N} \cup \{\infty\}$ and $M \in \real$, our decision problem of interest is to decide whether there is an $x \in \mathbb{R}^N$ such that
\begin{equation}\label{eq:Smales9restatement}
	\langle x , c \rangle_k \leq M \text{ subject to } Ax = y, \quad x \geq 0\text{ coordinatewise},
\end{equation}
where 
$
\langle x , c \rangle_k := \lfloor  10^{k} \langle x , c \rangle \rfloor  10^{-k}
$
with $c$ set to $(1,\hdots, 1) \in \mathbb{R}^N$ for each fixed $N$ as in  \S\ref{sec:Thrm_in_SCI}, corresponding to a computational problem with solution map
\begin{equation}\label{eq:Xi_K2}
	\Xi_k : \Omega \rightarrow \{0,1\},  \quad \Xi_k(\iota)=\begin{cases} 1, & \exists \, x\in\real^N \text{ satisfying } \eqref{eq:Smales9restatement}\\ 0, & \text{else} \end{cases},
\end{equation}
and set  $\Omega$ of inputs of the form $\iota=(y,A)\in\real^m\times \real^{m\times N}$. Now, for $k\in\mathbb{N}$, define the set
\begin{equation}\label{eq:The9Sets}
\SetNine_k:=\bigcup_{n\in\mathbb{N}\cup\{0\}}\big[(n+0.9)\cdot 10^{-(k-1)}, (n+1)\cdot10^{-(k-1)}\big) ,
\end{equation}
in other words, $\SetNine_k$ is the set of positive real numbers whose $k$-th digit after the decimal point is $9$. Note that, for all $k\in\mathbb{N}$, we have
\begin{equation*}
M\in \SetNine_k \;\implies\; \{x\in\real^N \,\vert\, \langle x , c \rangle_k\leq M\} = \{x\in\real^N \,\vert\, \langle x , c \rangle_{k-1}\leq M\} \;\implies\; \Xi_k=\Xi_{k-1},
\end{equation*}
and therefore the computational problem \eqref{eq:Xi_K2} is indistinguishable for $k=K$ and $k=K-1$ whenever $M\in\SetNine_K$. When this is not the case, the following proposition shows that the problem for $k=K$ and $k=K-1$ can have fundamentally different computational properties on the same set of inputs.

\begin{proposition}\label{prop:SmalesNinthProposition}
Consider a real $M\geq 0$ and an integer $K  \geq 2$. There exist a set of inputs
	\begin{equation*}
		\Omega = \bigcup_{m<N } \Omega_{m,N}\text{ so that }  \quad \Xi_k:  \Omega_{m,N} \rightarrow \{0,1\},
	\end{equation*}
	where $\Omega_{m,N}$ is defined as in \eqref{eq:Omega_mN} and sets of $\Delta_1$-information $\hat\Lambda_{m,N}\in \mathcal{L}^1(\Lambda_{m,N})$ such that
	\begin{itemize}[leftmargin=8mm]
		\item[(i)] $\{\Xi_K,\Omega_{m,N},\{0,1\},\hat\Lambda_{m,N}\} \notin \Sigma_1^G$, for all $m,N\in\mathbb{N}$ with $m<N$, and
		\item[(ii)] for $\{\Xi_K,\Omega_{m,N},\{0,1\},\hat\Lambda_{m,N}\}$ with any $m,N\in\mathbb{N}$ such that $m<N$, we have $\epsilon_{\mathbb{P}\mathrm{B}}^{\mathrm{s}}(\mathrm{p}) \geq \frac{1}{2}$, for all $\mathrm{p} > \frac{1}{2}$. 
\end{itemize}
Additionally, depending on whether $M$ is in one of $\SetNine_{K}$ or $\SetNine_{K-1}$, one or both of the following hold as well:
\begin{itemize}[leftmargin=8mm]
		\item[(iii)] 		
		For $\{\Xi_{K-1},\Omega_{m,N},\{0,1\},\hat\Lambda_{m,N}\}$ with any $m,N\in\mathbb{N}$ such that $m<N$,  we have $\epsilon_{\mathbb{P}\mathrm{B}}^{\mathrm{w}}(\mathrm{p}) \geq \frac{1}{2}$, for all $\mathrm{p} > \frac{1}{2}$.	\label{item:SNinthThirdPart}
		However, when considering the problems 
		\[
		\{\Xi_{K-1},\Omega_{m,N},\{0,1\},\Lambda_{m,N}\}^{\Delta_1}=\{\tilde\Xi_{K-1},\tilde\Omega_{m,N},\{0,1\},\tilde\Lambda_{m,N}\},
\]
		 there exists an algorithm $\Gamma$ (a Turing machine or a BSS machine) that takes the dimensions $m$, $N$ and any $\tilde\iota \in \tilde\Omega_{m,N}$ as input and satisfies
		\[
		\Gamma(m, N, \tilde \iota) = \tilde\Xi_{K-1}(\tilde\iota).
		\]
		\item[(iv)] Considering the problems $\{\Xi_{K-2},\Omega_{m,N},\{0,1\},\Lambda_{m,N}\}^{\Delta_1}=\{\tilde\Xi_{K-2},\tilde\Omega_{m,N},\{0,1\},\tilde\Lambda_{m,N}\}$, there exists a polynomial $\mathrm{pol}: \mathbb{R} \rightarrow \mathbb{R}$ and an algorithm $\Gamma$ (a Turing machine or a BSS machine) that takes the dimensions $m$, $N$ and any $\tilde\iota \in \tilde\Omega_{m,N}$ as input and satisfies
		\[
		\Gamma(m, N, \tilde \iota) = \tilde\Xi_{K-2}(\tilde\iota).
		\]
		Furthermore, in the Turing case, the arithmetic runtime and the space complexity of the Turing machine are bounded by $\mathrm{pol}(n_{\mathrm{var}})$, where $n_{\mathrm{var}}=mN+m$ is the number of variables, and the number of digits read from the oracle tape is bounded by $\mathrm{pol}(\log(n_{\mathrm{var}}))$. In particular, the runtime in the Turing model is polynomial in $n_{\mathrm{var}}$. In the BSS case, the runtime is likewise bounded by $\mathrm{pol}(n_{\mathrm{var}})$.
	\end{itemize}
	Concretely, (iii) holds whenever $M\notin\SetNine_{K}$, whereas (iv) holds whenever $M\notin\SetNine_{K-1}$.
\end{proposition}

\section{Tools for proving impossibility results -- Breakdown epsilons and exit flags}\label{sec:main_tools}
The main tools can be summed up in the following three statements. We commence with a proposition describing the relationship between the different Breakdown epsilons.

\subsection{The key mechanisms for the Breakdown epsilons }

The connection between the different Breakdown epsilons can be summarised in Proposition \ref{lemma:BreakdownLinks} further below.

\begin{proposition}\label{lemma:BreakdownLinks}
	Given a computational problem  $\{\Xi,\Omega,\mathcal{M},\Lambda\}$ with $\Lambda=\{f_k\, \vert \, k\in\mathbb{N}, k\leq |\Lambda|\}$, and $p,q \in (0,1)$ with $p \leq q$, we have 
	\begin{align}
		\epsilon_{\mathbb{P}\mathrm{B}}^{\mathrm{s}}(q) &\leq \epsilon_{\mathbb{P}\mathrm{B}}^{\mathrm{s}}(p) \leq \epsilon^{\mathrm{s}}_{\mathrm{B}}, \label{first}
		\\
		\epsilon_{\mathbb{P}\mathrm{B}}^{\mathrm{w}}(q) &\leq \epsilon_{\mathbb{P}\mathrm{B}}^{\mathrm{w}}(p) \leq \epsilon^{\mathrm{w}}_{\mathrm{B}}, \label{second}
		\\
		\epsilon_{\mathbb{P} h \mathrm{B}}^{\mathrm{s}}(p) & \leq \epsilon_{\mathbb{P}\mathrm{B}}^{\mathrm{s}}(p) \leq \epsilon_{\mathbb{P}\mathrm{B}}^{\mathrm{w}}(p),\quad  \text{and }\label{fourth}
		\\
		\epsilon^{\mathrm{s}}_{\mathrm{B}} &\leq \epsilon^{\mathrm{w}}_{\mathrm{B}}. \label{fifth}
	\end{align}
\end{proposition}

\begin{proof}[Proof of Proposition \ref{lemma:BreakdownLinks}]
	We start with \eqref{first}, and observe that $\epsilon_{\mathbb{P}\mathrm{B}}^{\mathrm{s}}(q)\leq \epsilon_{\mathbb{P}\mathrm{B}}^{\mathrm{s}}(p)$ follows directly from the definition. To see that $\epsilon_{\mathbb{P}\mathrm{B}}^{\mathrm{s}}(p) \leq \epsilon^{\mathrm{s}}_{\mathrm{B}}$ we argue by contradiction and suppose that $\epsilon_{\mathbb{P}\mathrm{B}}^{\mathrm{s}}(p) > \epsilon^{\mathrm{s}}_{\mathrm{B}}$. Then there exists an $\epsilon > 0$ such that $\epsilon_{\mathbb{P}\mathrm{B}}^{\mathrm{s}}(p) > \epsilon > \epsilon^{\mathrm{s}}_{\mathrm{B}}$. Hence,  
	\begin{equation}\label{eq:contra}
		\forall \, \gprob \in \mathrm{RGA} \,\,\exists \, \iota \in \Omega \text{ such that } \mathbb{P}_{\iota}(\disM(\gprob_{\iota},\Xi(\iota)) > \epsilon) > p.
	\end{equation}
	However, since $\epsilon > \epsilon^{\mathrm{s}}_{\mathrm{B}}$ there exists a general algorithm $\Gamma$ such that for all $\iota \in \Omega$, we have that 
$
\disM(\Gamma(\iota),\Xi(\iota)) \leq \epsilon.
$
 Since $\Gamma$ can be seen as an RGA with $X = \{\Gamma\}$ and $\mathbb{P}_{\iota}(\disM(\Gamma(\iota),\Xi(\iota)) > \epsilon)=0$, which violates \eqref{eq:contra}. Note that the argument to establish \eqref{second} is identical to the proof of \eqref{first}. 
	Finally,  we notice that \eqref{fourth} and \eqref{fifth} follow directly from the definitions.  
\end{proof}

As the next two propositions will rely on the use of $D_\Gamma(\iota)$, the minimum number of digits on the input, together with randomised general algorithms, it will be opportune to immediately state and prove that $D_\Gamma(\iota)$ is a random variable. Concretely, we have the following two lemmas:

\begin{lemma}\label{lemma:RGAindepOfOrdering}
	Let  $\gprob $ be an RGA for a computational problem $\{\Xi, \Omega, \mathcal{M}, \Lambda\}$, where $\Lambda=\{f_j\,\vert\,j\leq |\Lambda|\}$ is countable, and let $X$, $\mathcal{F}$, and $\{\pr_\iota\}_{\iota\in \Omega}$ be as in Definition \ref{definition:ProbablisticAlgorithm}. Then,
	\begin{enumerate}[label = (\roman*)]
		\item for every $\iota\in\Omega$ and every $f\in\Lambda$, $\{\Gamma\in X \,\vert\, f\in \Lambda_\Gamma(\iota)\}\in\mathcal{F}$, and \label{item:GammasUsingfMeasurable}
		\item for every bijection $\theta:\mathbb{N}\to\mathbb{N}$ and every $n\in\mathbb{N}$, we have 
		$
			\{\Gamma\in X \,\vert\,  \Lambda_\Gamma(\iota) \subset\{f_{\theta(1)},\dots, f_{\theta(n)}\}  \}\in\mathcal{F}.
		$
		In particular, $\gprob $ is an RGA for  $\{\Xi, \Omega, \mathcal{M}, \Lambda\}$ independently of the enumeration of $\Lambda$.   \label{item:ChangeEnumerationRV}
	\end{enumerate}
\end{lemma}
\begin{proof}
	To prove (i) take arbitrary $\iota\in\Omega$ and $f_j\in\hat\Lambda$. The claim follows from \ref{property:PAlgorithmRTMeasurable} in Definition \ref{definition:ProbablisticAlgorithm} after observing that $\{\Gamma\in X \,\vert\, f_j\in \Lambda_\Gamma(\iota)\}=\{ \Gamma\in X \,\vert\, T_\Gamma(\iota)\leq j\}\setminus \{ \Gamma\in X \,\vert\, T_\Gamma(\iota)\leq j-1\}$. To prove (ii) consider an arbitrary $n\in\mathbb{N}$ and define the set $S=\{\theta(j)\, \vert \, j\in\mathbb{N},  j\leq n\}$. Then, for every $\iota\in\Omega$,
	\begin{equation*}
		\big\{\Gamma\in X \,\vert\, \Lambda_\Gamma(\iota) \subset\{f_{\theta(1)},\dots, f_{\theta(n)}\}  \big\}=X\setminus \bigcup_{j\in \mathbb{N}\setminus S} \{\Gamma\in X \,\vert\, f_{j}\in \Lambda_\Gamma(\iota)\} ,
	\end{equation*}
	which must be an element of $\mathcal{F}$, since each of the sets $\{\Gamma\in X \,\vert\, f_{j}\in \Lambda_\Gamma(\iota)\}$ is in $\mathcal{F}$, by the already established item \ref{item:GammasUsingfMeasurable}, and  $\mathcal{F}$ is a $\sigma$-algebra.
\end{proof}

\begin{lemma}\label{lemma:DGammaArandomVar}
	Let  $\{\Xi, \Omega, \mathcal{M}, \Lambda\}$ be a computational problem with $\Lambda$ countable and let $\hat\Lambda=\{f_{k,m}\,\vert\,k\leq |\Lambda|, m\in \mathbb{N}\}\in \mathcal{L}^1(\Lambda)$ be arbitrary. Furthermore, let $\gprob$ be an RGA for $\{\Xi, \Omega, \mathcal{M}, \hat \Lambda\}$  with $X$ and $\mathcal{F}$ as in Definition \ref{definition:ProbablisticAlgorithm}. Then, for each $\iota\in \Omega$, the function $D_{\gprob}(\iota):X\to \mathbb{N}\cup\{\infty\}$ defined by $\Gamma\mapsto D_{\Gamma}(\iota)$ is $\mathcal{F}$-measurable.\label{item:DGammaRV}
\end{lemma}
\begin{proof}
	Consider an arbitrary $n\in\mathbb{N}$ and define the set $S=\{(k,m)\in \mathbb{N}^2\,\vert\, k\leq |\Lambda|, m\leq n\}$. Then, for each $\iota\in\Omega$, we have
	\begin{equation*}
		\begin{aligned}
			\{\Gamma\in X\, \vert\, D_\Gamma(\iota)\leq n\}&= \big\{\Gamma\in X\,\vert\,  \hat\Lambda_\Gamma(\iota)\subset \{f_{k,m}\,\vert\, (k,m)\in S\}  \big\}\\
			&=X \setminus \bigcup_{(k,m)\in \mathbb{N}^2\setminus S} \{\Gamma\in X\,\vert\, f_{k,m}\in \hat\Lambda_\Gamma(\iota)\}\in \mathcal{F},
		\end{aligned}
	\end{equation*}
	since each of the sets $\{\Gamma\in X \,\vert\, f_{k,m}\in \Lambda_\Gamma(\iota)\}$ is in $\mathcal{F}$, by item \ref{item:GammasUsingfMeasurable} of Lemma \ref{lemma:RGAindepOfOrdering}, and  $\mathcal{F}$ is a $\sigma$-algebra. Then clearly 
\[
\{\Gamma\in X\, \vert\, D_\Gamma(\iota)=\infty\}=X\setminus \bigcup_{n\in\mathbb{N}}\{\Gamma\in X\, \vert\, D_\Gamma(\iota)\leq n\}\in\mathcal{F},
\]
 and thus $D_{\gprob}(\iota)$ is measurable, as desired.
\end{proof}

\begin{remark}\label{remark:OmegaSubsetBDE}
	It will be useful throughout the paper to note that any of the breakdown epsilons of a computational problem $\{\Xi, \Omega, \mathcal{M}, \Lambda\}$ is at least as large as the corresponding breakdown epsilon of any other computational problem $\{\Xi, \Omega', \mathcal{M}, \Lambda'\}$ with $\Omega'\subset\Omega$ and $\Lambda'=\{f\mid_{\Omega'}\, \vert f\in\Lambda\}$.
\end{remark}

The next proposition serves as the key building block in the impossibility results in all our main theorems except for Theorem \ref{thm:ExitFlag}. Note that the proposition is about arbitrary computational problems, and is hence also a tool for demonstrating lower bounds on the breakdown epsilons for general computational problems. 
\begin{proposition}\label{prop:DrivingNegativeProposition}
	Let $\{\Xi, \Omega, \mathcal{M}, \Lambda\}$ be a computational problem with $\Lambda=\{f_k\,\vert\,k\in\mathbb{N}, k\leq|\Lambda| \}$ countable, and let $\{\iota^1_n\}_{n=1}^{\infty}$, $\{\iota^2_n\}_{n=1}^{\infty}$ be sequences in $\Omega$. Consider the following conditions:
	\begin{enumerate}[leftmargin=8mm, label=(\alph*)]
		\item There are sets $S^1, S^2 \subset \mathcal{M}$ and $\kappa > 0$ such that
		$
		\inf_{x_1 \in S^1, x_2 \in S^2}d_{\mathcal{M}}(x_1,x_2) \geq \kappa
		$ and $\Xi(\iota^j_n) \subset S^j$ for $j=1,2$. \label{property:MinimisersOfNonZero}
		\item For every $k\leq |\Lambda|$ there is a $c_k \in \mathbb{C}$ such that $|f_k(\iota^j_n) - c_k|\leq 1/4^n$, for all $j = 1,2$ and $n \in\mathbb{N}$. \label{property:Del1Info}
		\item There is an $\iota^0 \in \Omega$ such that for every $ k\leq|\Lambda|$ we have that \ref{property:Del1Info} is satisfied with $c_k = f_k(\iota^0)$. \label{property:Del1withIota0}
		\item There is an $\iota^0 \in \Omega$ for which condition \ref{property:Del1withIota0} holds and additionally $\iota_n^1=\iota^0$, for all $n\in\mathbb{N}$. \label{property:SameWithSubsetS1}
		\item $\Xi$ is single-valued, $\mathcal{M}$ is a totally ordered set, and $\inf \opBall{\delta}{S^1} > \sup S^2$, for some $\delta > 0$.\label{property:MtotallyOrdered}
	\end{enumerate}
	Depending on which of the conditions \ref{property:MinimisersOfNonZero} -- \ref{property:MtotallyOrdered} are fulfilled, there exists a $\hat\Lambda\in \mathcal{L}^1(\Lambda)$
	such that some of the following claims about the computational problem $\{\Xi,\Omega,\mathcal{M},\hat\Lambda\}$ hold:
	\begin{enumerate}[leftmargin=8mm, label = (\roman*)]
		\item \label{item:SequenceCounterexampleL1WeakBDEps}
		$\epsilon_{\mathrm{B}}^{\mathrm{w}} \geq \epsilon_{\mathbb{P}\mathrm{B}}^{\mathrm{w}}(\mathrm{p}) \geq \kappa/2$ for $\mathrm{p} \in [0,1/2)$,
		\item \label{item:SequenceCounterexampleL1StrongBDEps}
		$\epsilon^{\mathrm{s}}_{\mathrm{B}} \geq \strbdepsph \geq \kappa/2 $ for $\mathrm{p} \in [0,1/2)$ and $\strbdepsp \geq \kappa/2$ for $\mathrm{p} \in [0,1/3)$,
		\item \label{item:SequenceCounterexampleL1StrongHaltEps}
		$\strbdepsp \geq \kappa/2$ for $\mathrm{p} \in [0,1/2)$. 
		\item\label{item:SequenceCounterexampleL1NoSigmaNoPi}  $\{\Xi,\Omega,\mathcal{M},\hat\Lambda\} \notin \Sigma^G_{1}$. 			 
	\end{enumerate}
	Concretely,
	if \ref{property:MinimisersOfNonZero} and \ref{property:Del1Info} are fulfilled, then \ref{item:SequenceCounterexampleL1WeakBDEps} holds,
	if \ref{property:MinimisersOfNonZero} -- \ref{property:Del1withIota0} are fulfilled, then \ref{item:SequenceCounterexampleL1WeakBDEps} and \ref{item:SequenceCounterexampleL1StrongBDEps} hold,
	if \ref{property:MinimisersOfNonZero} -- \ref{property:SameWithSubsetS1} are fulfilled, then \ref{item:SequenceCounterexampleL1WeakBDEps} -- \ref{item:SequenceCounterexampleL1StrongHaltEps} hold, and finally,
	if \ref{property:MinimisersOfNonZero} -- \ref{property:MtotallyOrdered} are fulfilled, then \ref{item:SequenceCounterexampleL1WeakBDEps} -- \ref{item:SequenceCounterexampleL1NoSigmaNoPi} hold.
\end{proposition}

\begin{remark} [Computability of $\Delta_1$ information]\label{remark:CompDelta1Markov}

Note that if $(c_k)_{k \leq |\Lambda|}$ are computable numbers in the sense of Turing-computability, then $\hat \Lambda$ can be chosen so that $\hat \Lambda = \{f_{k,n} \, \vert \, k \leq |\Lambda| \text{ and } n \in \mathbb{N}\}$ and so that, for any $\iota \in \Omega$ and each $f_k \in \Lambda$ there exists a Turing machine, such that, given input $n \in \mathbb{N}$, outputs $ \{f_{k,n}(\iota)\}_{n=1}^{\infty}$. This follows immediately from the proof of Proposition \ref{prop:DrivingNegativeProposition} - specifically, this computability is a consequence of equation \eqref{eq:def_fm}.
Thus the non-computability result presented in Proposition \ref{prop:DrivingNegativeProposition} can easily be seen to apply to the Markov model.
\end{remark}

\begin{proof}[Proof of Proposition \ref{prop:DrivingNegativeProposition}]
	
	We begin by defining the inputs $\iota_n\in\Omega$, for $n\geq 1$, according to
	$\iota_{2n} = \iota^1_{n+1}, \iota_{2n-1} = \iota^2_{n+1}$.
	Now, if only \ref{property:MinimisersOfNonZero} and \ref{property:Del1Info} are fulfilled, we can assume w.l.o.g. $\Omega = \{\iota_n \, \vert \, n\geq 1\}$ (see Remark \ref{remark:OmegaSubsetBDE}), and, in the cases when \ref{property:Del1withIota0} holds as well, we define $\iota_0:=\iota^0$ and assume w.l.o.g. $\Omega = \{\iota_n \, \vert \, n\geq 0\}$.
	
	Our aim now is to produce the desired $\Delta_1$-information for $\Omega$. 
	For $m,n\in \mathbb{N}$ and $ k\leq|\Lambda|$, choose $d^{n,m}_k \in \dyadic_m + i \dyadic_m$ such that $|f_k(\iota_n) - d^{n,m}_k | \leq 2^{-m}$ 
	as well as $c^m_k \in  \dyadic_m + i\dyadic_m$  such that $|c^m_k - c_k| \leq 2^{-m}/\sqrt{2}$. 
	Now, for $ k\leq|\Lambda|$ and $m\in\mathbb{N}$, define $f_{k,m}: \Omega \to \dyadic_m + i \dyadic_m$ according to 
	\begin{equation}\label{eq:def_fm}
		f_{k,m}(\iota_n) =\begin{cases}
			d^{n,m}_k, & \text{ if } 1\leq n \leq m\\
			c^m_k, & \text{if } n=0 \text{ or }n>m
		\end{cases},\quad \text{for }\iota_n\in\Omega.
	\end{equation}
	We claim that $
	\hat{\Lambda}:= \{ f_{k,m} \, \vert \,  k\leq|\Lambda| , m \in \mathbb{N} \}
	$
	provides $\Delta_1$-information for $\{\Xi, \Omega, \mathcal{M}, \Lambda\}$. To this end, first note that, due to assumption (b) and our construction of $\{\iota_n\}_{n\geq 1}$, we have $|c_k - f_k(\iota_n)| \leq 2^{-(n+2)}$, for all $n\geq 1$. Now, for  $n,m \in \mathbb{N}$ with $n>m$, we have $2^{-(n+2)} \leq 2^{-(m+3)}$, and therefore
	\begin{equation*}
		\begin{aligned}
			|f_{k,m}(\iota_n) - f_k(\iota_n)| &= |c^m_k - f_k(\iota_n)| \leq |c^m_k - c_k| + |c_k -  f_k(\iota_{n})|\\
			& \leq 2^{-m}/\sqrt{2} + 2^{-(n+2)}\leq (1/\sqrt{2} + 1/8)\cdot 2^{-m}<2^{-m}.
		\end{aligned}
	\end{equation*}
	Similarly, in the cases when \ref{property:Del1withIota0} holds so that $\iota_0\in\Omega$, we have
	$
	|f_{k,m}(\iota_0) - f_k(\iota_0)| = |c^m_k - f_k(\iota_0)| = |c^m_k - c_k| \leq 2^{-m}/\sqrt{2}<2^{-m}
	$.
	Finally, for $1\leq n \leq m$ we have $|f_{k,m}(\iota_n) - f_k(\iota_n)| = |d^{n,m}_k - f_k(\iota_n)| \leq 2^{-m}$ by the above definition of $d^{n,m}_k$.
	Therefore $\hat{\Lambda}$ provides $\Delta_1$-information for the computational problem $\{\Xi, \Omega, \mathcal{M}, \Lambda\}$, as desired.
	
	{\it Proof of \ref{property:MinimisersOfNonZero}, \ref{property:Del1Info} $\implies$ \ref{item:SequenceCounterexampleL1WeakBDEps}}:
	We argue by contradiction and assume that 
	\begin{equation}\label{eq:assumption}
		\epsilon^{\mathrm{w}}_{\mathbb{P}\mathrm{B}}(\mathrm{p}) < \kappa/2, \text{ for some } \mathrm{p} \in (0,1/2).
	\end{equation}
	for the computational problem $\{\Xi, \Omega, \mathcal{M}, \hat \Lambda\}$. 
	Therefore, there exist an $N\in\mathbb{N}$ and a $\gprob \in  \mathrm{RGA}$ such that 
	\begin{equation}\label{eq:prob1}
		\mathbb{P}_{\iota}(\disM(\gprob_{\iota},\Xi(\iota)) \geq \kappa/2 \; \text{ or }\; T_{\gprob}(\iota) > N) \leq \mathrm{p}, \quad\forall \iota\in\Omega.
	\end{equation}
	Now, recalling the definitions of $T_\Gamma$ and $D_\Gamma$, we see that there must exist an $M\in\mathbb{N}$ such that 
	\begin{equation}\label{eq:TGammaNDGammaM}
		\quad D_{\Gamma}(\iota)> M \implies T_{\Gamma}(\iota)> N,
	\end{equation}
	for all algorithms $\Gamma$ for $\{\Xi, \Omega, \mathcal{M}, \hat\Lambda\}$ and all $\iota\in\Omega$.
	Next, let the measurable space $(X,\mathcal{F})$ and the probability measures $\{\mathbb{P}_{\iota_n}\}_{{\iota_n}\in\Omega}$ associated with $\gprob$ be as in Definition \ref{definition:ProbablisticAlgorithm} of an RGA. We now define the following sets that form the basis for the rest of the argument:
	\[
	F_j := \lbrace \Gamma \in X\, \vert \, \disM(\Gamma(\iota_{M+j}),\Xi(\iota_{M+j})) \geq \kappa/2 \text{ or } D_{\Gamma}(\iota_{M+j}) > M \rbrace, \qquad j = 1,2,
	\]	
	 $
	\hat F_2 = \lbrace \Gamma \in X\, \vert \, \disM(\Gamma(\iota_{M+2}),\Xi(\iota_{M+2})) \geq \kappa/2 \rbrace
	$
	and
	$
	\mathring F_2 =  \lbrace \Gamma \in X\, \vert \,D_{\Gamma}(\iota_{M+2}) > M\rbrace.
	$
Note that (Pi) in Definition \ref{definition:ProbablisticAlgorithm}, the continuity of the metric $d_{\mathcal{M}}$, and Lemma \ref{lemma:DGammaArandomVar} together imply that $F_1$, $F_2$, $\hat F_2$ and $ \mathring F_2$  are measurable.  Also observe that \eqref{eq:prob1} and \eqref{eq:TGammaNDGammaM} imply that $\pr_{\iota_{M+j}}(F_j) \leq \mathrm{p}$ for $j = 1,2$.  	
	
	{\bf Claim 1:} We now claim the following.
	\begin{enumerate}[label = (\Roman*)]
		\item $X = F_1 \cup \hat F_2$. \label{proofItemMainThmwb:TnDecomp}
		\item  
		$
		\pr_{\iota_{M+1}}(\hat F_2 \cap  {\mathring F_2}^c) = \pr_{\iota_{M+2}}(\hat F_2 \cap  {\mathring F_2}^c).
		$
		\label{proofItemMainThmwb:F2F2CEquality}
		\item $\pr_{\iota_{M+1}}({\mathring F_2}) = \pr_{\iota_{M+2}}({\mathring F_2})$
		\label{proofItemMainThmwb:ProbMeasureEquality}
	\end{enumerate}
	
	To see \ref{proofItemMainThmwb:TnDecomp} we argue as follows. Consider an arbitrary $\Gamma \in X$. Suppose by way of contradiction that $\Gamma \notin F_1$ and $\Gamma \notin \hat F_2$. In particular, $D_{\Gamma}(\iota_{M+1}) \leq M$.
	Consequently, whenever $f_{k,m} \in \hat \Lambda_{\Gamma}(\iota_{M+1})$ we must have $m \leq M$ . Therefore, by \eqref{eq:def_fm} it follows that $f_{k,m}(\iota_{M+1}) =  f_{k,m}(\iota_{M+2}) = c^m_k$.
	Hence, by property \ref{property:AlgorithmSameInputSameInputTaken} in Definition \ref{definition:Algorithm} of a general algorithm, it follows that $\hat \Lambda_{\Gamma}(\iota_{M+1}) = \hat \Lambda_{\Gamma}(\iota_{M+2})$. Thus, by \ref{property:AlgorithmDependenceOnInput} in Definition \ref{definition:Algorithm}, 
	$\Gamma(\iota_{M+1}) =  \Gamma(\iota_{M+2})$. The assumption that  $\Gamma \notin F_1$ and $\Gamma \notin \hat F_2$ yields
	\begin{equation*}
		\begin{aligned}
			\disM(\Gamma(\iota_{M+1}),\Xi(\iota_{M+1})) &< \kappa/2 \qquad\text{ and} \\
			\disM(\Gamma(\iota_{M+1}),\Xi(\iota_{M+2})) &= \disM(\Gamma(\iota_{M+2}),\Xi(\iota_{M+2}))<\kappa/2.
		\end{aligned}
	\end{equation*}
	Therefore $\inf_{\xi_1\in \Xi(\iota_{M+1}), \xi_2\in \Xi(\iota_{M+2})}d_{\mathcal{M}}(\xi_1,\xi_2) < \kappa$, which contradicts \ref{property:MinimisersOfNonZero}, and hence we conclude that $\Gamma \in F_1$ or $\Gamma \in \hat F_2$, establishing (I).
	
	To prove \ref{proofItemMainThmwb:F2F2CEquality} and \ref{proofItemMainThmwb:ProbMeasureEquality} it suffices to demonstrate that 
	\begin{equation}\label{eq:change_prob}
		\pr_{\iota_{M+1}}(E \cap {\mathring F_2}^c) = \pr_{\iota_{M+2}}(E \cap {\mathring F_2}^c) \qquad \forall \, E \in \mathcal{F}.
	\end{equation}
	Indeed, given \eqref{eq:change_prob} \ref{proofItemMainThmwb:F2F2CEquality} follows immediately since $\hat{F}_2\in\mathcal{F}$, whereas \ref{proofItemMainThmwb:ProbMeasureEquality} follows by letting $E = X$ since then $\pr_{\iota_{M+1}}({\mathring F_2}) = 1 - \pr_{\iota_{M+1}}({\mathring F_2}^c) =  1 - \pr_{\iota_{M+2}}({\mathring F_2}^c) = \pr_{\iota_{M+2}}({\mathring F_2}).$ To show \eqref{eq:change_prob} consider any $E \in \mathcal{F}$ with $E \cap {\mathring F_2}^c \neq \varnothing$  (if instead  $E \cap {\mathring F_2}^c$ were empty there is nothing to prove).  Now, for $\Gamma \in E \cap {\mathring F_2}^c$ and $f_{k,m} \in \hat \Lambda_{\Gamma}(\iota_{M+2})$, it follows from
$
{\mathring F_2}^c =  \lbrace \Gamma \in X\, \vert \,D_{\Gamma}(\iota_{M+2}) \leq M\rbrace
$
	that $m \leq M$. Hence, by \eqref{eq:def_fm}, we have $f_{k,m}(\iota_{M+1}) =  f_{k,m}(\iota_{M+2}) = c^m_k$, and thus \eqref{eq:change_prob} follows immediately by \ref{property:PAlgorithmConsistent} in Definition \ref{definition:ProbablisticAlgorithm} of a randomised general algorithm, completing the proof of the claim.

	Armed with the claim we are ready to derive the desired contradiction. In particular, 
	\begin{align*}
		1&\leq \pr_{\iota_{M+1}}(F_1) + \pr_{\iota_{M+1}} (\hat F_2)&& \text{by \ref{proofItemMainThmwb:TnDecomp}}\\
		&= \pr_{\iota_{M+1}}(F_1) + \pr_{\iota_{M+1}} (\hat F_2 \cap {\mathring F_2}^c) + \pr_{\iota_{M+1}} (\hat F_2 \cap {\mathring F_2})\\
		&\leq \pr_{\iota_{M+1}}(F_1) + \pr_{\iota_{M+2}} (\hat F_2 \cap  {\mathring F_2}^c) + \pr_{\iota_{M+1}} ({\mathring F_2}) && \text{by \ref{proofItemMainThmwb:F2F2CEquality}}\\
		&= \pr_{\iota_{M+1}}(F_1) + \pr_{\iota_{M+2}} (\hat F_2 \cap  {\mathring F_2}^c) + \pr_{\iota_{M+2}} ({\mathring F_2}) && \text{by \ref{proofItemMainThmwb:ProbMeasureEquality} }\\
		&\leq \pr_{\iota_{M+1}}(F_1) + \pr_{\iota_{M+2}} (F_2) \leq 2\mathrm{p} <1&& \text{since } \hat F_2 \subset F_2,
	\end{align*}
	which yields the final contradiction.

	{\it Proof of \ref{property:MinimisersOfNonZero} -- \ref{property:Del1withIota0} $\implies$ \ref{item:SequenceCounterexampleL1WeakBDEps} -- \ref{item:SequenceCounterexampleL1StrongBDEps} }: Item \ref{item:SequenceCounterexampleL1WeakBDEps} holds by the argument above. For item \ref{item:SequenceCounterexampleL1StrongBDEps},
	we will first show that $\epsilon_{\mathbb{P}\mathrm{B}}^{\mathrm{s}}(\mathrm{p}) \geq \kappa/2$ for all $\mathrm{p} \in [0,1/3)$ and then demonstrate that $\epsilon^{\mathrm{s}}_{\mathrm{B}} \geq \strbdepsph  \geq \kappa/2$ for all $\mathrm{p} \in [0,1/2)$, for which it will suffice to show that $\strbdepsph  \geq \kappa/2$ for $\mathrm{p} \in [0,1/2)$ because by Proposition \ref{lemma:BreakdownLinks} we have that $\epsilon^{\mathrm{s}}_{\mathrm{B}} \geq \strbdepsph $. To this end, first recall from the beginning of the proof that the set $\Omega$ now contains the element $\iota_0$.
	Suppose by way of contradiction that $\epsilon_{\mathbb{P}\mathrm{B}}^{\mathrm{s}}(p)  < \kappa/2$ for some $\mathrm{p} \in [0,1/3)$, so that there exists a $\gprob \in \mathrm{RGA}$ with the corresponding measurable space $(X,\mathcal{F})$ and probability measures $\{\mathbb{P}_{\iota_n}\}_{\iota_n\in\Omega}$ such that 
	\begin{equation}\label{eq:existence}
		\mathbb{P}_{\iota_n}(F^n) \leq \mathrm{p},\quad\text{ for all }n\in\mathbb{N}\cup\{0\},
	\end{equation}
	where we define
	\begin{equation}\label{eq:the_Fn}
		F^{n} := \lbrace \Gamma \in X\, \vert \, \disM(\Gamma(\iota_{n}),\Xi(\iota_{n})) \geq \kappa/2 \rbrace, \qquad n \in \mathbb{N} \cup \{0\}.
	\end{equation}
	Note that by (Pi) in Definition \ref{definition:ProbablisticAlgorithm} and the continuity of the metric $d_{\mathcal{M}}$ it follows that the $F^{n}$ are measurable.
	Additionally, for $\iota \in \Omega$ and $n \in \mathbb{N} \cup \{0\}$, we define 
	$
	\mathcal{G}^n(\iota):= \lbrace \Gamma \in X \, \vert \, D_{\Gamma}(\iota) \leq n \rbrace,
	$
	i.e., the collection of general algorithms whose number of required input digits is bounded by $n \in \mathbb{N}$ when applied to $\iota$, and note that these sets are measurable due to Lemma \ref{lemma:DGammaArandomVar}.
	Now, 
	$
	X\setminus \bigcup_{n=1}^{\infty} \mathcal{G}^n(\iota_0) \subset F^0,
	$
	 as every $\Gamma\in X$ for which $D_\Gamma(\iota_0)=\infty$ must have $|\hat\Lambda_\Gamma(\iota_0)|=\infty$ (by Definition \ref{def:no-of-input-bits}) and therefore $\Gamma(\iota_0)=\nh$ (by Definition \ref{definition:Algorithm}), which then implies $\disM(\Gamma(\iota_0),\Xi(\iota_0))=\infty$ (by Definition \eqref{eq:extended-metric} of the extended metric on $\mathcal{M}\cup\{\infty\}$). Therefore,
	$\mathbb{P}_{\iota_0}( F^0)\leq \mathrm{p}$ implies
	\begin{equation}\label{eq:union}
		\pr_{\iota_0}\left(X \setminus  \bigcup\limits_{n=1}^{\infty} \mathcal{G}^n(\iota_0)\right) \leq \mathrm{p}.
	\end{equation}
	We next make the following claim. 
	
	{\bf Claim 2:} There is an $n\in\mathbb{N}$ such that we have the following:
	\begin{enumerate}[label = (\Roman*)]
		\item  $\pr_{\iota_0}(\mathcal{G}^n(\iota_0)) > 2 \mathrm{p}$, \label{proofItemMainThm:TotalProbability}
		\item  $\mathcal{G}^n(\iota_0) \subset F^{n+1} \cup  F^{n+2} $, \label{proofItemMainThm:TnDecomp}
		\item $\pr_{\iota_0}(F^{n+1} \cap \mathcal{G}^n(\iota_0)) = \pr_{\iota_{n+1}}(F^{n+1} \cap \mathcal{G}^n(\iota_0))$, and $\pr_{\iota_0}(F^{n+2} \cap \mathcal{G}^n(\iota_0)) = \pr_{\iota_{n+2}}(F^{n+2} \cap \mathcal{G}^n(\iota_0))$.\label{proofItemMainThm:ProbMeasureEquality}
	\end{enumerate}
	The contradiction arises by combining these results: indeed, by \ref{proofItemMainThm:TotalProbability} and \ref{proofItemMainThm:TnDecomp}, 
$
2 \mathrm{p} < \pr_{\iota_0}(\mathcal{G}^n(\iota_0)) = \pr_{\iota_0}( (F^{n+1}\cap \mathcal{G}^n(\iota_0)) \cup  (F^{n+2} \cap \mathcal{G}^n(\iota_0)))
$
 and by \ref{proofItemMainThm:ProbMeasureEquality} we get that
$
\pr_{\iota_{0}}(F^{n+1} \cap \mathcal{G}^n(\iota_0)) + \pr_{\iota_{0}}(F^{n+2} \cap \mathcal{G}^n(\iota_0))= \pr_{\iota_{n+1}}(F^{n+1} \cap \mathcal{G}^n(\iota_0)) + \pr_{\iota_{n+2}}(F^{n+2} \cap \mathcal{G}^n(\iota_0)).
$ Therefore 
	\begin{align}\label{eq:calculation1}
		2 \mathrm{p}&< \pr_{\iota_0}(\mathcal{G}^n({\iota_0})) = \pr_{\iota_0}\left((\mathcal{G}^{n}(\iota_0) \cap F^{n+1}) \cup (\mathcal{G}^{n}(\iota_0) \cap F^{n+2})\right)\\
		\label{eq:calculation2}
		&\leq \pr_{\iota_0}(\mathcal{G}^{n}(\iota_0) \cap F^{n+1}) + \pr_{\iota_0} (\mathcal{G}^{n}(\iota_0) \cap F^{n+2})\\
		\label{eq:calculation3}
		&=\pr_{\iota_{n+1}}(\mathcal{G}^{n}(\iota_0) \cap F^{n+1}) + \pr_{\iota_{n+2}}(\mathcal{G}^{n}(\iota_0) \cap F^{n+2})\\
		&\leq \pr_{\iota_{n+1}}(F^{n+1}) + \pr_{\iota_{n+2}}(F^{n+2}) \leq 2 \mathrm{p} && \text{by \eqref{eq:existence}}, \label{eq:calculation4}
	\end{align}
	which is the desired contradiction establishing item \ref{item:SequenceCounterexampleL1StrongBDEps}.
	It thus remains to prove Claim 2.

	For \ref{proofItemMainThm:TotalProbability}, suppose on the contrary that we have $\pr_{\iota_0}(\mathcal{G}^n(\iota_0)) \leq 2 \mathrm{p}$ for all $n \in \mathbb{N}$. Then, as $\mathcal{G}^n(\iota_0) \subset \mathcal{G}^{n+1}(\iota_0)$ for all $n$, monotone convergence implies 
	\begin{equation}\label{eq:fix_it}
		\mathrm p \geq \pr_{\iota_{0}}\left(X\setminus \bigcup\limits_{n=1}^{\infty}\mathcal{G}^n(\iota_0)\right) = \lim_{n \to \infty} 1- \pr_{\iota_{0}}(\mathcal{G}^n(\iota_{0})) \geq 1-2 \mathrm p,
	\end{equation}
	which is a contradiction since $\mathrm{p} \in [0,1/3)$. Therefore, there exists an $n\in\mathbb{N}$ for which \ref{proofItemMainThm:TotalProbability} holds. We now prove \ref{proofItemMainThm:TnDecomp} and \ref{proofItemMainThm:ProbMeasureEquality} for this $n$.
	To this end, we make the intermediary step of showing that 
	\begin{equation}\label{eq:useful_fact}
		\Gamma(\iota_{n+1}) = \Gamma(\iota_{n+2}),\text{ for all }\Gamma \in \mathcal{G}^n(\iota_0).
	\end{equation}
	Indeed, for $\Gamma \in \mathcal{G}^n(\iota_0)$, it follows immediately from the definition of the set $\mathcal{G}^n$ and the definition of $D_\Gamma$ that if $f_{k,m} \in \hat \Lambda_{\Gamma}(\iota_0)$, then $m \leq n < n+1<n+2$.  Thus, by the construction of $\hat \Lambda$, it follows that $f_{k,m}(\iota_0) = f_{k,m}(\iota_{n+1}) = f_{k,m}(\iota_{n+2})$, for all $f_{k,m} \in \hat \Lambda_{\Gamma}(\iota_0)$. Hence, by property \ref{property:AlgorithmSameInputSameInputTaken} in Definition \ref{definition:Algorithm} of a general algorithm, it follows that $\hat \Lambda_{\Gamma}(\iota_0) = \hat \Lambda_{ \Gamma}(\iota_{n+1})$, and similarly we get that $\hat \Lambda_{\Gamma}(\iota_0) = \hat \Lambda_{\Gamma}(\iota_{n+2})$.
	However, the fact that $\hat \Lambda_{\Gamma}(\iota_0) = \hat \Lambda_{\Gamma}(\iota_{n+1}) = \hat \Lambda_{\Gamma}(\iota_{n+2})$ together with \ref{property:AlgorithmDependenceOnInput} in Definition \ref{definition:Algorithm} imply that $\Gamma(\iota_0) = \Gamma(\iota_{n+1})$ and $ \Gamma(\iota_0) = \Gamma(\iota_{n+2})$, and hence $\Gamma(\iota_{n+1}) =  \Gamma(\iota_{n+2})$, establishing \eqref{eq:useful_fact}.
	
	We can now show \ref{proofItemMainThm:TnDecomp}. Consider an arbitrary $\Gamma \in \mathcal{G}^{n}(\iota_0)$. We have shown that $\Gamma(\iota_{n+1}) = \Gamma(\iota_{n+2})$. Hence, if $\Gamma \notin F^{n+1}$ and $\Gamma \notin F^{n+2}$ then $\disM(\Gamma(\iota_{n+1}),\Xi(\iota_{n+1})) < \kappa/2$ and 
$
\disM(\Gamma(\iota_{n+1}),\Xi(\iota_{n+2}))\allowbreak = \disM(\Gamma(\iota_{n+2}),\Xi(\iota_{n+2}))<\kappa/2
$ 
together yield 
$
\inf_{\xi_1\in\Xi(\iota_{n+1}), \xi_2\in \Xi(\iota_{n+2})} d_{\mathcal{M}}(\xi_1,\xi_2) < \kappa, 
$
which would contradict \ref{property:MinimisersOfNonZero}. Therefore \ref{proofItemMainThm:TnDecomp} must hold.
	
	Finally, to prove \ref{proofItemMainThm:ProbMeasureEquality}, we note that both
	$\mathcal{G}^{n}(\iota_0) \cap F^{n+1}$ and $\mathcal{G}^n(\iota_0) \cap F^{n+2}$ are measurable because $\mathcal{G}^n(\iota_0),F^{n+1}$, and $F^{n+2}$ are all measurable. We will show the result only for $\mathcal{G}^{n}(\iota_0) \cap F^{n+1}$ as the corresponding argument for $\mathcal{G}^{n}(\iota_0) \cap F^{n+2}$ is similar. If $\Gamma \in \mathcal{G}^{n}(\iota_0) \cap F^{n+1}$ then $\Gamma \in \mathcal{G}^{n}(\iota_0)$, which by \eqref{eq:useful_fact} implies that $f_{k,m}(\iota_{n+1}) = f_{k,m}(\iota_{0})$, for all $f_{k,m}\in \hat{\Lambda}_{\Gamma}(\iota_0)$. The result \ref{proofItemMainThm:ProbMeasureEquality} now follows immediately from \ref{property:PAlgorithmConsistent} in Definition \ref{definition:ProbablisticAlgorithm} of an RGA.
	
	The proof that $\strbdepsph  \geq \kappa/2$ for all $\mathrm p \in [0,1/2)$ is almost identical to the argument above, and we will point out the minor differences. We again argue by contradiction assuming that $\strbdepsph  < \kappa/2$, and the proof is identical up to \eqref{eq:union}. At this point it follows by Definition \ref{def:halting_randomised} of a halting randomised general algorithm that $\pr_{\iota_0}(\gprob_{\iota_0}=\nh) = 0$. Hence, \eqref{eq:union} becomes \begin{equation}\label{eq:union2}
		\pr_{\iota_0}\left(X \setminus  \bigcup\limits_{n=1}^{\infty} \mathcal{G}^n(\iota_0)\right) = 0
	\end{equation}
	and thus \ref{eq:fix_it} is to be replaced by
	\begin{equation*}
		0 \geq \pr_{\iota_{0}}\left(X\setminus \bigcup\limits_{n=1}^{\infty}\mathcal{G}^n(\iota_0)\right) = \lim_{n \to \infty} 1- \pr_{\iota_{0}}(\mathcal{G}^n(\iota_{0})) \geq 1-2 \mathrm p,
	\end{equation*}
	which contradicts $\mathrm p \in [0,1/2)$. The rest of the proof is identical to the argument above.  
	
	{\it Proof of \ref{property:MinimisersOfNonZero} -- \ref{property:SameWithSubsetS1} $\implies$ \ref{item:SequenceCounterexampleL1WeakBDEps} --\ref{item:SequenceCounterexampleL1StrongHaltEps}}: 
	Items \ref{item:SequenceCounterexampleL1WeakBDEps}  and \ref{item:SequenceCounterexampleL1StrongBDEps} have already been established. For  \ref{item:SequenceCounterexampleL1StrongHaltEps}, the proof stays close to the proof of \ref{item:SequenceCounterexampleL1StrongBDEps}. Indeed, the proof is identical to the proof of \ref{item:SequenceCounterexampleL1StrongBDEps} up to and including \eqref{eq:existence}. Now, however, condition \ref{property:SameWithSubsetS1} implies that $\iota_{2n}=\iota_n^1=\iota^0=\iota_0$ and thus $F^0 = F^{2n}$ for all $n\in\mathbb{N}$. Next, instead of \eqref{eq:union} we write 
	\begin{equation}\label{eq:x_union}
		X = F^0 \cup \bigcup\limits_{n=1}^{\infty} \mathcal{G}^n(\iota_0).
	\end{equation}
	We then change Claim 2 to the following:
	
	{\bf Claim $2^\prime$:} There is an \emph{odd} $n\in\mathbb{N}$ such that we have the following:
	\begin{enumerate}[label = (\Roman*${ }^\prime$)]
		\item  $\pr_{\iota_0}(F^0 \cup \mathcal{G}^n(\iota_0)) > 2\mathrm p$ \label{eq:change_I}
		\item  $\mathcal{G}^n(\iota_0) \subset F^{n+1} \cup  F^{n+2} $, \label{eq:PartIIVariant}
		\item $\pr_{\iota_0}(F^{n+1} \cap \mathcal{G}^n(\iota_0)) = \pr_{\iota_{n+1}}(F^{n+1} \cap \mathcal{G}^n(\iota_0))$, and $\pr_{\iota_0}(F^{n+2} \cap \mathcal{G}^n(\iota_0)) = \pr_{\iota_{n+2}}(F^{n+2} \cap \mathcal{G}^n(\iota_0))$.\label{eq:PartIIIVariant}
	\end{enumerate}
	The proof of \ref{eq:change_I} follows from \eqref{eq:x_union} and monotonicity of the sets $F^0 \cup \mathcal{G}^n({\iota_0})$: indeed, if by contradiction $\pr_{\iota_{0}}(F^0 \cup \mathcal{G}^n({\iota_0}))\leq 2p$ for all odd $n$ then \[1=\pr(X) = \pr_{\iota_{0}}\left(\lim_{n \to \infty} F^0 \cup \mathcal{G}^n({\iota_0})\right) = \lim_{n \to \infty}\pr_{\iota_{0}}( F^0 \cup \mathcal{G}^n({\iota_0})) \leq 2p \] by the monotone convergence theorem, which contradicts $\mathrm{p} \in [0,1/2)$. The proofs of part \ref{eq:PartIIVariant} and \ref{eq:PartIIIVariant} are then identical to those presented in the proof of Claim 2.
	
	The calculation in \eqref{eq:calculation1}, \eqref{eq:calculation2}, \eqref{eq:calculation3}, \eqref{eq:calculation4} is then changed to 
	\begin{align*}
		2p&< \pr_{\iota_0}(F^0 \cup \mathcal{G}^n({\iota_0})) = \pr_{\iota_0}\left(F^0 \cup (\mathcal{G}^{n}(\iota_0) \cap F^{n+1}) \cup (\mathcal{G}^{n}(\iota_0) \cap F^{n+2})\right) \\
		&\leq \pr_{\iota_0}(F^0) + \pr_{\iota_0} (\mathcal{G}^{n}(\iota_0) \cap F^{n+2}) \leq 2p \qquad \qquad \text{since } F^0 = F^{n+1} \text{ as $n$ is odd},
	\end{align*}
	which yields the desired contradiction that establishes that \ref{property:MinimisersOfNonZero} -- \ref{property:SameWithSubsetS1} $\implies$ \ref{item:SequenceCounterexampleL1WeakBDEps} --\ref{item:SequenceCounterexampleL1StrongHaltEps}.

	{\it Proof of \ref{property:MinimisersOfNonZero} -- \ref{property:MtotallyOrdered} $\implies$ \ref{item:SequenceCounterexampleL1WeakBDEps} -- \ref{item:SequenceCounterexampleL1NoSigmaNoPi}}: 
	Suppose by way of contradiction that there exists a tower of algorithms $\{\Gamma_n\}_{n\in\mathbb{N}}$ of height 1 for $\{\Xi,\Omega,\mathcal{M},\hat\Lambda\}$ such that $\Gamma_n(\iota)\nearrow \Xi(\iota)$ as $n\to\infty$, for all $\iota\in\Omega$. Thus, recalling that $\iota^0=\iota_n^1$ due to condition \ref{property:SameWithSubsetS1} and that $\Xi(\iota_n^1) \in S^1$ by \ref{property:MinimisersOfNonZero}, there exists an $n_0\in\mathbb{N}$ so that $\Gamma_{n_0}(\iota_0)\subset \opBall{\delta}{\Xi(\iota^0)}\subset \opBall{\delta}{S^1}$. 
	
	Now, let $M$ be a sufficiently large odd natural number such that $M>D_{\Gamma_{n_0}}(\iota_0)$. It then follows immediately from the definition of $D_{\Gamma_{n_0}}$ that if $f_{k,m}\in \hat\Lambda_{\Gamma_{n_0}}(\iota_0)$, then $m<M$. Thus, by the construction \eqref{eq:def_fm} of $\hat{\Lambda}$, it follows that $f_{k,m}(\iota_0)=f_{k,m}(\iota_M)$, for all $f_{k,m}\in \hat\Lambda_{\Gamma_{n_0}}(\iota_0)$, and hence by property \ref{property:AlgorithmSameInputSameInputTaken} in Definition \ref{definition:Algorithm} of a general algorithm, it follows that $\hat \Lambda_{\Gamma_{n_0}}(\iota_0) = \hat \Lambda_{ \Gamma_{n_0}}(\iota_{M})$, which further by \ref{property:AlgorithmDependenceOnInput} in Definition \ref{definition:Algorithm} implies that $\Gamma_{n_0}(\iota_M) = \Gamma_{n_0}(\iota_0)\in \opBall{\delta}{S^1} $. 
	
	On the other hand, we have $\Gamma_{n_0}(\iota_M)\leq \Xi(\iota_M)$ as we assumed $\Gamma_n(\iota_M)\nearrow \Xi(\iota_M)$ as $n\to\infty$, and hence $\Gamma_{n_0}(\iota_M)\leq \Xi(\iota_M)= \Xi\left(\iota_{(M+1)/2}^2\right)$. Furthermore, by \ref{property:SameWithSubsetS1} $\Xi\left(\iota_{(M+1)/2}^2\right) \in S^2$ and thus $\Gamma_{n_0}(\iota_M)\leq \sup S^2 < \inf \opBall{\delta}{S^1}$ where the last inequality follows from \ref{property:MtotallyOrdered}. This stands in contradiction to $\Gamma_{n_0}(\iota_M)\in \opBall{\delta}{\Xi(S^1)}$, thereby concluding the proof.
\end{proof}

\subsection{Impossibility results for the exit flag}
In this section we will state and prove a result that will be useful when proving the negative results contained in Theorem \ref{thm:ExitFlag}.

\begin{proposition}\label{prop:EF}
	Suppose that $\Gamma:\tilde \Omega \to \mathcal{M}$ is a general algorithm for a computational problem $\{\Xi,\Omega,\mathcal{M},\Lambda\}^{\Delta_1}=\{\tilde \Xi,\tilde \Omega,\mathcal{M},\tilde \Lambda\}$ specified according to Remark \ref{rem:EFAlgorithmAssumption} and assume that the assumption \ref{assumption:AlgorithmCloseToTheRange} holds. Furthermore suppose that there exists an $\iota_0 \in \Omega$ and, for $j=1,2$, there exist a set $S^j \subset \mathcal{M}$ and a sequence $\{\iota^j_n\}_{n=0}^{\infty}\subset \Omega$ satisfying the  following:
	\begin{enumerate}[label=(\alph*),series=EFBaseAssumptions]
		\item 
		$
		\inf_{\xi_1\in S^1,\xi_2\in S^2} d_{\mathcal{M}}(\xi_1,\xi_2) > 2 \kappa
		$, for some $\kappa>0$.  \label{assumption:EFS1S2Distant}
		\item $\Xi(\iota^j_n) \subseteq S^j$ for all $n \in \mathbb{N}$ and $j \in \{1,2\}$. \label{assumption:EFXiIotajInSj}
		\item There exists $x^j \in S^j$ such that $\disM (x^j,\Xi(\iota^j_n)) \to 0$ as $n \to \infty$ for $j=1,2$.  \label{assumption:EFxjLimitOfXiIotajn}
		\item For every $f \in \Lambda$ then $|f(\iota^j_n) - f(\iota_0)|\leq 1/2^n$ for all $n\in \mathbb{N}$ and $j = 1,2$.  \label{assumption:EFDel1Info}		
		\item $\Xi(\Omega) \subseteq \clBall{\kappa-\omega}{\Xi(\iota_0)} \cup \clBall{\kappa - \omega}{x^1} \cup \clBall{\kappa - \omega}{x^2}$, for some $\omega\in (0,\kappa)$.  \label{assumption:EFXiOmegaInBallAroundXiIota0X1AndX2}
	\end{enumerate}
	Then, for the exit flag problem $\{\Xi^E,\tilde \Omega,\{0,1\},\tilde \Lambda\}$ relative to $\Gamma$, as specified in \eqref{eq:comp_prob_exit}, we have  $\epsilon_{\mathbb{P}\mathrm{B}}^{\mathrm{s}}(\mathrm{p}) \geq 1/2$.
	Furthermore, whenever $\mathcal{M} = \real^N$, for some $N\in\mathbb{N}$, if we additionally assume that
	\begin{enumerate}[label=(\alph*),resume=EFBaseAssumptions]
		\item  $\disM(x^j, \Xi(\iota^0)) <  \omega$, for $j \in \{1,2\}$, \label{assumption:EFLLPO}
	\end{enumerate}
	then there exits a $\tilde \Lambda^+ \in \mathcal{L}^{\mathcal{O}, \omega, \tilde \Xi}(\tilde \Lambda)$ such that,
	for the computational problem  $\{\Xi^E,\tilde \Omega,\{0,1\},\tilde \Lambda^+\}$, we have $\epsilon_{\mathbb{P}\mathrm{B}}^{\mathrm{s}}(\mathrm{p}) \geq 1/2$.
\end{proposition}
\begin{proof}
	Using the sequences $\{\iota^1_n\}$, $\{\iota^2_n\}$ and the input $\iota^0 \in \Omega$, we define the following inputs in $\tilde \Omega$. 
	\begin{equation}\label{eq:EFInitialIotaDef}
		\begin{split}
			\tilde \iota^1_n&:=\{ ( \underbrace{f_{j}(\iota^0), f_{j}(\iota^0),\dotsc,f_{j}(\iota^0)}_{n \text{ times }},f_{j}(\iota^1_n),f_{j}(\iota^1_n),\dotsc ) \}_{j \in \beta},\\
			\tilde \iota^2_n&:=\{ ( \underbrace{f_{j}(\iota^0), f_{j}(\iota^0),\dotsc,f_{j}(\iota^0)}_{n \text{ times }},f_{j}(\iota^2_n),f_{j}(\iota^2_n),\dotsc ) \}_{j \in \beta},\\
			\tilde \iota^0&:= \{( f_{j}(\iota^0), f_{j}(\iota^0),\dotsc,f_{j}(\iota^0),f_{j}(\iota^0),f_{j}(\iota^0),\dotsc ) \}_{j \in \beta}, \end{split}
	\end{equation}
	where $\beta$ is the indexing set (recalling Definition \ref{definition:Omega_tilde_Delta_m}).
	By assumption \ref{assumption:EFDel1Info}, each of  $\tilde \iota^0$, $\tilde \iota^1_n$ and $\tilde \iota^2_n$ are in $\tilde \Omega$, for every $n \in \mathbb{N}$. We will prove each of the following:
	\begin{enumerate}[leftmargin=8mm, label=\Alph*)]
		\item There exists $N_0 \in \mathbb{N}$ such that, for all $n > N_0$, 
		$\Gamma(\tilde \iota^1_n) = \Gamma(\tilde \iota^2_n) = \Gamma(\tilde \iota^0)$ and at least one of the following occurs:
		\begin{enumerate}[leftmargin=8mm, label=\Alph{enumi}\roman*)]
			\item $\disM(\Gamma(\tilde \iota^1_n),\tilde\Xi(\tilde \iota^1_n)) > \kappa$ for all $n > N_0$. \label{result:WrongOnIota1n}
			\item $\disM(\Gamma(\tilde \iota^2_n),\tilde\Xi(\tilde \iota^2_n)) > \kappa$ for all $n > N_0$. \label{result:WrongOnIota2n}
		\end{enumerate} \label{result:EFCommonValue}
		\item There exists $N_1 \in \mathbb{N}$ such that at least one of the following occurs:
		\begin{enumerate}[label=\Alph{enumi}\roman*)]
			\item $\disM(\Gamma(\tilde \iota^0),\tilde \Xi(\tilde \iota^0)) \leq \kappa$. \label{result:EFRightOnIota0}
			\item $\disM(\Gamma(\tilde \iota^1_n),\tilde \Xi(\tilde \iota^1_n)) \leq \kappa$ for all $n > N_1$. \label{result:RightOnIota1n}
			\item $\disM(\Gamma(\tilde \iota^2_n),\tilde \Xi(\tilde \iota^2_n)) \leq \kappa$ for all $n > N_1$. \label{result:RightOnIota2n}
			
		\end{enumerate}\label{result:EFRightOnOne}
		
		\item There exists a sequence $\{\tilde \iota_n\}_{n=0}^{\infty}\subseteq \tilde \Omega$ (constructed from $\iota^0$ and subsequences of $\{\iota^1_n\}$ and $\{\iota^2_n\}$) satisfying the following: \begin{enumerate}[label=\Alph{enumi}\roman*)]
			\item For all $n \in \mathbb{N}$ both  $\Xi^E(\tilde \iota_{2n}) =1$ and $\Xi^E(\tilde \iota_{2n-1})= 0$.  \label{result:EFTildeIotanExitValues}
			\item For all $m \leq n$ and $j \in \beta$ we have $\tilde f_{j,m}(\tilde \iota_n) =  \tilde f_{j,m}(\tilde \iota_0)$, where the $\tilde{f}_{j,m}$ are as given in Definition \ref{definition:Omega_tilde_Delta_m}. \label{result:EFQMark}
			\item Either $\tilde \iota_{2n} = \tilde \iota_0 $  for all $n$ or $\tilde \iota_{2n-1} = \tilde \iota_0$ for all $n$. \label{result:EFTildeIota2nOr2nP1Constant}
		\end{enumerate}
		Moreover, under assumption \ref{assumption:EFLLPO} we have
		\begin{enumerate}[resume,label=\Alph{enumi}\roman*)]
			\item There exists a dyadic vector $v \in \mathbb{D}^M$ such that, for all $n$, $\disM(v,\tilde \Xi(\tilde \iota_n)) < \omega$. \label{result:EFOracleVector}
		\end{enumerate}
		\label{result:EFExistenceOfSequence}
	\end{enumerate}
	
	\textit{Proof of \ref{result:EFCommonValue}:}
	The proof of this step is similar to the proof of Proposition \ref{prop:DrivingNegativeProposition}. Recall the definition of the number of digits on the input
	\[
	D_{\Gamma}(\tilde \iota)  := \sup\lbrace k \in \mathbb{N} \, \vert \, \exists j \in \beta \text{ with } \tilde f_{j,k} \in \tilde \Lambda_{\Gamma}(\tilde \iota)\rbrace.
	\]
	required by $\Gamma$, and set  $N_0 = T_{\Gamma}(\tilde \iota^0)$, which is guaranteed to be finite since the assumption \ref{assumption:AlgorithmCloseToTheRange} implies that $\Gamma(\tilde\iota^0)\neq\nh$.
	Note that by the definition of the sequences $\{\tilde \iota^1_n\}_{n=1}^{\infty}$ and $\{\tilde \iota^2_n\}_{n=1}^{\infty}$ and $\tilde{\iota}^0$ in \eqref{eq:EFInitialIotaDef}, we have that $f(\tilde \iota_0) = f(\tilde \iota^1_n) = f(\tilde \iota^1_n)$ for all $f \in \tilde \Lambda_{\Gamma}(\tilde \iota_0)$ whenever $n \geq N_0$. Thus, by  \ref{property:AlgorithmSameInputSameInputTaken} in Definition \ref{definition:Algorithm} of a general algorithm, it follows that $\tilde \Lambda_{\Gamma}(\tilde \iota_0) = \tilde \Lambda_{\Gamma}(\tilde \iota^1_n) = \tilde \Lambda_{\Gamma}(\tilde  \iota^2_n)$. Consequently, by \ref{property:AlgorithmDependenceOnInput} and \ref{property:AlgorithmSameInputSameInputTaken} in Definition \ref{definition:Algorithm} we get that, for some $x \in \mathcal{M}$,
	\begin{equation}\label{eq:EFEqualityOfInitialAlgorithm}
		x=\Gamma(\tilde \iota_0) = \Gamma(\tilde \iota^1_n) =\Gamma(\tilde \iota^2_n),\quad\text{for }n\geq N_0.
	\end{equation}
	Suppose now that neither \ref{result:WrongOnIota1n} nor \ref{result:WrongOnIota2n} hold. Then there are $n_1\geq N_0$ and $n_2\geq N_0$ such that both $\disM(x,\tilde \Xi(\tilde \iota^1_{n_1})) = \disM(\Gamma(\tilde \iota^1_{n_1}),\tilde \Xi(\tilde \iota^1_{n_1})) \leq \kappa$ and $\disM(x,\tilde \Xi(\tilde \iota^2_{n_2}))=\disM(\Gamma(\tilde \iota^2_n),\tilde \Xi(\tilde \iota^2_n)) \leq \kappa$. But then by assumption \ref{assumption:EFXiIotajInSj} we have $\disM(x,S^1)\leq \kappa ,\disM(x,S^2) \leq \kappa$ and hence $\disM(S^1,S^2) \leq 2\kappa$, contradicting assumption \ref{assumption:EFS1S2Distant}.
	
	\textit{Proof of \ref{result:EFRightOnOne}:}
	By the assumption \ref{assumption:AlgorithmCloseToTheRange} on $\Gamma$, there exist a $y \in \Xi(\Omega)$ and an $\epsilon>0$ such that $d_{\mathcal{M}}(x,y) = \alpha - \epsilon\leq \omega -\epsilon$ where $x$ is the common value in \eqref{eq:EFEqualityOfInitialAlgorithm}. Moreover, by assumption \ref{assumption:EFXiOmegaInBallAroundXiIota0X1AndX2}, we have $y \in \clBall{\kappa-\omega}{\Xi(\iota^0)} \cup \clBall{\kappa - \omega}{x^1} \cup \clBall{\kappa - \omega}{x^2}$. We now split into two cases:
	
	{\bf Case 1:} \textit{ $y \in \clBall{\kappa - \omega}{\Xi(\iota^0)}$.}  In this case, there exists $y^2 \in \Xi(\iota^0)$ (depending on $\epsilon$) such that $\disM(y,y^2) \leq \kappa - \omega $. Thus 
	\begin{align*}
		\disM(\Gamma(\tilde\iota^0),\tilde\Xi(\tilde \iota^0)) &= \disM(\Gamma(\tilde\iota^0),\Xi(\iota^0)) \leq \disM(x,y^2) \\
		&\leq \disM(x,y) + \disM(y,y^2) \leq (\omega - \epsilon) + (\kappa - \omega) < \kappa,
	\end{align*}
	and so \ref{result:EFRightOnIota0} is satisfied.
	
	{\bf Case 2:} \textit{$y \in \clBall{\kappa - \omega}{x^j}$, for some $j \in \{1,2\}$.} We take $N_1 \geq N_0$ to be sufficiently large so that, for every $n > N_1$, there is a $v^n \in \Xi(\iota^j_n)$ with $d_{\mathcal{M}}(x^j,v^n) \leq \epsilon$. The existence of such an $N_1$ is guaranteed by assumption \ref{assumption:EFxjLimitOfXiIotajn}. Then, for $n > N_1$,
	\begin{align*} 
		&\disM(\Gamma(\tilde \iota^j_n),\tilde\Xi(\tilde \iota^j_n))=\disM(\Gamma(\tilde\iota^j_n),\Xi(\iota^j_n)) \leq \disM(x,v^n)\\ & \leq \disM(x,y) + \disM(y,x^j) + \disM(x^j,v^n) \leq (\omega - \epsilon) + (\kappa- \omega) + \epsilon = \kappa .
	\end{align*}
	If $j=1$ then we have shown that \ref{result:RightOnIota1n} occurs, whereas if $j=2$ we have shown that \ref{result:RightOnIota2n} holds. 
	
	\textit{Proof of \ref{result:EFExistenceOfSequence}:}
	To construct $(\tilde \iota_n)_{n=1}^{\infty}$, we distinguish two cases depending on whether \ref{result:EFRightOnIota0} is true or false.
	
	{\bf Case 1:} \textit{$\disM(\Gamma(\tilde\iota^0),\tilde \Xi(\tilde \iota^0)) \leq \kappa$ occurs.} We set $\tilde \iota_{2n} = \tilde \iota^0$ for all $n$ so that \ref{result:EFTildeIota2nOr2nP1Constant} holds and $\Xi^E(\tilde \iota_{2n}) = 1$.
	To define $\tilde \iota_{2n-1}$, we distinguish between whether \ref{result:WrongOnIota1n} or \ref{result:WrongOnIota2n} holds. Suppose for now that \ref{result:WrongOnIota1n} holds. If assumption \ref{assumption:EFLLPO} additionally holds, we let $N_2 > N_0$ be such that $\disM (x^1, \Xi(\iota^1_n)) < \omega/2$ for all $n > N_2$ (the existence of $N_2$ is guaranteed by assumption \ref{assumption:EFxjLimitOfXiIotajn}), and else we let $N_2=0$. Now, \ref{result:WrongOnIota1n} implies that $\Xi^E(\tilde \iota^1_{n+N_2}) = 0$, for all $n$, and so setting $\tilde \iota_{2n-1} = \tilde \iota^1_{n+N_2}$ establishes \ref{result:EFTildeIotanExitValues}. Moreover, \ref{result:EFQMark} follows by \eqref{eq:EFInitialIotaDef} and the definition of $\tilde{f}_{j,m}$.
	Under assumption \ref{assumption:EFLLPO}, we use the density of the dyadic numbers in $\mathbb{R}$ to choose a dyadic vector $v$ satisfying both $d_{\mathcal{M}}(x^1,v) \leq \omega/2$ and $ \disM(x^1, \Xi(\iota^0)) +    d_{\mathcal{M}}(x^1,v) < \omega$. Then 
	\[
	\disM (v, \tilde \Xi (\tilde \iota_{2n-1})) \leq \disM (x^1, \Xi(\iota^1_{n+N_2})) + d_{\mathcal{M}}(x^1,v) < \omega/2 + \omega/2 = \omega,
	\]
	and similarly $\disM (v, \tilde \Xi(\tilde \iota_{2n})) \leq  \disM (x^1, \Xi(\iota^0)) +    d_{\mathcal{M}}(x^1,v) < \omega$, allowing us to conclude \ref{result:EFOracleVector}.
	If instead \ref{result:WrongOnIota2n} holds, we similarly let $N_2 > N_0$ be such that $\disM (x^2, \Xi(\iota^2_n)) < \omega/2$ for all $n>N_2$, provided \ref{assumption:EFLLPO} holds, and else we set $N_2=0$. Letting  $\tilde \iota_{2n-1} = \tilde \iota^2_{n + N_2}$ and choosing $v$ to be sufficiently close to $x^2$ then yields \ref{result:EFTildeIotanExitValues},  \ref{result:EFQMark}, and \ref{result:EFOracleVector}  by the same argument as above.
	
	{\bf Case 2:} \textit{$\disM(\Gamma(\tilde\iota^0),\tilde \Xi(\tilde \iota^0)) > \kappa$ occurs.} We set $\tilde \iota_{2n-1} = \tilde \iota^0$ for all $n$ so that \ref{result:EFTildeIota2nOr2nP1Constant} holds and $\Xi^E(\tilde \iota_{2n-1}) = 0$. 
	To define $\tilde \iota_{2n}$ for all $n$, we distinguish between whether \ref{result:RightOnIota1n} or \ref{result:RightOnIota2n} holds, one of which must occur since \ref{result:EFRightOnIota0} is false by the assumption $\disM(\Gamma(\tilde\iota^0),\tilde \Xi(\tilde \iota^0)) > \kappa$.
	Suppose for now that \ref{result:RightOnIota1n} holds. Provided assumption \ref{assumption:EFLLPO} holds, we let $N_2 > N_1$ be such that $\disM (x^1, \Xi(\iota^1_n)) < \omega/2$, for all $n > N_2$ (the existence of $N_2$ is guaranteed by assumption \ref{assumption:EFxjLimitOfXiIotajn}), and else we set $N_2=0$. Then, setting $\tilde \iota_{2n} = \tilde \iota^1_{n+N_2}$, we have $\Xi^E(\tilde \iota^1_{n+N_2}) = 1$ for all $n$, thus establishing \ref{result:EFTildeIotanExitValues}, whereas \ref{result:EFQMark} follows by \eqref{eq:EFInitialIotaDef} and the definition of $\tilde{f}_{j,m}$.
	Under assumption \ref{assumption:EFLLPO} we again the density of the dyadic numbers in $\mathbb{R}$ to choose a $v \in \mathbb{D}^N$ so that both  $d_{\mathcal{M}}(v,x^1) < \omega/2$ and $\disM(x^1,\Xi(\iota^0)) + d_{\mathcal{M}}(v,x^1) < \omega$. Then 
$
\disM(v,\tilde \Xi(\tilde \iota_{2n-1})) \leq \disM(x^1,\Xi(\iota^0)) + d_{\mathcal{M}}(v,x^1) < \omega
$
 and similarly $\disM(v, \tilde \Xi(\tilde \iota_{2n})) \leq \disM (x^1, \Xi(\iota^1_{n+N_2})) + d_{\mathcal{M}}(v,x^1) < \omega/2 + \omega/2 = \omega$, establishing \ref{result:EFOracleVector}.
	Similarly, if \ref{result:RightOnIota2n} holds, we let $N_2 > N_0$ be such that $\disM(\Xi(\iota^2_n),\allowbreak x^2) < \omega$ for all $n > N_2$,  provided assumption \ref{assumption:EFLLPO} holds, and else we set $N_2=0$. Letting $\tilde \iota_{2n-1} = \tilde \iota^2_{n+N_2}$ and  $v$ sufficiently close to $x^2$ then yields \ref{result:EFTildeIotanExitValues}, \ref{result:EFQMark}, and \ref{result:EFOracleVector} by the same argument as before, thus completing the proof of \ref{result:EFExistenceOfSequence}.
	
	Now that we have shown \ref{result:EFCommonValue} to \ref{result:EFExistenceOfSequence}, we can show $\epsilon_{\mathbb{P}\mathrm{B}}^{\mathrm{s}}(\mathrm{p}) \geq 1/2$ for both the exit flag problem $\{\Xi^E,\tilde \Omega,\{0,1\},\tilde \Lambda\}$ itself as well as the oracle problem $\{\Xi^E,\tilde \Omega,\{0,1\},\tilde \Lambda^+\}$, where $\tilde \Lambda^+ \in \mathcal{L}^{\mathcal{O}, \omega, \tilde \Xi}(\tilde \Lambda)$ will be specified in due course.
	The two proofs are very similar so we will discuss differences only when they occur. For the oracle problem (under assumption \ref{assumption:EFLLPO}), we fix a $g^*=(g_1^*,\dots, g_M^*): \tilde \Omega \to \mathbb{D}^M$ such that, for $\tilde \iota \in \tilde \Omega$, we have  $g^*(\tilde \iota) = v$ if $\tilde \iota = \tilde \iota_n$ for some $n$, and else $g^*(\tilde \iota)$ is assigned a value satisfying  $\disM( g^*(\tilde \iota), \tilde \Xi(\tilde \iota)) < \omega$, but is otherwise arbitrary. Then 
	$\{g_k(\tilde \iota)\}_{k=1}^M \in \mathcal{B}^{\infty}_{\omega}(\tilde \Xi(\tilde \iota))$,  for all $\tilde \iota \in \tilde \Omega$, and thus, letting $\tilde \Lambda^+ = \tilde \Lambda \cup \{g_k^*\}_{k=1}^M$, we have $\tilde \Lambda^+ \in \mathcal{L}^{\mathcal{O}, \omega, \tilde \Xi}(\tilde \Lambda)$.

	Next, we will argue by contradiction and assume that, for some fixed $\mathrm{p} \in [0,1/2)$,
	\begin{equation}\label{eq:EFExistenceOfAlgorithm}
		\exists \, \,  \gprob \in \mathrm{RGA}  \text{ such that }  \forall \, \iota \in \tilde \Omega \quad  \mathbb{P}_{\iota}(\gprob_{\iota} \neq  \Xi^E(\iota)) \leq p.
	\end{equation}
	Continuing with this $\gprob \in \mathrm{RGA}$ we define the failure sets $F^n$ by 
	$
	F^{n} := \lbrace \Gamma^E \in X\, \vert \, \Gamma^E(\tilde \iota_{n}) \neq  \Xi^E(\tilde \iota_{n}) \rbrace
	$, for $n \in \mathbb{N}$.
	For each $\iota \in \tilde \Omega$ and $n \in \mathbb{N}$ we also define the set $\mathcal{G}^n(\iota):= \lbrace \Gamma^E \in X \, \vert \, T_{\Gamma^E}(\iota)  \leq n \rbrace.$
	Note that it is clear from \ref{property:PAlgorithmMeasurable} and \ref{property:PAlgorithmRTMeasurable} in Definition \ref{definition:ProbablisticAlgorithm} of an RGA that $F^{n}$ and $\mathcal{G}^n(\iota)$ are measurable.  
	By \eqref{eq:EFExistenceOfAlgorithm} it follows that 
	\begin{equation}\label{eq:EFFailureBoundedByp}
		\mathbb{P}_{\tilde \iota_n }(F^n) \leq p   \quad  \forall \iota \in \tilde \Omega, \, n \in \mathbb{N}.
	\end{equation}
	We will show that this leads to the desired contradiction. From here the proof is very similar to the proof of \ref{item:SequenceCounterexampleL1StrongBDEps} and \ref{item:SequenceCounterexampleL1StrongHaltEps} in Proposition \ref{prop:DrivingNegativeProposition}. Suppose we have the following. 
	\begin{enumerate}[resume,label = \Alph*)]
		\item There exits $N_3$ such that if $n\geq N_3$ then $\pr_{\iota_0}(F^0 \cup \mathcal{G}^n(\iota_0)) > 2p$. \label{proofItemMainThm:TotalProbability2}
		\item For all  $n \in \mathbb{N}$, $\mathcal{G}^n(\tilde \iota_0) = (F^{n+1}\cap \mathcal{G}^n(\tilde \iota_0)) \cup  (F^{n+2} \cap \mathcal{G}^n(\tilde \iota_0))$. \label{proofItemMainThm:TnDecomp2}
		\item For all $n \in \mathbb{N}$,  $\pr_{\tilde \iota_0}(F^{n+2} \cap \mathcal{G}^n(\tilde \iota_0)) = \pr_{\tilde \iota_{n+2}}(F^{n+2} \cap \mathcal{G}^n(\tilde \iota_0))$.\label{proofItemMainThm:ProbMeasureEquality2}
	\end{enumerate}
	Assuming (D)-(F), we may choose an integer $r >N_3$ so that $r$ is even if $\tilde \iota_0 = \tilde \iota_{2n-1}$ for all $n$ and $r$ is odd if $\tilde \iota_0 = \tilde \iota_{2n}$ for all $n$ (we know at least one of these occurs by \ref{result:EFTildeIota2nOr2nP1Constant}). Then both $\tilde \iota_0 = \tilde \iota_{r+1}$ and the conclusion of \ref{proofItemMainThm:TotalProbability2} holds for this $r$. Thus, by arguing similarly to in 
	\eqref{eq:calculation1}, \eqref{eq:calculation2}, \eqref{eq:calculation3}, \eqref{eq:calculation4}, we obtain
	\begin{align*}
		2p&< \pr_{\tilde \iota_0}(F^0 \cup \mathcal{G}^r({\iota_0})) = \pr_{\tilde \iota_0}\left(F^0 \cup (\mathcal{G}^{r}(\iota_0) \cap F^{r+1}) \cup (\mathcal{G}^{r}(\tilde \iota_0) \cap F^{r+2})\right) && \text{ by \ref{proofItemMainThm:TotalProbability2} and \ref{proofItemMainThm:TnDecomp2}} \\
		&\leq \pr_{\tilde \iota_0}(F^0) + \pr_{\tilde \iota_0} (\mathcal{G}^{r}(\tilde \iota_0) \cap F^{r+2})  \qquad \qquad  \text{since } F^0 = F^{r+1} \text{ as } \tilde\iota_0 = \tilde \iota_{r+1}\\
		& = \pr_{\tilde \iota_0}(F^0) + \pr_{\tilde \iota_{r+2}} (\mathcal{G}^{r}(\iota_0) \cap F^{r+2}) \leq \pr_{\tilde \iota_0}(F^0) + \pr_{\tilde \iota_{r+2}} (F^{r+2}) \leq 2p  && \text{by \ref{proofItemMainThm:ProbMeasureEquality2} and \eqref{eq:EFFailureBoundedByp}},
	\end{align*}
	which yields the desired contradiction. Thus, to complete the proof, it remains to establish  (D)-(F).  To this end, note that \ref{proofItemMainThm:TotalProbability2} follows from $\pr_{\iota_0}(X) = 1 > 2p$,
	$
	X = F^0 \cup \bigcup_{n=1}^{\infty} \mathcal{G}^n(\iota_0),
	$
	and the fact that $\{ \mathcal{G}^n(\iota_0)\}_{n\in \mathbb{N}}$ is an increasing sequence of sets.
	To show \ref{proofItemMainThm:TnDecomp2}  we start by assuming that $\Gamma^E \in \mathcal{G}^{n}(\tilde \iota_0)$ and then argue by contradiction. If $\Gamma^E \notin F^{n+1}$ and $\Gamma^E \notin F^{n+2}$ then $\Gamma^E(\tilde \iota_{n+1})=\Xi^E(\tilde \iota_{n+1})$ and $\Gamma^E(\tilde \iota_{n+2})= \Xi^E(\tilde \iota_{n+2})$. Since $\Gamma^E \in \mathcal{G}^n(\tilde \iota_0)$ we have $\Gamma^E(\tilde \iota_{n+1}) = \Gamma^E(\tilde \iota_{n+2})$ (this almost follows verbatim from the argument demonstrating \eqref{eq:useful_fact} -- the only differences here are that we use \ref{result:EFQMark} and, for the oracle case, the fact that $g_j(\tilde{\iota}_{n+1})=v=g_j(\tilde{\iota}_{n+2})$, for all $j=1,\dots, N$ and $n\in\mathbb{N}$). Thus $\Xi^E(\tilde \iota_{n+1})= \Xi^E(\tilde \iota_{n+2})$, which contradicts \ref{result:EFTildeIotanExitValues}.
	Finally, the proof of \ref{proofItemMainThm:ProbMeasureEquality2} is verbatim from the proof of (III) in Claim 2 in (iii) in Proposition \ref{prop:DrivingNegativeProposition}.
\end{proof}

\section{Geometry of solutions to problems  \eqref{problems} - \eqref{problems5} -- Part I}\label{sec:BasicProblemsSolutions}
The purpose of this section is to provide some simple inputs and solutions of each of the problems listed in \eqref{problems}- \eqref{problems5} that will be used to show that the breakdown epsilons for such computational problems can be non-zero. The intent will be to use these inputs and solutions to prove Theorems \ref{Cor:main}, \ref{thm:ExitFlag}, and \ref{thm:Smales9}. We will separately address counterexamples for $\ell^1$ regularisation and $\mathrm{TV}$ regularisation as the examples we construct are somewhat different in these two settings. As the results of this section are simple consequences of elementary convex analysis, their proofs are deferred to Appendix \ref{appendix:geometry-lemmas}.

\subsection{Linear programming}

Our counterexample relies on a family of matrices $\matLP(\alpha,\beta,m,N) \in \real^{m \times N}$ and a family of vectors $\vecYLP(y_1,m) \in \real^m$ for positive parameters $\alpha,\beta, y_1$ where the dimensions $m$ and $N$ satisfy $m < N$ and $N \geq 3$. Additionally, where there is no ambiguity we write $\matLP = \matLP(\alpha,\beta,m,N)$ and $\vecYLP = \vecYLP(y_1,m)$ to simplify notation.The families $\matLP(\alpha,\beta,m,N)$ and $\vecYLP(y_1,m)$ are defined as follows: 
\begin{equation}\label{eq:y_and_A_LP}
\begin{split}
\matLP(\alpha,\beta,m,N) &= \begin{pmatrix} \alpha & \beta & -1 \end{pmatrix} \oplus \begin{pmatrix} I_{m-1} & 0_{m-1 \times N-m-2} \end{pmatrix}\\
\vecYLP(y_1,m) &= y_1 e_1.
\end{split}
\end{equation} where $y_1$ is always assumed to be positive. We now state a lemma that relates these inputs to the corresponding solutions of the LP problem.

\begin{lemma} \label{lemma:ProblemBasicExampleLP}
	Let $c=\ones_{N}$ be the $N$-dimensional vector of ones. Then the solution  $\xilp$ to the linear programming problem satisfies
	\begin{equation}\label{eq:lpsolnsbasic}
	\xilp(\vecYLP,\matLP) =\begin{cases}
	\left\{\frac{y_1}{\alpha\vee \beta }e_1\right\} & \text{ if } \alpha > \beta \\
	\left\{\frac{y_1}{\alpha\vee \beta }e_2\right\}  &\text{ if } \beta > \alpha\\
	\left\{\frac{y_1}{\alpha\vee \beta }(te_1 + (1-t)e_2) \, \vert \, t \in [0,1]\right\} & \text{ if } \alpha =\beta
	\end{cases}.
	\end{equation}
	\end{lemma}

\subsection{$\ell^1$ regularisation}

Here, our counterexamples rely on a family of matrices $\matl(\alpha,\beta,m,N) \in \real^{m \times N}$ and a family of vectors $\vecYl(y_1,m) \in \real^m$ for positive parameters $\alpha,\beta, y_1$ where the dimensions $m$ and $N$ satisfy $4\leq m < N$. Additionally, where there is no ambiguity we write $\matl = \matl(\alpha,\beta,m,N)$ and $\vecYl = \vecYl(y_1,m)$. Specific to constrained lasso, we also introduce the family of vectors $\vecYCL(y_1,m)$ that depend on the regularisation parameter $\tau$ and again if there is no ambiguity we write $\vecYCL = \vecYCL(y_1,m)$.

The families $\matl(\alpha,\beta,m,N)$, $\vecYl(y_1,m)$ and $\vecYCL(y_1,m)$ are defined as follows: 
\begin{equation}\label{eq:y_and_A_l1}
\begin{split}
\matl(\alpha,\beta,m,N) &= \begin{pmatrix} \alpha & \beta \end{pmatrix} \oplus \begin{pmatrix} I_{m-1} & 0_{m-1 \times N-m-1} \end{pmatrix}\\
\vecYl(y_1,m) &= y_1 e_1,\quad \vecYCL (y_1,m) =  y_1 e_1 + \tau e_2.
\end{split}
\end{equation} where $y_1$ is always assumed to be positive.
We now state some simple lemmas that relate these inputs to their solutions under unconstrained lasso, basis pursuit and constrained lasso.
	
\begin{lemma}\label{lemma:ProblemBasicExampleULASSO}
	Assuming that $\alpha \vee \beta \geq \lambda/(2y_1)$, the solution $\xiul$ to the unconstrained lasso problem  satisfies
	\begin{equation}\label{eq:unequalAlphaBetaLASSO}
	\xiul(\vecYl,\matl)  =\begin{cases}
	\left\{\frac{2\alpha y_1-\lambda}{2\alpha^2}e_1\right\} & \text{ if } \alpha > \beta \\
	\left\{\frac{2\beta y_1-\lambda}{2\beta^2}e_2\right\}  &\text{ if } \alpha < \beta\\
	\left\{t\left(\frac{2\alpha y_1-\lambda}{2\alpha^2}\right)e_1 + (1-t)\left(\frac{2\beta y_1-\lambda}{2\beta^2}\right)e_2 \, \vert \, t \in [0,1]\right\} & \text{ if } \alpha =\beta
	\end{cases}.
	\end{equation}
\end{lemma}

\begin{lemma}\label{lemma:ProblemBasicExampleBPDNL1}
		Assuming that $y_1 \geq \delta$, we have
	\begin{equation}\label{eq:bpdnsolnsbasic}
	 \xibpdn(\vecYl,\matl) = \begin{cases} \left\{\frac{y_1 - \delta}{\alpha} e_1\right\} & \text{if } \alpha > \beta\\
	 \left\{\frac{y_1 - \delta}{\beta} e_2\right\} & \text{if } \alpha < \beta \\
	  \left\{t \frac{(y_1 - \delta)}{\alpha} e_1 + (1-t) \frac{(y_1 - \delta)}{\beta} e_2\, \vert \, t \in [0,1]\right\} & \text{if } \alpha = \beta
	   \end{cases}.
	\end{equation}
\end{lemma}

\begin{lemma}\label{lemma:ProblemBasicExampleCLASSO}
	Assuming that $y_1\geq 0$ is so that $r = \frac{(\alpha \vee \beta)y_1}{1+(\alpha \vee \beta)^2}$ satisfies $r \leq \tau$, the solution to the $\ell^1$ constrained lasso problem $\xicl$ satisfies
	\begin{equation}\label{eq:unequalAlphaBetaCLASSO}
	\xicl(\vecYCL,\matl)  =\begin{cases}
	\left\{re_1 + (\tau - r)e_3\right\} & \text{ if } \alpha > \beta \\
	\left\{re_2 + (\tau - r)e_3\right\}  &\text{ if } \alpha < \beta \\
	\left\{tre_1 + (1-t)re_2 + (\tau - r)e_3 \, \vert \, t \in [0,1]\right\} & \text{ if } \alpha =\beta
	\end{cases}.
	\end{equation}	
\end{lemma}

\subsection{TV regularisation}\label{sec:TV reg geometry}
Our counterexamples rely on a family of matrices $\matTV(\alpha,\beta,m,N) \in \real^{m \times N}$ and a family of vectors $\vecYTV(y_1,m) \in \real^m$ for positive parameters $\alpha,\beta, y_1$ where the dimensions $m$ and $N$ now satisfy $4\leq m < N$. As before, where there is no ambiguity we write $\matTV = \matTV(\alpha,\beta,m,N)$ and $\vecYTV = y(y_1,m)$.

The precise definitions of the families $\matTV(\alpha,\beta,m,N)$ and $\vecYTV(y_1,m)$ are as follows:
\begin{equation}\label{eq:y_and_A_TV}
\begin{split}
\matTV &= \alpha\, e_1\otimes e_1 + 
\beta \, e_1\otimes e_N +  e_m\otimes e_{N-1} +  \sum_{r=2}^{m-1} 
e_r\otimes e_r \in \mathbb{R}^{m\times N},\\
\vecYTV &= y_1 e_1 \in \mathbb{R}^{m}, \; y_1>0.
\end{split}
\end{equation}
To simplify the results that follow, we also define the value 
$\theta = \theta(\alpha,\beta,m)$ by
\begin{equation}\label{eq:the_theta}
\theta(\alpha,\beta,m)= \big((m-1)^2 + \allowbreak (\alpha + \beta)^2 (m-1) \big)^{\frac{1}{2}}.
\end{equation}

We now present some important lemmas relating these inputs to the corresponding solutions of problems \eqref{problems3} to \eqref{problems5} under TV regularisation. In the following, we use the flipping operator $\flipOp: \real^{N} \to \real^{N}$, defined as $(\flipOp v)_i = v_{N-i+1}$.
\begin{lemma}\label{lemma:BPTVSolutions}

	Assume that $y_1 >
	\delta\theta/(m-1)$. Then
	\begin{equation}
	\xibptv(\vecYTV,\matTV) =  \begin{cases}  \{\eta(y_1,\alpha,\beta)\}
	&\text{if } \alpha < \beta\\
	\{\flipOp \eta(y_1,\alpha,\beta )\} &\text{if } \alpha > \beta\\
	\{t\eta(y_1,\alpha,\beta)  + (1-t)(\flipOp \eta(y_1,\alpha,\beta)) \, \vert \, t \in [0,1] \} &\text{if } \alpha = \beta
	\end{cases},	\label{eq:TVBPDNSolutions}
	\end{equation}
	where 
	\begin{equation}\label{eq:the_eta}
	\begin{split}
	\eta(y_1,\alpha,\beta)_1 &= \eta(y_1,\alpha,\beta)_2 = \dotsb = \eta(y_1,\alpha,\beta)_{N-1}= \frac{\delta (\alpha + \beta)}{\theta}, \\
	 \eta(y_1,\alpha,\beta)_{N} &=  \frac{\delta (\alpha + \beta)}{\theta} + \frac{y_1-\delta\theta/(m-1) }{\alpha\vee \beta}
	\end{split}	
	\end{equation}
\end{lemma}

\begin{lemma}
	\label{lemma:ULTVSolutions}
	Let  $\matTV$ and $\vecYTV$ be as above, let $y_1> \frac {\lambda\theta^2}{2 (m-1)^2 (\alpha \vee \beta) }$. Let $\xiultv$ denote the solution map for TV Unconstrained Lasso with parameter $\lambda$. Then 
	\begin{equation}
	\xiultv(\vecYTV,\matTV) =  \begin{cases}  \{\psi(y_1, \alpha,\beta)\} &\text{if } \alpha<\beta\\
	\{(\flipOp \psi(y_1,\alpha,\beta))\}& \text{if } \alpha > \beta\; ,\\
	\{[t\psi(y_1, \alpha,\beta)  + (1-t)(\flipOp \psi(y_1, \alpha,\beta)) \, \vert \, t \in [0,1] \} & \text{if } \alpha = \beta
	\end{cases}	\label{eq:ULTVLassoSolutions}
	\end{equation}
	where
	\begin{equation}\label{eq:the_tilde_eta}
	\begin{split}
	\psi(y_1,\alpha,\beta)_1 &= \psi(y_1,\alpha,\beta)_2 = \dotsb = \psi(y_1,\alpha,\beta)_{N-1}= \frac{\lambda (\alpha + \beta)}{2(m-1)(\alpha\vee \beta)}\; , \\
	 \psi(y_1,\alpha,\beta)_{N} &=  \frac{\lambda (\alpha + \beta)}{2(m-1)(\alpha\vee \beta)}  + \frac{1}{\alpha\vee\beta} \left(y_1-  \frac {\lambda\theta^2}{2 (m-1)^2 (\alpha \vee \beta) } \right).
	\end{split}	
	\end{equation}
\end{lemma}
\subsection{Linear programming and basis pursuit examples for the exit flag theorem}
For $\alpha \geq 0$ let
\begin{equation}\label{eq:y_and_A_LPexitFlag}
\begin{split}
\matLPBPExit(\alpha,m,N) &=  \alpha   \oplus \begin{pmatrix} I_{m-1} & 0_{m-1 \times N-m-1} \end{pmatrix}\\
\vecYl(y_1,m) &= y_1 e_1,
\end{split}
\end{equation} where $y_1$ is always assumed to be non-negative (but not necessarily non-zero). Additionally, assume that $c$ is the $N$-dimensional vector of $1$s.
We then have the following:
\begin{lemma}\label{lemma:ProblemBasicExampleLPBPExitFlag}
	For basis pursuit and linear programming we have
	 \[
	\xilp(\vecYl,\matLPBPExit)=\xibp(\vecYl,\matLPBPExit) = \begin{cases} (y_1/\alpha) e_1 & \text{ if } \alpha > 0\\ 
	0 &\text{ if } \alpha = y_1 = 0\end{cases} .
	\]
\end{lemma}
\subsection{Objective values for linear programming and Smale's 9th problem}
Our counterexamples used when proving Theorem \ref{thm:Smales9} rely on the matrix and vector pairs 
defined as follows for $\alpha,\beta \geq 0$:
\begin{equation}\label{eq:y_and_A_LPObj}
\begin{split}
\matLPObj(\alpha,\beta,m,N) &= \begin{pmatrix} \alpha & -\beta \end{pmatrix} \oplus \begin{pmatrix} I_{m-1} & 0_{m-1 \times N-m-1} \end{pmatrix},\\
\vecYl(y_1,m) &= y_1 e_1.
\end{split}
\end{equation} where $y_1$ is always assumed to be non-negative. We then have the following lemma:
\begin{lemma}\label{lemma:ProblemBasicExampleLPObj}
	For $k \in\mathbb{N}$ and real $M\geq 0$, consider the decision problem \eqref{eq:Smales9restatement} with $c=\ones_N$ (for each fixed dimension $N\in\mathbb{N}$) and the corresponding solution map $\Xi_k$ as defined in \eqref{eq:Xi_K2}. Then, for $y_1>0$, $\alpha > 0$, $\beta \geq 0$, we have
	\begin{equation}\label{eq:LPObjBasic}
	\Xi_k(\vecYl(y_1,m),\matLPObj(\alpha,\beta,m,N)) = \begin{cases} 1& \text{ if } y_1/\alpha < 10^{-k} (\floor{10^{k} M}+1 )\\ 
	0& \text{ if } y_1/\alpha \geq  10^{-k} (\floor{10^{k} M}+1 )
	\end{cases},
	\end{equation}
	where $\Xi_k$ is defined in \eqref{eq:Xi_K2}.
	In addition, for all $\alpha,\beta \in \real$, we have 
	$\Xi_k(\vecYl(0,m),\matLPObj(\alpha,\beta,m,N)) = 1$.
\end{lemma}

\section{Proof of Theorem \ref{Cor:main} -- Preliminaries: constructing $\Omega$}

\subsection{Strategy for the proof}

The proof of Theorem \ref{Cor:main} formally consists of the proofs of Proposition \ref{Cor:main_SCI_cont} (corresponding to parts (i) and (ii) of the theorem) and Proposition \ref{Cor:main_SCI} (corresponding to parts (iii) and (iv) of the theorem), which will be proved in \S\ref{sec:cor_main(i)(ii)} and \S\ref{sec:cor_main(iii)(iv)}, respectively. 

The strategy for these proofs is as follows. First, for each of problems \eqref{problems} - \eqref{problems5}, we need to construct a suitable input set (which will depend both on $K$ and any relevant regularisation parameters) of the desired form, i.e., a set $\Omega=\bigcup_{m,N}\Omega_{m,N}$, where each $\Omega_{m,N}$ consists of inputs of fixed dimension.
The computational problem corresponding to each $\Omega_{m,N}$ (for fixed $m$ and $N$) will have a strong breakdown epsilon exceeding $10^{-K}$ as well as a weak breakdown epsilon exceeding $10^{-K+1}$.  This will be achieved by setting  $\Omega_{m,N}=\Omega_{m,N}^{\mathrm{s}}\cup \Omega_{m,N}^{\mathrm{w}}$, where $\Omega_{m,N}^{\mathrm{s}}$ is a set of input for which the computational problem $\{\Xi, \Omega_{m,N}^{\mathrm{s}},\mathcal{M}, \hat\Lambda\}$ has a strong breakdown epsilon of size exceeding $10^{-K}$ and $\Omega_{m,N}^{\mathrm{w}}$ is a set of input for which  $\{\Xi, \Omega_{m,N}^{\mathrm{w}},\mathcal{M}, \hat\Lambda\}$ has a weak breakdown epsilon exceeding $10^{-K+1}$, where the evaluation set $\hat \Lambda$ will be provided by Proposition \ref{prop:DrivingNegativeProposition}. The whole construction will make heavy use of the results presented in \S\ref{sec:BasicProblemsSolutions}. Specifically, our inputs will take the form of  matrix vector pairs $(\vecYl,\matLP)$, $(\vecYl,\matl)$ and $(\vecYTV,\matTV)$ presented in that section. 

 The algorithms whose existence is claimed in the statements of Propositions  \ref{Cor:main_SCI} and  \ref{Cor:main_SCI_cont} will be constructed with the help of various subroutines introduced in \S\ref{sec:UsefulSubroutinesForMainThm}.

\subsection{Constructing the sets of inputs for Theorem \ref{Cor:main}\label{sec:constr-Cor:main-sets}}
\subsubsection{The set of inputs}
Each individual problem will require a separate input set. Moreover, they will depend on the integer $K$ from the statement of Theorem \ref{Cor:main} as well as any relevant regularisation parameters.  We will denote the input sets for LP, $\ell^1$ BP, $\ell^1$ UL, CL, $\mathrm{TV}$ BP and $\mathrm{TV}$ UL
by 
\begin{equation}\label{eq:InputSetLabels}
	\omlp, \ombpl, \omull,\omcll, \ombptv \text{ and } \omultv,
\end{equation}
respectively. We remark that these will depend on $K$ and the regularisation parameters $\delta$, $\lambda$, and $\tau$, however, in order to lighten the notation, we omit making this dependence explicit. For linear programming and the $\ell^1$ regularised problems, the set $\albetSet \subset \real^2$ defined by $\albetSet=\left([1/4,1/2]\times\{1/2\} \right)\cup\left(\{1/2\} \times [1/4,1/2] \right)$ will prove useful whereas for the TV problems we denote, for integers $n$ (and for given basis pursuit denoising parameter $\delta$ and unconstrained lasso parameter $\lambda$) 
\begin{equation}\label{eq:rn_and_sn}
	r_n = \begin{cases} \frac{1}{4} \vee \frac{1}{2} \sqrt{1 - \frac{3 \cdot 10^{-n}}{4\delta}} & \text{ if } {3 \cdot 10^{-n}} \leq {4\delta}\\
		\frac{1}{4} & \text{ otherwise} \end{cases} , \quad s_n = \begin{cases} \frac{1}{4} \vee \frac{1}{2} \sqrt{1 - \frac{3 \cdot 10^{-n}}{4\lambda}} & \text{ if } {3 \cdot 10^{-n}} \leq {4\lambda}\\
		\frac{1}{4} & \text{ otherwise} \end{cases}
\end{equation} and define $\albetSetBPTV{n}=\left([r_n,1/2]\times\{1/2\} \right)\cup\left(\{1/2\} \times [r_n,1/2] \right)$ and $\albetSetULTV{n}=\left([s_n,1/2]\times\{1/2\} \right)\cup\left(\{1/2\} \times [s_n,1/2] \right)$.
Recall the definitions of  $\matLP(\alpha,\beta,m,N)$ and $\vecYLP(y_1,m)$ from \eqref{eq:y_and_A_LP}, $\matl(\alpha,\beta,m,N)$ and $\vecYl(y_1,m)$ from \eqref{eq:y_and_A_l1}, as well as $\matTV(\alpha,\beta,m,N)$ and $\vecYTV(y_1,m)$ from \eqref{eq:y_and_A_TV}, and, for $k\geq 1$, define 
\begin{equation}\label{eq:strongOmega-fixed-dim}
\begin{aligned}
	\omlps[k] &:= \left\{ \left(\vecYMainLP{k}(m) ,\matLP(\alpha,\beta,m,N)\right)\, \vert \,  (\alpha,\beta)\in \albetSet \right \},\\
	\ombpls[k] &:= \left\{ \left(\vecYMainBP{k}(m) ,\matl(\alpha,\beta,m,N)\right)\, \vert \, (\alpha,\beta)\in \albetSet \right \},\\
	\omulls[k] &:= \left\{ \left(\vecYMainUL{k}(m),\matl(\alpha,\beta,m,N)\right)\, \vert \,  (\alpha,\beta)\in \albetSet \right \} ,\\
	\omclls[k] &:= \left\{ \left(\vecYMainCL{k}(m),\matl(\alpha,\beta,m,N)\right)\, \vert \,  (\alpha,\beta)\in \albetSet \right \}  ,\\
	\ombptvs[k] &:= \left\{ \left(\vecYMainBPTV{k}(m),\matTV(\alpha,\beta,m,N)\right)\, \vert \, (\alpha,\beta)\in \albetSetBPTV{k} \right \},  \\
	\omultvs[k] &:= \left\{ \left(\vecYMainULTV{k}(m),\matTV(\alpha,\beta,m,N)\right)\, \vert \,  (\alpha,\beta)\in \albetSetULTV{k} \right \} 
\end{aligned}
\end{equation}
where the superscript $\mathrm{s}$ notes that the input sets are designed for results concerning the strong breakdown epsilon, and  
\begin{equation}\label{eq:thm3.2_y_values_strong}
	\begin{split}
		&\vecYMainLP{k}(m) := \vecYLP(2 \cdot 10^{-k},m), \qquad \vecYMainBP{k}(m) = \vecYl(2 \cdot 10^{-K}+\delta,m),\\
		&\vecYMainUL{k}(m) :=  \vecYl\left(2\cdot  10^{-k}+\lambda,m\right),\qquad \vecYMainCL{k}(m) :=  \vecYCL\left(10^{-k+1},m\right), \\
		& \vecYMainBPTV{k}(m) := \vecYTV\left(\frac{\delta m} {\theta(1/2,1/2,m)} + \left[7 - \frac{3} {\theta(1/2,1/2,m)}\right]\frac{10^{-k}}{4},m\right),\\
		&\vecYMainULTV{k}(m) := \vecYTV\left( \frac{7 \cdot 10^{-k} }{4} + \lambda \left[1 + \frac{1}{m-1}\right] - \frac{3 \cdot 10^{-k}}{4(m-1)},m\right),
	\end{split}
\end{equation}
where $\lambda \in ( 0,1/3]$ is the regularisation parameter from the UL problem \eqref{problems4} and $\delta \in [ 0,1]$ is the regularisation parameter from the BP problem \eqref{problems3}. Where it is either clear or superfluous to the result proven, we will omit the dependency on $m$ in the definition of the $y$ vectors.

To create input sets that capture the required weak breakdown epsilons, we use the superscript $\mathrm{w}$ and, for $k\geq 2$, define
\begin{equation}\label{eq:weakOmega-fixed-dim}
\begin{aligned}
	\omlpw[k] &= \left\{ \left(\vecYMainLP{k-1}(m) , \matLP(\alpha,\beta,m,N)\right)\, \vert \,  (\alpha,\beta)\in \albetSet\setminus \{z^a\}\right \} \\
	\ombplw[k] &= \left\{ \left(\vecYMainBP{k-1}(m) ,\matl(\alpha,\beta,m,N)\right)\, \vert \,  (\alpha,\beta)\in \albetSet\setminus \{z^a\} \right \} \\
	\omullw[k] &= \left\{ \left(\vecYMainUL{k-1}(m),\matl(\alpha,\beta,m,N)\right)\, \vert \,  (\alpha,\beta)\in \albetSet\setminus \{z^a\} \right \} \\
	\omcllw[k] &= \left\{ \left(\vecYMainCL{k-1}(m),\matl(\alpha,\beta,m,N)\right)\, \vert \,  (\alpha,\beta)\in \albetSet\setminus \{z^a\} \right \} ,\\
	\ombptvw[k] &= \left\{ \left(\vecYMainBPTV{k-1}(m),\matTV(\alpha,\beta,m,N)\right)\, \vert \,  (\alpha,\beta)\in \albetSetBPTV{k-1}\setminus \{z^a\} \right \}  \\
	\omultvw[k] &= \left\{ \left(\vecYMainULTV{k-1}(m),\matTV(\alpha,\beta,m,N)\right)\, \vert \,  (\alpha,\beta)\in \albetSetULTV{k-1}\setminus \{z^a\} \right \},
\end{aligned}
\end{equation}
where $z^a$ is the corner point $(1/2,1/2)$ of $\mathcal{L}$, whereas for $k=1$ we let each of $\omlpw[k]$, $\ombplw[k]$, $\omullw[k]$, $\omcllw[k]$, $\ombptvw[k]$, and $\omultvw[k]$ be the empty set.
We now define the fixed-dimension input sets as the  union of the corresponding ``strong'' and and ``weak '' sets:
\begin{equation}\label{eq:defs-5-Om-fixedDim}
\begin{gathered}
	\omlpb[k] = \omlps[k] \cup \omlpw[k] , \quad \omcllb[k]  = \omclls[k]  \cup \omcllw[k] , \\
	\ombplb[k] = \ombpls[k]  \cup \ombplw[k] ,\quad \omullb[k]  = \omulls[k]  \cup \omullw[k]  ,\\
	\ombptvb[k] = \ombptvs[k]  \cup \ombptvw[k] , \quad  \omultvb[k]  =\omultvs[k]  \cup \omultvw[k],
\end{gathered}
\end{equation}
for each $k\geq 1$, $m\geq 4$, and $N>m$, and the combined input sets as the union of the fixed-dimension input sets with the given value of $K$ over all admissible dimensions:
\begin{equation}\label{eq:defs-5-Om}
\begin{gathered}
	\omlp = \bigcup_{\substack{m,N\in\mathbb{N} \\ N>m\geq 4}} \omlpb[K], \quad \omcll = \bigcup_{\substack{m,N\in\mathbb{N} \\ N>m\geq 4}} \omcllb[K],\quad \ombpl = \bigcup_{\substack{m,N\in\mathbb{N} \\ N>m\geq 4}} \ombplb[K] ,\\
	 \omull = \bigcup_{\substack{m,N\in\mathbb{N} \\ N>m\geq 4}} \omullb[K] ,\quad \ombptv = \bigcup_{\substack{m,N\in\mathbb{N} \\ N>m\geq 4}} \ombptvb[K], \quad  \omultv = \bigcup_{\substack{m,N\in\mathbb{N} \\ N>m\geq 4}} \omultvb[K] 
\end{gathered}
\end{equation}

\subsection{Size and conditioning of the inputs}
In this section we will analyse the bounds on the inputs as well as each of the condition numbers relevant to the input sets defined in \S\ref{sec:constr-Cor:main-sets}.
\subsubsection{Size of the inputs}
\begin{lemma}\label{lemma:mainThmSizeEstimates}
	For natural numbers $m$ and $N$, let $y \in \real^m$ and $A \in \real^{m \times N}$. Then for an integer $k \geq 1$ the following hold:
	\begin{enumerate}[leftmargin=8mm]
		\item If $(y,A)$ is an element of one of the sets $\omlps[k],\ombpls[k],\omulls[k],\omclls[k]$ then $\|y\|_{\infty} \leq 2$ and $\|A\|_{\max} \leq 1 $. \label{result:L1SBSize}
		\item If $(y,A)$ is an element of $\ombptvs[k]$ then $\|y\|_{\infty} \leq 3/2$ and $\|A\|_{\max} = 1$. \label{result:BPTVSBSize}
		\item If $(y,A)$ is an element of $\omultvs[k]$ then $\|y\|_{\infty} \leq 107/180$ and $\|A\|_{\max} = 1$. \label{result:ULTVSBSize}
	\end{enumerate}
	Similarly, for an integer $k \geq 2$ the following statements hold:
	\begin{enumerate} [leftmargin=8mm, resume]
		\item If $(y,A)$ is an element of one of the sets $\omlpw[k],\ombplw[k],\omullw[k],\omcllw[k]$ then $\|y\|_{\infty} \leq 2$ and $\|A\|_{\max} \leq 1 $. \label{result:firstWBSize}
		\item If $\Omega = \ombptvw[k]$ then $\|y\|_{\infty} \leq 3/2$ and $\|A\|_{\max} = 1$. 
		\item If $\Omega = \omultvw[k]$ then $\|y\|_{\infty} \leq 107/180$ and $\|A\|_{\max} = 1$. \label{result:lastWBSize}
	\end{enumerate}
\end{lemma}
\begin{proof}
	We start with the case that $(y,A)$ is in one of $\omlps[k],\ombpls[k],\omulls[k],\omclls[k]$. In this case, we must have exactly one of  $A = \matLP(\alpha,\beta,m,N)$ or $A = \matl(\alpha,\beta,m,N)$. In any of these cases it is easy to see that $\|A\|_{\max} = 1$ directly from the definitions. We also have $\|y\|_\infty \leq \max(2\cdot 10^{-k},2\cdot 10^{-k} + \delta ,2\cdot 10^{-k} + \lambda,10^{-k+1},\tau)\leq 2$, as each of these terms is at most 2.
	
	Next, we analyse $(y,A) \in \ombptvs[k]$: if $(y,A) \in \ombptvs[k]$ then (writing $\theta = \theta(1/2,1/2,m)$ and by the definition of $\vecYMainBPTV{k}$)
	\begin{align*}
		\|y\|_{\infty} &= \frac{\delta m}{\theta} + \left(7 - \frac{3}{\theta} \right) \frac{10^{-k}}{4} \leq \frac{1}{m-1} \left(m - \frac{3\cdot 10^{-k}}{4}\right) + \frac{7 \cdot 10^{-k}}{4} \\&= 1 + \frac{4 - 3\cdot 10^{-k}}{4(m-1)} + \frac{7 \cdot 10^{-k}}{4} \leq \frac{4}{3} + \frac{3 \cdot 10^{-k}}{2}
	\end{align*}
	where the first inequality follows because $\delta \leq 1$ and $\theta \geq m-1$ and the final inequality because $m \geq 4$. Since $k \geq 1$ we obtain $\|y\|_{\infty} < 3/2$.  By the definition of $\ombptvs[k]$, we have $A = \matTV(\alpha,\beta,m,N)$ for some $(\alpha,\beta) \in \albetSetBPTV{k}$. Thus from the definition of $\matTV$ and the fact that $\alpha,\beta \leq 1/2$ we conclude that $\|A\|_{\max} \leq 1$. Thus we have shown \ref{result:BPTVSBSize}.
	
	To analyse the case where $(y,A) \in \omultvs[k]$, note that if $(y,A) \in \omultvs[k]$ then 
	\begin{align*}
		\|y\|_{\infty} &= \frac{7 \cdot 10^{-k}}{4} + \lambda \left(1 + \frac{1}{m-1}\right) - \frac{3 \cdot 10^{-k}}{4(m-1)}\\&\leq  \frac{7 \cdot 10^{-k}}{4} + \frac{1}{3} + \left(\frac{1}{3} - \frac{3 \cdot 10^{-k}}{4}\right)\frac{1}{m-1} \leq \frac{3 \cdot 10^{-k}}{2} + \frac{4}{9}
	\end{align*}
	where the first inequality follows because $\lambda \leq 1/3$ and the final inequality because $m \geq 4$. Thus for $k \geq 1$ we obtain $\|y\|_{\infty} \leq 4/9 + 3/20 = 107/180$. The argument that $\|A\|_{\max} \leq 1$ is identical to the analysis performed to prove the bounds on $\ombptv$, except now we replace all statements and sets referring to basis pursuit with those referring to unconstrained lasso. This proves \ref{result:ULTVSBSize}.
	
	Finally, we can prove \ref{result:firstWBSize} to \ref{result:lastWBSize} by using \ref{result:L1SBSize} to \ref{result:ULTVSBSize} and noting the inclusions
	\begin{gather*}
	\omlpw[k] \subseteq \omlps[k-1], \quad \ombplw[k] \subseteq \ombpls[k-1], \quad  \omullw[k]\subseteq \omulls[k-1],\quad \omcllw[k] \subseteq \omclls[k-1]\\
		\ombptvw[k] \subseteq \ombptvs[k-1],\quad \omultvw[k] \subseteq \omultvs[k-1]
		\end{gather*} that hold whenever $k \geq 2$.
\end{proof}

\subsubsection{Condition of the solution map}

We  compute the condition of the solution map for the problems LP, $\ell^1$ BP, $\ell^1$ UL, TV BP, and TV UL, with the input sets as specified in \S\ref{sec:constr-Cor:main-sets}. Concretely, we prove the following lemma.

\begin{lemma}\label{lem:cond-Xi-various}
	We have
	\begin{align}
	&\cond{\xilp,\omlp} \leq 15,\quad \cond{\xiul,\omull} \leq 28, \quad \cond{\xicl,\omcll}\leq 2, \notag\\
	&\cond{\xibp,\ombpl} \leq 24,\quad \cond{\xibptv,\ombptv} \leq 35,\quad \cond{\xiultv,\omultv} \leq 179, \label{eq:MainInputConditionMap}
	\end{align}
	where $\omlp$, $\ombpl$,  $\omull$, $\ombptv$, $\omcll$, and  $\omultv$ are defined in \eqref{eq:defs-5-Om}.
\end{lemma}

This will give us the desired bounds on the condition numbers. Proving Lemma \ref{lem:cond-Xi-various} will be straightforward for linear programming and the $\ell^1$ regularised problems but will require a bit more effort for the TV problems. We thus state and prove three simple lemmas that will be useful.

\begin{lemma}\label{lem:alpha-vee-beta-estimate}
	Let $\alpha,\beta>0$, and let $z=(z_\alpha,z_\beta)\in\real^2$. Then
	\begin{equation*}
	|(\alpha+z_\alpha)\vee (\beta+ z_\beta) - \alpha\vee \beta|\leq \|z\|_\infty.
	\end{equation*}
\end{lemma}
\begin{proof}
	Suppose w.l.o.g. that $\alpha\geq \beta$. Then $\alpha-\|z\|_\infty \leq \alpha+z_\alpha \leq (\alpha+z_\alpha)\vee (\beta+ z_\beta)\leq \alpha+\|z\|_\infty$, and so
	$|(\alpha+z_\alpha)\vee (\beta+ z_\beta) - \alpha|\leq \|z\|_\infty$, as desired.
\end{proof}

\begin{lemma}\label{lem:eta-pert-bound} 
	Let $(\alpha,\beta)\in\albetSetBPTV{n}$ for some natural number $n$ and let $m$, $N$, $\theta=\theta(\cdot,\cdot)$, and $\eta=\eta(\cdot,\cdot,\cdot )$ be as defined in \eqref{eq:the_theta} and \eqref{eq:the_eta}. Let $y_1>0 $ be such that $\mu(y_1,\alpha,\beta):=y_1-\delta\theta(\alpha,\beta,m)/(m-1)>0$ and $z=(z_y,z_\alpha,z_\beta)\in\real^3$ with $\|z\|_\infty\leq \frac{1}{8}\wedge\frac{1}{2}\mu(y_1,\alpha,\beta) $. Then $\mu((y_1,\alpha,\beta) + z)>0$ and
	\begin{equation*}
	\|\eta(y_1, \alpha,\beta) - \eta(y_1+z_y, \alpha+z_\alpha,\beta+z_\beta)\|_\infty\leq 14(y_1+1)\|z\|_\infty.
	\end{equation*}
\end{lemma}\begin{proof}[Proof of Lemma \ref{lem:eta-pert-bound}]
	First note that, using the mean value theorem and the fact that $\alpha + z_\alpha,\beta + z_\beta \in [1/8,5/8]$, we have
	\begin{equation}\label{eq:thetaBound}
	\begin{aligned}
	&|\theta(\alpha+z_\alpha,\beta+z_\beta,m)- \theta(\alpha,\beta,m)|\leq  \max_{(u,v)\in[1/8,5/8]}\|\nabla_{u,v}\theta(u,v,m)\|_1 \|z\|_\infty\\
	& \qquad \leq  \max\limits_{(u,v)\in[1/8,5/8]}\left|\frac{2(u+v)(m-1)}{\theta(u,v,m)}\right|\|z\|_\infty \leq  \max\limits_{(u,v)\in[3/8,5/8]}2(u+v) \|z\|_\infty  \leq \frac{5}{2}\|z\|_\infty,
	\end{aligned}
	\end{equation}
	since 
\[
(u+v)(m-1)/\theta(u,v,m) = [1/(u+v)^2 + 1/(m-1)]^{-1/2} \leq (u+v).
\]
 We can now write
	\begin{equation*}
	\eta(y_1 + z_y,\alpha+z_\alpha,\beta+z_\beta)= \frac{\delta \cdot (\alpha+z_\alpha+ \beta+z_\beta )}{\theta(\alpha+z_\alpha,\beta+z_\beta,m)}\ones_{N} + \frac{\mu((y_1,\alpha,\beta) + z) }{(\alpha+z_\alpha)\vee (\beta+z_\beta)} e_{N},
	\end{equation*}
	where $\ones_{N}\in\real^N$ is the vector of all ones, so that 
	\begin{equation}\label{eq:TwoTermCondBound}
	\begin{aligned}
	&\|\eta(y_1,\alpha,\beta) - \eta(y_1+z_y,\alpha+z_\alpha,\beta+z_\beta)\|_\infty\\
	& \leq \left\| \frac{\delta \cdot (\alpha+z_\alpha+ \beta+z_\beta )}{\theta(\alpha+z_\alpha,\beta+z_\beta,m)}\ones_{N} - \frac{\delta\cdot (\alpha+\beta)}{\theta(\alpha,\beta,m)}  \ones_{N} \right\|_\infty  + \left\|\frac{\mu((y_1,\alpha,\beta) + z)}{(\alpha+z_\alpha)\vee (\beta+z_\beta)} e_{N} - \frac{\mu(y_1,\alpha,\beta) }{\alpha\vee \beta} e_{N}\right\|_\infty\\
	& =  \left| \frac{\delta \cdot (\alpha+z_\alpha+ \beta+z_\beta )}{\theta(\alpha+z_\alpha,\beta+z_\beta,m)} - \frac{\delta\cdot (\alpha+\beta)}{\theta(\alpha,\beta,m)} \right|  + \left|\frac{\mu((y_1,\alpha,\beta) + z)}{(\alpha+z_\alpha)\vee (\beta+z_\beta)}- \frac{\mu(y_1,\alpha,\beta) }{\alpha\vee \beta} \right|.
	\end{aligned}
	\end{equation}
	We will bound each of the latter two terms separately. Write $f(u,v) = u/v$, for $v$ non-zero. Note that, since $(\alpha,\beta) \in \albetSetBPTV{n}$, we have $\alpha + z_\alpha + \beta + z_\beta \in (1/2,5/4)$: indeed, $1/2 = 3/4 - 2/8\leq \alpha + \beta -2\|z\|_{\infty} \leq \alpha + z_\alpha + \beta + z_\beta$ and similarly $\alpha + z_\alpha + \beta + z_\beta \leq \alpha + \beta + 2\|z\|_{\infty}\leq 1 + 2/8 = 5/4$. In a similar way,  $m-1 \leq  \sqrt{(m-1)^2+(m-1)(\alpha + z_\alpha + \beta + z_\beta)} \leq \sqrt{(m-1)^2 + 5(m-1)/4} \leq m$ so that $\theta(\alpha+z_\alpha,\beta+z_\beta,m) \in [m-1,m]$. Thus, by the mean value theorem 
	\begin{equation}\label{eq:FirstTermCondBPBound}
	\begin{aligned}
	& \left| \frac{\delta \cdot (\alpha+z_\alpha+ \beta+z_\beta )}{\theta(\alpha+z_\alpha,\beta+z_\beta,m)} - \frac{\delta\cdot (\alpha+\beta)}{\theta(\alpha,\beta,m)} \right|  \\&\leq \delta \max\limits_{(u,v)\in [1/2,5/4]\times [m-1,m] }\|\nabla f(u,v)\|_1 \cdot \left(|z_\alpha + z_\beta| \vee |\theta(\alpha+z_\alpha,\beta+z_\beta,m)- \theta(\alpha,\beta,m) |\right) \\
	&\leq \delta \left(\frac{1}{m-1} + \frac{5/4}{(m-1)^2}\right)\cdot (2\|z\|_{\infty} \vee |\theta(\alpha+z_\alpha,\beta+z_\beta,m)- \theta(\alpha,\beta,m)|) \leq  \frac{85\|z\|_{\infty}}{72}, 
	\end{aligned}
	\end{equation}
	where in the final line we used \eqref{eq:thetaBound}, $\delta\leq 1$, and the assumption that $m \geq 4$.
	Next, again using \eqref{eq:thetaBound}, we obtain 
\[	
|\mu((y_1,\alpha,\beta) + z)-\mu(y_1, \alpha,\beta)|\leq |z_y| + \frac{\delta}{m-1}|\theta(\alpha+z_\alpha,\beta+z_\beta,m)- \theta(\alpha,\beta,m)|\leq \frac{11}{6}\|z\|_\infty,
\] 
since  $\delta \in [0,1]$ and $m \geq 4$. In particular, since $\mu(y_1,\alpha,\beta) \geq 2\|z\|_{\infty} $ and $\mu(y_1,\alpha,\beta) \leq y_1$ we obtain 
\[
\mu((y_1,\alpha,\beta) + z) \in [\|z\|_\infty/6,y_1 +11\|z\|_{\infty}/6] \subset [0,y_1 + 1/4], 
\]
establishing the claim that $\mu((y_1,\alpha,\beta)+z) > 0$.
	Now, by Lemma \ref{lem:alpha-vee-beta-estimate}, $(\alpha + z_\alpha) \vee (\beta + z_\beta) \in [3/8,5/8]$.
	Therefore, using the mean value theorem together with the bounds above we obtain 
	\begin{equation}\label{eq:SecondTermCondBPBound}
	\begin{aligned}
	& \left|\frac{\mu((y_1,\alpha,\beta) + z)}{(\alpha+z_\alpha)\vee (\beta+z_\beta)}- \frac{\mu(y_1,\alpha,\beta) }{\alpha\vee \beta} \right|  \leq \max_{(u,v)\in [0,y_1+\frac{1}{4}]\times [\frac{3}{8},\frac{5}{8}] } \!\|\nabla f(u,v)\|_1 \\ 
	& \qquad \qquad \cdot \left( |\mu((y_1,\alpha,\beta) + z)-\mu(y_1, \alpha,\beta)| \vee |(\alpha+z_\alpha)\vee (\beta+z_\beta) -\alpha\vee \beta| \right) \\
	& \qquad \qquad \quad  \leq \left(\frac{1}{3/8} +\frac{(y_1 +\frac{1}{4})}{(3/8)^2}\right) \cdot \left(\frac{11}{6}\|z\|_\infty \vee \|z\|_{\infty}\right) = \left(\frac{220}{27} +  \frac{352 y_1 }{27}\right)\|z\|_{\infty}.
	\end{aligned}
	\end{equation}
	Combining \eqref{eq:TwoTermCondBound}, \eqref{eq:FirstTermCondBPBound}, \eqref{eq:SecondTermCondBPBound} as well as the assumption $\delta \leq 1$ yields $\|\eta(y_1,\alpha,\beta) - \eta(y_1+z_y,\alpha+z_\alpha,\beta+z_\beta)\|_\infty \leq  14(y_1+1)$, as desired.
\end{proof}

\begin{lemma}\label{lem:etatilde-pert-bound}
	For some natural number $n$, let $(\alpha,\beta)\in\albetSetULTV{n}$ and let $m$, $N$, $\theta=\theta(\cdot,\cdot)$, and $\psi=\psi(\cdot,\cdot,\cdot )$ be as in subsection \ref{sec:TV reg geometry}. Suppose $y_1>0$ is such that \[\varsigma(y_1,\alpha,\beta):=y_1-\lambda\, \theta(\alpha,\beta,m)^2\left[2(m-1)^2 (\alpha\vee \beta)\right]^{-1}>0\] and $z=(z_y,z_\alpha,z_\beta)\in\real^2$ with $\|z\|_\infty\leq \frac{1}{8}\wedge\frac{1}{10}\varsigma(y_1,\alpha,\beta)$. Then $\varsigma((y_1,\alpha,\beta) + z)>0$ and
	\begin{equation}\label{eq:etatilde-pert-bound}
	\|\psi(y_1, \alpha,\beta) - \psi(y_1+z_y, \alpha+z_\alpha,\beta+z_\beta)\|_\infty\leq 112(y_1+1)\|z\|_\infty.
	\end{equation}
\end{lemma}
\begin{proof}[Proof of Lemma \ref{lem:etatilde-pert-bound}]
	As in the proof of Lemma \ref{lem:eta-pert-bound}, we first obtain $|\theta(\alpha+z_\alpha,\beta+z_\beta,m)- \theta(\alpha,\beta,m)|\leq \frac{5}{2}\|z\|_\infty$ and $\theta(\alpha + z_\alpha, \beta + z_\beta,m) \in [m-1,m]$. Moreover, using Lemma \ref{lem:alpha-vee-beta-estimate} we get $(\alpha+z_\alpha)\vee(\beta+z_\beta) \in [3/8,5/8]$. Next, we write 
	\begin{equation*}
	\psi(y_1 + z_y,\alpha+z_\alpha,\beta+z_\beta)=\frac{\lambda\cdot (\alpha+z_\alpha+ \beta+z_\beta )}{2(m-1)((\alpha+z_\alpha)\vee(\beta+z_\beta))}\mathbf{1} + \frac{\varsigma((y_1,\alpha,\beta) + z) }{(\alpha+z_\alpha)\vee (\beta+z_\beta)} e_{N}
	\end{equation*}
	where $\ones_N\in\real^N$ is the vector of all ones, so that 
	\begin{equation}\label{eq:TwoTermCondULBound}
	\begin{aligned}
	&\|\psi(y_1,\alpha,\beta) - \psi(y_1+z_y,\alpha+z_\alpha,\beta+z_\beta)\|_\infty\\
	\leq&  \left\|  \frac{\lambda\cdot (\alpha+z_\alpha+ \beta+z_\beta )}{2(m-1)((\alpha+z_\alpha)\vee(\beta+z_\beta))}\mathbf{1}-\frac{\lambda\cdot (\alpha+\beta)}{2(m-1)(\alpha\vee\beta )} \mathbf{1} \right\|_\infty + \\&\hspace{5cm} +\left\|\frac{\varsigma((y_1,\alpha,\beta) + z)}{(\alpha+z_\alpha)\vee (\beta+z_\beta)} e_{N} - \frac{\varsigma(y_1,\alpha,\beta) }{\alpha\vee \beta} e_{N}\right\|_\infty\\
	= & \left|  \frac{\lambda\cdot (\alpha+z_\alpha+ \beta+z_\beta )}{2(m-1)((\alpha+z_\alpha)\vee(\beta+z_\beta))} - \frac{\lambda\cdot (\alpha+\beta)}{2(m-1)(\alpha\vee\beta )}  \right|  \\
	& \qquad \qquad \qquad \qquad \qquad \qquad \qquad  + \left|\frac{\varsigma((y_1,\alpha,\beta) + z)}{(\alpha+z_\alpha)\vee (\beta+z_\beta)}- \frac{\varsigma(y_1,\alpha,\beta) }{\alpha\vee \beta} \right|.
	\end{aligned}
	\end{equation}
	We will bound both of these terms separately. The first term can be bounded as follows. Let $f(u,v) = u/v$ and note that $\alpha + \beta + z_\alpha + z_\beta \in [1/2 + 1/4 - 2/8,1/2+1/2 + 2/8] = [1/2,5/4]$. We thus obtain
	\begin{equation}\label{eq:FirstTermCondULBound}
	\begin{aligned}
	& \left|  \frac{\lambda\cdot (\alpha+z_\alpha+ \beta+z_\beta )}{2(m-1)((\alpha+z_\alpha)\vee(\beta+z_\beta))} - \frac{\lambda\cdot (\alpha+\beta)}{2(m-1)(\alpha\vee\beta )}  \right| \\
	\leq&\max_{(u,v)\in [\frac{1}{2},\frac{5}{4}]\times [\frac{3}{8},\frac{5}{8}] } \frac{\lambda}{2(m-1)} \|\nabla f(u,v)\|_1 \cdot \left(|z_\alpha+z_\beta|\vee \left|(\alpha+z_\alpha)\vee (\beta+z_\beta) -\alpha\vee \beta  \right|\right) \\
	\leq & \frac{\lambda}{2(m-1)} \left(\frac{1}{3/8} + \frac{5/4}{(3/8)^2}\right)\cdot (2\|z\|_\infty \vee \|z\|_{\infty} )\leq \frac{12\lambda  \|z\|_{\infty}}{m-1} \leq \frac{4\|z\|_{\infty}}{3}
	\end{aligned}
	\end{equation}
	since $\lambda \leq 1/3$ and $m \geq 4$.
	For the second term, using the mean value theorem with the function $g(u,v)=u^2/v$, we first obtain the following bound
	\begin{equation*}
	\begin{aligned}
	&\left| \frac{\theta(\alpha+z_\alpha,\beta+z_\beta,m)^2}{(\alpha+z_\alpha)\vee(\beta+z_\beta)}- \frac{\theta(\alpha,\beta,m)^2}{\alpha\vee \beta}\right|\\
	\leq & \max_{\substack{u\in [m-1,m]\\ v \in [\frac{3}{8},\frac{5}{8}]}} \|\nabla g (u,v)\|_1 \cdot \left[\left|\theta(\alpha+z_\alpha,\beta+z_\beta,m) -\theta(\alpha,\beta,m)\right| \vee \left| (\alpha+z_\alpha)\vee(\beta+z_\beta)-\alpha\vee \beta\right|\right]\\
	\leq &  \left(\frac{2m}{3/8} +\frac{m^2}{(3/8)^2} \right)\cdot \left(\frac{5}{2}\|z\|_\infty \vee \|z\|_{\infty}\right) \leq 18 (1+m)^2 \|z\|_\infty,
	\end{aligned}
	\end{equation*}
	and therefore, since $\lambda \leq 1/3$ and $m \geq 4$, we find 
	\begin{equation*}
	\frac{\lambda}{2(m-1)^2}\left|\frac{\theta(\alpha+z_\alpha,\beta+z_\beta,m)^2}{(\alpha+z_\alpha)\vee(\beta+z_\beta)}- \frac{\theta(\alpha,\beta,m)^2}{\alpha\vee \beta}\right|\\
	\leq 9\lambda  \|z\|_\infty\left(1+\frac{2}{m-1}\right)^2 \leq \frac{25 \|z\|_{\infty}}{3}.
	\end{equation*}
	Therefore,
	\begin{align*}
	|\varsigma((y_1,\alpha,\beta) + z)-\varsigma(y_1, \alpha,\beta)|&\leq  |z_y| +  \frac{\lambda }{2(m-1)^2}\left| \frac{\theta(\alpha+z_\alpha,\beta+z_\beta,m)^2}{(\alpha+z_\alpha)\vee(\beta+z_\beta)}- \frac{\theta(\alpha,\beta,m)^2}{\alpha\vee \beta}\right|\\&
	\leq \left(1+\frac{25}{3}\right)\|z\|_{\infty}\leq  28 \|z\|_\infty /3.
	\end{align*}
	In particular, since $y_1 \geq \varsigma(y_1,\alpha,\beta) \geq 10\|z\|_{\infty}$, we must have $\varsigma((y_1,\alpha,\beta) + z) \in[2\|z\|_\infty/3, y_1+ 28\|z\|_\infty/3]\subset [0,y_1+ 7/6]$ and hence $\varsigma((y_1,\alpha,\beta) + z) > 0$.
	Using the mean value theorem together with the bounds above and Lemma \ref{lem:alpha-vee-beta-estimate} we obtain
	\begin{equation}\label{eq:SecondTermCondULBound}
	\begin{aligned}
	&\left|\frac{\varsigma((y_1,\alpha,\beta) + z)}{(\alpha+z_\alpha)\vee (\beta+z_\beta)}- \frac{\varsigma(y_1,\alpha,\beta) }{\alpha\vee \beta} \right|\\
	& \leq \!\!\!\!\!\! \max_{(u,v)\in [0,y_1+\frac{7}{6}]\times [\frac{3}{8},\frac{5}{8}] } \!\|\nabla f(u,v)\|_1 \cdot \left[ \left|\varsigma((y_1,\alpha,\beta) + z)-\varsigma(y_1, \alpha,\beta)| \vee|(\alpha+z_\alpha)\vee (\beta+z_\beta) -\alpha\vee \beta  \right|\right]\\
	&\leq  \left(\frac{1}{3/8} +\frac{(y_1 +\frac{7}{6})}{(3/8)^2}\right) \cdot \left(10\|z\|_\infty \vee \|z\|_{\infty}\right)\leq 110 (y_1+1 )\|z\|_\infty.
	\end{aligned}
	\end{equation}
	Combining \eqref{eq:TwoTermCondULBound}, \eqref{eq:FirstTermCondULBound} and \eqref{eq:SecondTermCondULBound} gives the desired inequality \eqref{eq:etatilde-pert-bound}.
\end{proof}

With Lemma \ref{lem:alpha-vee-beta-estimate}, Lemma \ref{lem:eta-pert-bound} and Lemma \ref{lem:etatilde-pert-bound} at hand, we are ready to prove Lemma \ref{lem:cond-Xi-various}.

\begin{proof}[Proof of Lemma \ref{lem:cond-Xi-various}]
We first establish that
	\begin{equation}\label{eq:StrongBDEpsConditionMap}
	\begin{split}
	&\cond{\xilp,\omlps[k]} \leq 15,\quad \cond{\xiul,\omulls[k]} \leq 28, \quad \cond{\xicl,\omclls[k]}\leq 2,\\
	&\cond{\xibp,\ombpls[k]} \leq 24, \quad \cond{\xibptv,\ombptvs[k]} \leq 35,\quad \cond{\xiultv,\omultvs[k]} \leq 179 
	\end{split}
	\end{equation}
	 for integers $k \geq 1$ and $N>m\geq 4$, and then argue via the inclusions between the ``strong'' and ``weak'' input sets to prove \eqref{eq:MainInputConditionMap}. We work through each of the problems in turn, calculating their condition numbers. 
	
	\textbf{Case 1 (Linear programming)}: 
	The active coordinates of $\omlps[k]$are $\{(1,1),(1,2)\}$ for the matrix part of the input and $\{1\}$ for the vector part. Thus, 
\[
\iota^z \in \actv{\omlps[k]} \text{ only if } \iota^z = (\vecYMainLP{k}+z_ye_1,\matLP(\alpha+z_{\alpha},\beta+z_{\beta},m,N)),
\]
 for some suitable $z$ and $\alpha,\beta$ such that there is an $\iota \in \omlps[k]$ with $\iota= (\vecYMainLP{k},\matLP(\alpha,\beta,m,N))$. For the sake of brevity, we will let $y_1 = \vecYMainLP{k}_1.$
	Let us assume first that $\iota$ is such that $\alpha \neq \beta$. Then, for $\epsilon\in\big(0, y_1 \wedge \frac{1}{2} |\alpha-\beta|\big) $ and $z=(z_y,z_\alpha,z_\beta)\in \real^3$ such that $0<\|z\|_\infty\leq \epsilon$ we have both $\alpha+z_\alpha\neq\beta+z_{\beta}$ and $\sgn(\alpha + z_{\alpha} - (\beta + z_{\beta})) = \sgn(\alpha - \beta)$, and so by Lemma \ref{lemma:ProblemBasicExampleLP}
	\begin{equation*}
	\xilp( \iota^z)= \frac{y_1+z_y}{(\alpha+z_\alpha)\vee (\beta+z_\beta)}\, {v} \quad \text{and} \quad \xilp(\iota)=\frac{y_1}{\alpha\vee \beta}\, {v},
	\end{equation*}
	where ${v}=e_1$ if $\alpha>\beta$, and ${v}=e_2$ if $\alpha<\beta$. 
	
	Next, consider an input $\iota = (y,\matLP(\alpha,\alpha,m,N))\in \omlps \cup \omlpw$, and let $z=(z_y,z_\alpha,z_\beta)\in \real^3$ be such that $0<\|z\|_\infty\leq y_1\wedge \frac{1}{2}\alpha$. It then follows by Lemma \ref{lemma:ProblemBasicExampleLP} that
	\begin{equation*}
	\begin{aligned}
	\xilp(\iota^z)&\subset  \left\{  \frac{y_1+z_y}{(\alpha+ z_\alpha )\vee (\alpha + z_\beta)} \left(t e_1+ (1-t)e_2 \right)\, \vert \,t\in[0,1]\right\}\text{ and} \\
	\xilp(\iota)&= \left\{ \frac{y_1}{\alpha} \left(t e_1+ (1-t)e_2 \right)\, \vert \,t\in[0,1]\right\}.
	\end{aligned}
	\end{equation*}
	Therefore in all possible cases for $\iota $ and $ \iota^z$ with $\epsilon$ sufficiently small we must have 
	\begin{equation*}
	\dist_\infty \left (\xilp(\iota),\xilp(\iota^z)\right) \leq \max_{\alpha,\beta \in \albetSet} \left|\frac{y_1+z_y}{(\alpha+z_\alpha)\vee (\beta+z_\beta)} - \frac{y_1}{\alpha\vee \beta}\right|.
	\end{equation*}
	Moreover, in all the cases above we have $0\leq y_1+z_y \leq 2y_1 \leq 4 \cdot 10^{-k} \leq 1$, and $|(\alpha+z_\alpha)\vee (\beta+ z_\beta) - \alpha\vee \beta|\leq \|z\|_\infty$ by Lemma \ref{lem:alpha-vee-beta-estimate}. Thus, for $\|z\|_{\infty}$ sufficiently small (in particular insisting that $\|z\|_{\infty} < 1/8$), the mean value theorem applied to the function $f_{\text{LP}}(u,v):= u/v$ gives
	\begin{equation*}
	\max_{\alpha,\beta \in \albetSet} \left|\frac{y_1+z_y}{(\alpha+z_\alpha)\vee (\beta+z_\beta)} - \frac{y_1}{\alpha\vee \beta}\right| \leq \max_{\substack{u \in [0,1] \\ v \in [\frac{3}{8},\frac{5}{8}]}} \|\nabla f_{\text{LP}}(u,v)\|_1 \cdot \|z\|_\infty\leq  \left(\frac{8}{3}+\frac{64}{9}\right)\cdot \|z\|_\infty<10\|z\|_\infty.
	\end{equation*}
	We hence deduce that  
	\begin{equation*}
	\cond{\xilp,\omlps[k]} =\sup_{N > m \geq 4} \,\sup_{\iota \in  \omlps[k]}\, \limsup_{\epsilon\to 0^+} \!\!\sup_{\substack{  \iota^z \in  \actv{\omlps[k]}  \\0<\|z\|_\infty \leq \epsilon} }\!\left\{   \frac{\dist_\infty \left(\xilp( \iota),\xilp(\iota^z)\right) }{\|z\|_{\infty}} \right\}\leq 10.
	\end{equation*}
	
	\textbf{Case 2 (Unconstrained lasso with $\ell^1$ regularisation):}	
	The active coordinates of $\omulls[k]$are $\{(1,1),(1,2)\}$ for the matrix part of the input and $\{1\}$ for the vector part. Thus,
\[
\iota^z \in  \actv{\omulls[k]} \text{ only if } \iota^z = (\vecYMainUL{k}+z_ye_1,\matl(\alpha+z_{\alpha},\beta+z_{\beta},m,N)),
\] 
 for some suitable $z$ and $\alpha,\beta$ such that there is an $\iota \in \ombpls[k]$ with $\iota= (\vecYMainUL{k},\matl(\alpha,\beta,m,N))$. For the sake of brevity, we will let $y_1 = \vecYMainUL{k}_1.$
	For such an $\iota$ note that $\alpha\vee \beta =\frac{1}{2}> \frac{\lambda}{2y_1}$. Let $\epsilon\in(0,y_1)$ be small enough so  that $\alpha\vee \beta -\epsilon > \frac{\lambda}{2(y_1+\epsilon)}$ and consider $z=(z_y,z_\alpha,z_\beta)\in\real^3$ such that $0<\|z\|_\infty\leq \epsilon$. We then have $0\leq y_1+ z_y\leq 2y_1\leq 2(1+\lambda)\leq \frac{8}{3}  $ and again $|(\alpha+z_\alpha)\vee (\beta+ z_\beta) - \alpha\vee \beta|\leq \|z\|_\infty$ by Lemma \ref{lem:alpha-vee-beta-estimate}. Thus an argument  analogous to the one presented when analysing $\cond{\xilp}$, but employing Lemma \ref{lemma:ProblemBasicExampleULASSO} instead of Lemma \ref{lemma:ProblemBasicExampleLP}, gives
	\begin{equation*}
	\begin{aligned}
	&\dist_\infty \left(\xiul(\iota),\xiul( \iota^z ) \right) \\
	\leq & \max_{(\alpha,\beta) \in \albetSet}   \left|\frac{2(\alpha\vee \beta)y_1-\lambda}{2 (\alpha\vee \beta)^2}- \frac{2\left((\alpha+z_\alpha)\vee (\beta+z_\beta)\right)(y_1+z_y)-\lambda}{2 \left((\alpha+z_\alpha)\vee (\beta+z_\beta)\right)^2} \right|\\
	\leq & \max_{ (u,v)\in [0,\frac{8}{3}]\times [\frac{3}{8},\frac{5}{8}]} \|\nabla f_{\text{UL},\ell^1}(u,v)\|_1 \cdot \|z\|_\infty = \max_{ (u,v)\in [0,\frac{8}{3}]\times [\frac{3}{8},\frac{5}{8}]} \left(\left|\frac{1}{v}\right| + \left|\frac{\lambda}{v^3} - \frac{u}{v^2}\right| \right) \cdot \|z\|_\infty,
	\end{aligned}
	\end{equation*}
	where $f_{\text{UL},\ell^1}(u,v)=\frac{2uv-\lambda}{2v^2}$. Noting that $|\lambda/v^3 - u/v^2| \leq ((\lambda/v^3) \vee (u/v^2))$ and $\lambda \leq 1/3$ we obtain
	$\dist_\infty \left(\xiul(\iota),\xiul( \iota^z ) \right) \leq 8/3 + (8^3/3^4 \vee 8^3/3^3) \leq 22$ from which $\cond{\xiul,\omulls[k]}\leq 22$ follows.

	\textbf{Case 3 (Constrained lasso with $\ell^1$ regularisation):}	
	The active coordinates of $\omclls[k]$ are $\{(1,1),(1,2)\}$ for the matrix part of the input and $\{1\}$ for the vector part. Thus, 
	\[
	\iota^z \in \actv{\omclls[k]} \text{  only if } \iota^z = (\vecYMainCL{k}+z_ye_1,\matl(\alpha+z_{\alpha},\beta+z_{\beta},m,N)),
	\]
	  for some suitable $z$ and $\alpha,\beta$ such that there is an $\iota \in \omclls[k]$ with $\iota= (\vecYMainCL{k},\matl(\alpha,\beta,m,N))$. For the sake of brevity, we will let $y_1 = \vecYMainCL{k}_1$.
	For such an $\iota$ let $r = \frac{(\alpha \vee \beta)y_1}{1+(\alpha \vee \beta)^2}=2y_1/5=4\cdot 10^{-k}$ and note that $r< 1/2 \leq \tau$ and so Lemma \ref{lemma:ProblemBasicExampleCLASSO} applies. Next, let $\epsilon\in(0,y_1\wedge 1/2)$ be small enough so  that 
$
\frac{(\alpha \vee \beta +\epsilon )(y_1+\epsilon)}{1+(\alpha \vee \beta +\epsilon)^2}< \tau,
$
and consider $z=(z_y,z_\alpha,z_\beta)\in\real^3$ such that $0<\|z\|_\infty\leq \epsilon$. We then have $0\leq y_1+ z_y\leq 2y_1\leq 2 $ and $|(\alpha+z_\alpha)\vee (\beta+ z_\beta) - \alpha\vee \beta|\leq \|z\|_\infty$ by Lemma \ref{lem:alpha-vee-beta-estimate}. Thus an argument  analogous to the one presented when analysing $\cond{\xilp}$, but employing Lemma \ref{lemma:ProblemBasicExampleCLASSO} instead of Lemma \ref{lemma:ProblemBasicExampleLP}, gives
	\begin{equation*}
	\begin{aligned}
	&\dist_\infty \left(\xicl(\iota),\xicl( \iota^z ) \right) \\
	\leq & \max_{(\alpha,\beta) \in \albetSet}   \left|\left(\tau- \frac{(\alpha \vee \beta)y_1}{1+(\alpha \vee \beta)^2}\right)- \left(\tau-\frac{\left((\alpha+z_\alpha)\vee (\beta+z_\beta)\right)(y_1+z_y)}{1+ \left((\alpha+z_\alpha)\vee (\beta+z_\beta)\right)^2} \right)\right|\\
	\leq & \max_{ (u,v)\in [0,1]\times [\frac{3}{8},\frac{5}{8}]} \|\nabla f_{\text{CL}}(u,v)\|_1 \cdot \|z\|_\infty = \max_{ (u,v)\in [0,1]\times [\frac{3}{8},\frac{5}{8}]} \left(\left|\frac{1}{1+v^2}\right| + \left|\frac{u(1-v^2)}{(1+v^2)^2}\right| \right) \cdot \|z\|_\infty
	\end{aligned}
	\end{equation*}
	where $f_{\text{CL}}(u,v)=\frac{uv}{1+v^2}$. Noting that $\left|1/(1+v^2)\right| + \left|{u(1-v^2)}/{(1+v^2)^2}\right|\leq 2/(1+v^2)^2$ for the specified range of $(u,v)$, we obtain
$
\dist_\infty \left(\xicl(\iota),\xicl( \iota^z ) \right) \leq 2/(1+(3/8)^2)^2 \leq 2,
$
 from which we deduce that $\cond{\xicl,\omclls[k]}\leq 2$. 
	
	\textbf{Case 4 (Basis pursuit with $\ell^1$ regularisation):}
	As before the active coordinates of $\ombpls[k]$ are $\{(1,1),(1,2)\}$ for the matrix part of the input and $\{1\}$ for the vector part. Thus,  
\[	
\iota^z \in \actv{\ombpls[k]} \text{ only if } \iota^z = (\vecYMainBP{k}+z_ye_1,\matl(\alpha+z_{\alpha},\beta+z_{\beta},m,N))
\] 
for some suitable $z$ and $\alpha,\beta$ such that there is an $\iota \in \ombpls[k]$ with $\iota= (\vecYMainBP{k},\matl(\alpha,\beta,m,N))$. For the sake of brevity, we will let $y_1 = \vecYMainBP{k}_1.$
For such an $\iota$ note that $\delta< \delta + 10^{-k}\leq  y_1 \leq 2$  so for $\epsilon\in(0,y_1-\delta)$ and $z=(z_y,z_\alpha,z_\beta)\in\real^3$ such that $0<\|z\|_\infty\leq \epsilon$, an argument analogous to the analysis of $\cond{\xilp}$ except employing Lemma \ref{lemma:ProblemBasicExampleBPDNL1} instead of Lemma \ref{lemma:ProblemBasicExampleLP} shows that
	\begin{align*}
	\dist_\infty \left(\xibpdn(\iota)),\xibpdn( \iota^z)) \right)  &\leq   \left|\frac{y_1-\delta}{\alpha\vee \beta}- \frac{y_1+z_y-\delta}{(\alpha+z_\alpha)\vee (\beta+z_\beta)} \right|\\&\leq  \max_{ (u,v)\in [0,2]\times [\frac{3}{8},\frac{5}{8}]} \|\nabla f_{\text{BP},\ell^1 }(u,v)\|_1 \cdot \|z\|_\infty
	\\&\leq   \left(\frac{1}{v} +  \left|\frac{u-\delta}{v^2}\right| \right)\cdot \|z\|_\infty \leq \left(\frac{8}{3} +  \frac{2}{(3/8)^2} \right) <17\|z\|_\infty,
	\end{align*}
	where $f_{\text{BP},\ell^1 }(u,v)=\frac{u-\delta}{v}$, from which we deduce that $\cond{\xibpdn,\ombpls[k]}\leq 17$. 
	
	\textbf{Case 5 (Basis pursuit with TV regularisation):}
	The active coordinates of $\ombptvs[k]$ are $\{(1,1),(1,N)\}$ for the matrix and $\{1\}$ for the vector.  Thus, 
\[
\iota^z \in \actv{\ombptvs[k]} \text{ only if } \iota^z = (\vecYMainBPTV{k}+z_ye_1,\matTV(\alpha+z_{\alpha},\beta+z_{\beta},m,N)),
\]
for some suitable $z$ and $\alpha,\beta$ such that there is an $\iota \in \ombptvs[k]$ with $\iota= (\vecYMainBPTV{k},\matTV(\alpha,\beta,m,N))$. For the sake of brevity, we will let $y_1 = \vecYMainBPTV{k}_1.$
	Suppose that $\alpha \neq \beta$. For $\epsilon\in\big( 0,\;\frac{1}{2} |\alpha-\beta|\wedge \frac{1}{8}\wedge \frac{1}{2}\big(y_1 -\frac{\delta\theta(\alpha,\beta,m)}{m-1}\big) \big) $ and $z=(z_y,z_\alpha,z_\beta)\in \real^3$ such that $0<\|z\|_\infty\leq \epsilon$ we have $\alpha+z_\alpha\neq\beta+z_{\beta}$ and, by Lemma \ref{lem:eta-pert-bound}, 
\[
y_1+z_y>\frac{\delta\theta(\alpha+z_\alpha,\beta+z_\beta,m)}{m-1}.
\]
 Suppose for now that $\alpha<\beta$. Then, by Lemma \ref{lemma:BPTVSolutions}, we have
	$\xibptv( \iota^z)= \eta(y+z_y,\alpha+z_\alpha, \beta+z_\beta)$
	and thus by Lemma \ref{lem:eta-pert-bound} we find 
	\begin{equation*}
	\begin{split}
	\dist_\infty\left(\xibptv(\iota), \xibptv( \iota^z )\right)&=\| \eta(y_1+z_y, \alpha+z_\alpha,\beta+z_\beta)-\eta(y_1, \alpha,\beta) \|_\infty\\
	& \leq  14(y_1+1)\|z\|_\infty<35\|z\|_\infty,
	\end{split}
	\end{equation*}
	where the final inequality follows from $y_1 \leq \|y\|_{\infty}\leq 3/2$, proven in Lemma \ref{lemma:mainThmSizeEstimates}.
	In the case $\alpha > \beta$, the same logic and an application of Lemma \ref{lemma:BPTVSolutions} yields 
\begin{equation*}
\begin{split}
\dist_\infty\left(\xibptv(\iota), \xibptv( \iota^z )\right)&=\| \flipOp \eta(y_1+z_y, \alpha+z_\alpha,\beta+z_\beta)-\flipOp \eta(y_1, \alpha,\beta) \|_\infty \\
& \leq  14(y_1+1)\|z\|_\infty<35\|z\|_\infty,
\end{split}
\end{equation*}
as the flipping operator $\flipOp$ is an isometry.
Next, assume $\alpha = \beta$ and let $z=(z_y,z_\alpha,z_\beta)\in \real^3$  be such that $0<\|z\|_\infty\leq \frac{1}{8}\wedge\frac{1}{2}\left( y_1 - \frac{\delta\theta(\alpha,\beta)}{m-1}\right)$. Then, by Lemma \ref{lemma:BPTVSolutions}, we have both 
	\begin{equation*}
	\xibptv( \iota^z)\cap  \left\{ \eta(y_1+z_y, \alpha+ z_\alpha,\beta+z_\beta ),\flipOp \eta(y_1+z_y,\beta+z_\beta, \alpha+z_\alpha )\right\} \neq \varnothing
	\end{equation*}
	and 
$
\xibptv(\iota) =  \left\{ t \eta(y_1, \alpha,\beta) + (1-t)  \flipOp \eta(y_1,\beta, \alpha) \, \vert \, t\in[0,1]\right\}.
$
	Therefore, 
	\begin{equation*}
	\begin{split}
	 \dist_\infty\left(\xibptv( \iota)), \xibptv(\iota^z) \right)
	&\leq \|\eta(y_1+z_y, \alpha+ z_\alpha,\beta+z_\beta )- \eta(y_1, \alpha,\beta )\|_{\infty}\\
	& \qquad \vee \|\flipOp \eta(y_1+z_y, \alpha+ z_\alpha,\beta+z_\beta )- \flipOp\eta(y_1, \alpha,\beta )\|_{\infty}
	\end{split}
	\end{equation*}
	and thus using Lemma \ref{lem:eta-pert-bound} we obtain $\dist_\infty\left(\xibptv(\iota^z),\xibptv(\iota)\right)\leq 14(y_1+1)\|z\|_\infty<35\|z\|_\infty$. Since we have now shown that this bound holds in the cases $\alpha > \beta$, $\alpha < \beta$ and $\alpha = \beta$ we can conclude that $\cond{\xibptv,\ombptvs}\leq 35$.
	
	\textbf{Case 6 (Unconstrained lasso with TV regularisation):}
	The active coordinates of $\omultvs[k]$ are $\{(1,1),(1,N)\}$ for the matrix and $\{1\}$ for the vector.  Thus $\iota^z$ is in the active set 
\[
\actv{\omultvs[k]} \text{ only if } \iota^z = (\vecYMainULTV{k}+z_ye_1,\matTV(\alpha+z_{\alpha},\beta+z_{\beta},m,N)),\] 
for some suitable $z$ and $\alpha,\beta$ such that there is an $\iota \in \omultvs[k]$ with $\iota= (\vecYMainULTV{k},\matTV(\alpha,\beta,m,N))$. For the sake of brevity, we will let $y_1 = \vecYMainULTV{k}_1.$
For such an $\iota$ we consider $z=(z_y,z_\alpha,z_\beta)\in\real^3$ such that $0<\|z\|_\infty\leq \frac{1}{8}\wedge\frac{1}{10}\varsigma(y_1,\alpha,\beta)$, where $\varsigma(y_1,\alpha,\beta)=y_1-\lambda\, \theta(\alpha,\beta)^2\left[2(m-1)^2 (\alpha\vee \beta)\right]^{-1}$. By Lemma \ref{lem:etatilde-pert-bound} we conclude that $\varsigma(y_1+z_y,\alpha+z_\alpha,\beta+z_\beta)>0$, and so Lemma \ref{lemma:ULTVSolutions} applies to the input $(y+z_ye_1,\matTV(\alpha+z_\alpha,\beta+z_\beta,m,N))$. An argument analogous to the one presented for Basis Pursuit with TV regularisation employing Lemma \ref{lem:etatilde-pert-bound} instead of Lemma \ref{lem:eta-pert-bound} then yields 
	\begin{equation*}
	\dist_\infty\left(\xiultv( \iota^z ), \xiultv(\iota) \right)\leq 112(y_1+1)\leq 112\left(\frac{107}{180}+1\right) \leq 179 ,
	\end{equation*}
	where the penultimate inequality follows from Lemma \ref{lemma:mainThmSizeEstimates}. Hence $\cond{\xiultv,\omultvs[k]}\leq 179$, as desired.
	
	This establishes each of the claims in \eqref{eq:StrongBDEpsConditionMap}. It remains to prove \eqref{eq:MainInputConditionMap}. We do this for linear programming only as the argument for the other cases is analogous. If $K=1$ there is nothing to prove, so assume w.l.o.g. that $K\geq 2$.  Note that the active coordinates for $\ombpls[K-1]$ are the same as the active coordinates for $\ombpls[K]$. Thus \[\cond{\xilp,\ombpls[K-1] \cup \ombpls[K]} = \cond{\xilp,\ombpls[K-1]} \vee \cond{\xilp,\ombpls[K]}\] and \eqref{eq:StrongBDEpsConditionMap} implies that $\cond{\ombpls[K-1]}$ and $\cond{\ombpls[K]}$ are both bounded above by $10$ since $K \geq 2$. Next, note that $\ombplw[K] \subset \ombpls[K-1]$. This allows us to conclude that 
	\[
	\cond{\xilp,\ombplb[K]} = \cond{\xilp,\ombpls[K-1] \cup \ombpls[K]} \leq 10, 
	\]
	and  hence
	$
	\cond{\xilp,\ombpl} = \sup_{N> m\geq 4} \cond{\ombpls[K] \cup \ombplw[K]} \leq 10.
	$
\end{proof}

\begin{lemma}\label{lem:cond-FP-various}
	For any natural numbers $m$, $N$, and $k$ with $k \geq 1$ and  $4 \leq m < N$ the FP condition number of every input to any of the problems LP, $\ell^1$ BP, $\ell^1$ UL, CL, $\mathrm{TV}$ BP and $\mathrm{TV}$ UL with respective input sets $\omlp$, $\ombpl$,  $\omull$, $\omcll$, $\ombptv$, and  $\omultv$ is bounded by 4.
\end{lemma}
\begin{proof}[Proof of Lemma \ref{lem:cond-FP-various}]
Suppose for now that the input is in one of the ``strong'' input sets, and consider first an input $(y,\matLP)=(y,\matLP(\alpha,\beta,m,N))\in \omlps[k]$. Recall that $C_{\mathrm{FP}}(y,\matLP)=\frac{\|y\|_\infty\vee \|\matLP\|_{\max}}{\rho(y,\matLP)}$, where
	\begin{equation*}
	\begin{aligned}
	\rho(y,\matLP)&=\sup\{\epsilon\, \vert\;  \exists z=(z_y,z_\alpha,z_\beta)\in\real^3,\|z\|_\infty\leq \epsilon \\
	&\qquad\qquad\qquad \text{ s.t. }(y+z_ye_1,\matLP(\alpha+z_\alpha,\beta+z_\beta,m,N))\text{ is feasible for LP} \}.
	\end{aligned}
	\end{equation*}
	Concretely, by recalling that $c$ is the vector of ones of length $N$ and inspecting the explicit form of $\matLP(\alpha+z_\alpha,\beta+z_\beta,m,N)$, we see that it is necessary and sufficient for at least one of $\alpha+z_\alpha$ and $\beta+z_\beta$ to be positive for the LP problem to be feasible. As $\alpha\vee\beta=\frac{1}{2}$, we deduce that $\rho(y,\matLP)=\frac{1}{2}$, and therefore $C_{\mathrm{FP}}(y,\matLP)= \frac{(2\cdot 10^{-k} )\vee 1}{1/2} = 2$.
	
	Next, note that, for $z=(z_y,z_\alpha,z_\beta)\in\real^3$, the matrices $\matl(\alpha+z_\alpha,\beta+z_\beta,m,N)$ and $\matTV(\alpha+z_\beta,\beta+z_\beta,m,N)$ are onto as long as at least one of $\alpha+z_\alpha$ and $\beta+z_\beta$ is nonzero, and therefore $\rho(y,A)\geq \frac{1}{2}$ for the problems $\ell^1$ BP,  CL, and $\mathrm{TV}$ BP with inputs $(y,A)$ in $\ombpls[k]$, $\omclls[k]$, and $\ombptvs[k]$, respectively. By Lemma \ref{lemma:mainThmSizeEstimates}  $\|y\|_\infty \leq 2$ and $\|A\|_{\max}\leq 1$ , we deduce that $C_{\mathrm{FP}}(y,A)\leq \frac{2}{1/2}=4$ in the case of $\ell^1$ BP, CL,  and $\mathrm{TV}$ BP.
	
	Finally, inputs $(y,\matl(\alpha,\beta,m,N))$ and $(y,\matTV(\alpha,\beta,m,N))$ are feasible for the $\ell^1$ UL and $\mathrm{TV}$ UL problems for any values of $y$,  $\alpha$, and $\beta$, and therefore $C_{\mathrm{FP}}(y,A)=0$ in the case of $\ell^1$ UL and $\mathrm{TV}$ UL for any input set.
	The proof is entirely analogous in the case when the input is in one of the ``weak'' input sets, and thus the proof of the lemma is complete.
\end{proof}

\begin{lemma}\label{lem:cond-mat-various}
	Let $(\alpha,\beta)\in\albetSet$. Then $\cond{\matLP \matLP^*},\cond{\matl \matl^*}\leq 16/5$. In addition, for $k\geq 1$ and $(\alpha,\beta) \in \albetSetBPTV{k}$ or $(\alpha,\beta) \in \albetSetULTV{k}$, we have $\cond{\matTV\matTV^*}\leq\frac{16}{5}$.
\end{lemma}
\begin{proof}[Proof of Lemma \ref{lem:cond-mat-various}]
	Note that $M:=\matLP \matLP^*=\matl \matl^*= \matTV\matTV^*= \big(\alpha^2 + \beta^2\big) \oplus I_{m-1}$. Therefore $\cond{M}=\|M\|_{2}\|M^{-1}\|_2= \left((\alpha^2+\beta^2)\vee 1\right)\cdot \left((\alpha^2+\beta^2)^{-1}\vee 1\right)\leq 1\cdot \frac{16}{5}=\frac{16}{5}$.
\end{proof}
\subsection{Inequalities to control the breakdown epsilons for the TV problems}
Below follows a collection of inequalities that are needed to estimate the breakdown epsilons for the TV problems.
\begin{lemma}\label{lemma:BPTVEtaBounds}
	Let $p\in [1,\infty]$, $n\in\mathbb{N}$, and set 
	\begin{equation*}
	y_1 = \delta m [\theta(1/2,1/2,m)]^{-1} + [7 - 3 (\theta(1/2,1/2,m))^{-1}]10^{-n}/4.
	\end{equation*}
	 Then
	\begin{equation}\label{eq:BPTVIneqBDEpsLower}
	\min \limits_{\substack{(\alpha,\beta) \in \albetSetBPTV{n} \\ (\alpha',\beta') \in \albetSetBPTV{n}}} \|\eta(y_1,\alpha,\beta) - \flipOp \eta(y_1,\alpha',\beta')\|_p > 2 \cdot 10^{-n}
	\end{equation} and
	\begin{equation}\label{eq:BPTVIneqBDEpsUpper}
	\max \limits_{(\alpha,\beta) \in \albetSetBPTV{n}} |\eta(y_1,\alpha,\beta)_1 - \eta(y_1,\alpha,\beta)_N| \leq 5 \cdot 10^{-n+1}2^{-1/p}/6.
	\end{equation}
\end{lemma}
\begin{proof}
We first prove the following inequalities for general $y_1\geq 0$ and then specialise to the value of $y_1$ in the statement of the lemma:
\begin{equation}\label{eq:BPTVDistS1S2Estimate}
		\min \limits_{\substack{(\alpha,\beta) \in \albetSetBPTV{n} \\ (\alpha',\beta') \in \albetSetBPTV{n}}} \|\eta(y_1,\alpha,\beta) - \flipOp\eta (y_1,\alpha',\beta')\|_p \geq 2^{1/p}\left[2y_1 - \frac{2\delta m}{\theta(1/2,1/2,m)}\right] 
		\end{equation}
		(with the convention $2^1/p=1$ when $p=\infty$) and, recalling $r_n$ from \eqref{eq:rn_and_sn},
		\begin{equation}\label{eq:BPTVComputeBoundEstimate}
		\max \limits_{(\alpha,\beta) \in \albetSetBPTV{n}} |\eta(y_1,\alpha,\beta)_1 - \eta(y_1,\alpha,\beta)_N| \leq  2y_1 - \frac{2\delta \left[4r_n^2 + (m-1)\right]}{\theta(1/2,1/2,m)}.
		\end{equation}
	To this end, we have for $(\alpha,\beta),(\alpha',\beta')\in\albetSetBPTV{n}$ and $p<\infty$
	\begin{equation*}
	\begin{aligned}
	&\|\eta(y_1,\alpha,\beta) - \flipOp \eta(y_1,\alpha',\beta')\|_p \geq \Big[\left|\eta(y_1,\alpha,\beta)_1- \left(\flipOp \eta(y_1,\alpha',\beta')\right)_1\right|^p\;+\\
	&\hspace{7cm} +\;\left|\eta(y_1,\alpha,\beta)_N- \left(\flipOp \eta(y_1,\alpha',\beta')\right)_N\right|^p \Big]^{1/p}\\
	= & \Big[\left|\eta(y_1,\alpha,\beta)_1- \eta(y_1,\alpha',\beta')_N\right|^p + \left|\eta(y_1,\alpha,\beta)_N- \eta(y_1,\alpha',\beta')_1\right|^p \Big]^{1/p}\\
	\geq & \Big[2 \cdot \min \limits_{\substack{(u,v) \in \albetSetBPTV{n} \\ (u',v') \in \albetSetBPTV{n}}} |\eta(y_1,u,v)_N - \eta(y_1,u',v')_1|^p \Big]^{1/p},
	\end{aligned}
	\end{equation*}
	and therefore
	\begin{equation*}
	\min \limits_{\substack{(\alpha,\beta)  \in \albetSetBPTV{n} \\  (\alpha',\beta')  \in \albetSetBPTV{n}}}\|\eta(y_1,\alpha,\beta) - \flipOp \eta(y_1,\alpha',\beta')\|_p  \geq	 \min \limits_{\substack{(\alpha,\beta) \in \albetSetBPTV{n} \\ (\alpha',\beta') \in \albetSetBPTV{n}}} 2^{1/p} |\eta(y_1,\alpha,\beta)_N - \eta(y_1,\alpha',\beta')_1|,
	\end{equation*}
	where this inequality also holds for $p=\infty$ by an analogous argument.
	Next, by the definition of $\eta$ we have, for $(\alpha,\beta), (\alpha',\beta')\in\albetSetBPTV{n}$,
	\[
	\eta(y_1,\alpha,\beta)_N - \eta(y_1,\alpha',\beta')_1 \geq 2y_1+\frac{\delta  (\alpha + \beta)}{\theta(\alpha,\beta)}  -\frac{2\delta\theta(\alpha,\beta)}{m-1}-\frac{\delta  (\alpha' + \beta')}{\theta(\alpha',\beta')} = 2y_1+\delta g(\alpha+ \beta)+\delta h(\alpha'+\beta')
	\] 
	where 
\[
g(u) = \frac{u}{\sqrt{(m-1)u^2+(m-1)^2}}  -\frac{2\sqrt{(m-1)u^2+(m-1)^2}}{m-1}
\]
and 
$
h(u) = \frac{-1}{\sqrt{(m-1)+(m-1)^2/u^2}}. 
$
A simple calculation yields, for $u \in [1/2,1]$, \[g'(u) = \frac{-(m-1)\left[(2u-1)(m-1)+2u^3\right]}{[(m-1)^2+u^2(m-1)]^{3/2}} \leq 0,\]
	so that $g(u) \geq g(1)$. It is also clear that $h(u)$ is decreasing on $[1/2,\infty)$ and that if $(\alpha,\beta) , (\alpha',\beta') \in \albetSetBPTV{n}$ then  $(\alpha + \beta) , (\alpha' + \beta') \in [3/4,1]$. Thus, 
	\begin{equation*}
	\min \limits_{\substack{\alpha,\beta \in \albetSetBPTV{n}\\ \alpha',\beta' \in \albetSetBPTV{n}}}2^{1/p} |\eta(y_1,\alpha,\beta)_N - \eta(y_1,\alpha',\beta')_1|\geq 2^{1/p} \left(2y_1+\delta g(1)+\delta h(1)\right).
	\end{equation*}
	Noting that $g(1) + h(1) = -2\sqrt{m-1 + (m-1)^2}/(m-1) = -2m/\theta(1/2,1/2,m)$ gives \eqref{eq:BPTVDistS1S2Estimate}.

	To obtain \eqref{eq:BPTVComputeBoundEstimate}, we note that $\eta(y_1,\alpha,\beta)_1 \leq \eta(y_1,\alpha,\beta)_N$ and $\theta(\alpha,\beta,m)$ is increasing in its arguments provided $(\alpha,\beta) \in [0,\infty)^2$. Thus 
	\begin{equation*}
	\max \limits_{(\alpha,\beta) \in \albetSetBPTV{n}} |\eta(y_1,\alpha,\beta)_1 - \eta(y_1,\alpha,\beta)_N| = \max \limits_{(\alpha,\beta) \in \albetSetBPTV{n}}\left[2y_1 - 2\delta \frac{\theta(\alpha,\beta,m)}{m-1}\right] = 2y_1 - 2\delta \frac{\theta(r_n,r_n,m)}{m-1}.
	\end{equation*}
	We conclude \eqref{eq:BPTVComputeBoundEstimate} by noting that 
\[
\theta(r_n,r_n,m)/(m-1) \geq \theta(r_n,r_n,m)^2 \left[\theta(1/2,1/2,m)(m-1)\right]^{-1} = [4r_n^2 + (m-1)] [\theta(1/2,1/2,m)]^{-1}.
\]

	With the specific choice $y_1 = \delta m [\theta(1/2,1/2,m)]^{-1} + [7 - 3 (\theta(1/2,1/2,m))^{-1}]10^{-n}/4$ and using $m\geq 4$, we calculate
	\begin{align*}
	2^{1/p}\left[2y_1 - \frac{2\delta m}{\theta(1/2,1/2,m)}\right]  
	&\geq   \left(\frac{7}{2} - \frac{3}{2\theta(1/2,1/2,m)}\right) 10^{-n} >  \left(\frac{7}{2} - \frac{3}{2(m-1)}\right) 10^{-n} \geq 2 \cdot 10^{-n},
	\end{align*}
	from which we conclude \eqref{eq:BPTVIneqBDEpsLower}, as well as 
	\begin{align*}2y_1 - \frac{2\delta \left[4r_n^2 + (m-1)\right]}{\theta(1/2,1/2,m)}
	&= \frac{2\delta (1-4r_n^2)}{\theta(1/2,1/2,m)} +  \left(\frac{7}{2} - \frac{3}{2\theta(1/2,1/2,m)}\right) 10^{-n} \\
	&\leq \frac{6\delta  \cdot 10^{-n}}{4\delta\theta(1/2,1/2,m)} +  \left(\frac{7}{2} - \frac{3}{2\theta(1/2,1/2,m)}\right) 10^{-n}\\
	&= \frac{7\cdot 10^{-n}}{2} < \frac{5 \cdot 10^{-n+1}}{12} \leq \frac{5 \cdot 10^{-n+1}2^{-1/p}}{6},
	\end{align*}
	where we used $1-4r_n^2 \leq  3 \cdot 10^{-n}/(4\delta)$, which is by the definition \eqref{eq:rn_and_sn}.
	This concludes the proof of \eqref{eq:BPTVIneqBDEpsUpper}.
\end{proof}

\begin{lemma}\label{lemma:ULTVPsiBounds}
Let $p\in [1,\infty]$, $n\in\mathbb{N}$, and
\begin{equation*}
y_1 = 7 \cdot 10^{-n}/4 + \lambda (1 + 1/(m-1)) - 3 \cdot 10^{-n}/(4(m-1)).
\end{equation*}
Then
	\begin{equation}\label{eq:ULTVIneqBDEpsLower}\min \limits_{\substack{(\alpha,\beta) \in \albetSetULTV{n} \\ (\alpha',\beta') \in \albetSetULTV{n}}} \|\psi(y_1,\alpha,\beta) - \flipOp\psi(y_1,\alpha',\beta')\|_p > 2 \cdot 10^{-n}
	\end{equation}
	and 
	\begin{equation} \label{eq:ULTVIneqBDEpsUpper}
	\max \limits_{\alpha,\beta \in \albetSetULTV{n}} |\psi(y_1,\alpha,\beta)_1 - \psi(y_1,\alpha,\beta)_N| \leq 5 \cdot  10^{-n+1}2^{-1/p}/6.
	\end{equation}
\end{lemma}

\begin{proof}
	As in the proof of Lemma \ref{lemma:BPTVEtaBounds},  we will first prove the following inequalities for general $y_1\geq 0$, and then specialise to $y_1$ in the statement of the lemma:
	\begin{equation}\label{eq:ULTVDistS1S2Estimate}
		\min \limits_{\substack{\alpha,\beta \in \albetSetULTV{n} \\ \alpha',\beta' \in \albetSetULTV{n}}} \|\psi(y_1,\alpha,\beta) - \flipOp\psi(y_1,\alpha',\beta')\|_p \geq 2^{1/p}\left[2y_1 - 2\lambda \left(1 + \frac{1}{m-1}\right)\right] 
		\end{equation}
		and,  recalling $s_n$ from \eqref{eq:rn_and_sn},
		\begin{equation}\label{eq:ULTVComputeBoundEstimate}
		\max \limits_{(\alpha,\beta) \in \albetSetULTV{n}} |\psi(y_1,\alpha,\beta)_1 - \psi(y_1,\alpha,\beta)_N| \leq  2y_1 - 2\lambda \frac{\theta(s_n,s_n,m)^2}{(m-1)^2}.
		\end{equation}
	To this end, note that
	\begin{equation}\label{eq:ULTVDistS1S2Int}
	\min \limits_{\substack{(\alpha,\beta) \in \albetSetULTV{n} \\ (\alpha',\beta') \in \albetSetULTV{n}}}\|\psi(y_1,\alpha,\beta) - \flipOp\psi (y_1,\alpha',\beta')\|_p  \geq \min \limits_{\substack{(\alpha,\beta) \in \albetSetULTV{n} \\( \alpha',\beta' )\in \albetSetULTV{n}}}2^{1/p} |\psi(y_1,\alpha,\beta)_N - \psi(y_1,\alpha',\beta')_1|.
	\end{equation}
	Next, a simple calculation and the definition of $\psi$ yields, for $(\alpha,\beta),(\alpha',\beta')\in\albetSetULTV{n}$,
	\begin{align}\psi(y_1,\alpha,\beta)_N - \psi(y_1,\alpha',\beta')_1 &= 2y_1 - \lambda\left(\frac{2(\alpha+\beta)^2 +2(m-1)-(\alpha + \beta)}{(m-1)} + \frac{(\alpha'+\beta')}{m-1}\right)
	\notag\\&=2y_1 - \lambda\left(\frac{2(\alpha+\beta-1/4)^2 +2(m-1)-1/8}{(m-1)} + \frac{(\alpha'+\beta')}{m-1}\right).\label{eq:ULTVDistS1S2Int2}
	\end{align}
	By noting that $\alpha + \beta , \alpha'+\beta' \in [1/2,1]$ we see that the right hand side of this expression is minimised when $\alpha + \beta = \alpha' + \beta' = 1$. Applying these values of $\alpha+\beta$ and $\alpha' + \beta'$ in \eqref{eq:ULTVDistS1S2Int2} and combining with \eqref{eq:ULTVDistS1S2Int} yields \eqref{eq:ULTVDistS1S2Estimate}.
	For \eqref{eq:ULTVComputeBoundEstimate}, we note that $\psi(y_1,\alpha,\beta)_1 \leq \psi(y_1,\alpha,\beta)_N$ and $\theta(\alpha,\beta,m)$ is increasing in its arguments provided $(\alpha,\beta) \in [0,\infty)^2$. Thus 
	\begin{equation*}
	\max \limits_{(\alpha,\beta) \in \albetSetULTV{n}} \!\!|\psi(y_1,\alpha,\beta)_1 - \psi(y_1,\alpha,\beta)_N| =\!\!\!\!\! \max\limits_{(\alpha,\beta) \in \albetSetULTV{n}}\!\!\left[2y_1 \!-\! 2\lambda \frac{\theta(\alpha,\beta,m)^2}{(m-1)^2}\right] \!\leq\! 2y_1 - 2\lambda \frac{\theta(s_n,s_n,m)^2}{(m-1)^2}.
	\end{equation*}
	
	Now, with the specific choice $y_1 = 7 \cdot 10^{-n} /4 + \lambda (1 + 1/(m-1)) - 3 \cdot 10^{-n}/(4(m-1))$ and using $m\geq 4$, we calculate
	\begin{align*}
	2^{1/p}\left[2y_1 - 2\lambda \left(1 + \frac{1}{m-1}\right)\right] 
	&\geq \left(\frac{7}{2} - \frac{3}{2(m-1)}\right) 10^{-n} >  2 \cdot 10^{-n},
	\end{align*}
	as well as 
	\begin{align*} 2y_1 - 2\lambda \frac{\theta(s_n,s_n,m)^2}{(m-1)^2}
	&= \frac{7 \cdot 10^{-n} }{2} + 2\lambda \left(\frac{1}{m-1} - \frac{4s_n^2}{m-1}\right) - \frac{3 \cdot 10^{-n}}{2(m-1)} \\&\leq \frac{7 \cdot 10^{-n}}{2}< \frac{5 \cdot 10^{-n+1}}{12} \leq  \frac{5 \cdot 10^{-n+1}2^{-1/p}}{6}
	\end{align*}
	since $1-4s_n^2 \leq  3 \cdot 10^{-n}/(4\lambda)$ by the definition \eqref{eq:rn_and_sn}. This concludes the proof of \eqref{eq:ULTVIneqBDEpsUpper}.
\end{proof}
\subsection{The strong breakdown epsilons}\label{sec:MainBDEpsStrong}
Our aim will be to prove that, for each of the computational problems under consideration, we have $\epsilon_{\mathbb{P}h\mathrm{B}}^\mathrm{s}(\mathrm{p}) > 10^{-k}$ for $\mathrm{p} \in [0,1/2)$ and  $\epsilon_{\mathbb{P}\mathrm{B}}^\mathrm{s}(\mathrm{p}) > 10^{-k}$ for $\mathrm{p} \in [0,1/3)$. We do so in the following lemma:

\begin{lemma}\label{lemma:StrongBDEpsilonBounds10K}
	Let $k$, $m$, and $N$ be natural numbers with $N>m\geq 4$ and $k\geq 1$, and consider the computational problem $\{\Xi,\Omega_{m,N,k}^{\mathrm{s}},\mathcal{M},\Lambda_{m,N} \}$,
	 where $\Omega_{m,N,k}^{\mathrm{s}}$ is one of \eqref{eq:strongOmega-fixed-dim} and $\Xi$ is the appropriate solution map.
	 Then there exists a $\hat\Lambda\in\mathcal{L}^1(\Lambda_{m,N})$ such that,  for  the computational problem $\{\Xi,\Omega_{m,N,k}^{\mathrm{s}},\mathcal{M},\hat\Lambda\}$, we have $\epsilon_{\mathbb{P}h\mathrm{B}}^\mathrm{s}(\mathrm{p}) > 10^{-k}$ for $\mathrm{p} \in [0,1/2)$ and $\epsilon_{\mathbb{P}\mathrm{B}}^\mathrm{s}(\mathrm{p}) > 10^{-k}$ for $\mathrm{p} \in [0,1/3)$.
\end{lemma}

\begin{proof}
	The proof is based on using Proposition \ref{prop:DrivingNegativeProposition}. 
	Concretely, for linear programming, we will construct input sequences $\{\iota^1_n\}_{n=1}^{\infty}\subset \Omega_{m,N,k}^{\mathrm{s}}$, $\{\iota^2_n\}_{n=1}^{\infty}\subset \Omega_{m,N,k}^{\mathrm{s}}$, an input $\iota^0 \in \Omega_{m,N,k}^{\mathrm{s}}$, as well as sets $S^1\subset  \real^{m}$ and $S^2 \subset \real^m$ such that the following hold: 
	\begin{itemize}
		\item[(i)] $\inf_{x^1 \in S^1, x^2 \in S^2} \|x^1 - x^2\|_{p} > 2 \cdot 10^{-k}$.
		\item[(ii)] $\iota^0=(y^0,\matLP^0)\in \real^m\times \real^{m \times N}$ and $\iota^j_n=(y^{j,n},\matLP^{j,n})\in \real^m\times \real^{m \times N} $, for $j=1,2$ and $n\in\mathbb{N}$, satisfy 
		\begin{equation}\label{eq:ShowCloseMatVectorS} 
		\|y^{j,n} - y^0\|_{\infty},\|\matLP^{j,n}-\matLP^0\|_{\max} \leq 4^{-n}.
		\end{equation}
		\item[(iii)] $\Xi(\iota^1_n) \subset S^1$ and $\Xi(\iota^2_n) \subset S^2$, for all $n\in\mathbb{N}$. 
	\end{itemize}
	For the $\ell^1$ and TV problems we will do the same, except that we respectively write $\matl$ and $\matTV$ instead of $\matLP$.
	Once we have done this, the result will follow by part \ref{item:SequenceCounterexampleL1StrongBDEps} of Proposition \ref{prop:DrivingNegativeProposition}.
	 We thus work through each of the computational problems $\{\xilp,\omlps[k]\}$, $\{\xibp,\ombpls[k]\}$, $\{\xiul,\omulls[k]\}$, $\{\xicl,\omclls[k]\}$, $\{\xibptv,\ombptvs[k]\}$, and $\{\xiultv,\omultvs[k]\}$	in turn.
	
	{\bf Case $\{\xilp,\omlps[k]\}$}: We set 
	\begin{equation*}
	\begin{aligned}
	&\iota^0_n = (\vecYMainLP{k},\matl(1/2,1/2,m,N)), \\
	&\iota^1_n = (\vecYMainLP{k},\matl(1/2,1/2-4^{-n},m,N)),\qquad \iota^2_n = (\vecYMainLP{k},\matl(1/2 - 4^{-n},1/2,m,N)),\\
	&S^1 = \{4\cdot 10^{-k}e_1\},\quad\text{and}\quad S^2 = \{4 \cdot 10^{-k}e_2\},
	\end{aligned}
	\end{equation*}
	where $e_i$ is the $i$-th canonical basis vector of $\real^m$. We note that \eqref{eq:ShowCloseMatVectorS} immediately holds, hence  establishing (ii).
	Next, we have $\inf_{x^1 \in S^1, x^2 \in S^2} \|x^1 - x^2\|_{p} = \|4\cdot 10^{-k}e_1 - 4\cdot 10^{-k}e_2\|_{p} > 2 \cdot 10^{-k}$, yielding (i). 
	Finally, to see (iii), note that Lemma \ref{lemma:ProblemBasicExampleLP} implies $\xilp(\iota^1_n) =  \{4\cdot 10^{-k}e_1\}$ and $\xilp(\iota^2_n)  = \{4\cdot 10^{-k}e_2\}$, for $j=1,2$ and $n\in\mathbb{N}$.
	
	{\bf Case $\{\xibp,\ombpls[k]\}$}: We set 
	\begin{equation*}
	\begin{aligned}
	&\iota^0 = (\vecYMainBP{k},\matl(1/2,1/2,m,N)), \\
	&\iota^1_n =  (\vecYMainBP{k},\matl(1/2,1/2-4^{-n},m,N)) ,\qquad \iota^2_n = (\vecYMainBP{k},\matl(1/2 - 4^{-n},1/2,m,N)) ,\\
	&S^1 = \{4\cdot 10^{-k}e_1\},\quad\text{and}\quad S^2 = \{4 \cdot 10^{-k}e_2\},
	\end{aligned}
	\end{equation*}
	from which \eqref{eq:ShowCloseMatVectorS} immediately holds, hence establishing (ii).
	Next, we again have $\inf_{x^1 \in S^1, x^2 \in S^2} \|x^1 - x^2\|_{p} = \|4\cdot 10^{-k}e_1 - 4\cdot 10^{-k}e_2\|_{p} > 2 \cdot 10^{-k}$, yielding (i).
	Finally, to see (iii), we use Lemma \ref{lemma:ProblemBasicExampleBPDNL1} to find that $\xibpdn(\iota^1_n) = \{2(\delta + 2 \cdot 10^{-k} - \delta)e_1\} = \{4\cdot 10^{-k}e_1\}=S^1$ and $\xibpdn(\iota^2_n) = \{2(\delta + 2 \cdot 10^{-k} - \delta)e_2\} = \{4\cdot 10^{-k}e_2\}=S^2$.

	{\bf Case $\{\xiul,\omulls[k]\}$}: The argument is almost identical to that for  $\{\xibp,\ombpls[k]\}$, except that $\vecYMainUL{k}$ replaces $\vecYMainBP{k}$, and we use Lemma \ref{lemma:ProblemBasicExampleULASSO} to obtain
	\[\xiul(\iota^1_n) = \{(\lambda  + 2\cdot 10^{-k} -\lambda) [2(1/2)^2]^{-1} e_1\} = \{4\cdot 10^{-k}e_1\}=S^1\] and similarly $\xiul(\iota^2_n) =  \{4\cdot 10^{-k}e_2\}=S^2$.

	{\bf Case $\{\xicl,\omclls[k]\}$}: Here, the argument differs slightly from the argument for $\{\xibp,\ombpls[k]\}$. Again, we replace $\vecYMainBP{k}$ with $\vecYMainCL{k}$. As before, we immediately obtain (ii). We now define $S^1 = \{4\cdot 10^{-k}e_1 + (\tau - 4\cdot 10^{-k})e_3\}$ and $S^2 = \{4 \cdot 10^{-k}e_2+(\tau - 4\cdot 10^{-k})e_3\}$ and note that 
\[	
\inf_{x^1 \in S^1, x^2 \in S^2} \|x^1 - x^2\|_{p} =  \|4\cdot 10^{-k}e_1 - 4\cdot 10^{-k}e_2\|_{p} > 2 \cdot 10^{-k},
\] 
and thus (i) is satisfied. To establish (iii), we note that for $(\alpha,\beta) \in \albetSet$ we have $\alpha \vee \beta = 1/2$ and thus the parameter $r$ given in \ref{lemma:ProblemBasicExampleCLASSO} satisfies 
	\[r=(\alpha \vee \beta) \vecYMainCL{k}_1 /[1+(\alpha \vee \beta)^2] = 2\cdot 10^{-k+1}/5  = 4\cdot 10^{-k} \leq \tau.\] Therefore Lemma \ref{lemma:ProblemBasicExampleCLASSO} applies and so 
\[
\xicl(\iota^1_n) = \{re_1 + (\tau -r) e_3\} = \{4\cdot 10^{-K}e_1 + (\tau - 4\cdot 10^{-K}) e_3\}
\]
 and similarly $\xicl(\iota^2_n) = \{4\cdot 10^{-K}e_2 + (\tau - 4\cdot 10^{-K}) e_3\}$, as desired. 
	
	{\bf Case $\{\xibptv,\ombptvs[k]\}$}: We choose $n_0$ so that $1/2 - 4^{-n_0} \in [r_k,1/2]$ and, writing $y_1 = \vecYMainBP{k}_1$, we set
	\begin{equation*}
	\begin{aligned}
	&\iota^0 =(\vecYMainBPTV{k},\matTV(1/2,1/2,m,N)), \\
	&\iota^1_n = (\vecYMainBPTV{k},\matTV(1/2,1/2-4^{-n-n_0}),m,N)),\qquad \iota^2_n =  (\vecYMainBPTV{k},\matTV(1/2 - 4^{-n-n_0}),1/2,m,N)) ,\\
	&S^1 = \{\flipOp \eta(y_1,\alpha,\beta)\, \vert \, (\alpha,\beta) \in\albetSetBPTV{k} \},\quad\text{and}\quad S^2 = \{\eta(y_1,\alpha,\beta) \, \vert \, (\alpha,\beta) \in\albetSetBPTV{k}\}.
	\end{aligned}
	\end{equation*}
	It then follows by the choice of $n_0$ that $\iota^0, \iota^1_n$ and $\iota^2_n$ are all in $\ombptvs$. We also immediately have  \eqref{eq:ShowCloseMatVectorS} and hence (ii) holds.
	Now, by Lemma \ref{lemma:BPTVEtaBounds} (where we choose $n=k$ in Lemma \ref{lemma:BPTVEtaBounds}) and more specifically equation \eqref{eq:BPTVIneqBDEpsLower} we have
	\begin{equation*}
	\min_{\substack{x^1 \in S^1\\ x^2 \in S^2}} \|x^1 - x^2\|_p = \min \limits_{\substack{(\alpha,\beta) \in \albetSetBPTV{K} \\ (\alpha',\beta') \in \albetSetBPTV{K}}} \|\eta(y_1,\alpha,\beta) - \flipOp(y_1,\alpha',\beta')\|_p > 2 \cdot 10^{-k},
	\end{equation*} and thus (i) follows.
	Finally, using Lemma \ref{lemma:BPTVSolutions} we immediately obtain $\xibptv(\iota^j_n) \subseteq S^j$, for $j=1,2$ and $n\in\mathbb{N}$, establishing (iii).
	
	\textbf{Case $\{\xiultv,\omultvs[k]\}$}: This time the argument is very similar to that for $\{\xibptv,\ombptvs[k]\}$. We list only the differences: 
	first, we replace $\vecYMainBPTV{k}$ with $\vecYMainULTV{k}$. Again, it follows immediately that \ref{eq:ShowCloseMatVector} holds. The definitions of $S^1$ and $S^2$ are also identical with one difference -- we replace $\eta$ with $\psi$ defined in Lemma \ref{lemma:ULTVSolutions}. The argument for (i) is the same except we replace all mentions of Lemma \ref{lemma:BPTVEtaBounds} with Lemma \ref{lemma:ULTVPsiBounds} and inequality \eqref{eq:BPTVIneqBDEpsLower} with \eqref{eq:ULTVIneqBDEpsLower}. Finally, part (iii) is identical except now we use Lemma \ref{lemma:ULTVSolutions} instead of Lemma \ref{lemma:BPTVSolutions}.
\end{proof}

\subsection{The weak breakdown epsilons}\label{sec:MainBDEpsWeak} 
Similarly to the previous subsection, our aim will be to prove that, for each of the computational problems under consideration, we have both $\epsilon^{\mathrm{w}}_{\mathbb{P}\mathrm{B}}(\mathrm{p}) > 10^{-k+1}$, for $\mathrm{p} \in [0,1/2)$, and $\epsilon^{\mathrm{w}}_{\mathrm{B}}> 10^{-k+1}$. 
\begin{lemma}\label{lemma:WeakBDEpsilonBounds10K}
	Let $k$, $m$, and $N$ be natural numbers with $N>m\geq 4$ and $k\geq 2$, and consider the computational problem $\{\Xi,\Omega_{m,N,k}^{\mathrm{w}},\mathcal{M},\Lambda_{m,N} \}, $
	 where $\Omega_{m,N,k}^{\mathrm{w}}$ is one of \eqref{eq:weakOmega-fixed-dim} and $\Xi$ is the appropriate solution map. 
	 Then there exists a $\hat\Lambda\in\mathcal{L}^1(\Lambda_{m,N})$ such that,  for  the computational problem $\{\Xi,\Omega_{m,N,k}^{\mathrm{s}},\mathcal{M},\hat\Lambda\}$, we have $\epsilon^{\mathrm{w}}_{\mathbb{P}\mathrm{B}}(\mathrm{p}) > 10^{-k+1}$, for $\mathrm{p} \in [0,1/2)$, and $\epsilon^{\mathrm{w}}_{\mathrm{B}}> 10^{-k+1}$.
\end{lemma}
\begin{proof}
	The proof reads almost identically to that of Lemma \ref{lemma:StrongBDEpsilonBounds10K}, so we will only list the differences. Concretely, for linear programming, the proof proceeds by constructing input sequences $\{\iota^1_n\}_{n=1}^{\infty}\subset \Omega$, $\{\iota^2_n\}_{n=1}^{\infty}\subset \Omega$, an $\iota^0 =(y^0,A^0)\in \real^m \times \real^{m \times N}$ (a crucial difference from the proof of Lemma \ref{lemma:StrongBDEpsilonBounds10K} is that we no longer require $\iota^0 \in \Omega$) as well as sets $S^1\subset  \real^N$ and $S^2 \subset \real^N$ such that the following hold: 
	\begin{itemize}[leftmargin=8mm]
		\item[(i)] $\inf_{x^1 \in S^1, x^2 \in S^2} \|x^1 - x^2\|_{\infty} > 2 \cdot 10^{-k+1}$.
		\item[(ii)] For $j=1,2$ and each natural $n$, $\iota^j_n$ is a tuple $(y^{j,n},A^{j,n})$ for $y^{j,n} \in \real^m$ and $A^{j,n}$. These tuples will satisfy \begin{equation}\label{eq:ShowCloseMatVector} \|y^{j,n} - y^0\|_{\infty},\|A^{j,n}-A^0\|_{\max} \leq 4^{-n}.\end{equation}
		\item[(iii)] For each $n$, both $\Xi(\iota^1_n) \subset S^1$ and $\Xi(\iota^2_n) \subset S^2$. 
	\end{itemize}
	Again, for the $\ell^1$ and TV problems we do the same, except that we respectively write $\matl$ and $\matTV$ instead of $\matLP$.
	
	These are the conditions required for us to apply part \ref{item:SequenceCounterexampleL1WeakBDEps} of Proposition \ref{prop:DrivingNegativeProposition}  and conclude that $\epsilon^{\mathrm{w}}_{\mathbb{P}\mathrm{B}}(\mathrm{p}) > 10^{-k+1}$ for $\mathrm{p} \in [0,1/2)$ and that $\epsilon^{\mathrm{w}}_{\mathrm{B}}> 10^{-K+1}$. We go through the differences to the proof of Lemma \ref{lemma:StrongBDEpsilonBounds10K} for each of the computational problems  in turn.
	
	{\bf Case $\{\xilp,\omlpw[k]\}$}: in the definitions of $\iota^0$, $\iota^1_n$ and $\iota^2_n$ we replace $\vecYMainLP{k}$ with $\vecYMainLP{k-1}$, and every instance of $10^{-k}$ is replaced by $10^{-k+1}$.

	{\bf Case $\{\xibp,\ombplw[k]\}$}: 
	we replace $\vecYMainBP{k}$ with $\vecYMainBP{k-1}$ and every instance of $10^{-k}$ is replaced by $10^{-k+1}$.
	
	{\bf Case $\{\xiul,\omullw[k]\}$}: 
	we replace $\vecYMainUL{k}$ with $\vecYMainUL{k-1}$ and every instance of $10^{-k}$ is replaced by $10^{-k+1}$.
	
	{\bf Case $\{\xicl,\omcllw[k]\}$}:
	we replace $\vecYMainCL{k}$ with $\vecYMainCL{k-1}$. We replace $4\cdot 10^{-k}$ in the definitions of $S^1$ and $S^2$ by $4 \cdot 10^{-k+1}$ so that $\inf_{x^1 \in S^1, x^2 \in S^2} \|x^1 - x^2\|_p  > 2 \cdot 10^{-K+1}$. Finally, the value of $r$ now becomes $2\cdot 10^{-k+2}/5 = 4 \cdot 10^{-k+1}$.
	
	{\bf Case $\{\xibptv,\ombptvw[k]\}$}: 
	we replace $\vecYMainBPTV{k}$ with $\vecYMainBPTV{k-1}$ and every instance of $10^{-k}$ is replaced by $10^{-k+1}$. In addition, Lemma \ref{lemma:BPTVEtaBounds} is applied by setting $n$ to $k-1$ instead of $k$.
	
	{\bf Case $\{\xiultv,\omultvw[k]\}$}: 
	we replace $\vecYMainULTV{k}$ with $\vecYMainULTV{k-1}$ and every instance of $10^{-k}$ is replaced by $10^{-k+1}$. In addition,  Lemma \ref{lemma:ULTVPsiBounds} is applied by setting $n$ to $k-1$ instead of $k$.
	\end{proof}
Note that in none of the cases  in the proof of Lemma \ref{sec:MainBDEpsWeak} is $\iota^0$ contained in $\Omega$, so these arguments are not sufficient to prove that the strong breakdown epsilons are greater than or equal to $10^{-k+1}$ (indeed, this statement is not true).

\section{Proof of Theorem \ref{Cor:main} -- Preliminaries: constructing the subroutines}\label{sec:UsefulSubroutinesForMainThm} 
To construct the algorithms  required to prove Proposition \ref{Cor:main_SCI_cont} and part (ii) of Proposition \ref{Cor:main_SCI} and, we require various subroutines. 
Throughout this section, we consider the computational problem $\{\Xi,\Omega_{m,N},\allowbreak\mathcal{M}_N,\Lambda_{m,N}\}^{\Delta_1}=\{\tilde{\Xi},\tilde\Omega_{m,N},\mathcal{M}_N,\tilde{\Lambda}_{m,N}\}$, where $\Omega_{m,N}$ is one of \eqref{eq:defs-5-Om-fixedDim} and $\Xi$ is the corresponding solution map. For the linear programming case, we fix the notation for an element of $\tilde\Omega_{m,N}$ by writing $\tilde{\iota}=\big(\{y_j^{(n)}\}_{n=0}^{\infty}, \{\matLP_{j,k}^{(n)}\}_{n=0}^{\infty}\big)_{j,k}$, corresponding to an $\iota=(y,\matLP)\in \Omega$. For the $\ell^1$ and TV problems the notation is analogous, except that we respectively write $\matl$ and $\matTV$ instead of $\matLP$.

We begin with a subroutine termed \textit{OutputEta} which applies only to the basis pursuit with TV regularisation problem. Informally, the purpose of this subroutine is to approximate the function $\eta$ defined in \S\ref{sec:TV reg geometry}. The exact specification of the subroutine is given below and a proof of its correctness and complexity is given in Lemma \ref{lemma:OutputEtaCorrect}.

{\it Subroutine} {\bf OutputEta:}

\indent Inputs: Dimensions $m$, $N$, and natural numbers $k_K$ and a $k_\epsilon$. \\
\indent Oracles: $\orvec$ and $\ormat$ providing access to the components $y_j^{(n)}$ and $\matTV_{j,k}^{(n)}$ of an input $\tilde{\iota}$.\\
\indent Output: $\eta^* \in \dyadic^N$ (in the Turing case) or $\eta^* \in \real^N$ (in the BSS case).

\begin{enumerate}[leftmargin=10mm, label= \arabic*.,ref = step \arabic*]
	\item Set $n_1=(4k_\epsilon + \length(N)+ 7)\vee (4k_K+1)$ and use the oracle $\orvec$ to obtain $y_1':=y_1^{(n_1)}$ and the oracle $\ormat$ to obtain $\alpha':=\matTV_{1,1}^{(n_1)}$ and $\beta':=\matTV_{1,N}^{(n_1)}$. We then compute a $w \in \dyadic$ so that
	\begin{equation}\label{eq:wDef}
		w \approx \left(\frac{m-1+(\alpha'+\beta')^2}{m-1}\right)^{-1/2}
	\end{equation}
	to $k_\epsilon + \length(N)+4$ bits of precision.
	\item Next, set $n_2=k_\epsilon+\length(N)+2$ and compute
	$\eta^*_1 = \eta^*_2=\dotsb = \eta^*_{N-1} \in \dyadic$ (or $ \real$, in the BSS case) so that
	\begin{equation} \label{eq:eta1Assignment}
		\eta^*_1 \approx \frac{\delta(\alpha'+\beta')w}{m-1}
		\end{equation} 
		to $n_2$ bits of precision, incurred by converting a rational to a dyadic (in the BSS case, we simply assign the right hand side to $\eta^*_1$).
	\item 
	Similarly we compute $\eta^*_N \in \dyadic$ (or $ \real$, in the BSS case)
	\begin{equation} \label{eq:etaNAssignment}
	\eta^*_N \approx \frac{ \delta(\alpha'+\beta')w}{m-1}+ \frac{y'_1}{(\alpha' \vee \beta')} - \frac{\delta w}{(\alpha'\vee \beta')} \left(\frac{m-1+(\alpha'+\beta')^2}{m-1}\right)
	\end{equation}
	to $n_2$ bits of precision, again incurred by converting a rational to a dyadic  (in the BSS case, we assign the right hand side to $\eta^*_N$).
\end{enumerate}

\begin{lemma}\label{lemma:OutputEtaCorrect}
Assume that the subroutine `\textit{OutputEta}' is applied to an input $\tilde\iota$ corresponding to $\iota=(y,\matTV) = (\vecYMainBPTV{k},\matTV(\alpha,\beta,m,N)) \in \ombpls[k]$, where $(\alpha,\beta) \in \albetSetBPTV{k}$ for some natural number $k$ with $k \leq k_K$. Then the output $\eta^*$  satisfies $\|\eta^*- \eta(y_1,\alpha,\beta)\|_p\leq 10^{-k_\epsilon}$, and the number of digits needed by the oracles as well as  the BSS  and the Turing runtime
 are all bounded above by some polynomial in $k_K$, $k_\epsilon$, and $\log(N)$.
\end{lemma}
\begin{proof}
	Let $\alpha$ and $\beta$ be as in the statement of the lemma. For ease of notation we also write $z_y = y_1' - y_1$, $z_{\alpha} = \alpha' - \alpha$ and $z_{\beta} = \beta' - \beta$ as well as $z = (z_y,z_{\alpha},z_{\beta})$. 
	Note that then $\|z\|_\infty\leq 2^{-n_1}\leq 10^{-k_\epsilon}/(70N)\wedge 10^{-k_K}/2$, by definition of $n_1$.
	
	Our argument to prove the correctness of $\textit{OutputEta}$ will rely on bounding  
	$\|\eta^* - \eta(y_1',\alpha',\beta')\|_1$ and $\|\eta(y_1,\alpha,\beta) - \eta(y_1',\alpha',\beta')\|_1$ separately. For the first of these two terms, we note that $\eta(y_1',\alpha',\beta')$ can be written as 
	\begin{equation*}
		\eta(y_1',\alpha',\beta')_1 = \eta(y_1',\alpha',\beta')_2 = \dotsb = \eta(y_1',\alpha',\beta')_{N-1} = \frac{\delta(\alpha'+ \beta')}{m-1}\left(\frac{m-1+(\alpha'+\beta')^2}{m-1}\right)^{-\frac{1}{2}}
	\end{equation*}
	and 
	\begin{equation*}
		\eta(y_1',\alpha',\beta')_N= \eta(y_1',\alpha',\beta')_1 + \frac{y_1'}{\alpha' \vee \beta'}  - \delta\left(\frac{m-1+(\alpha'+\beta')^2}{(\alpha'\vee \beta')(m-1)}\right)\left(\frac{m-1+(\alpha'+\beta')^2}{m-1}\right)^{-\frac{1}{2}}
	\end{equation*}
	Thus, using the triangle inequality, we obtain 
	\begin{align*}
		\|\eta^* - \eta(y_1',\alpha',\beta')\|_1 &\leq \!\delta \! \left(\frac{N (\alpha'+\beta')}{m-1} + \frac{m-1 + (\alpha'+\beta')^2}{(\alpha'\vee \beta')(m-1)}\right)\!\left|w - \left(1+\frac{(\alpha'+\beta')^2}{m-1}\right)^{-\frac{1}{2}}\right| \!+\! N2^{-n_2}
	\end{align*}
	Since $(\alpha, \beta) \in \albetSetBPTV{k}$ we must have $\alpha+ \beta \leq 1$ and $\alpha \vee \beta = 1/2$. Thus because $z_{\alpha},z_{\beta}\leq 10^{-k_K} \leq 1/4$, we note that both $\alpha' + \beta' \leq 3/2$ and, using Lemma \ref{lem:alpha-vee-beta-estimate}, $\alpha' \vee \beta' \geq 1/4$. Therefore
	\begin{equation*}
		\delta\left(\frac{N (\alpha'+\beta')}{m-1} + \frac{m-1 + (\alpha'+\beta')^2}{(\alpha'\vee \beta')(m-1)}\right)\leq \delta\left[\frac{3 N }{2(m-1)} + 4\left(1+\frac{9}{4(m-1)}\right)\right].
	\end{equation*}
	Furthermore, as $\delta \leq 1$ and $N > m \geq 4$ we obtain \[\delta\left[\frac{3 N }{2(m-1)} + 4\left(1+\frac{9}{4(m-1)}\right)\right]\leq \frac{N}{2} + 4+3\leq \frac{5N}{2}	.
	\]
	Therefore, as $w$ is computed to precision $k_\epsilon + \length(N)+4$ so that $2^{-(k_\epsilon + \length(N)+4)}\leq 10^{-k_\epsilon}/(10N)$ and  $2^{-n_2}\leq 10^{-k_\epsilon}/(10N)$ by definition of $n_2$, we obtain
 \[\|\eta^* - \eta(y_1',\alpha',\beta')\|_1 \!\leq\! \frac{5N}{2} \!\!\left|w - \left(1+\frac{(\alpha'+\beta')^2}{m-1}\right)^{-\frac{1}{2}}\right|+N2^{-n_2} \!\leq\! \frac{5N\cdot10^{-k_\epsilon}}{20N} + \frac{N\cdot 10^{-k_\epsilon} }{4N} \!=\! \frac{10^{-k_\epsilon}}{2}.\]
	
	Next, we will bound $\|\eta(y_1,\alpha,\beta) - \eta(y_1',\alpha',\beta')\|_\infty$ using Lemma \ref{lem:eta-pert-bound}. To apply Lemma \ref{lem:eta-pert-bound} we must show that  $\|z\|_{\infty} \leq (1/8) \wedge (\mu(y_1,\alpha,\beta)/2)$ where the function $\mu$ is defined within the statement of Lemma \ref{lem:eta-pert-bound}. Because $(\vecYTV,\matTV) \in \ombpls[k]$ with $k \leq k_K$ and $\theta(1/2,1/2,m) \geq m-1 \geq 1$ we have \[y_1 = \frac{\delta m} {\theta(1/2,1/2,m)} + \left(7 - \frac{3} {\theta(1/2,1/2,m)}\right)\frac{10^{-k}}{4} \geq \frac{\delta m} {\theta(1/2,1/2,m)} + \frac{(7 - 3)\cdot10^{-{k_K}}}{4}\] and thus, because $\theta(\alpha,\beta,m)$ is increasing in both $\alpha$  and $\beta$ and $(\alpha, \beta) \in \albetSetBPTV{k} \subset [1/4,1/2]\times [1/4,1/2]$,
	\begin{equation*}
		\mu(y_1,\alpha,\beta) \geq \frac{\delta m}{\theta(1/2,1/2,m)} - \frac{\delta \theta(1/2,1/2,m)}{m-1} + \frac{(7-3) \cdot 10^{-{k_K}}}{4} = 10^{-{k_K}}
	\end{equation*}
	Hence, as $\|z\|_{\infty} \leq 10^{-{k_K}}/2 \leq (1/8) \wedge (\mu(y_1,\alpha,\beta)/2)$, the conditions of Lemma \ref{lem:eta-pert-bound} are met.
We conclude, this time using  $\|z\|_\infty\leq 10^{-k_\epsilon}/(70N)$, that 
\[
\|\eta(y_1,\alpha,\beta) - \eta(y_1',\alpha',\beta')\|_\infty \leq 14(y_1 +1) \|z\|_{\infty} \leq 10^{-k_\epsilon} (y_1+1)/(5 N).
\]
 By Lemma \ref{lemma:mainThmSizeEstimates} we have $y_1 \leq 3/2$ so  $\|\eta(y_1,\alpha,\beta) - \eta(y_1',\alpha',\beta')\|_\infty \leq 10^{-k_\epsilon}/(2N)$. Therefore $\|\eta(y_1,\alpha,\beta) - \eta(y_1',\alpha',\beta')\|_p \leq 10^{-k_\epsilon}/2$. 
	
	The proof of correctness completed by combining the already established inequalities $\|\eta(y_1,\alpha,\beta) - \eta(y_1',\alpha',\beta')\|_1 \leq 10^{-k_\epsilon}/2$ and $\|\eta^* - \eta(y_1',\alpha',\beta')\|_1\leq 10^{-k_\epsilon}/2$. Indeed, we have 
\begin{equation*}
\begin{split}
\|\eta^* - \eta(y_1',\alpha',\beta')\|_p &\leq \|\eta^* - \eta(y_1',\alpha',\beta')\|_1 \\
&\leq  \|\eta(y_1,\alpha,\beta) - \eta(y_1',\alpha',\beta')\|_1 + \|\eta^* - \eta(y_1',\alpha',\beta')\|_1 \leq 10^{-k_\epsilon}.
\end{split}
\end{equation*}	
	 Next, we note that number of digits  $n_1=(4k_\epsilon + \length(N)+ 7)\vee (4k_K+1)$ needed by the oracles is polynomial in  $k_K$, $k_\epsilon$, and $\log(N)$.
	All that remains is to bound the complexities. Note that it suffices to show that the Turing runtime is polynomial in $k_\epsilon$, $k_K$, and $\length(N)$, as this will then imply the desired polynomial bound on the BSS runtime.
	
	To this end, as noted in \cite[Page 92-93]{ElementaryFunctionsImplementation}, we recall that it is possible to use Newton-Raphson iteration to compute the reciprocal square root of an $u$-bit number to $v$ bits of precision at the same cost (up to constants and asymptotically in $u$ and $v$) as a multiplication operation on two integers each with number of bits bounded above by $u\vee v$. Since we have already shown that $(\alpha' + \beta')^2 \leq 4$, the number of bits of the numerator of $\frac{m-1+(\alpha'+\beta')^2}{m-1}$ is bounded above by a polynomial in $k_\epsilon$, $k_K$, and $\length(N)$ and the number of bits of the denominator is $\oh(\log(m))$. The length of $k_\epsilon+\length(N)+4$ is clearly polynomial in  $k_\epsilon$, $k_K$, and $\length(N)$. Hence (since multiplication of $b$-bit integers can be done in $\oh(b\log(b))$, the runtime  of computing $w$ in the Turing model is polynomial in $k_\epsilon$, $k_K$, and $\length(N)$. In particular, $\length(w)$ must be bounded above by a polynomial in $k_\epsilon$, $k_K$ and $\length(N)$.
	
	 The right hand sides of both \eqref{eq:eta1Assignment} and \eqref{eq:etaNAssignment} involve finitely many additions, multiplications and subtractions of rational numbers each with lengths bounded above by a polynomial in $k_\epsilon$, $k_K$, $\length(N)$, and $\length(\delta)$. Since $\delta$ is assumed to be fixed, the numerator and denominator of the right hand side of both \eqref{eq:eta1Assignment} and \eqref{eq:etaNAssignment} can be computed in Turing runtime  polynomial in $k_\epsilon$, $k_K$, and $\length(N)$.

	For each of equations \eqref{eq:eta1Assignment} and \eqref{eq:etaNAssignment} the conversion between the rational (say, $q_1/q_2$) right hand side to a dyadic is  done through integer division to  $n_2$ bits of precision. Again, as noted in \cite[Page 92-93]{ElementaryFunctionsImplementation} this can be done in Turing runtime polynomial in $\length(q_1/q_2)$ and $n_2$. But we have already established that both of these quantities are bounded above by some polynomial in $k_\epsilon$, $k_K$, and $\length(N)$, and  thus the Turing runtime of this step is also polynomial in the same quantities.
	
	We have therefore bounded each of the steps of the subroutine by polynomials in $k_\epsilon$, $k_K$, and $\length(N)$. But the subroutine itself only performs finitely many steps and thus the overall Turing runtime is also bounded by a polynomial in $k_\epsilon$, $k_K$, and $\length(N)$. The proof is complete by noting that $\length(N) = \oh(\log(N))$.
\end{proof}

Similarly, we formulate \textit{OutputPsi}, which applies only to the unconstrained lasso with TV regularisation problem. Informally, the purpose of this subroutine is to approximate the function $\psi$ defined in \S\ref{sec:TV reg geometry}. Its exact specification is given below and a proof of its correctness and complexity is given in Lemma \ref{lemma:OutputPsiCorrect}.

{\it Subroutine} {\bf OutputPsi:} 

\indent Inputs: Dimensions $m$, $N$, and natural numbers $k_K$ and a $k_\epsilon$. \\
\indent Oracles: $\orvec$ and $\ormat$ providing access to the components $y_j^{(n)}$ and $\matTV_{j,k}^{(n)}$ of an input $\tilde{\iota}$.\\
\indent Output: $\psi^* \in \dyadic^N$ (in the Turing case) or $\psi^* \in \real^N$ (in the BSS case).

\begin{enumerate}[leftmargin=10mm, label= \arabic*.,ref = step \arabic*]
	\item Set $n_1=(4k_\epsilon + \length(N)+ 9)\vee (4k_K+3)$  and use the oracle $\orvec$ to obtain $y_1':=y_1^{(n_1)}$ and the oracle $\ormat$ to obtain $\alpha':=\matTV_{1,1}^{(n_1)}$ and $\beta':=\matTV_{1,N}^{(n_1)}$.  
	\item Next, set $n_2=4k_\epsilon +\length(N)+1$ and compute $\psi^*_1 = \psi^*_2=\dotsb = \psi^*_{N-1} \in \dyadic$  (or $ \real$, in the BSS case)  so that
	\begin{equation}\label{eq:etaTilde1Assignment} 
	\psi^*_1 \approx \frac{\lambda(\alpha' + \beta')}{2(m-1)(\alpha' \vee \beta')}
	\end{equation}
	to $n_2$ bits of precision, incurred by converting a rational to a dyadic (in the BSS case, we assign the right hand side to $\psi^*_1$).
	\item Similarly, compute $\psi^*_N \in \dyadic$  (or $ \real$, in the BSS case) so that
	\begin{equation}\label{eq:etaTildeNAssignment}
	\psi^*_N \approx \frac{\lambda(\alpha' + \beta')}{2(m-1)(\alpha' \vee \beta')}+ \frac{1}{\alpha'\vee \beta'} \left(y'_1 - \frac{\lambda \left[(\alpha'+\beta')^2+(m-1)\right]}{2(m-1)(\alpha' \vee \beta')}\right)
	\end{equation}
	to $n_2$ bits of precision, again incurred by converting a rational to a dyadic  (in the BSS case, we assign the right hand side to $\psi^*_N$).
\end{enumerate}
\begin{lemma}\label{lemma:OutputPsiCorrect}
Assume that the subroutine `\textit{OutputPsi}' is applied to an input $\tilde\iota$ corresponding to $\iota=(y,\matTV) = (\vecYMainULTV{k},\matTV(\alpha,\beta,m,N))\in \omulls[k]$, where $(\alpha,\beta) \in \albetSetULTV{k}$ for some natural number $k$ with $k \leq k_K$. Then the output $\psi^*$  satisfies $\|\psi^*- \psi(y_1,\alpha,\beta)\|_p\leq 10^{-k_\epsilon}$, and  the number of digits needed by the oracles as well as  the BSS  and the Turing runtime
 are all bounded above by some polynomial in $k_K$, $k_\epsilon$, and $\log(N)$.
\end{lemma}
\begin{proof}
	We start by proving correctness. Let $\alpha$ and $\beta$ be as in the statement of the lemma. For ease of notation we also write $z_y = y_1' - y_1$, $z_{\alpha} = \alpha' - \alpha$ and $z_{\beta} = \beta' - \beta$ as well as $z = (z_y,z_{\alpha},z_{\beta})$. 
	Note that then $\|z\|_\infty\leq 2^{-n_1}\leq \big(10^{-k_\epsilon}/(512 N)\big) \wedge \big( 10^{-k_K}/8\big)$, by definition of $n_1$.
	Next, by definition of $n_2$, we have $2^{-n_2}\leq 10^{-k_\epsilon}/(2N)$, and thus $\|\psi^* - \psi(y_1',\alpha',\beta')\|_p \leq N\cdot 2^{-n_2} \leq 10^{-k_\epsilon}/2$.  It thus suffices to show that $\|\psi(y_1,\alpha,\beta) - \psi(y_1',\alpha',\beta')\|_p\leq 10^{-k_\epsilon}/2$. We will accomplish this by using Lemma \ref{lem:etatilde-pert-bound}. 
	
	As in the proof of Lemma \ref{lemma:OutputEtaCorrect}, let $z_y = y_1' - y_1$, $z_{\alpha} = \alpha' - \alpha$, and $z_{\beta} = \beta' - \beta$.  To apply Lemma \ref{lem:etatilde-pert-bound} we must verify the condition  $\|z\|_{\infty} \leq (1/8) \wedge (\varsigma(y_1,\alpha,\beta)/10)$ where the function $\varsigma$ is defined in Lemma \ref{lem:etatilde-pert-bound}. Because $(y,T) \in \omulls$ as well as $m \geq 4$ and $k \leq k_K$ 
	\begin{equation*}
		y_1 = \frac{7 \cdot 10^{-k}}{4} + \lambda \left( 1 + \frac{1}{m-1}\right) - \frac{3 \cdot 10^{-k}}{4(m-1)} \geq \frac{3 \cdot 10^{-k_K}}{2} + \lambda \left( 1 + \frac{1}{m-1}\right).
	\end{equation*}
	Thus the definition of $r$ from \ref{lem:etatilde-pert-bound} implies that 
	\begin{align*}
		\varsigma(y_1,\alpha,\beta) &\geq \frac{3 \cdot 10^{-k_K}}{2} + \lambda \left(1+\frac{1}{m-1}\right)  - \lambda \frac{\theta(\alpha,\beta,m)^2}{(m-1)^2}
		\\&=  \frac{3 \cdot 10^{-k_K}}{2} + \lambda \frac{\theta(1/2,1/2,m)^2}{(m-1)^2}  - \lambda \frac{\theta(\alpha,\beta,m)^2}{(m-1)^2}\geq  \frac{3 \cdot 10^{-k_K}}{2}
	\end{align*}
	since once again $\theta(\alpha,\beta,m)$ is increasing in $\alpha $ and $\beta$, and $(\alpha, \beta) \in \albetSetBPTV{K} \cup \albetSetBPTV{K-1} \subset [1/4,1/2]\times [1/4,1/2]$. Hence $\|z\|_{\infty} \leq 10^{-k_K}/8 < 3\cdot 10^{-k_K}/20 \leq (1/8) \wedge (\varsigma(y_1,\alpha,\beta)/10)$ and thus the conditions of Lemma \ref{lem:etatilde-pert-bound} are met.
	
	We conclude that $\|\psi(y_1,\alpha,\beta) - \psi(y_1',\alpha',\beta')\|_\infty \leq  112(y_1 +1) \|z\|_{\infty} \leq  112(y_1+1)\cdot 10^{-k_\epsilon}/(512 N)$ where we used $\|z\|_{\infty}\leq 10^{-k_\epsilon}/(512 N)$. By Lemma \ref{lemma:mainThmSizeEstimates}, $y_1 \leq 107/180$ so $112(y_1+1) \leq 112\cdot 287/180 \leq 256$. Therefore  $\|\psi(y_1,\alpha,\beta) - \psi(y_1',\alpha',\beta')\|_\infty \leq 10^{-k_\epsilon}/(2N)$ and thus  $\|\psi(y_1,\alpha,\beta) - \psi(y_1',\alpha',\beta')\|_p\leq \|\psi(y_1,\alpha,\beta) - \psi(y_1',\alpha',\beta')\|_\infty\cdot  N^{1/p} \leq 10^{-k_\epsilon}/2$, as desired.

	Next, we note that number of digits  $n_1=(4k_\epsilon + \length(N)+ 9)\vee (4k_K+3)$ needed by the oracles is polynomial in  $k_K$, $k_\epsilon$, and $\log(N)$.
	All that remains is to bound the complexities. Note that it suffices to show that the Turing runtime is polynomial in $k_\epsilon$, $k_K$, and $\length(N)$, as this will then imply the desired polynomial bound on the BSS runtime.
	
	The right hand sides of both \eqref{eq:etaTilde1Assignment} and \eqref{eq:etaTildeNAssignment} involve finitely many additions, multiplications and subtractions of rational numbers each with lengths bounded above by a polynomial in $k_\epsilon$, $k_K$, $\length(N)$, and $\length(\lambda)$. Since $\lambda$ is assumed to be fixed, numerator and thee denominator of the right hand side of both \eqref{eq:eta1Assignment} and \eqref{eq:etaNAssignment} can thus be computed in Turing runtime  complexity polynomial in $k_\epsilon$, $k_K$, and $\length(N)$.
	
	For each of equations \eqref{eq:eta1Assignment} and \eqref{eq:etaNAssignment} the conversion between the rational (say, $q_1/q_2$) right hand side to a dyadic is  done through integer division to  $n_2$ bits of precision. Again, as noted in \cite[Page 92-93]{ElementaryFunctionsImplementation} this can be done with complexity polynomial in $\length(q_1/q_2)$ and $n_2$. But we have already established that both of these quantities are bounded above by some polynomial in $k_\epsilon$, $k_K$, and $\length(N)$, and thus the complexity of this step is also polynomial in the same quantities.
	
	We have therefore bounded each of the steps of the subroutine by polynomials in $k_\epsilon$, $k_K$, and $\length(N)$, and,  as the subroutine itself only performs finitely many steps, the overall Turing runtime is also bounded by a polynomial  in $k_\epsilon$, $k_K$, and $\length(N)$. Noting that $\length(N) = \oh(\log(N))$ concludes the proof.	
\end{proof}

Next, we need the following subroutine which operates in the case when the input $\tilde\iota$ corresponds to an $\iota$ in one of  $\omlpw[k]$, $\ombplw[k]$ $\omullw[k]$, $\omcllw[k]$, $\ombptvw[k]$ or $\omultvw[k]$. We title this subroutine \textit{Weak} as it is used in \S\ref{sec:MainBDEpsWeak} to give lower bounds on the weak breakdown epsilons. The subroutine is defined as follows:

{\it Subroutine} {\bf Weak:}

\indent Inputs: Dimensions $m$, $N$, and a natural number $k_\epsilon$. \\
\indent Oracles: $\orvec$ and $\ormat$ providing access to the components $y_j^{(n)}$ and $\matLP_{j,k}^{(n)}$ (respectively  $\matl_{j,k}^{(n)}$ or $\matTV_{j,k}^{(n)}$) of an input $\tilde{\iota}$.\\
\indent Output: $x \in \dyadic^N$ (in the Turing case) or $x \in \real^N$ (in the BSS case).

\begin{enumerate}[leftmargin=10mm, label= \arabic*.,ref = step \arabic*]
	\item We execute a loop that proceeds as follows -- at each iteration, we increase $n$, starting with $n=1$. What we do now depends on the problem at hand. In the linear programming case we use the oracle $\ormat$ to read  $\matLP_{1,1}^{(n)}$ and $\matLP_{1,2}^{(n)}$ and set $\dalgo = \matLP_{1,1}^{(n)}-\matLP_{1,2}^{(n)}$. For  the $\ell^1$ problems we likewise read  $\matl_{1,1}^{(n)}$ and $\matl_{1,2}^{(n)}$ and set $\dalgo = \matl_{1,1}^{(n)}-\matl_{1,2}^{(n)}$. For  the TV problems we  read  $\matTV_{1,1}^{(n)}$ and $\matTV_{1,N}^{(n)}$ and set $\dalgo = \matTV_{1,1}^{(n)}-\matTV_{1,N}^{(n)}$.
	Next, we branch depending on the value of $\dalgo$:
	\begin{enumerate}[label = \alph*.]
		\item If $\dalgo > 2^{-n+1}$ then we output $x \in \dyadic^N$ with $\|x - 4\cdot 10^{-K+1} e_1\|_p \leq 10^{-k_\epsilon}$ for linear programming, basis pursuit with $\ell^1$ regularisation or unconstrained lasso with $\ell^1$ regularisation. For constrained lasso we output $x \in \dyadic^N$ with $\|x - 4\cdot 10^{-K+1}e_1 - (\tau - 4\cdot 10^{-K+1})e_3\|_p \leq 10^{-k_\epsilon}$. For basis pursuit TV we apply the subroutine $\textit{OutputEta}$ to obtain $\eta^*=\textit{OutputEta}(m,N,k_K=K-1,k_\epsilon) \in \dyadic^N$, to which we apply $\flipOp$ and output as $x$ (so that $x = \flipOp \eta^*$). Finally, for unconstrained lasso TV we apply the subroutine $\textit{OutputPsi}$ to obtain $\psi^*=\textit{OutputPsi}(m,N,k_K=K-1,k_\epsilon)\in \dyadic^N$, to which we apply $\flipOp$ and output the result as $x$ (so that $x = \flipOp \psi^*$). In all of the above cases we terminate execution after outputting $x$ \label{inst:WeakExitAlphaGBeta}.
		\item Alternatively, if $\dalgo < -2^{-n+1}$ then we output $x \in \dyadic^N$ with $\|x- 4 \cdot 10^{-K+1} e_2 \|_p \leq 10^{-k_\epsilon}$ for linear programming, basis pursuit with $\ell^1$ regularisation or unconstrained lasso with $\ell^1$ regularisation. For constrained lasso we output $x \in \dyadic^N$ with $\|x - 4\cdot 10^{-K+1}e_2 - (\tau - 4\cdot 10^{-K+1})e_3\|_p \leq 10^{-k_\epsilon}$. For basis pursuit TV we apply  the subroutine $\textit{OutputEta}$ to obtain $\eta^*=\textit{OutputEta}(m,N,k_K=K-1,k_\epsilon) \in \dyadic^N$, which we output as $x$. Finally, for unconstrained lasso TV we apply  the subroutine $\textit{OutputPsi}$ to obtain $\psi^*=\textit{OutputPsi}(m,N,k_K=K-1,k_\epsilon)\in \dyadic^N$, which we output as $x$. In all of the above cases we terminate execution after outputting $x$. \label{inst:WeakExitAlphaLBeta}
	\end{enumerate}
	If neither of these conditions are met then the loop continues by executing the next iteration.
\end{enumerate}

\noindent We have presented the Turing version of the subroutine. For the BSS version, all instances of $\dyadic$ are replaced by $\real$.

\begin{lemma}\label{lemma:WeakCorrect}
	Assume that the subroutine `\textit{Weak}' is applied to an input $\tilde\iota$ corresponding to $\iota$ in one of $\omlpw$, $\ombplw$, $\omullw$, $\omcllw$, $\ombptvw$ or $\omultvw$.
	Then the subroutine always terminates with an output $x$ such that $\disM(x,\tilde\Xi(\tilde\iota)) \leq 10^{-k_\epsilon}$, for the solution map $\Xi$ corresponding to the problem at hand.
\end{lemma}
\begin{proof}

	We start the proof by analysing the case where the problem is either linear programming or an $\ell^1$ regularisation problem. By the assumption that $\iota$ is in one of the ``weak'' input sets, it must be of one of the following forms: $\big(\vecYMainLP{K-1} ,\matLP(\alpha,\beta,m,N)\big)$ for linear programming, $\big(\vecYMainBP{K-1},\matl(\alpha,\beta,m,N)\big)$ for basis pursuit denoising with $\ell^1$ regularisation, $\big(\vecYMainUL{K-1} ,\matl(\alpha,\beta,m,N)\big)$ for unconstrained lasso with $\ell^1$ regularisation or $\big(\vecYMainCL{K-1},\matl(\alpha,\beta,m,N)\big)$ for constrained lasso with $\ell^1$ regularisation, where $(\alpha,\beta) \in \mathcal{L}$, $\alpha \neq \beta$. Our argument for each of the cases $\alpha>\beta$ and $\alpha<\beta$ is identical, so we will start by assuming that $\alpha < \beta$. First, we fix an $n$ and assume that the values of each of the variables accessed in the statement of the loop are given at the $n$-th iteration. In particular, $\dalgo \leq \alpha + 2^{-n} - (\beta - 2^{-n})  \leq 2^{-n+1}$ and hence the subroutine never exits at \ref{inst:WeakExitAlphaGBeta}. 
	
	By contrast, for $n$ sufficiently large $\dalgo \leq \alpha - \beta + 2^{-n+1} < -2^{-n+1}$ since $\alpha$ is strictly smaller than $\beta$ and thus eventually the subroutine does exit at \ref{inst:WeakExitAlphaLBeta}. Hence the output $x$ of the subroutine satisfies $\|x - 4\cdot 10^{-K+1} e_2\|_1 \leq 10^{-k_\epsilon}$ (or $\|x - 4\cdot 10^{-K+1}e_2- (\tau - 4\cdot 10^{-K+1})e_3\|_1\leq 10^{-k_\epsilon}$ if the problem is constrained lasso). We now use Lemma \ref{lemma:ProblemBasicExampleLP} in the case of LP, Lemma \ref{lemma:ProblemBasicExampleBPDNL1} in the case of BP with $\ell^1$ regularisation, Lemma \ref{lemma:ProblemBasicExampleULASSO} in the case of unconstrained lasso with $\ell^1$ regularisation, or Lemma \ref{lemma:ProblemBasicExampleCLASSO} in the case of constrained lasso with $\ell^1$ regularisation) to conclude that the subroutine is correct in the case $\alpha < \beta$. 
	
	The case $\alpha > \beta$ is identical except now the value $\dalgo$ will never be smaller than $-2^{-n+1}$ and instead we will eventually (after sufficiently many iterations, depending on $\iota$) have $\dalgo > 2^{-n+1}$. Thus the subroutine always outputs $x$ with $\|x-  4\cdot 10^{-K+1}e_1\|_1 \leq 10^{-k_\epsilon}$ (or $\|x - 4\cdot 10^{-K+1}e_1 - (\tau - 4\cdot 10^{-K+1})e_3\|_1 \leq 10^{-k_\epsilon}$ for constrained lasso), which for $\alpha > \beta$ is the correct output by one of Lemma \ref{lemma:ProblemBasicExampleLP}, Lemma \ref{lemma:ProblemBasicExampleBPDNL1}, Lemma \ref{lemma:ProblemBasicExampleULASSO} or Lemma \ref{lemma:ProblemBasicExampleCLASSO} depending on the computational problem being analysed.
	
	The argument that the subroutine is correct for the TV problems is similar. This time $\iota$ is either 
	\[
	\big(\vecYMainBPTV{K-1} ,\matTV(\alpha,\beta,m,N)\big)
	\]
	 for basis pursuit or $\big(\vecYMainULTV{K-1},\matTV(\alpha,\beta,m,N)\big)$ for unconstrained lasso.	Once again, if $\alpha < \beta$ the subroutine will exit at \ref{inst:WeakExitAlphaLBeta} (and never at \ref{inst:WeakExitAlphaGBeta}). This time, however,  the subroutine $\textit{OutputEta}$ for basis pursuit denoising is called to find $\eta^*$ and then outputs $x = \eta^*$  (similarly,  for unconstrained lasso,  $\textit{OutputPsi}$ is called to find $\psi^*$ and output $x=\psi^*$). Note that in this case, for basis pursuit denoising, by Lemma \ref{lemma:BPTVSolutions} we have $\Xi(\iota) = \{\eta(y_1,\alpha,\beta)\}$ (or for unconstrained lasso we use Lemma \ref{lemma:ULTVSolutions} to obtain $\Xi(\iota) = \{\psi (y_1,\alpha,\beta)\}$), where $y_1$ is the first coordinate of $y^{\mathrm{BP},\mathrm{TV},\mathrm{w}}$ or $y^{\mathrm{UL},\mathrm{TV},\mathrm{w}}$, as appropriate. Lemma \ref{lemma:OutputEtaCorrect} for basis pursuit denoising (correspondingly Lemma \ref{lemma:OutputPsiCorrect} for unconstrained lasso) as well as the inclusion $\ombplw \subset \ombpls[K-1]$ (correspondingly $\omullw \subset \omulls[K-1]$ for unconstrained lasso)  now yields $\|x - \eta(y_1,\alpha,\beta)\|_p = \|\eta^* - \eta(y_1,\alpha,\beta)\|_p \leq  \|\eta^* - \eta(y_1,\alpha,\beta)\|_1\leq 10^{-k_\epsilon}$ (correspondingly $\|x -  \psi(y_1,\alpha,\beta)\|_p \leq 10^{-k_\epsilon}$). Thus for the TV problems the subroutine returns $x$ with 
$
\disM(x,\tilde\Xi(\tilde\iota))=\disM(x,\Xi(\iota)) \leq 10^{-k_\epsilon}
$ 
in the case $\alpha < \beta$.
	
	The case $\alpha > \beta$ for the TV problems is identical to the above except now the subroutine will execute \ref{inst:WeakExitAlphaGBeta} and never \ref{inst:WeakExitAlphaLBeta}, giving $x = \flipOp\eta^*$ or $x=\flipOp\psi^*$. The same argument as before shows that $\disM(x,\tilde\Xi(\tilde\iota))=\disM(x,\Xi(\iota)) \leq 10^{-k_\epsilon}$.
\end{proof}

Finally, we need a subroutine `IdentifyStrongOrWeak' which determines whether the input $\tilde\iota$ corresponds to an $\iota$ in  one of 
\[
\text{$\omlps$, $\ombpls$, $\omulls$ , $\omclls$, $\ombptvs$ , or $\omultvs$,}
\]
 (which we call the `InputStrong' case) or an $\iota$ in of 
 \[
\text{$\omlpw$ , $\ombplw$ , $\omullw$ , $\omullw$, $\ombptvw$ , or $\omultvw$,}
 \]
 (which we call the `InputWeak' case), provided either of these two cases occurs.

{\it Subroutine} {\bf IdentifyStrongOrWeak:}

\indent Inputs: Dimensions $m$, $N$. \\
\indent Oracles: $\orvec$  providing access to the components $y_j^{(n)}$ of an input $\tilde{\iota}$.\\
\indent Output: Either `InputStrong' or `InputWeak'.

\begin{enumerate}[leftmargin=10mm, label= \arabic*.,ref = step \arabic*]
	\item \label{inst:IdentifyStrongOrWeakCases}We set two values, $t_0$ and $t_1$, depending on the problem at hand:
	\begin{enumerate}[leftmargin=8mm, label = \alph*.]
		\item For linear programming, we set $t_0 = 2\cdot 10^{-K}$ and $t_1 = 2 \cdot 10^{-K+1}$. 
		\item For basis pursuit with $\ell^1$ regularisation we set $t_0 = 2\cdot 10^{-K} + \delta$ and $t_1 = 2 \cdot 10^{-K+1} + \delta$.
		\item For unconstrained lasso with $\ell^1$ regularisation we set $t_0 = 2 \cdot 10^{-K} + \lambda$ and $t_1 = 2\cdot 10^{-K+1} + \lambda$.
		\item For constrained lasso we set $t_0 = 10^{-K+1}$ and $t_1 = 10^{-K+2}$.
		\item For basis pursuit with TV we compute $\left[(m-1)/m \right]^{-1/2}$ to $4K$ bits of precision, yielding some value $w$. We then set 
		\[
		t_0 = \delta w + \left(7 - \frac{3}{m}\right)\frac{10^{-K}}{4}, \quad t_1 = \delta  w + \left(7 - \frac{3}{m-1}\right)\frac{10^{-K+1}}{4}. 
		\]
		\item For unconstrained lasso with TV we set $t_0 = 7 \cdot 10^{-K}/4 + \lambda (1+1/(m-1))$ and $t_1 = 3 \cdot 10^{-K+1}/2 + \lambda (1+1/(m-1))$.
	\end{enumerate} 
	\item Using the oracle $\orvec$, we read  $y'_1:=y_1^{(4K)}$. If $y'_1 \leq t_0 + 2 \cdot 2^{-4K}$ we output `InputStrong' and terminate the subroutine. Else if $y'_1 \geq t_1 - 2\cdot 2^{-4K}$ we output `InputWeak'. The subroutine then terminates. \label{inst:IdentifyStrongOrWeakY}
\end{enumerate}
\begin{lemma}\label{lemma:IdentifyStrongOrWeakCorrect}
	The subroutine `IdentifyStrongOrWeak' correctly identifies the cases `InputStrong' and `InputWeak' as described above. Moreover,  the number of digits needed by the oracles as well as  the BSS  and the Turing runtime 
 are all bounded above by some polynomial in $\log(m)$. 
 \end{lemma}
\begin{proof}
	Suppose first that $\iota=(y,\cdot)$ is in one of  $\omlps$, $\ombpls$, $\omulls$ , $\omclls$, $\ombptvs$ , or $\omultvs$. For linear programming, basis pursuit with $\ell^1$ regularisation, unconstrained lasso with $\ell^1$ regularisation and constrained lasso we have $y_1 = t_0$ and so $y_1 \leq t_0 + 2^{-4K}$.
	For basis pursuit TV, we have 
	\begin{align*}
		y_1  &\leq \frac{\delta m }{\theta(1/2,1/2,m)} + \left[7 - \frac{3}{m}\right]\frac{10^{-K}}{4}\leq \delta  w + \left|\delta\left(w - \frac{m}{\theta(1/2,1/2,m)}\right)\right|  + \left[7 - \frac{3}{m}\right]\frac{10^{-K}}{4} \\
		&\leq \delta w + \delta\cdot  2^{-4K}  + \left[7 - \frac{3}{m}\right]\frac{10^{-K}}{4} \leq t_0 + 2^{-4K} 
	\end{align*}
	where the first inequality holds because $\theta(1/2,1/2,m) \leq m$ and the final inequality because $\delta \leq 1$. For unconstrained lasso TV 
	\begin{equation*}
		y_1  = \frac{7 \cdot 10^{-K}}{4 } + \lambda \left[1 + \frac{1}{m-1}\right] - \frac{3 \cdot 10^{-K}}{4(m-1)} \leq  \frac{7 \cdot 10^{-K}}{4 } + \lambda \left[1 + \frac{1}{m-1}\right] = t_0 \leq t_0 + 2^{-4K}.
	\end{equation*}
	Thus in each of the cases $y_1' \leq y_1 + 2^{-4K} \leq t_0 + 2\cdot 2^{-4K}$ and so the subroutine outputs `InputStrong'.
	
	We now consider the case where $\iota=(y,\cdot)$ is in one of $\omlpw$ , $\ombplw$ , $\omullw$ , $\omullw$, $\ombptvw$ , or $\omultvw$. We first show that $y_1 \geq t_1 - 2^{-4K} $ and hence $y'_1 \geq t_1- 2\cdot 2^{-4K} $. For linear programming, basis pursuit with $\ell^1$ regularisation, unconstrained lasso (with both $\ell^1$ and TV regularisation) and constrained lasso we have $y_1 = t_1$ and thus the claim is true. For basis pursuit with TV, we have
	\begin{equation*}
		y_1 \geq \frac{\delta m }{\theta(1/2,1/2,m)} + \left[7 - \frac{3}{m-1}\right] \frac{10^{-K+1}}{4} \geq \delta w  - \delta \cdot 2^{-4K} + \left[7 - \frac{3}{m-1}\right] \frac{10^{-K+1}}{4} \geq t_1 - 2^{-4K} 
	\end{equation*}
	since $\delta \leq 1$, as claimed. Finally, for unconstrained lasso we have	\begin{equation*}
		y_1 = \frac{7 \cdot 10^{-K+1}}{4} + \lambda \left[1 + \frac{1}{m-1}\right]  - \frac{3 \cdot 10^{-K+1}}{4(m-1)} \geq \frac{6 \cdot 10^{-K+1}}{4} + \lambda \left[1 + \frac{1}{m-1}\right] = t_1
	\end{equation*}
	since $m \geq 4$ and so the claim follows. Therefore the subroutine outputs `InputWeak' provided we can also show that $y'_1 > t_0 + 2\cdot 2^{-4K}$ (i.e. the subroutine does not branch and output `InputStrong').

	We thus will aim to show that $y'_1 > t_0 + 2\cdot2^{-4K}$ for each of the computational problems. For linear programming we have $y_1' \geq 2\cdot 10^{-K+1} - 2^{-4K} >2\cdot  10^{-K} +  2\cdot2^{-4K} = t_0 + 2\cdot2^{-4K}$. The same argument for basis pursuit with $\ell^1$ regularisation or unconstrained lasso with $\ell^1$ regularisation gives $y_1' > t_0 + 2\cdot 2^{-4K}$. Similarly, for constrained lasso we get $y'_1 \geq 	10^{-K+2} - 2^{-4K} > 10^{-K+1} + 9\cdot 10^{-K+1} - 5\cdot 10^{-K}/4 > t_0 + 2\cdot 2^{-4K}$. For basis pursuit with TV, starting from $y_1 \geq t_1 - 2^{-4K}$ we get
	\begin{equation*}
		y_1\geq t_0 - 2^{-4K} +  \left[7 - \frac{3}{m-1}\right] \frac{10^{-K+1}}{4} - \left[7 - \frac{3}{m}\right] \frac{10^{-K}}{4} =  t_0 - 2^{-4K} +  \left[63 - \frac{30}{m-1}  +\frac{3}{m}\right] \frac{10^{-K}}{4}
	\end{equation*}
	and since $-30/(m-1) + 3/m$ is increasing in $m$, for $m \in [4,\infty)$ (as can be seen by analysing the derivative), we obtain 
	$y_1\geq t_0 - 2^{-4K} +  \left(63 - 30/3 + 3/4\right) \frac{10^{-K}}{4} > t_0 - 2^{-4K} + 5 \cdot 10^{-K} > t_0 + 3\cdot 2^{-4K}$ and thus $y_1' \geq y_1 - 2^{-4K} > t_0 + 2\cdot 2^{-4K}$. Finally for unconstrained lasso with TV we start from $y_1 \geq t_1$ and argue 
	\begin{equation*}
		y_1 \geq t_1 = \frac{30 \cdot 10^{-K}}{2} + \lambda \left[1 + \frac{1}{m-1}\right] = \frac{7 \cdot 10^{-K}}{4} + \lambda \left[1 + \frac{1}{m-1}\right] + \frac{23 \cdot 10^{-K}}{2} \geq t_0 + 4\cdot 2^{-4K},
	\end{equation*}
	so that $y_1' \geq y_1 - 2^{-4K} > t_0 + 2\cdot 2^{-4K}$.
	
	Next, we note that the only call to the oracle $\orvec$ requires $4K$ digits, which is constant, as $K$ is assumed to be fixed.
	It remains to estimate the complexity. Note that it suffices to show that the Turing runtime is polynomial in $\log(m)$, as this will then imply the desired polynomial bound on the BSS runtime. 
		
	We do so by separately considering the different possible problems.

	\begin{enumerate}[leftmargin=10mm, label = \alph*.]
		\item For linear programming,  basis pursuit, unconstrained lasso with $\ell^1$ and constrained lasso,  the runtime of  computing the numerator and the denominator of $t_0$ and $t_1$ is $\oh(1)$ since $K$, $\delta$, $\lambda$, and  $\tau$ are all assumed to be fixed.
		\item For basis pursuit with TV, the computation of $w$ can be done as in Lemma \ref{lemma:OutputEtaCorrect} - we use Newton-Raphson iteration to compute $w$ as in \cite[Page 92-93]{ElementaryFunctionsImplementation}. The complexity of this operation, done to $4K$ bits of precision, is polynomial in $\length[m/(m-1)]$ and $4K$. Since $\length[m/(m-1)]=\oh(\log(m))$ and  $K$ is fixed, the overall complexity of computing $w$ is polynomial in $\log(m)$. Note that such a $w$ also has $\length(w)$ bounded above by a polynomial in $\log(m)$.
		The computation of $t_0$ and $t_1$ is then done by finitely many arithmetic operations on fractions each of lengths bounded above by a polynomial in $\length(\delta)$, $\length(w)$, $\length(m)$ and $\length(10^{-K+1})$. Since $\delta$ and $K$ are assumed to be fixed across all inputs and $\length(w)$ is bounded above by a polynomial in $\log(m)$, the overall complexity of this step is bounded above by a polynomial in $\log(m)$.
		\item For unconstrained lasso with TV the complexity of computing $t_0$ and $t_1$ is bounded above by a polynomial in $\length(\lambda)$, $\length(m)$, and $K$. Since $\lambda$ and $K$ are fixed, the complexity of this step is bounded above by a polynomial in $\log(m)$.
	\end{enumerate}

	In all the cases above both $t_0$ and $t_1$ are computed in Turing runtime bounded above by a polynomial in $\log(m)$, so their length must also be bounded above by a polynomial in $\log(m)$. Furthermore, by Lemma \ref{lemma:mainThmSizeEstimates}, we must have $\length(y_1') \leq \length(y_1) + 4K= \oh(1)$. Hence the comparisons in \ref{inst:IdentifyStrongOrWeakY} can be done in complexity polynomial in $\log(m)$. The final output of `InputWeak' or `InputStrong' can be done as an $\oh(1)$ boolean assignment.
	
	We conclude that in the bit complexity model the subroutine `$\textit{IdentifyStrongOrWeak}$' takes at most some polynomial in $\log(m)$ bit operations. Estimating the arithmetic complexity is simpler - there are finitely many arithmetic operations done, except in \ref{inst:IdentifyStrongOrWeakCases} where the number of Newton-Raphson iterates required can be bound by a polynomial in $\log(m)$. Each Newton-Raphson iteration takes finitely many arithmetic operations and hence the overall arithmetic complexity is bounded by a polynomial in $\log(m)$.

\end{proof}

\section{Proof of Theorem \ref{Cor:main}: parts (i) and (ii)\label{sec:cor_main(i)(ii)}}

Parts (i) and (ii) of Theorem \ref{Cor:main} are formally stated in Proposition \ref{Cor:main_SCI}, which we now prove with the set $\Omega$ being one of \eqref{eq:defs-5-Om} depending on the problem. Before proceeding to the breakdown epsilon bounds and algorithm constructions, we note that Lemmas \ref{lem:cond-Xi-various}, \ref{lem:cond-FP-various}, and \ref{lem:cond-mat-various} guarantee the desired bounds on the condition numbers and Lemma \ref{lemma:mainThmSizeEstimates} establishes the upper and lower bounds on the size of the inputs.

\subsection{Proof of Proposition  \ref{Cor:main_SCI}\label{sec:lambdaHatStrong-Cor:main_SCI}}
We consider the fixed-dimensional computational problems $\{\Xi,\Omega_{m,N},\allowbreak\mathcal{M}_N,\Lambda_{m,N} \}$,  where $\Omega_{m,N}$ is one of \eqref{eq:defs-5-Om-fixedDim} with $k=K$. Writing $\Omega_{m,N}=\Omega_{m,N}^{\mathrm{s}}\cup \Omega_{m,N}^{\mathrm{w}}$, for the corresponding ``strong'' and ``weak'' components as defined in \eqref{eq:strongOmega-fixed-dim} and \eqref{eq:weakOmega-fixed-dim}, Lemma \ref{lemma:StrongBDEpsilonBounds10K} establishes the existence of a 
\[
\hat\Lambda^{\mathrm{s}}=\{f_{j}^{\mathrm{s}}\,\vert\, j\leq n_{\mathrm{var}} ,n\in\mathbb{N}\}\in\mathcal{L}^1(\Lambda_{m,N})
\] 
such that,  for  the computational problem $\{\Xi,\Omega_{m,N}^{\mathrm{s}},\mathcal{M}_N,\hat\Lambda^{\mathrm{s}}\}$, we have $\epsilon_{\mathbb{P}h\mathrm{B}}^\mathrm{s}(\mathrm{p}) > 10^{-K}$, for $\mathrm{p} \in [0,1/2)$, as well as $\epsilon_{\mathbb{P}\mathrm{B}}^\mathrm{s}(\mathrm{p}) > 10^{-K}$, for $\mathrm{p} \in [0,1/3)$. 
Similarly, Lemma \ref{lemma:WeakBDEpsilonBounds10K} establishes the existence of a 
\[
\hat\Lambda^{\mathrm{w}}=\{f_{j,n}^{\mathrm{w}}\,\vert\, j\leq n_{\mathrm{var}},n\in\mathbb{N}\}\in\mathcal{L}^1(\Lambda_{m,N})
\]
 such that,  for the computational problem $\{\Xi,\Omega_{m,N}^{\mathrm{w}},\mathcal{M}_N,\hat\Lambda^{\mathrm{w}}\}$, we have $\epsilon^{\mathrm{w}}_{\mathbb{P}\mathrm{B}}(\mathrm{p}) > 10^{-(K-1)}$ for $\mathrm{p} \in [0,1/2)$ and $\epsilon^{\mathrm{w}}_{\mathrm{B}}> 10^{-(K-1)}$. 
Now, defining $\hat\Lambda_{m,N}:=\{f_{j,n}\,\vert\, j\leq n_{\mathrm{var}},n\in\mathbb{N}\}$, where we let $f_{j,n}(\iota)=f_{j,n}^{\mathrm{s}}(\iota)$ if $\iota\in \Omega_{m,N}^{\mathrm{s}}$ and $f_{j,n}(\iota)=f_{j,n}^{\mathrm{w}}(\iota)$ if $\iota\in \Omega_{m,N}^{\mathrm{w}}$, for $j\leq n_{\mathrm{var}}$ and $n\in\mathbb{N}$, we have that $\hat\Lambda_{m,N}$ provides $\Delta_1$-information for $\{\Xi,\Omega_{m,N},\mathcal{M}_N,\Lambda_{m,N} \}$, and in view of Remark \ref{remark:OmegaSubsetBDE}, we have that all the breakdown epsilon bounds mentioned above also hold for $\{\Xi,\Omega_{m,N},\mathcal{M}_N,\hat\Lambda_{m,N}\}$. This already establishes part (i) of Proposition \ref{Cor:main_SCI} as well as the breakdown epsilon bound in part (ii) (and, indeed, the breakdown epsilon bounds in Proposition \ref{Cor:main_SCI_cont}, which will be useful later).

In order to complete the proof of  part (ii) of Proposition \ref{Cor:main_SCI}, it remains to show the existence of a recursive (i.e., implementable on a Turing machine) algorithm (which we will call {\it Randomised $K$ digit algorithm}) that returns $K$ correct digits with probability greater than or equal to $2/3$ on all inputs of the problems $\{\tilde\Xi,\tilde\Omega_{m,N},\mathcal{M}_N,\tilde\Lambda_{m,N}\}=\{\Xi,\Omega_{m,N},\mathcal{M}_N,\Lambda_{m,N} \}^{\Delta_1}$ for varying $m$ and $N$, where $\Omega_{m,N}$ is one of \eqref{eq:defs-5-Om-fixedDim} and $\Xi$ is the corresponding solution map. As in the previous section, we fix the notation for an element $\tilde\iota$ of $\tilde\Omega_{m,N}$. For the linear programming case, we write  $\tilde{\iota}=\big(\{y_j^{(n)}\}_{n=0}^{\infty}, \{\matLP_{j,k}^{(n)}\}_{n=0}^{\infty}\big)_{j,k}$, corresponding to an $\iota=(y,\matLP)\in \Omega_{m,N}$. For the $\ell^1$ and TV problems the notation is analogous, except that we respectively write $\matl$ and $\matTV$ instead of $\matLP$.

To construct this algorithm, we  need a randomised subroutine, which we call \textit{BiasedCoinFlip}, that takes a natural number $n$ and returns `true' with probability $1/n$ and `false' with probability $1-1/n$. Note that this subroutine halts with probability $1$, and, for each execution of the subroutine, the probabilities of returning `true', respectively `false', are assumed independent of previous executions of the subroutine.  Such a subroutine can easily be constructed using a randomised Turing machine with access to coin flips that return true with probability $1/2$ and false with probability $1/2$.

To construct the desired {\it Randomised $K$ digit algorithm}, we first need to design a subroutine \textit{Guess} that randomly chooses between $K$ digit approximations to two plausible solutions, each with probability $1/2$. Concretely, we define:

{\it Subroutine} {\bf Guess:} 

\indent Inputs: Dimensions $m$, $N$.\\
\indent Oracles: $\orvec$ and $\ormat$ providing access to the components $y_j^{(n)}$ and $\matLP_{j,k}^{(n)}$ (respectively  $\matl_{j,k}^{(n)}$ or $\matTV_{j,k}^{(n)}$) of an input $\tilde{\iota}$.\\
\indent Output: A potential solution vector $x \in \dyadic^N$ (in the Turing case) or $x \in \real^N$ (in the BSS case).
\begin{enumerate}[leftmargin=10mm, label= \arabic*.,ref = step \arabic*]
	\item First, we make a  random coin flip that returns `true' with probability $1/2$ and `false' with probability $1/2$. 
	\begin{enumerate}[label = \alph*.]
		\item If the coin flip outputted `true' then the output of \textit{Guess} depends on the problem at hand: for linear programming, basis pursuit with $\ell^1$ regularisation or unconstrained lasso with $\ell^1$ regularisation we output $x\in\dyadic^N$ with $\| x- 4 \cdot 10^{-K} e_1\|_p\leq 10^{-K}$, for constrained lasso we output $x\in\dyadic^N$ with $\|x - 4\cdot 10^{-K}e_1 + (\tau - 4\cdot 10^{-K})e_3\|_p\leq 10^{-K}$, for basis pursuit with TV regularisation we output the result of $\textit{OutputEta}^{\orvec,\ormat}(m,N,k_0=K, k_\epsilon=K)$ and for unconstrained lasso with TV regularisation we output the result of  $\textit{OutputPsi}^{\orvec,\ormat}(m,N,k_0=K,k_\epsilon=K)$. 
		\item If the coin flip outputted `false' then our output depends on the problem at hand: for linear programming, basis pursuit with $\ell^1$ regularisation or unconstrained lasso with $\ell^1$ regularisation we output $x\in\dyadic^N$ with $\|x - 4 \cdot 10^{-K} e_2\|_p\leq 10^{-K}$, for constrained lasso we output $x\in\dyadic^N$ with $\|x - 4\cdot 10^{-K}e_2 + (\tau - 4\cdot 10^{-K})e_3\|_p\leq 10^{-K}$, for basis pursuit with TV regularisation we output the result of  $\flipOp \textit{OutputEta}^{\orvec,\ormat}(m,N,k_0=K, k_\epsilon=K)$ and for unconstrained lasso with TV regularisation we output the result of  $\flipOp \textit{OutputPsi}^{\orvec,\ormat}(m,N,k_0=K, k_\epsilon=K)$. 
	\end{enumerate}
\end{enumerate}

With the subroutine \textit{Guess} and the subroutines constructed in \S\ref{sec:UsefulSubroutinesForMainThm}, we are ready to specify the desired {\it Randomised $K$ digit algorithm}. In accordance to the claim of part (ii) of Proposition \ref{Cor:main_SCI}, this algorithm does not necessarily halt.

{\bf Randomised $K$ digit algorithm}

\indent Inputs: Dimensions  $m$, $N$.\\
\indent Oracles: $\orvec$ and $\ormat$ providing access to the components $y_j^{(n)}$ and $\matLP_{j,k}^{(n)}$ (respectively  $\matl_{j,k}^{(n)}$ or $\matTV_{j,k}^{(n)}$) of an input $\tilde{\iota}$.\\
\indent Output: With probability at least $2/3$, some vector $x \in \real^N$ with $\disM(x,\tilde\Xi(\tilde\iota))\leq 10^{-K}$.
\begin{enumerate}[leftmargin=10mm, label= \arabic*.,ref = step \arabic*]
	\item \label{inst:StronOrWeak} We execute $\textit{IdentifyStrongOrWeak}^{\orvec,\ormat}(m,N)$. If this evaluates to `InputWeak', we execute the subroutine $\textit{Weak}^{\orvec,\ormat}(m,N,k_\epsilon=K)$  and terminate. Otherwise, we continue to \ref{inst:NonHaltRandom}.
	\item \label{inst:NonHaltRandom} Initialise $n=1$ and execute the following loop: First, execute $\textit{BiasedCoinFlip}(2^{n-1}+2)$. If this subroutine returns `true' then we execute $\textit{Guess}^{\orvec,\ormat}(m,N)$ and terminate. If instead \textit{BiasedCoinFlip} returns `false' then we increment $n$. Next, in the linear programming case we use the oracle $\ormat$ to read $\matLP_{1,1}^{(n)}$ and $\matLP_{1,2}^{(n)}$ and set $d =\matLP_{1,1}^{(n)} - \matLP_{1,2}^{(n)}$. For the $\ell^1$ problems we read $\matl_{1,1}^{(n)}$ and $\matl_{1,2}^{(n)}$ and set  $d =\matl_{1,1}^{(n)} - \matl_{1,2}^{(n)}$, and for the TV problems we read $\matTV_{1,1}^{(n)}$ and $\matTV_{1,N}^{(n)}$ and set $d =\matTV_{1,1}^{(n)} - \matTV_{1,N}^{(n)}$. We then branch depending on the value of $d$:
	\begin{enumerate}[label = \alph*., ref = step 2\alph*]
		\item If $d > 2 \cdot 2^{-n}$ then we choose $x\in\dyadic^N$ with $\|x - 4\cdot 10^{-K} e_1\|_p\leq 10^{-K}$ for linear programming, basis pursuit with $\ell^1$ regularisation or unconstrained lasso with $\ell^1$ regularisation. For constrained lasso we choose $x\in\dyadic^{-K}$ $\|x - 4\cdot 10^{-K}e_1 + (\tau - 4\cdot 10^{-K})e_3\|_p\leq 10^{-K}$. For basis pursuit TV, we set $x$ to be the result of $ \flipOp\textit{OutputEta}^{\orvec,\ormat}(m,N,k_0=K, k_\epsilon=K)$. Finally, for unconstrained lasso with TV regularisation we set $x$ to be the result of $\flipOp \textit{OutputPsi}^{\orvec,\ormat}(m,N,k_0=K, k_\epsilon=K)$. In all cases, we output $x$ and then terminate the procedure.  \label{inst:ExitAlphaGBetaRandom}
		\item Alternatively, if $d < -2 \cdot 2^{-n}$ then we choose $x\in\dyadic^N$ with $\|x - 4\cdot 10^{-K} e_2 \|_p\leq 10^{-K}$ for linear programming, basis pursuit with $\ell^1$ regularisation or unconstrained lasso with $\ell^1$ regularisation. For constrained lasso we choose $x\in\dyadic^{N}$ with $\|x - 4\cdot 10^{-K}e_2 + (\tau - 4\cdot 10^{-K})e_3\|_p\leq 10^{-K}$. For basis pursuit TV, we set $x$ to be the result of the subroutine $ \textit{OutputEta}^{\orvec,\ormat}(m,N,k_0=K, k_\epsilon=K)$. Finally, for unconstrained lasso with TV regularisation we set $x$ to be the result of $ \textit{OutputPsi}^{\orvec,\ormat}(m,N,k_0=K, k_\epsilon=K)$. In all cases, we output $x$ and then terminate the procedure. \label{inst:ExitAlphaLBetaRandom}
	\end{enumerate}
	If neither of these conditions are met then the loop continues by incrementing $n$ and executing the next iteration.
\end{enumerate}

Firstly, it is clear that this algorithm can be executed on a randomised Turing machine, and hence can be executed on a randomised BSS machine. Secondly, we need to prove that this algorithm does indeed achieve what is stated in its preamble, i.e., we need to show that, for each given input, the algorithm terminates with probability greater than or equal to $2/3$ with a correct output, i.e., a vector $x$ at most $10^{-K}$ away from a true solution. To this end, we let $\iota$ be the element of $\Omega$ that $\tilde\iota$ corresponds to, and consider the following four cases separately:

{\bf Case 1 :} For this case, $\iota$ is of the following form: $\big(y^{\mathrm{LP},\mathrm{s}} ,\matLP(\alpha,1/2,m,N)\big)$ for linear programming, $\big(y^{\mathrm{BP},\ell^1,\mathrm{s}} ,\matl(\alpha,1/2,m,N)\big)$ for basis pursuit denoising with $\ell^1$ regularisation, $\big(y^{\mathrm{UL},\ell^1,\mathrm{s}},\matl(\alpha,1/2,m,\allowbreak N)\big)$ for unconstrained lasso, $\big(y^{\mathrm{CL},\mathrm{s}},\matl(\alpha,1/2,m,N)\big)$ for constrained lasso, $\big(y^{\mathrm{BP},\mathrm{TV},\mathrm{s}},\matTV(\alpha,1/2,m,\allowbreak N)\big)$ for basis pursuit denoising with TV regularisation and $\big(y^{\mathrm{UL},\mathrm{TV},\mathrm{s}},\matTV(\alpha,1/2,m,N)\big)$ for unconstrained lasso with TV regularisation, where $1/4 < \alpha < 1/2$ for the $\ell^1$ problems, $r_n < \alpha < 1/2$ for basis pursuit with TV, and $s_n < \alpha < 1/2$ for unconstrained lasso with TV. 

In this case, Lemma \ref{lemma:ProblemBasicExampleLP} (respectively Lemma \ref{lemma:ProblemBasicExampleBPDNL1} or Lemma \ref{lemma:ProblemBasicExampleULASSO}) show that the solution is $4\cdot 10^{-K} e_2$ for linear programming (respectively, basis pursuit with $\ell^1$ regularisation or unconstrained lasso with $\ell^1$ regularisation). For constrained lasso, applying Lemma \ref{lemma:ProblemBasicExampleCLASSO} shows that the solution is $4\cdot 10^{-K}e_2 + (\tau - 4\cdot 10^{-K}e_2)$. Similarly, Lemma \ref{lemma:BPTVSolutions} shows that the solution is $ \eta(y_1,\alpha,1/2)$ for basis pursuit with TV regularisation and Lemma \ref{lemma:ULTVSolutions} shows that the solution is $ \psi(y_1,\alpha,1/2)$  for unconstrained lasso with TV regularisation. Therefore, for linear programming, basis pursuit with $\ell^1$ regularisation or unconstrained lasso with $\ell^1$ regularisation the algorithm is correct whenever it outputs $x\in\dyadic^N$ within $10^{-K}$ of $4\cdot 10^{-K}e_2$, and is likewise correct for constrained lasso whenever it outputs $x\in\dyadic^N$ within $10^{-K}$ of $4\cdot 10^{-K}e_2 + (\tau - 4\cdot 10^{-K}e_2)$. The correctness of \textit{OutputEta} (Lemma \ref{lemma:OutputEtaCorrect}) thus implies that the algorithm is correct for basis pursuit TV whenever the output is $ \textit{OutputEta}^{\orvec,\ormat}(m,N,k_0=K, k_\epsilon=K)$ and similarly the correctness of \textit{OutputPsi} (Lemma \ref{lemma:OutputPsiCorrect}) implies that the algorithm is correct for unconstrained lasso TV whenever the output is $ \textit{OutputPsi}^{\orvec,\ormat}(m,N,k_0=K, k_\epsilon=K)$.

By Lemma \ref{lemma:IdentifyStrongOrWeakCorrect}, the subroutine  \textit{IdentifyStrongOrWeak}  in \ref{inst:StronOrWeak} of the algorithm will evaluate to `InputStrong', and hence the algorithm will proceed with the loop in \ref{inst:NonHaltRandom}. We therefore proceed with an analysis of this loop.

Next, let $F_n$ be the event that the subroutine \textit{Guess} is executed in the $n$-th iteration of the loop and note that, on the event $\left(\bigcup_{r=1}^{n} F_r\right)^c$, the value of $d$ after $n$ iterations satisfies $d \leq \alpha + 2^{-n} - (1/2 - 2^{-n}) < 2 \cdot 2^{-n}$ (since $\alpha <1/2$) and hence \ref{inst:ExitAlphaGBetaRandom} never results in the termination of the algorithm. By contrast, $d \leq \alpha + 2^{-n} - 1/2 +2^{-n}$ and since $\alpha$ is independent of $n$ and strictly smaller than $1/2$ this expression will be smaller than $-2 \cdot 2^{-n}$ for sufficiently large $n$. Now, define
\begin{equation*}
n_0=\inf\{n\in\mathbb{N}\, \vert \, F_n \text{ does not occur and } d < -2 \cdot 2^{-n}\}.
\end{equation*}
Note that the value of $n_0$ depends on $\tilde\iota$ through $d$.

Then $\probab(F_{n}) = 0$, for $n > n_0$, whereas for $n \leq n_0$, $\probab(F_n)$ is equal to the probability that the algorithm has executed $n-1$ iterations of the loop without terminating and the execution of $\textit{BiasedCoinFlip}(2^{n-1}+2)$ returns `true'. 
Note that these two events are independent since the result of $\textit{BiasedCoinFlip}(2^{n-1}+2)$ is independent of all prior calls to the subroutine $\textit{BiasedCoinFlip}$. Thus 
\[
\probab(F_n) = \left[1-\probab\big(\bigcup_{r=1}^{n-1} F_r\big)\right] (2^{n-1}+2)^{-1}, \qquad n \leq n_0
\] 
and, since the events $F_r$, for $r=1,\dots,n-1$, are disjoint, we obtain the recurrence 
\[
\probab(F_n) = \left[1 - \sum_{r=1}^{n-1} \probab(F_r)\right](2^{n-1}+2)^{-1}, \qquad n \leq n_0
\]
 and $\probab(F_n) = 0$ for $n > n_0$. Using strong induction one can show that this implies that 
 \[
 \probab(F_n) = 3^{-1} \cdot 2^{-(n-1)}, \,\, n \leq n_0 \, \text{  and  }\,  \probab(F_n) = 0, \,\, n > n_0.
 \]

We have argued that $n_0$ is finite -- this implies that the algorithm halts with probability $1$. Because the algorithm is correct if it outputs an $x$ within $10^{-K}$ of $4 \cdot 10^{-K}e_2$ for LP and $\ell^1$ problems or $4\cdot 10^{-K}e_2 + (\tau - 4\cdot 10^{-K})e_3$ for the constrained lasso case, as well as $ \textit{OutputEta}^{\orvec,\ormat}(m,N,k_0=K, k_\epsilon=K)$ in the basis pursuit with TV case or $  \textit{OutputPsi}^{\orvec,\ormat}(m,N,k_0=K, k_\epsilon=K)$ for the unconstrained lasso with TV case (and we have already argued that \textit{\textit{Weak}} is never executed) we conclude that the only possible incorrect outputs are an $x$ within $10^{-K}$ of $4 \cdot 10^{-K}e_1$ (in the linear programming or $\ell^1$ cases) or $4\cdot 10^{-K}e_1 + (\tau - 4\cdot 10^{-K})e_3$ (in the constrained lasso case), or 
\[
x =\flipOp \textit{OutputEta}^{\orvec,\ormat}(m,N,k_0=K, k_\epsilon=K)
\] (in the basis pursuit TV case), respectively $x = \flipOp \textit{OutputPsi}^{\orvec,\ormat}(m,N,k_0=K, k_\epsilon=K)$ (in the unconstrained lasso TV case). Each of these can occur only if the subroutine \textit{Guess} returns `true'. Since \textit{Guess} returns `true' with probability $1/2$, an incorrect output occurs with probability $\sum_{n=1}^{\infty} \probab(F_n)/2 \leq  \sum_{n=1}^{\infty} 3^{-1} \cdot 2^{-n} = 1/3$. Thus, with probability at least $1-1/3=2/3$, the algorithm produces a correct output.

{\bf Case 2:} For this case, $\iota$ is of the following form: $\big(y^{\mathrm{LP},\mathrm{s}} ,\matLP(1/2,1/2,m,N)\big)$ for linear programming, $\big(y^{\mathrm{BP},\ell^1,\mathrm{s}} ,\matl(1/2,1/2,m,N)\big)$ for basis pursuit denoising and $\big(y^{\mathrm{UL},\ell^1,\mathrm{s}},\matl(1/2,1/2,m,N)\big)$ for unconstrained lasso,  $\big(y^{\mathrm{CL},\mathrm{s}},\matl(1/2,1/2,m,N)\big)$ for constrained lasso, $\big(y^{\mathrm{BP},\mathrm{TV},\mathrm{s}},\matTV(1/2,1/2,m,N)\big)$ for basis pursuit denoising with TV regularisation and $\big(y^{\mathrm{UL},\mathrm{TV},\mathrm{s}},\matTV(1/2,1/2,m,N)\big)$ for unconstrained lasso with TV regularisation. 

As in the proof of Case 1, and again by Lemma \ref{lemma:IdentifyStrongOrWeakCorrect}, the algorithm proceeds to execute the loop in \ref{inst:NonHaltRandom}. This time, however, the algorithm never terminates at either \ref{inst:ExitAlphaGBetaRandom} or \ref{inst:ExitAlphaLBetaRandom} -- indeed, in the $n$-th iteration we have $d \leq 1/2 + 2^{-n} - (1/2 - 2^{-n}) = 2 \cdot 2^{-n}$ and similarly $d \geq 1/2 - 2^{-n} -(1/2 + 2^{-n}) = - 2 \cdot 2^{-n}$. 
Thus the algorithm only terminates in the $n$-th iteration if the subroutine \textit{Guess} is executed. As in the proof for Case 1, we let $F_n$ be the event that the subroutine \textit{Guess} is executed in the $n$-th iteration. The same argument as before shows that $\probab(F_n) = 3^{-1}\cdot 2^{-(n-1)}$ (and this occurs for all $n$ since \ref{inst:ExitAlphaGBetaRandom} and \ref{inst:ExitAlphaLBetaRandom} never result in the termination of the algorithm). Hence the probability the algorithm terminates with the execution of the subroutine \textit{Guess} is given by 
\[
\probab\big(\bigcup_{r=1}^{\infty}F_r\big) = \sum_{r=1}^{\infty} \probab(F_r) = \sum_{r=1}^{\infty} 3^{-1}\cdot 2^{-n+1} = 2/3.
\] 

Note that, whenever \textit{Guess} is executed by the algorithm, then, for each of linear programming, basis pursuit with $\ell^1$ regularisation and unconstrained lasso with $\ell^1$ regularisation, the output of the algorithm is an $x\in\dyadic^N$ within $10^{-K}$ of either $ 4 \cdot 10^{-K}e_1$ or $ 4\cdot 10^{-K} e_2$, which, by Lemma \ref{lemma:ProblemBasicExampleLP}, Lemma \ref{lemma:ProblemBasicExampleBPDNL1}, or Lemma \ref{lemma:ProblemBasicExampleULASSO}, satisfies 
$
\disM(x,\tilde\Xi(\tilde\iota)) = \disM(x,\Xi(\iota)) \leq 10^{-K}.
$
 Similarly, for constrained lasso the output of the algorithm is an $x\in\dyadic^{-K}$ within $10^{-K}$ of either $4\cdot 10^{-K}e_1 + (\tau - 4\cdot 10^{-K})e_3$ or $4\cdot 10^{-K}e_2 + (\tau - 4\cdot 10^{-K})e_3$ and this time Lemma \ref{lemma:ProblemBasicExampleCLASSO} implies that $\disM(x,\Xi(\iota)) \leq 10^{-K}$.
The situation is only marginally more complicated for basis pursuit with TV regularisation -- indeed, Lemma \ref{lemma:BPTVSolutions} shows that $\eta(y_1,1/2,1/2),\flipOp \eta(y_1,1/2,1/2) \in \Xi(\iota)$ and thus, writing $x$ for the output of \textit{Guess}, we have by Lemma \ref{lemma:OutputEtaCorrect} that $\disM(x,\tilde\Xi(\tilde\iota)) \leq 10^{-K}$. Finally, for unconstrained lasso with TV regularisation Lemma \ref{lemma:ULTVSolutions} shows that 
$
\psi(y_1,1/2,1/2),\flipOp \psi(y_1,1/2,1/2) \in \Xi(\iota)
$
 and thus, writing $x$ for the output of \textit{Guess}, Lemma \ref{lemma:OutputPsiCorrect} implies that $\disM(x,\tilde\Xi(\tilde\iota)) \leq 10^{-K}$.
We conclude that, with probability $2/3$, the algorithm produces a correct output.

{\bf Case 3:} For this case, $\iota$ is of the following form: $\big(y^{\mathrm{LP},\mathrm{s}} ,\matLP(1/2,\beta,m,N)\big)$ for linear programming, $\big(y^{\mathrm{BP},\ell^1,\mathrm{s}} ,\matl(1/2,\beta,m,N)\big)$ for basis pursuit denoising with $\ell^1$ regularisation, $\big(y^{\mathrm{UL},\ell^1,\mathrm{s}},\matl(1/2,\beta,m,\allowbreak N)\big)$ for unconstrained lasso, $\big(y^{\mathrm{CL},\mathrm{s}},\matl(1/2,\beta,m,N)\big)$ for constrained lasso, $\big(y^{\mathrm{BP},\mathrm{TV},\mathrm{s}},\matTV(1/2,\beta,m,\allowbreak N)\big)$ for basis pursuit denoising with TV regularisation and $\big(y^{\mathrm{UL},\mathrm{TV},\mathrm{s}},\matTV(1/2,\beta,m,N)\big)$ for unconstrained lasso with TV regularisation, where $1/4 < \beta < 1/2$ in the $\ell^1$ case, $r_n < \beta < 1/2$ in the basis pursuit with TV case and $s_n < \beta < 1/2$ in the unconstrained lasso with TV case. 

In this case, the argument for correctness proceeds as in Case $1$, with the only exceptions being that now \ref{inst:ExitAlphaLBetaRandom} never results in the termination of the algorithm (as opposed to Case 1 where \ref{inst:ExitAlphaGBetaRandom} never results in the termination of the algorithm) and that now the correct solution is $ 4\cdot 10^{-K} e_1$ for linear programming and basis pursuit with $\ell^1$ regularisation, $ 4\cdot 10^{-K}e_1 + (\tau - 4\cdot 10^{-K})e_3$ for constrained lasso, $\flipOp \eta(y_1,1/2,\beta)$ for basis pursuit with TV, or $ \flipOp \psi(y_1,1/2,\beta)$ for unconstrained lasso with TV. Similarly to the argument for Case 1, the algorithm will only provide the wrong answer or fail to halt if  \textit{Guess} is called and outputs an $x\in\dyadic^{N}$ within $10^{-K}$ of $4\cdot 10^{-K} e_2$ for linear programming and basis pursuit with $\ell^1$ regularisation, $4\cdot 10^{-K}e_2 + (\tau - 4\cdot 10^{-K})e_3$ for constrained lasso, $\eta(y_1,1/2,\beta)$ for basis pursuit with TV, or $\psi(y_1,1/2,\beta)$ for unconstrained lasso with TV. This occurs with probability bounded above by 
\[
\sum_{n=1}^{\infty} \probab(F_n)/2 = 1/3,
\]
 and thus the algorithm produces a correct output with probability at least $2/3$.

{\bf Case 4:} For this case, $\iota$ is of the following form: $\big(y^{\mathrm{LP},\mathrm{w}} ,\matLP(\alpha,\beta,m,N)\big)$ for linear programming, $\big(y^{\mathrm{BP},\ell^1,\mathrm{w}} ,\matl(\alpha,\beta,m,N)\big)$ for basis pursuit denoising with $\ell^1$ regularisation, $\big(y^{\mathrm{UL},\ell^1,\mathrm{w}},\matl(\alpha,\beta,m,N)\big)$ for unconstrained lasso, $\big(y^{\mathrm{CL},\mathrm{w}},\matl(\alpha,\beta,m,N)\big)$ for unconstrained lasso, $\big(y^{\mathrm{BP},\mathrm{TV},\mathrm{w}},\matTV(\alpha,\beta,m,N)\big)$ for basis pursuit denoising with TV regularisation and $\big(y^{\mathrm{UL},\mathrm{TV},\mathrm{w}},\matTV(\alpha,\beta,m,N)\big)$ for unconstrained lasso with TV regularisation, where $(\alpha,\beta) \in \albetSet$  in the $\ell^1$ case, $(\alpha,\beta) \in \albetSetBPTV{K-1}$  in the basis pursuit with TV case and $(\alpha,\beta) \in \albetSetULTV{K-1}$ in the unconstrained lasso with TV case. 

By the correctness of \textit{IdenifyStrongOrWeak} (Lemma \ref{lemma:IdentifyStrongOrWeakCorrect}), the algorithm proceeds immediately by executing $\textit{Weak}^{\orvec,\ormat}(m,N, k_\epsilon=K)$ . The correctness of \textit{Weak} (Lemma \ref{lemma:WeakCorrect}) then shows that the algorithm always outputs some $x$ with $\disM(x,\tilde\Xi(\tilde\iota)) \leq 10^{-K}$. Therefore, in this case the algorithm is correct with probability $1$.

These are the only possible cases for $\iota$, and  thus, with probability at least $2/3$, the algorithm halts  and outputs an $x$ such that $\disM(x,\tilde\Xi(\tilde\iota)) \leq 10^{-K}$ , as desired.

\section{Proof of Theorem \ref{Cor:main}: parts (iii) and  (iv) \label{sec:cor_main(iii)(iv)}}

Parts (iii) and (iv) of Theorem \ref{Cor:main} are formally stated in Proposition \ref{Cor:main_SCI_cont}, which we now prove. 
Recall that the breakdown epsilon bounds in Proposition \ref{Cor:main_SCI_cont} were already proved in \S \ref{sec:lambdaHatStrong-Cor:main_SCI}. Thus, all that remains to prove in part (iii) is the existence of a deterministic algorithm that achieves at least $10^{-(K-1)}$ accuracy on the problems  $\{\tilde\Xi,\tilde\Omega_{m,N},\mathcal{M}_N,\tilde\Lambda_{m,N}\}=\{\Xi,\Omega_{m,N},\mathcal{M}_N,\Lambda_{m,N} \}^{\Delta_1}$ for varying dimensions $m$ and $N$, where $\Omega_{m,N}$ is one of \eqref{eq:defs-5-Om-fixedDim} and $\Xi$ and $\Lambda_{m,N}$ are the corresponding solution map and set of evaluations. 

\subsection{Proof of part (iii) of Proposition  \ref{Cor:main_SCI_cont}}

As in the previous  two sections, we fix the notation for an element $\tilde\iota$ of $\tilde\Omega_{m,N}$, writing   $\tilde{\iota}=\big(\{y_j^{(n)}\}_{n=0}^{\infty}, \{\matLP_{j,k}^{(n)}\}_{n=0}^{\infty}\big)_{j,k}$, corresponding to an $\iota=(y,\matLP)\in \Omega_{m,N}$, in the linear programming case, and analogously for the $\ell^1$ and TV problems,  writing respectively $\matl$ and $\matTV$ instead of $\matLP$.
The desired algorithm makes use of the subroutines established in \ref{sec:UsefulSubroutinesForMainThm} and is specified for each of the computational problems as follows:

{\it Deterministic $K-1$ digit algorithm}

\indent Inputs: Dimensions  $m$, $N$.\\
\indent Oracles: $\orvec$ and $\ormat$ providing access to the components $y_j^{(n)}$ and $\matLP_{j,k}^{(n)}$ (respectively  $\matl_{j,k}^{(n)}$ or $\matTV_{j,k}^{(n)}$) of an input $\tilde{\iota}$.\\
\indent Output: A vector $x \in \dyadic ^N$ (in the Turing case) or $x \in \real ^N$ (in the BSS case) with $\disM(x,\tilde\Xi(\tilde\iota))\leq 10^{-K+1}$.

\begin{enumerate}[leftmargin=10mm, label= \arabic*.,ref = step \arabic*]
	\item We execute $\textit{IdentifyStrongOrWeak}^{\orvec,\ormat}(m,N)$. If this evaluates to `InputWeak', we execute the algorithm $\textit{Weak}^{\orvec,\ormat}(m,N,k_\epsilon=K-1)$  and terminate. Otherwise, we continue to \ref{inst:K-1DigitsAlgorithm}.
	\item \label{inst:K-1DigitsAlgorithm} For each of linear programming, basis pursuit with $\ell^1$ regularisation and unconstrained lasso with $\ell^1$ regularisation we output an $x\in\dyadic^{N}$ with $\|x-2\cdot 10^{-K}e_1 + 2\cdot 10^{-K}e_2\|_p\leq 10^{-K}$ and terminate. For constrained lasso we output an $x\in\dyadic^{N}$ with $\|x - 2\cdot 10^{-K}e_1 + 2\cdot 10^{-K}e_2 + (\tau-4\cdot 10^{-K})e_3\|_p\leq 10^{-K}$. For basis pursuit with TV regularisation we set $\eta^* = \textit{OutputEta}^{\orvec,\ormat}(m,N,k_0=K, k_\epsilon=K)$ and output $x = (\flipOp \eta^* + \eta^*)/2$ and terminate and finally for unconstrained lasso with TV regularisation we set $\psi^* = \textit{OutputPsi}^{\orvec,\ormat}(m,N,\allowbreak k_0=K, k_\epsilon=K)$ and output $x = (\flipOp \psi^* + \psi^*)/2$.
\end{enumerate}

Unlike the algorithm described in the proof of part (ii) of  Proposition \ref{Cor:main_SCI}, it is clear that this algorithm always terminates. Hence it suffices to show that if $x$ is the output of the algorithm then $\disM(x,\tilde\Xi(\tilde\iota)) \leq 10^{-K+1}$. We do this by separately considering the  following  two cases for $\iota\in \Omega$ that $\tilde\iota$ corresponds to.

{\bf Case 1:} For this case, $\iota$ is of the following form: $\big(y^{\mathrm{LP},\mathrm{s}} ,\matLP(\alpha,\beta,m,N)\big)$ for linear programming, \newline $\big(y^{\mathrm{BP},\ell^1,\mathrm{s}} ,\matl(\alpha,\beta,m,N)\big)$ for basis pursuit denoising with $\ell^1$ regularisation, $\big(y^{\mathrm{UL},\ell^1,\mathrm{s}},\matl(\alpha,\beta,m,N)\big)$ for unconstrained lasso, $\big(y^{\mathrm{CL},\mathrm{s}},\matl(\alpha,1/2,m,N)\big)$ for constrained lasso,  $\big(y^{\mathrm{BP},\mathrm{TV},\mathrm{s}},\matTV(\alpha,\beta,m,N)\big)$ for basis pursuit denoising with TV regularisation and $\big(y^{\mathrm{UL},\mathrm{TV},\mathrm{s}},\matTV(\alpha,\beta,m,N)\big)$ for unconstrained lasso with TV regularisation, where $(\alpha,\beta) \in \albetSet$  for the $\ell^1$ problems, $(\alpha,\beta) \in \albetSetBPTV{K}$  for basis pursuit with TV, and $(\alpha,\beta) \in \albetSetULTV{K}$ for unconstrained lasso with TV. By the correctness of \textit{IdentifyStrongOrWeak} (Lemma \ref{lemma:IdentifyStrongOrWeakCorrect}), in this case the algorithm will execute \ref{inst:K-1DigitsAlgorithm}. 

By Lemma \ref{lemma:ProblemBasicExampleLP} (respectively Lemma \ref{lemma:ProblemBasicExampleBPDNL1} or Lemma \ref{lemma:ProblemBasicExampleULASSO}) we have 
$
\Xi(\iota) \cap \{4\cdot 10^{-K} e_1, 4\cdot 10^{-K} e_2\}\neq \varnothing
$
 for linear programming (respectively, basis pursuit with $\ell^1$ regularisation or unconstrained lasso with $\ell^1$ regularisation). Thus for any of linear programming, basis pursuit with $\ell^1$ regularisation or unconstrained lasso with $\ell^1$ regularisation we conclude that 
\begin{equation*}
\begin{split}
\disM(x,\Xi(\iota)) &\leq \max \{ d_{\mathcal{M}}(x, 4\cdot 10^{-K}e_1), d_{\mathcal{M}}(x,4\cdot 10^{-K}e_2)\} \\
&\leq \|2\cdot 10^{-K} e_1 +2 \cdot  10^{-K} e_2\|_p + 2\cdot 10^{-K} = (2+ 2^{1+1/p} )10^{-K} \leq 10^{-K+1}.
\end{split}
\end{equation*}
 The same argument but this time with Lemma \ref{lemma:ProblemBasicExampleCLASSO} and the addition of $(\tau - 4\cdot 10^{-K})e_3$ where appropriate shows that $\disM(x,\Xi(\iota)) \leq 10^{-K+1}$ for constrained lasso.

Similarly, Lemma \ref{lemma:BPTVSolutions} shows that $\Xi(\iota) \cap \{\eta(y_1,\alpha,\beta),  \flipOp \eta(y_1,\alpha,\beta)\} \neq \varnothing$ for basis pursuit with TV regularisation. Thus if $x$ is the output of the algorithm and $\eta = \eta(y_1,\alpha,\beta)$ we have 
\begin{align*}
\disM(x,\Xi(\iota)) &\leq \max \{ d_{\mathcal{M}}(x,\eta), d_{\mathcal{M}}(x,\flipOp \eta)\} \\
&\leq \Bigg\|\frac{\eta^* + \flipOp \eta^*}{2} - \frac{\eta - \flipOp \eta}{2}\Bigg\|_p + \max\Bigg\{ \Bigg \|\frac{\eta + \flipOp \eta}{2} - \eta\Bigg\|_p ,  \Bigg\|\frac{\eta + \flipOp \eta}{2} - \flipOp \eta\Bigg\|_p\Bigg\} \\
&\leq \|\eta^* - \eta\|_p +  \frac{1}{2} \|\eta - \flipOp \eta\|_p=  \|\eta^* - \eta\|_p +  \frac{1}{2} \|(\eta_1 - \eta_N)(e_1 - e_N)\|_p.
\end{align*}
By Lemma \ref{lemma:OutputEtaCorrect} we obtain $\|\eta^* - \eta\|_p \leq 10^{-K}  < 10^{-K+1}/6$. Moreover, $\|(\eta_1 - \eta_N)(e_1 - e_N)\|_p  = 2^{1/p} |\eta_1 - \eta_N|$ and by applying Lemma \ref{lemma:BPTVEtaBounds} (equation \ref{eq:BPTVIneqBDEpsUpper}) with $n = K$ we obtain 
\[
|\eta_1 - \eta_N| \leq \allowbreak \max_{(\alpha,\beta) \in \albetSetBPTV{K}} |\eta(y_1,\alpha,\beta) - \eta(y_1,\alpha,\beta)_N| \leq 5\cdot 10^{-K+1} 2^{-1/p}/6. 
\]
We conclude that $\disM(x,\Xi(\iota))\leq 10^{-K+1}/6 + 5\cdot 10^{-K+1}/6 = 10^{-K+1}$.

Finally,  for unconstrained lasso with TV regularisation, applying Lemma \ref{lemma:ULTVSolutions} shows that 
\[
\Xi(\iota) \cap \{\psi(y_1,\alpha,\beta),  \flipOp \psi(y_1,\alpha,\beta)\} \neq \varnothing.
\] Therefore if $x$ is the output of the algorithm and $\psi = \psi(y_1,\alpha,\beta)$ then an  argument similar to the above establishes  
\begin{align*}
\disM(x,\Xi(\iota)) &\leq \max \{ d_{\mathcal{M}}(x,\psi), d_{\mathcal{M}}(x,\flipOp \psi)\} \leq  \|\psi^* - \psi\|_p +  \frac{1}{2}\|(\psi_1 - \psi_N)(e_1 - e_N)\|_p.
\end{align*}
This time, by Lemma \ref{lemma:OutputPsiCorrect} we obtain $\|\psi^* - \psi\|_p \leq 10^{-K}<  10^{-K+1}/6$. Moreover, $\|(\psi_1 - \psi_N)(e_1 - e_N)\|_p  = 2^{1/p} |\psi_1 - \psi_N|$ and by Lemma \ref{lemma:ULTVPsiBounds} (equation \ref{eq:ULTVIneqBDEpsUpper}) with $n = K$ we obtain $|\psi_1 - \psi_N| \leq \max_{(\alpha,\beta) \in \albetSetBPTV{K}} |\psi(y_1,\alpha,\beta) - \psi(y_1,\alpha,\beta)_N| \leq 5\cdot10^{-K+1} 2^{-1/p}/6 $. We therefore conclude that 
\[
\disM(x,\Xi(\iota))) \leq 5\cdot10^{-K+1}/6 + 10^{-K+1}/6  = 10^{-K+1}.
\] 

We have thus established that for any of linear programming, basis pursuit or unconstrained lasso (with either $\ell^1$ or with TV regularisation), the output $x$ satisfies $\disM(x,\tilde\Xi(\tilde\iota)) =\disM(x,\Xi(\iota)) \leq 10^{-K+1}$.

{\bf Case 2:} For this case,  $\iota$ is of the following form: $\big(y^{\mathrm{LP},\mathrm{w}} ,\matLP(\alpha,\beta,m,N)\big)$ for linear programming, $\big(y^{\mathrm{BP},\ell^1,\mathrm{w}} ,\matl(\alpha,\beta,m,N)\big)$ for basis pursuit denoising with $\ell^1$ regularisation, $\big(y^{\mathrm{UL},\ell^1,\mathrm{w}},\matl(\alpha,\beta,m,N)\big)$ for unconstrained lasso, $\big(y^{\mathrm{CL},\mathrm{w}},\matl(\alpha,\beta,m,N)\big)$ for constrained lasso, $\big(y^{\mathrm{BP},\mathrm{TV},\mathrm{w}},\matTV(\alpha,\beta,m,N)\big)$ for basis pursuit denoising with TV regularisation and $\big(y^{\mathrm{UL},\mathrm{TV},\mathrm{w}},\matTV(\alpha,\beta,m,N)\big)$ for unconstrained lasso with TV regularisation, where $(\alpha,\beta) \in \albetSet$  for linear programming and the $\ell^1$ regularised problems, $(\alpha,\beta) \in \albetSetBPTV{K-1}$  for basis pursuit with TV, and $(\alpha,\beta) \in \albetSetULTV{K-1}$ for  unconstrained lasso with TV.

By the correctness of \textit{IdenifyStrongOrWeak} (Lemma \ref{lemma:IdentifyStrongOrWeakCorrect}), the algorithm proceeds immediately by executing $\textit{Weak}^{\orvec,\ormat}(m,N,k_\epsilon=K-1)$. The correctness of \textit{Weak} (Lemma \ref{lemma:WeakCorrect}) then shows that the algorithm always outputs some $x$ with $\disM(x,\tilde\Xi(\tilde\iota)) =\disM(x,\Xi(\iota)) \leq 10^{-K+1}$. 

These are the only possible cases for each of the inputs, and we can thus conclude that the algorithm halts and outputs an $x\in\real^N$ such that $\disM(x,\tilde\Xi(\tilde\iota)) \leq 10^{-K+1}$, completing the proof of Proposition \ref{Cor:main_SCI_cont} part (iii).

\subsection{Proof of part (iv) of Proposition  \ref{Cor:main_SCI_cont}}
The desired algorithm that achieves $10^{-(K-2)}$ accuracy and has polynomial Turing  arithmetic and BSS runtime, as well as Turing space complexity, and requires at most a polynomial in $\log(n_\mathrm{var})$ digits form the oracles, where $n_\mathrm{var}=mN+m$ is the number of variables, is specified as follows:

{\bf Polynomial time $K-2$ digit algorithm}

\indent Inputs: Dimensions  $m$, $N$.\\
\indent Oracles: $\orvec$ and $\ormat$ providing access to the components $y_j^{(n)}$ and $\matLP_{j,k}^{(n)}$ (respectively  $\matl_{j,k}^{(n)}$ or $\matTV_{j,k}^{(n)}$) of an input $\tilde{\iota}$.\\
\indent Output: A vector $x\in\dyadic^N$ (in the Turing case) or $x \in \real^N$ (in the BSS case) with $\disM(x,\tilde\Xi(\tilde\iota))\leq 10^{-K+2}$.

\begin{enumerate}[leftmargin=10mm, label= \arabic*.,ref = step \arabic*]
	\item We execute $\textit{IdentifyStrongOrWeak}^{\orvec,\ormat}(m,N)$ and branch depending on the result: 
	\begin{enumerate}[label = \alph*., ref = step 1\alph*)]
	\item \label{inst:K-2DigitsAlgorithmWeak}If the output of \textit{IdentifyStrongOrWeak} was `InputWeak' then we output depending on the problem at hand. For each of linear programming, basis pursuit with $\ell^1$ regularisation and unconstrained lasso with $\ell^1$ regularisation we output an $x\in\dyadic^N$ with $\|x-2\cdot 10^{-K+1}e_1 + 2\cdot 10^{-K+1}e_2\|_p\leq 10^{-(K-1)}$ and terminate. For constrained lasso we output an $x\in\dyadic^N$ with $\|x - 2\cdot 10^{-K+1}e_1 + 2\cdot 10^{-K+1}e_2 + (\tau - 4\cdot 10^{-K+1})e_3\|_p\leq 10^{-(K-1)}$ and terminate. For basis pursuit with TV regularisation we set $\eta^* = \textit{OutputEta}^{\orvec,\ormat}(m,N,k_0=K-1, k_\epsilon=K-1 )$ and output $x = (\flipOp \eta^* + \eta^*)/2$ and terminate and finally for unconstrained lasso with TV regularisation we set $\psi^* = \textit{OutputPsi}^{\orvec,\ormat}(m,N,k_0=K-1, k_\epsilon=K-1)$ and output $x = (\flipOp \psi^* + \psi^*)/2$ and   terminate.
	\item \label{inst:K-2DigitsAlgorithmStr} If the output of \textit{IdentifyStrongOrWeak} was `InputStrong' then we output depending on the problem at hand. For each of linear programming, basis pursuit with $\ell^1$ regularisation and unconstrained lasso with $\ell^1$ regularisation we output an $x\in\dyadic^N$ with $\| x- 2\cdot 10^{-K}e_1 + 2\cdot 10^{-K}e_2\|_p\leq 10^{-(K-1)}$ and terminate. For constrained lasso we output an $x\in\dyadic^N$ with $\| x-  2\cdot 10^{-K}e_1 + 2\cdot 10^{-K}e_2 + (\tau - 4\cdot 10^{-K})e_3\|_p\leq 10^{-(K-1)}$ and terminate. For basis pursuit with TV regularisation we set $\eta^* = \textit{OutputEta}^{\orvec,\ormat}(m,N,k_0=K, k_\epsilon=K-1)$ and output $x = (\flipOp \eta^* + \eta^*)/2$ and terminate and finally for unconstrained lasso with TV regularisation we set $\psi^* = \textit{OutputPsi}^{\orvec,\ormat}(m,N,k_0=K, k_\epsilon=K-1)$ and output $x = (\flipOp \psi^* + \psi^*)/2$ and terminate.
	\end{enumerate}
\end{enumerate}

We show that the algorithm outputs a vector $x$ with $\disM(x,\tilde\Xi(\tilde\iota)) \leq 10^{-K+2}$ by analysing the following two cases for $\iota\in\Omega$ that $\tilde\iota$ corresponds to:

{\bf Case 1:} For this case, $\iota$ is of the following form: $\big(y^{\mathrm{LP},\mathrm{s}} ,\matLP(\alpha,\beta,m,N)\big)$ for linear programming, \newline $\big(y^{\mathrm{BP},\ell^1,\mathrm{s}} ,\matl(\alpha,\beta,m,N)\big)$ for basis pursuit denoising with $\ell^1$ regularisation, $\big(y^{\mathrm{UL},\ell^1,\mathrm{s}},\matl(\alpha,\beta,m,N)\big)$ for unconstrained lasso,  $\big(y^{\mathrm{CL},\mathrm{s}},\matl(\alpha,\beta,m,N)\big)$ for constrained lasso, $\big(y^{\mathrm{BP},\mathrm{TV},\mathrm{s}},\matTV(\alpha,\beta,m,N)\big)$ for basis pursuit denoising with TV regularisation and $\big(y^{\mathrm{UL},\mathrm{TV},\mathrm{s}},\matTV(\alpha,\beta,m,N)\big)$ for unconstrained lasso with TV regularisation, where $(\alpha,\beta) \in \albetSet$  in the $\ell^1$ case, $(\alpha,\beta) \in \albetSetBPTV{K}$  in the basis pursuit with TV case and $(\alpha,\beta) \in \albetSetULTV{K}$ in the unconstrained lasso with TV case. By the correctness of \textit{IdentifyStrongOrWeak} (Lemma \ref{lemma:IdentifyStrongOrWeakCorrect}), in this case the algorithm will execute \ref{inst:K-2DigitsAlgorithmStr}. 

We have already shown in the proof of part (iii) that this will result in an output $x$ with $\disM(x,\tilde\Xi(\tilde\iota))\allowbreak \leq 10^{-K+1}$. It thus remains to analyse the complexity of the algorithm. To this end, first note that Lemmas \ref{lemma:OutputEtaCorrect}, \ref{lemma:OutputPsiCorrect}, and \ref{lemma:IdentifyStrongOrWeakCorrect} imply that \textit{OutputEta}, \textit{OutputPsi}, and \textit{IdentifyStrongOrWeak} request at most a polynomial in $\log(n_{\mathrm{var}})$ digits from the oracles, and therefore so does the 
 algorithm overall.
Next, it suffices to show that the Turing and BSS runtime of the algorithm is polynomial in $n_{\mathrm{var}}$, as this will then imply the desired polynomial bound on the Turing arithmetic runtime and the Turing space complexity as well. Note that Lemma \ref{lemma:IdentifyStrongOrWeakCorrect} implies that the Turing and BSS runtime of executing \textit{IdentifyStrongOrWeak} is bounded by a polynomial of $\log(m)$, whereas Lemmas \ref{lemma:OutputEtaCorrect} and \ref{lemma:OutputPsiCorrect} imply that the Turing and BSS runtimes of executing \textit{OutputEta} and \textit{OutputPsi} are bounded by a polynomial of $K$ and $\log(N)$. This, in particular, means that $\length(\eta^*)$ and $\length(\psi^*)$, and hence the Turing runtime of computing $x$, are bounded by a polynomial of $\log(N)$. As $K$ is fixed, we deduce that the the overall Turing and BSS runtime of the algorithm is bounded by a polynomial of $\log(m)$ and $\log(N)$, which is itself bounded by a polynomial of $n_{\mathrm{var}}$, as desired.

{\bf Case 2:} For this case, $\iota$ is of the following form: $\big(y^{\mathrm{LP},\mathrm{w}} ,\matLP(\alpha,\beta,m,N)\big)$ for linear programming, $\big(y^{\mathrm{BP},\ell^1,\mathrm{w}} ,\matl(\alpha,\beta,m,N)\big)$ for basis pursuit denoising with $\ell^1$ regularisation, $\big(y^{\mathrm{UL},\ell^1,\mathrm{w}},\matl(\alpha,\beta,m,N)\big)$ for unconstrained lasso, $\big(y^{\mathrm{CL},\mathrm{w}},\matl(\alpha,\beta,m,N)\big)$ for unconstrained lasso, $\big(y^{\mathrm{BP},\mathrm{TV},\mathrm{w}},\matTV(\alpha,\beta,m,N)\big)$ for basis pursuit denoising with TV regularisation and $\big(y^{\mathrm{UL},\mathrm{TV},\mathrm{w}},\matTV(\alpha,\beta,m,N)\big)$ for unconstrained lasso with TV regularisation, where $(\alpha,\beta) \in \albetSet$  in the $\ell^1$ case, $(\alpha,\beta) \in \albetSetBPTV{K-1}$  in the basis pursuit with TV case and $(\alpha,\beta) \in \albetSetULTV{K-1}$ in the unconstrained lasso with TV case. In this case, the correctness of \textit{IdentifyStrongOrWeak} (Lemma \ref{lemma:IdentifyStrongOrWeakCorrect}) implies that the algorithm will execute \ref{inst:K-2DigitsAlgorithmWeak}. Our analysis is very similar to that of part (iii) Case 1.

By Lemma \ref{lemma:ProblemBasicExampleLP} (respectively Lemma \ref{lemma:ProblemBasicExampleBPDNL1} or Lemma \ref{lemma:ProblemBasicExampleULASSO}) we have 
$
\Xi(\iota) \cap \{4\cdot 10^{-K+1} e_1, 4\cdot 10^{-K+1} e_2\}\neq \varnothing
$
 for linear programming (respectively, basis pursuit with $\ell^1$ regularisation or unconstrained lasso with $\ell^1$ regularisation). Thus for any of linear programming, basis pursuit with $\ell^1$ regularisation or unconstrained lasso with $\ell^1$ regularisation we conclude that 
\begin{equation*}
\begin{split}
\disM(x,\Xi(\iota)) &\leq \max \{ d_{\mathcal{M}}(x, 4\cdot 10^{-K+1}e_1), d_{\mathcal{M}}(x,4\cdot 10^{-K+1}e_2)\} \\
& \leq \|2\cdot 10^{-K+1} e_1 + 2\cdot 10^{-K+1} e_2\|_p + 2\cdot 10^{-(K-1)} =(2+ 2^{1+1/p} )10^{-K+1} \leq 10^{-K+2}.
\end{split}
\end{equation*}
The argument for constrained lasso is identical but now we add $(\tau -4\cdot 10^{-K+1})e_3$ to both the true solution and the output value of the algorithm. Applying Lemma \ref{lemma:ProblemBasicExampleCLASSO} allows us to conclude that $\disM(x,\Xi(\iota)) \leq 10^{-K+2}$.
Similarly, Lemma \ref{lemma:BPTVSolutions} shows that $\Xi(\iota) \cap \{\eta(y_1,\alpha,\beta),  \flipOp \eta(y_1,\alpha,\beta)\} \neq \varnothing$ for basis pursuit with TV regularisation. Thus, writing $\eta = \eta(y_1,\alpha,\beta)$, we obtain as in the proof of part (iii), case 1 that
$\disM(x,\Xi(\iota)) \leq \|\eta^* - \eta\|_p + |\eta_1 - \eta_N|2^{1/p}/2$.
Next, by Lemma \ref{lemma:OutputEtaCorrect} we obtain $\|\eta^* - \eta\|_p \leq 10^{-(K-1)}< 10^{-K+2}/6$. Moreover, Lemma \ref{lemma:BPTVEtaBounds} (equation \ref{eq:BPTVIneqBDEpsUpper}) with $n = K-1$ gives 
\[
|\eta_1 - \eta_N| \leq \max_{(\alpha,\beta) \in \albetSetBPTV{K-1}} |\eta(y_1,\alpha,\beta) - \eta(y_1,\alpha,\beta)_N| \leq 5\cdot 10^{-K+2} 2^{-1/p}/6. 
\]
We conclude that $\disM(x,\Xi(\iota)) \leq 5\cdot 10^{-K+2}/6 + 10^{-K+2}/6 = 10^{-K+2}$.

Finally,  for unconstrained lasso with TV regularisation Lemma \ref{lemma:ULTVSolutions} shows that $\Xi(\iota) \cap \{\psi(y_1,\alpha,\beta), \allowbreak \flipOp \psi(y_1,\alpha,\beta)\} \neq \varnothing$. Therefore, writing $\psi = \psi(y_1,\alpha,\beta)$, we obtain as before that $\disM(x,\Xi(\iota)) \leq  \|\psi^* - \psi\|_p +  |\psi_1 - \psi_N|2^{1/p}/2$
Thus applying Lemma \ref{lemma:ULTVPsiBounds} (equation \ref{eq:ULTVIneqBDEpsUpper}) with $n = K-1$ gives 
\[
|\psi_1 - \psi_N| \leq \max_{(\alpha,\beta) \in \albetSetBPTV{K-1}} |\psi(y_1,\alpha,\beta) - \psi(y_1,\alpha,\beta)_N| \leq 5\cdot 10^{-K+2} 2^{-1/p}/6.
\] 
We conclude that $\disM(x,\Xi(\iota)) \leq 5\cdot10^{-K+2}/6 + 10^{-K+2}/6 = 10^{-K+2}$.
Therefore, in all cases we have $\disM(x,\tilde\Xi(\tilde\iota))=\disM(x,\Xi(\iota))\leq 10^{-(K-2)}$. Finally, the complexity analysis is entirely analogous to the one presented in Case 1.
As these were the only possible cases for $\iota$, the proof of part (iv) of Proposition \ref{Cor:main_SCI_cont}  is complete.

\section{Proof of Theorem \ref{thm:ExitFlag}}

Theorem \ref{thm:ExitFlag} is formally stated in Proposition \ref{prop:ExitFlag}, which we now prove. 

\subsection{Constructing the sets of inputs} 
The constructions here will be similar to those in the proof of Theorem \ref{Cor:main}. As before, each computational problem will require a separate input set.  We will denote the input sets for LP, $\ell^1$ BP, $\ell^1$ UL, CL, $\mathrm{TV}$ BP and $\mathrm{TV}$ UL
by 
\newcommand{\EFAlgorithmRange}{\omega}
\begin{equation}\label{eq:EFInputSetLabels}
\omlpef, \ombplef, \omullef,\omcllef, \ombptvef, \omultvef
\end{equation}
respectively. These input sets may depend on $K$, $\EFAlgorithmRange$, $\omega$, the dimensions $m$, $N$, and any relevant regularisation parameters. As in the proof of Theorem \ref{Cor:main}, we will omit this dependency in order to avoid cluttered notation. 

For the $\ell^1$ problems, there is no dependence on $\omega$ or $\EFAlgorithmRange$ and we simply set
\[\omlpef = \omlps,\,
\ombplef= \ombpls,\,
\omullef = \omulls,\,
\omcllef = \omclls 
\]
where the sets $\omlps,\ombpls,\omulls$ and $\omclls$ are defined as in \S\ref{sec:constr-Cor:main-sets}.

The situation is slightly more complicated for the TV problems: starting with basis pursuit TV, we recall $\eta$ from \S\ref{sec:TV reg geometry} and $r_K$ as defined in \S\ref{sec:constr-Cor:main-sets}, and choose  $r_K' \in (r_K,1/2) $ so that if $\alpha \in [r_K',r_k]$ and $y_1\in[0,  3/2]$, then 
$\|\eta(y_1,\alpha,1/2) - \eta(y_1,1/2,1/2)\|_{p} \leq 10^{-K} - \EFAlgorithmRange$. The existence of $r_K'$ is guaranteed by the continuity of $\eta$ and the assumption that $10^{-K} - \EFAlgorithmRange>0$.

For unconstrained lasso TV we recall $\psi$ from \S\ref{sec:TV reg geometry} and $s_K$ as defined in \S\ref{sec:constr-Cor:main-sets}, and choose  $s_K' \in (s_K , 1/2) $ so that if $\alpha \in [r_K',r_k]$ and $y_1\in [0,  107/180]$  then 
$\|\psi(y_1,\alpha,1/2) - \psi(y_1,1/2,1/2)\|_{p} \leq 10^{-K} - \EFAlgorithmRange $.
Again, the existence of $s_K'$ is guaranteed by the continuity of $\psi$ and the assumption that $ 10^{-K} - \EFAlgorithmRange>0$.
We then set
\begin{align*}
\albetSetEFBPTV{K}&=\left([r_K',1/2]\times\{1/2\} \right)\cup\left(\{1/2\} \times [r_K',1/2] \right)\\
\albetSetEFULTV{K}&=\left([s_K',1/2]\times\{1/2\} \right)\cup\left(\{1/2\} \times [s_K',1/2] \right)
\end{align*}
and 
\begin{align*}
\ombptvef = \left\{ \left(y^{\mathrm{BP},\mathrm{TV},\mathrm{s}},\matTV(\alpha,\beta,m,N)\right)\, \vert \, (\alpha,\beta)\in \albetSetEFBPTV{K} \right \},  \\
\omultvef = \left\{ \left(y^{\mathrm{UL},\mathrm{TV},\mathrm{s}},\matTV(\alpha,\beta,m,N)\right)\, \vert \,  (\alpha,\beta)\in \albetSetEFULTV{K} \right \} 
\end{align*}
Note then that $\ombptvef \subseteq \ombptvs$ and $\omultvef \subseteq \omultvs$ so once again the statements about the condition and size of the inputs follows from Lemmas \ref{lemma:mainThmSizeEstimates}, \ref{lem:cond-Xi-various},  \ref{lem:cond-FP-various}, and \ref {lem:cond-mat-various}.

\subsection{Proof of Proposition \ref{prop:ExitFlag} part (i)} Using the setup above we can now prove part (i) of the proposition.

\begin{proof}[Proof of Proposition \ref{prop:ExitFlag} part (i) ]

Note that part (i) is an immediate consequence of part (ii) as discussed in Remark \ref{remark:EFPartiWeakerThanPartii}. Therefore we will focus on proving part (ii). To do this we will employ Proposition \ref{prop:EF}  with $\kappa=10^{-K}$. For each of the input sets $\Omega$ discussed above, we will construct an input $\iota_0 \in \Omega$, sequences of inputs $\{\iota^1_n\}_{n=1}^{\infty}$ and $\{\iota^2_n\}_{n=1}^{\infty}$ in $\Omega$, subsets $S^1$ and $S^2$ of $\real^N$ and vectors $x^1,x^2 \in \real^N$ satisfying requirements \ref{assumption:EFS1S2Distant} to \ref{assumption:EFLLPO} in Proposition \ref{prop:EF} with $\kappa = 10^{-K}$. To do so, we present an argument similar to the one used in Lemma \ref{lemma:StrongBDEpsilonBounds10K}, covering each computational problem separately.

{\bf Case $\{\xilp,\omlpef \} $}: We recall $\matLP$ from \eqref{eq:y_and_A_LP} and  $y^{\mathrm{LP},\mathrm{s}}$ from \eqref{eq:thm3.2_y_values_strong} and set
\begin{equation*}
\begin{aligned}
\iota^0_n &= (y^{\mathrm{LP},\mathrm{s}}, \matLP(1/2,1/2,m,N)), \\
\iota^1_n &= (y^{\mathrm{LP},\mathrm{s}}, \matLP(1/2,1/2-4^{-n},m,N)),\qquad   \iota^2_n= (y^{\mathrm{LP},\mathrm{s}}, \matLP(1/2 - 4^{-n},1/2,m,N)),\\
x^1 &= 4\cdot 10^{-K}e_1, \qquad x^2  = 4 \cdot 10^{-K}e_2, \qquad S^1= \{x^1\},\quad\text{and }\quad  S^2 = \{x^2\},
\end{aligned}
\end{equation*}
from which Proposition \ref{prop:EF} \ref{assumption:EFDel1Info} immediately holds. Arguing as in Lemma \ref{lemma:StrongBDEpsilonBounds10K} gives us Proposition \ref{prop:EF} \ref{assumption:EFS1S2Distant}. Lemma \ref{lemma:ProblemBasicExampleLP} gives Proposition \ref{prop:EF} \ref{assumption:EFXiIotajInSj}. It is then obvious that Proposition \ref{prop:EF} \ref{assumption:EFxjLimitOfXiIotajn} holds. Now, by Lemma \ref{lemma:ProblemBasicExampleLP}, we have that, for $\alpha\in[0, 1/2)$, $\xilp(y^{\mathrm{LP},\mathrm{s}},\matLP(\alpha,1/2,m,N)) = \{ x^2\} $ and similarly, for $\beta \in[0, 1/2) $, $\xilp(y^{\mathrm{LP},\mathrm{s}},\matLP(1/2,\beta,m,N)) = \{x^1\} $, and thus
\begin{equation}\label{eq:EFXiLPInX1X2}
\xilp(\omlpef) \subseteq \{x^1\} \cup \{x^2\} \cup \xilp(\iota^0).
\end{equation}
which implies \ref{prop:EF} \ref{assumption:EFXiOmegaInBallAroundXiIota0X1AndX2}. Finally, to show \ref{prop:EF} \ref{assumption:EFLLPO} we note that Lemma \ref{lemma:ProblemBasicExampleLP} implies $x^1,x^2 \in \Xi(\iota^0)$.

{\bf Case $\{\xibpdn, \ombplef\} $}: We recall  $y^{\mathrm{BP},\ell^1,\mathrm{s}}$ from \eqref{eq:thm3.2_y_values_strong} and set
\begin{equation*}
\begin{aligned}
\iota^0 & = (y^{\mathrm{BP},\ell^1,\mathrm{s}}, \matl(1/2,1/2,m,N)),\\
\iota^1_n & = (y^{\mathrm{BP},\ell^1,\mathrm{s}}, \matl(1/2,1/2-4^{-n},m,N)),\qquad \iota^2_n = (y^{\mathrm{BP},\ell^1,\mathrm{s}},\matl(1/2 - 4^{-n},1/2,m,N)),\\
x^1 &= 4\cdot 10^{-K}e_1,\qquad  x^2 = 4 \cdot 10^{-K}e_2,\qquad S^1 = \{x^1\}, \quad\text{and}\quad S^2 = \{x^2 \}. 
\end{aligned}
\end{equation*}
The remainder of the argument is identical to the LP case except we use Lemma \ref{lemma:ProblemBasicExampleBPDNL1} instead of Lemma \ref{lemma:ProblemBasicExampleLP} and obtain
\begin{equation}\label{eq:EFXiBPInX1X2}
\xibpdn(\ombplef) \subseteq \{x^1\} \cup \{x^2\} \cup \xibpdn(\iota^0)
\end{equation}
instead of \eqref{eq:EFXiLPInX1X2}.

{\bf Case $\{\xiul, \omullef\}$ }: We recall  $y^{\mathrm{UL},\ell^1,\mathrm{s}}$ from \eqref{eq:thm3.2_y_values_strong} and set
\begin{equation*}
\begin{aligned}
\iota^0 &= (y^{\mathrm{UL},\ell^1,\mathrm{s}}, \matl(1/2,1/2,m,N)),\\
\iota^1_n &= (y^{\mathrm{UL},\ell^1,\mathrm{s}}, \matl(1/2,1/2-4^{-n},m,N) ),\qquad \iota^2_n = (y^{\mathrm{UL},\ell^1,\mathrm{s}}, \matl(1/2 - 4^{-n},1/2,m,N)) ,\\
x^1 &= 4\cdot 10^{-K}e_1, \qquad x^2 = 4 \cdot 10^{-K}e_2, \qquad S^1 = \{x^1\},\quad\text{and}\quad S^2 = \{x^2\}.
\end{aligned}
\end{equation*}
The remainder of the argument is identical to the LP case except we use Lemma \ref{lemma:ProblemBasicExampleULASSO} instead of Lemma \ref{lemma:ProblemBasicExampleLP} and obtain
\begin{equation}\label{eq:EFXiULInX1X2}
\xiul(\omullef) \subseteq \{x^1\} \cup \{x^2\} \cup \xiul(\iota^0)
\end{equation}
instead of \eqref{eq:EFXiLPInX1X2}.

{\bf Case $\{ \xicl,\omcllef\}$}:  We recall  $y^{\mathrm{CL},\mathrm{s}}$ from \eqref{eq:thm3.2_y_values_strong} and set
\begin{equation*}
\begin{aligned}
\iota^0_n &= (y^{\mathrm{CL},\mathrm{s}}, \matl(1/2,1/2,m,N)),\\
\iota^1_n &= (y^{\mathrm{CL},\mathrm{s}}, \matl(1/2,1/2-4^{-n},m,N)),\qquad \iota^2_n = (y^{\mathrm{CL},\mathrm{s}},\matl(1/2 - 4^{-n},1/2,m,N) ),\\
x^1 &= 4\cdot 10^{-K}e_1 + (\tau - 4\cdot 10^{-K})e_3, \qquad x^2 = 4 \cdot 10^{-K}e_2+(\tau - 4\cdot 10^{-K})e_3,\\
S^1 &= \{x^1\},\quad\text{and}\quad S^2 = \{x^2\},
\end{aligned}
\end{equation*}
from which Proposition \ref{prop:EF} \ref{assumption:EFDel1Info} immediately holds. 
Arguing as in Lemma \ref{lemma:StrongBDEpsilonBounds10K} gives us Proposition \ref{prop:EF} \ref{assumption:EFS1S2Distant}. Lemma \ref{lemma:ProblemBasicExampleCLASSO} gives Proposition \ref{prop:EF} \ref{assumption:EFXiIotajInSj}. It is then obvious that Proposition \ref{prop:EF} \ref{assumption:EFxjLimitOfXiIotajn} holds. Next, by Lemma \ref{lemma:ProblemBasicExampleCLASSO}, we have that, for $ \alpha\in[0, 1/2)$, $\xicl(y^{\mathrm{CL},\mathrm{s}},\matl(\alpha,1/2,m,N)) = \{x^2\}$ and similarly, for $\beta\in[0, 1/2)$, $\xicl(y^{\mathrm{CL},\mathrm{s}},\matl(1/2,\beta,m,N)) = \{x^1\} $. We thus obtain
\begin{equation}\label{eq:EFXiCLInX1X2}
\xicl (\omlpef) \subseteq \{x^1\} \cup \{x^2\} \cup \xicl(\iota^0) 
\end{equation}
which implies \ref{prop:EF} \ref{assumption:EFXiOmegaInBallAroundXiIota0X1AndX2}. Finally, to show \ref{prop:EF} \ref{assumption:EFLLPO} we note that Lemma \ref{lemma:ProblemBasicExampleCLASSO} implies $x^1,x^2 \in \Xi(\iota^0)$.

{\bf Case $\{\xibptv, \ombptvef\} $}: We recall $\matTV$ from \eqref{eq:y_and_A_TV} and $y^{\mathrm{BP},\mathrm{TV},\mathrm{s}}$ from \eqref{eq:thm3.2_y_values_strong}, and choose $n_0\in \mathbb{N}$ so that $1/2 - 4^{-n_0} \in [r_K',1/2]$. Now, writing $y_1 = y^{\mathrm{BP},\mathrm{TV},\mathrm{s}}_1$, we set
\begin{equation*}
\begin{aligned}
\iota^0 &= (y^{\mathrm{BP},\mathrm{TV},\mathrm{s}}, \matTV(1/2,1/2,m,N)), \\
\iota^1_n &= (y^{\mathrm{BP},\mathrm{TV},\mathrm{s}} ,\matTV(1/2,1/2-4^{-n-n_0},m,N)),\qquad \iota^2_n = (y^{\mathrm{BP},\mathrm{TV},\mathrm{s}}, \matTV(1/2 - 4^{-n-n_0},1/2,m,N)),\\
x^1 &= \flipOp \eta(y_1,1/2,1/2),\qquad x^2 = \eta(y_1 ,1/2,1/2),\\
S^1 &=  \{\flipOp \eta(y_1 ,\alpha,1/2)\, \vert \, \alpha \in [r_K',1/2]\} ,\quad\text{and}\quad S^2 =  \{\eta(y_1 ,1/2,\beta) \, \vert \, \beta \in [r_K',1/2]\} .
\end{aligned}
\end{equation*}
By the choice of $n_0$ and the definition of $r_K'$ we have that $\iota^0, \iota^1_n$ and $\iota^2_n$ are all elements of $\ombptvef$, and we also see immediately from the definition that Proposition \ref{prop:EF} \ref{assumption:EFDel1Info} holds.
Next, by Lemma \ref{lemma:BPTVSolutions} we get Proposition \ref{prop:EF} \ref{assumption:EFXiIotajInSj}, whereas the continuity of $\eta$ together with Lemma \ref{lemma:BPTVSolutions} yield Proposition \ref{prop:EF} \ref{assumption:EFxjLimitOfXiIotajn}. The same argument as in the proof of Lemma \ref{lemma:StrongBDEpsilonBounds10K} gives us Proposition \ref{prop:EF} \ref{assumption:EFS1S2Distant}. 

Now, by Lemma \ref{lemma:mainThmSizeEstimates} we have that $y_1\in[0, 3/2]$, and so, by Lemma \ref{lemma:BPTVSolutions} and the definition of $r_K'$ we obtain that, for $\alpha \in [r_K',1/2]$ and $\iota = (y^{\mathrm{BP},\mathrm{TV},\mathrm{s}}, \matTV(\alpha,1/2,m,N))$ we have $\|\xibptv(\iota)-x^1\|_{\infty} = \|\eta(y_1,\alpha,1/2) - \eta(y_1,1/2,1/2)\|_\infty \leq 10^{-K} - \EFAlgorithmRange$. Similarly, for $\beta \in [r_K',1/2]$ and 
$
\iota = (y^{\mathrm{BP},\mathrm{TV},\mathrm{s}}, \matTV(1/2,\beta,m,N))
$
 we obtain $\|\xibptv(\iota) - x^2\|_{\infty} \leq 10^{-K} - \EFAlgorithmRange$. Hence 
\begin{equation}\label{eq:EFXiBPTVInX1X2}
\xibptv (\ombptvef) \subseteq \clBall{10^{-K} -  \EFAlgorithmRange}{x^1} \cup \clBall{10^{-K} -  \EFAlgorithmRange }{x^2}\cup \clBall{10^{-K} -  \EFAlgorithmRange}{\xibptv(\iota^0)}.
\end{equation}
We conclude that Proposition \ref{prop:EF} \ref{assumption:EFXiOmegaInBallAroundXiIota0X1AndX2} holds. Finally, to show \ref{prop:EF} \ref{assumption:EFLLPO} we note that Lemma \ref{lemma:BPTVSolutions} implies $x^1,x^2 \in \Xi(\iota^0)$.

\textbf{Case $\{\xiultv, \omultvef\} $}: We recall  $y^{\mathrm{UL},\mathrm{TV},\mathrm{s}}$ from \eqref{eq:thm3.2_y_values_strong} and choose  $n_0\in \mathbb{N}$ such that $1/2 - 4^{-n_0} \in [s_K',1/2]$. Writing $y_1= y^{\mathrm{UL},\mathrm{TV},\mathrm{s}}_1$, we set 
\begin{equation*}
\begin{aligned}
\iota^0 &= (y^{\mathrm{UL}, \mathrm{TV},\mathrm{s}}, \matTV(1/2,1/2,m,N) ), \\
\iota^1_n &= (y^{\mathrm{UL},\mathrm{TV},\mathrm{s}}, \matTV(1/2,1/2-4^{-n-n_0},m,N)) ,\qquad  \iota^2_n = (y^{\mathrm{UL},\mathrm{TV},\mathrm{s}}, \matTV(1/2 - 4^{-n-n_0},1/2,m,N)),\\
x^1 &= \flipOp \psi(y_1,1/2,1/2),\qquad x^2 = \psi(y_1,1/2,1/2),\\
S^1 &= \{\flipOp \psi(y_1,\alpha,1/2)\, \vert \, \alpha \in [s_K',1/2]\},\quad \text{and} \quad S^2 = \{\psi(y_1,1/2,\beta) \, \vert \, \beta \in [s_K',1/2]\}.
\end{aligned}
\end{equation*}

The remainder of the  argument is very similar to that in the case $\{\xibptv,\ombptvef\} $ with the major difference being that we replace $\eta$ by $\psi$, $r_K'$ by $s_K'$, references to Lemma \ref{lemma:BPTVSolutions} by references to Lemma \ref{lemma:ULTVSolutions}, $y^{\mathrm{BP},\mathrm{TV},\mathrm{s}}$ by $y^{\mathrm{UL},\mathrm{TV},\mathrm{s}}$, and the bounds $y_1\in[0, 3/2]$ by $y_1\in[0, 107/180]$. This then yields
\begin{equation}\label{eq:EFXiULTVInX1X2}
\xiultv (\omultvef) \subseteq \clBall{10^{-K} - \EFAlgorithmRange}{x^1} \cup \clBall{10^{-K} - \EFAlgorithmRange }{x^2} \cup \clBall{10^{-K} -  \EFAlgorithmRange}{\xiultv(\iota^0)}
\end{equation}
instead of \eqref{eq:EFXiBPTVInX1X2}. As before, we are able to conclude that assumptions \ref{assumption:EFS1S2Distant} to \ref{assumption:EFLLPO} in Proposition \ref{prop:EF} hold.

Therefore, for each computational problem we have shown that assumptions \ref{assumption:EFS1S2Distant} to \ref{assumption:EFLLPO} in Proposition \ref{prop:EF} hold, and thus there exits a $\tilde \Lambda^+ \in \mathcal{L}^{\mathcal{O}, \omega, \tilde \Xi}(\tilde \Lambda)$ such that, for the computational problem  $\{\Xi^E,\tilde \Omega,\{0,1\},\tilde \Lambda^+\}$, we have $\epsilon_{\mathbb{P}\mathrm{B}}^{\mathrm{s}}(\mathrm{p}) \geq 1/2$, as desired, completing the proof of part (i) of Proposition \ref{prop:ExitFlag}.
\end{proof}

\subsection{Proof of Proposition \ref{prop:ExitFlag} part (ii)}
The fact that for the oracle problem $\{\Xi^E, \tilde\Omega, \{0,1\}, \tilde{\Lambda}\}^{\mathcal{O}, \omega}$ with respect to $\{\tilde\Xi, \tilde\Omega, \mathcal{M}, \tilde{\Lambda}\}$ we have $\epsilon_{\mathrm{B}}^{\mathrm{s}}\geq 1/2$ follows directly from part (i) of the proposition. It thus suffices to show that the oracle problem $\{\tilde\Xi, \tilde\Omega, \mathcal{M}, \tilde{\Lambda}\}^{\mathcal{O}, \omega} $ with respect to $\{\Xi^E, \tilde\Omega, \{0,1\}, \tilde{\Lambda}\}$ can be computed in the arithmetic model to within $10^{-K}$ accuracy.  The first step of our recursive algorithm that achieves this will make use of the exit flag to determine if the input is a representation of $\iota^0$ -- if not, the algorithm proceeds by executing a loop similar to the algorithm \textit{Weak} defined in \S\ref{sec:UsefulSubroutinesForMainThm}. 

To this end, for the linear programming case, we recall \eqref{eq:exit-flag-general-Omega} and fix the notation for an element of $\tilde\Omega^{\mathcal{O}}$ by writing $\tilde{\iota}=\big(\{y_j^{(n)}\}_{n=0}^{\infty}, \{\matLP_{j,k}^{(n)}\}_{n=0}^{\infty}\big)_{j,k}\oplus \Xi^E(\tilde \iota)$, corresponding to an $\iota=(y,\matLP)\in \Omega$ and the solution $\Xi^E(\tilde \iota)$ to the exit flag problem, i.e., $\Xi^E(\tilde \iota) = 1$ if $\disM(\Gamma(\tilde{\iota}),\tilde \Xi(\tilde \iota)) \leq 10^{-K}$  and $\Xi^E(\tilde \iota) = 0$ else. For the $\ell^1$ and TV problems the notation is entirely analogous, except that we respectively write $\matl$ and $\matTV$ instead of $\matLP$.
The exact specification of the algorithm is now as follows:

{\it Algorithm} {\bf ComputeTrueSolution:}

\indent Inputs: Dimensions $m$, $N$.\\
\indent Oracles: $\orvec$, $\ormat$, and $\orsol$ providing access to the components $y_j^{(n)}$, $\matLP_{j,k}^{(n)}$ (respectively $\matl_{j,k}^{(n)}$ or $\matTV_{j,k}^{(n)}$), and $\Xi^E(\tilde \iota)$ of an input $\tilde \iota$. \\
\indent Output: A vector $x \in \mathbb{D}^N$ (for the Turing machine) or  $x \in \real^N$ (for the BSS machine) with 
\[
\disM(x,\tilde \Xi(\tilde \iota)) \leq 10^{-K}.
\]

\begin{enumerate}[leftmargin=10mm, label= \arabic*.,ref = step \arabic*]
	\item If $\Xi^E(\tilde \iota) = 1$, we run $\Gamma$ on $\tilde\iota$ and output $\Gamma(\tilde \iota)$. Otherwise, we proceed to the next step. \label{inst:CTSExitAlgCorrect}
	\item We execute a loop that proceeds as follows -- at each iteration, we increase $n$, starting with $n=1$. What we do now depends on the problem at hand. In the linear programming case we use the oracle $\ormat$ to read the values $\matLP_{1,1}^{(n)}$ and $\matLP_{1,2}^{(n)}$ and set $d = \matLP_{1,1}^{(n)} - \matLP_{1,2}^{(n)}$. In the $\ell^1$ case we use the oracle $\ormat$ to read the values $\matl_{1,1}^{(n)}$ and $\matl_{1,1}^{(n)}$and set  $d = \matl_{1,1}^{(n)} -  \matl_{1,2}^{(n)}$. For the TV problems we similarly read the values $\matTV_{1,1}^{(n)}$ and $\matTV_{1,N}^{(n)}$ and set $d =  \matTV_{1,1}^{(n)} - {\matTV}_{1,N}^{(n)}$.
	
	Next, we branch depending on the value of $d$:
	\begin{enumerate}[label = \alph*.]
		\item If $d > 2 \cdot 2^{-n}$ then we output an $x\in\dyadic^N$ with $\|x -  4\cdot 10^{-K} e_1 \|_p\leq 10^{-K}$ for linear programming, basis pursuit with $\ell^1$ regularisation or unconstrained lasso with $\ell^1$ regularisation. For constrained lasso we output an $x\in\dyadic^N$ with $\|x -   4\cdot 10^{-K}e_1 + (\tau - 4\cdot 10^{-K})e_3 \|_p\leq 10^{-K}$. For basis pursuit TV we apply the subroutine $\textit{OutputEta}$ to obtain $\eta^*=\textit{OutputEta}^{\orvec,\ormat}(m,N,k_K=K,k_\epsilon=K)$, to which we apply $\flipOp$ and output as $x$ (so that $x = \flipOp \eta^*$). Finally, for unconstrained lasso TV we apply the subroutine $\textit{OutputPsi}$  to obtain $\psi^*=\textit{OutputPsi}^{\orvec,\ormat}(m,N,k_K=K,k_\epsilon=K)$, to which we apply $\flipOp$ and output as $x$ (so that $x = \flipOp \psi^*$). In all of the above cases we terminate the algorithm after outputting $x$. \label{inst:CTSExitAlphaGBeta}
		\item Alternatively, if $d < -2 \cdot 2^{-n}$ then we output an $x\in\dyadic^N$ with $\|x -   4 \cdot 10^{-K} e_2 \|_p\leq 10^{-K}$ for linear programming, basis pursuit with $\ell^1$ regularisation or unconstrained lasso with $\ell^1$ regularisation. For constrained lasso we output an $x\in\dyadic^N$ with $\|x -   4\cdot 10^{-K}e_2 + (\tau - 4\cdot 10^{-K})e_3\|_p\leq 10^{-K}$. For basis pursuit TV we apply the subroutine $\textit{OutputEta}$ to obtain $\eta^*=\textit{OutputEta}^{\orvec,\ormat}(m,N,k_K=K,k_\epsilon=K)$, which we output as $x$. Finally, for unconstrained lasso TV we apply the subroutine $\textit{OutputPsi}$ with parameter $\epsilon$ to obtain $\psi^*=\textit{OutputPsi}^{\orvec,\ormat}(m,N,k_K=K,k_\epsilon=K)$, which we output as $x$. In all of the above cases we terminate the algorithm after outputting $x$. \label{inst:CTSExitAlphaLBeta}
	\end{enumerate}
	If neither of these conditions are met then the loop continues by executing the next iteration.
\end{enumerate}

\begin{lemma} \label{lemma:EFComputeTrueSolutionCorrect}
	For every input $\tilde \iota $ in $\tilde \Omega$, \textit{ComputeTrueSolution} outputs a vector $x$ with $\disM(x,\tilde \Xi(\tilde \iota)) \leq 10^{-K}$.
\end{lemma}
\begin{proof}
	First, observe that if `ComputeTrueSolution' exits at \ref{inst:CTSExitAlgCorrect} then `ComputeTrueSolution' outputs a value within $10^{-K}$ of the true solution -- this follows from the definition of the exit flag oracle. 
	
	Next, for each of the computational problems, let $\iota^0$,  $x^1$, and $x^2$ be as in the proof of part (i).
	Now, for any $\tilde \iota$ corresponding to $\iota^0$, we claim that the algorithm always terminates at \ref{inst:CTSExitAlgCorrect}. Indeed, we proved there that 
	\[
	\Xi(\Omega) \subseteq \clBall{10^{-K} - \omega}{x^1} \cup \clBall{10^{-K} - \omega}{x^2} \cup \Xi(\iota^0)
	\]
	 and that $x^1,x^2 \in \Xi(\iota^0)$. Consequently, we have $\Xi(\Omega) \subseteq \clBall{10^{-K} - \omega}{ \Xi(\iota^0) }$. Furthermore, by assumption \ref{assumption:AlgorithmCloseToTheRange} we have $\disM(\Gamma(\tilde \iota),\Xi(\Omega)) \leq \omega$. Hence 
	\begin{align*}
	\disM(\Gamma(\tilde \iota),\tilde \Xi(\tilde \iota))&= \disM(\Gamma(\tilde \iota),\Xi(\iota^0)) \leq \disM(\Gamma(\iota),\Xi(\Omega)) + \sup_{w \in \Xi(\Omega)} \disM(w,\tilde \Xi(\tilde \iota^0))\\
	& \leq \omega + 10^{-K} - \omega = 10^{-K},
	\end{align*}
	completing the argument that the algorithm always exits at \ref{inst:CTSExitAlgCorrect}.

	All that remains is to prove that `ComputeTrueSolution' is correct whenever  $\tilde \iota$ corresponds to an element $\iota\in\Omega$  other than $\iota^0$ as well as $\Xi^E(\tilde \iota) = 0$. The proof of this step is very similar to the proof of Lemma \ref{lemma:WeakCorrect}. Indeed, since $ \iota\neq \iota^0$, $\iota$  has to be of one of the following forms for the $\ell^1$ problems: $\big(y^{\mathrm{LP},\mathrm{s}} ,\matl(\alpha,\beta,m,N)\big)$ for linear programming, $\big(y^{\mathrm{BP},\ell^1,\mathrm{s}} ,\matl(\alpha,\beta,m,N)\big)$ for basis pursuit denoising with $\ell^1$ regularisation and $\big(y^{\mathrm{UL},\ell^1,\mathrm{s}},\matl(\alpha,\beta,m,N)\big)$ for unconstrained lasso with $\ell^1$ regularisation where $\alpha,\beta \in \mathcal{L}$, $\alpha \neq \beta$. Similarly, for the TV problems $\iota$ must be of one of the following forms: $\big(y^{\mathrm{BP},\mathrm{TV},\mathrm{s}} ,\matTV(\alpha,\beta,m,N)\big)$ for basis pursuit TV where $\alpha \neq \beta$ and $(\alpha,\beta) \in  \albetSetEFBPTV{K}$ and  $\big(y^{\mathrm{UL},\mathrm{TV},\mathrm{s}} ,\matTV(\alpha,\beta,m,N)\big)$ for unconstrained lasso TV where $\alpha \neq \beta$ and $(\alpha,\beta) \in \albetSetEFULTV{K}$. The remainder of the proof from here is identical to that of Lemma \ref{lemma:WeakCorrect}, except that all references to the value $10^{-K+1}$ are now replaced by $10^{-K}$.
The above construction establishes that, given an oracle for the exit flag, one can compute $\tilde \Xi$ to precision $10^{-K}$, concluding the proof of part (ii).
\end{proof}

\subsection{Proof of Proposition \ref{prop:ExitFlag} part (iii)}

\begin{proof}[Proof of Proposition \ref{prop:ExitFlag} part (iii)]

We recall the definition of $\matLPBPExit$ from \eqref{eq:y_and_A_LPexitFlag}, define the input set $\Omega^{\sharp}$ according to
\begin{equation*}
\Omega^{\sharp}:= \left \{ (\vecYl(y_1,m),\matLPBPExit(\alpha,m,N)) \, \vert \, y_1 = 3 \cdot 10^{-K} \cdot \alpha\quad  \text{ and } \quad \alpha \in [0,1/2)\right \},
\end{equation*}
and let $\Xi$ be either $\xilp$ or $\xibpdn$ with $\delta=0$. 
We first construct a $\hat\Lambda\in\mathcal{L}^1(\Lambda)$ such that, for the computational problem $\{\Xi^{E}, \Omega^{\sharp},\{0,1\}, \hat{\Lambda}\}$ and $\mathrm{p}>1/2$, we have $\epsilon_{\mathbb{P}\mathrm{B}}^{\mathrm{s}}(\mathrm{p}) \geq 1/2$ and second, we design an algorithm that outputs a solution to the oracle problem $\{\Xi^{E}, \tilde\Omega^{\sharp},\{0,1\}, \tilde{\Lambda}\}^{\mathcal{O},\omega}$ with respect to $\{\tilde\Xi,\tilde\Omega^\sharp, \mathcal{M},\tilde\Lambda\}$.

To accomplish the first task, we will use Proposition \ref{prop:EF} with $\kappa=10^{-K}$. Concretely, we set 
\begin{equation*}
\begin{aligned}
\iota^0 &= (\vecYl(0,m),\matLPBPExit(0,m,N)),\quad \iota^1_n = (\vecYl(3\cdot 10^{-K}\cdot  2^{-n},m),\matLPBPExit(2^{-n},m,N)),\qquad   \iota^2_n = \iota^0\\
x^1 &= 3\cdot 10^{-K} e_1 \in \real^N, \quad x^2 = 0 \in \real^N, \quad S^1 = \{x^1\}, \quad \text{ and }\quad  S^2 = \{x^2\}.
\end{aligned}
\end{equation*}
We now prove that conditions \ref{assumption:EFS1S2Distant} to \ref{assumption:EFXiOmegaInBallAroundXiIota0X1AndX2} of Proposition \ref{prop:EF} hold. Starting with \ref{assumption:EFS1S2Distant}, we note that 
\[
\disM(S^1,S^2) = d_{\mathcal{M}}(x^1,0) = 3\cdot 10^{-K}. 
\]
By Lemma \ref{lemma:ProblemBasicExampleLPBPExitFlag}, $\Xi(\iota^j_n)  = \{x^j\} $ for all $n\in\mathbb{N}$ and $j=1,2$. This gives us both \ref{assumption:EFXiIotajInSj} and \ref{assumption:EFxjLimitOfXiIotajn}. We also have $\iota^0 = \iota^2_n$ and $|f(\iota^1_n) - f(\iota^0)| \leq \max(3\cdot 10^{-K}2^{-n},2^{-n}) = 2^{-n}$, for all $f\in \Lambda$, and thus \ref{assumption:EFDel1Info} holds. Finally, for $\iota \in \Omega^{\sharp}$ we can use Lemma \ref{lemma:ProblemBasicExampleLPBPExitFlag} to see that $\Xi(\iota) = \{x^1\} $ or $\Xi(\iota) = \{x^2\}$ and hence \ref{assumption:EFXiOmegaInBallAroundXiIota0X1AndX2} holds. An application of Proposition \ref{prop:EF} therefore establishes that  $\epsilon_{\mathbb{P}\mathrm{B}}^{\mathrm{s}}(\mathrm{p}) \geq 1/2$, for $\mathrm{p}>1/2$, as desired.
Next, we construct the desired algorithm, which will be identical  for both basis pursuit and linear programming. Similarly to the proof of part (ii), we write  
\[
\tilde{\iota}=\big(\{y_j^{(n)}\}_{n=0}^{\infty}, \{A_{j,k}^{(n)}\}_{n=0}^{\infty}\big)_{j,k}\oplus \{g_k(\tilde{\iota})\}_{k}, 
\]
for an input $\tilde\iota\in\tilde\Omega$ corresponding to an $\iota=(y,\matLP)\in \Omega$ and a solution oracle $\{g_k(\tilde{\iota})\}_k\in\mathcal{B}^\infty_\omega(\tilde\Xi(\tilde\iota))$.
We now define:

{\it Algorithm} {\bf ComputeExitFlag:} 

\indent Inputs: Dimensions $m$, $N$.\\
\indent Oracles: $\orsol$ providing access to the components $g_k$ of an input $\tilde \iota$ \\
\indent Output: Either $1$ if $\disM(\Gamma(\tilde\iota),\tilde\Xi(\tilde\iota)) \leq 10^{-K}$ or $0$ if $\disM(\Gamma(\tilde\iota),\tilde\Xi(\tilde\iota)) > 10^{-K}$.

\begin{enumerate}[leftmargin=10mm, label= \arabic*.,ref = step \arabic*]
	\item We use  $\orsol$ to read $g_1(\tilde \iota)$: if $|g_1(\tilde \iota)| \leq 10^{-K}$ then we set $x = 0\in\mathbb{R}^N$, otherwise we set $x = 3 \cdot 10^{-K} e_1 $.
	\item We output $1$ if $\|\Gamma(\tilde \iota) - x\|_\infty \leq 10^{-K}$. Otherwise, we output $0$. \label{inst:EFExit}
\end{enumerate}
To prove the correctness of `ComputeExitFlag', we will prove that $\tilde \Xi(\tilde \iota) = \{ x\} $ for every input $\tilde \iota \in \tilde \Omega^{\sharp}$. Indeed, let $\iota$ be the element of $\Omega^\sharp$ that $\tilde \iota$ corresponds to. There are two cases to consider: either $\iota = (\vecYl(3 \cdot 10^{-K} \alpha,m),\matLPBPExit(\alpha,m,N))$ for some $\alpha > 0$ or $\iota = (\vecYl(0,m),\matLPBPExit(0,m,N))$. In the first case Lemma \ref{lemma:ProblemBasicExampleLPBPExitFlag} tells us that $\Xi(\iota) = \{ 3 \cdot 10^{-K}\} $. In particular, since $\{g_k(\tilde{\iota})\}_k\in\mathcal{B}^\infty_\omega(\tilde\Xi(\tilde\iota))$, we must have $|g_1 (\tilde \iota)| \geq 3\cdot 10^{-K} - \omega >10^{-K}$ and hence `ComputeExitFlag' sets $x = 3\cdot 10^{-K}$, which is the only element of $\Xi(\iota)$.
If instead $\iota = (\vecYl(0,m),\matLPBPExit(0,m,N))$ then Lemma \ref{lemma:ProblemBasicExampleLPBPExitFlag} tells us that $\Xi(\iota) = 0$. Hence $|g_1(\tilde \iota)| \leq \omega \leq 10^{-K}$. Thus `ComputeExitFlag' sets $x = 0$, which is the only element of $\Xi(\iota) = \tilde \Xi(\tilde \iota)$. 
We have thus shown that $\tilde \Xi(\tilde \iota) = \{ x\} $. Therefore $\Gamma^E(\tilde \iota) = 1$ if $\|\Gamma(\tilde \iota) - x\|_{\infty} \leq 10^{-K}$ and $\Gamma^E(\tilde \iota) = 0$ otherwise. But this is exactly the output of `ComputeExitFlag' in \ref{inst:EFExit} and thus the algorithm `ComputeExitFlag' computes the exit flag as claimed.
\end{proof}

\section{Proof of Theorem \ref{thm:Smales9}}
Theorem  \ref{thm:Smales9} is formally stated in Proposition \ref{prop:SmalesNinthProposition}, which we now prove. Our strategy is similar to that for the proof of Theorem \ref{Cor:main}. We begin by defining the input set $\Omega$ and then verify the conditions of \eqref{prop:DrivingNegativeProposition}. This, together with an additional algorithm we write to prove part (iii) will complete the proof of Proposition \ref{prop:SmalesNinthProposition}.

\subsection{Constructing the input and evaluation sets}
Let $M$ and $K$ be as in the statement in the proposition, recall the definition of $\SetNine_k$ from \eqref{eq:The9Sets} and define $M_{k} = 10^{-k}(\floor{10^{k} M}+1)$, for $k\in\mathbb{N}\cup \{0\}$. Note that $M_{k}\leq M_{k-1}$, for all $k\in\mathbb{N}$, with equality if and only if $M\in\SetNine_k$.

We recall the definitions of $\vecYl$ and $\matLPObj$ from \eqref{eq:y_and_A_LPObj} and for natural $m$ and $N$ with $m<N$ define the sets $\omlpos$ and $\omlpow$ as follows:
\begin{align*}
\omlpos &= \{ (\vecYl (M_{K}\alpha,m),\matLPObj(\alpha ,1,m,N)) \,\vert \, \alpha \in [0,1] \}\\
\omlpow &= \{ (\vecYl(M_{K-1} \alpha ,m),\matLPObj(\alpha ,0,m,N))  \,\vert \, \alpha \in (0,1]  \}\\
&\;\cup \{ (\vecYl(0,m),\matLPObj(-\alpha,0,m,N)) \,\vert\, \alpha \in (0,1]  \}
\end{align*}
and set 
\[
\Omega = \bigcup_{1\leq m< N} \omlpos \cup \omlpow.
\]

\subsection{Proof of Proposition \ref{prop:SmalesNinthProposition} parts (i) and (ii)}
Fix $m,N\in\mathbb{N}$ with $m<N$ and define the constant sequence $\{\iota^1_n\}_{n=1}^{\infty} \subseteq \omlpos$ and input $\iota^0 \in \omlpos$ by $\iota^1_n = \iota^0= (\vecYl(0,m),\matLPObj(4^{-n},1,m,N))$, as well as the sequence $\{\iota^2_n\}_{n=1}^{\infty} \subseteq \omlpos$ by $\iota^2_n = (\vecYl(4^{-n}M_K),\matLPObj(4^{-n},1,m,N))$. 
If we define $S^1 = \{\text{1}\}$ and $S^2 = \{0\}$, then $(4^{-n}M_K)/4^{-n}=10^{-K} (\floor{10^{K} M}+1)$, and so by Lemma \ref{lemma:ProblemBasicExampleLPObj} we have $\Xi_K(\iota^1_n) = \Xi_K(\iota^0) \in S^1$ and $\Xi_K(\iota^2_n) \in S^2$, for all $n\in\mathbb{N}$. 
The conditions \ref{property:MinimisersOfNonZero} -- \ref{property:MtotallyOrdered} of Proposition  \ref{prop:DrivingNegativeProposition}  are readily seen to hold with for these sequences and the sets $S^1$ and $S^2$, allowing us to conclude that there exists a $\hat\Lambda^{\mathrm{s}}\in \mathcal{L}^1(\Lambda_{m,N})$ such that $\{\Xi_K,\omlpos,\{0,1\},\hat\Lambda^{\mathrm{s}}\}\notin \Sigma_1^{G}$ as well as $\strbdepsp \geq 1/2$, for $\mathrm{p} \in [0,1/2)$.

Next, consider the sequences $\{\iota^1_n\}_{n=1}^{\infty}$ and $\{\iota^2_n\}_{n=1}^{\infty}$ given by $\iota^1_n = (\vecYl(0,m),\matLPObj(-4^{-n},0,m,N))\allowbreak \in \omlpow$ and $\iota^2_n = (\vecYl(M_{K-1}\cdot 4^{-n},m),\matLPObj(4^{-n},0,m,N))\in \omlpow$. Defining $\iota^0 \notin \Omega$ according to $\iota^0 = (\vecYl(0,m),\matLPObj(0,0,m,N))$, we have $|f(\iota^j_n) - f(\iota^0)| \leq 4^{-n}$, for $f\in\Lambda_{m,N}$, $j\in\{1,2\}$, and $n\in\mathbb{N}$, whereas by Lemma \ref{lemma:ProblemBasicExampleLPObj} we obtain $\Xi(\iota^1_n) \in S^1=\{1\}$ and $\Xi(\iota^2_n) \in S^2=\{0\}$, and thus  Proposition  \ref{prop:DrivingNegativeProposition}  implies the existence of a $\hat\Lambda^{\mathrm{w}}\in \mathcal{L}^1(\Lambda_{m,N})$ such that, for the computational problem $\{\Xi_{K-1},\omlpow ,\{0,1\},\hat\Lambda^{\mathrm{w}}\}$, we have $\epsilon_{\mathbb{P}\mathrm{B}}^{\mathrm{w}}(\mathrm{p}) \geq 1/2$, for $\mathrm{p} \in [0,1/2)$.

Now, defining $\hat\Lambda_{m,N}:=\{f_{j,n}\,\vert\, j\leq n_{\mathrm{var}},n\in\mathbb{N}\}$, where we let $f_{j,n}(\iota)=f_{j,n}^{\mathrm{s}}(\iota)$ if $\iota\in \Omega_{m,N}^{\mathrm{s}}$ and $f_{j,n}(\iota)=f_{j,n}^{\mathrm{w}}(\iota)$ if $\iota\in \Omega_{m,N}^{\mathrm{w}}$, for $j\leq n_{\mathrm{var}}$ and $n\in\mathbb{N}$, we have that $\hat\Lambda_{m,N}$ provides $\Delta_1$-information for $\{\Xi_K,\Omega_{m,N},\mathcal{M}_N,\Lambda_{m,N} \}$, and in view of Remark \ref{remark:OmegaSubsetBDE}, we have that all the breakdown epsilon bounds mentioned above also hold for $\{\Xi_K,\Omega_{m,N},\{0,1\},\hat\Lambda_{m,N}\}$. This establishes parts (i) and (ii) of Proposition \ref{prop:SmalesNinthProposition} as well as the weak breakdown epsilon bound in part (iii).

\subsection{Proof of Proposition \ref{prop:SmalesNinthProposition}, $M\notin \SetNine_K \,\implies\,$ (iii)}
To complete the proof of part (iii) under the additional assumption $M\notin \SetNine_K$, it remains to construct an algorithm that, for each $m$ and $N$ with $m<N$, outputs an approximate solution to $\{\Xi_{K-1},\Omega_{m,N},\{0,1\},\Lambda_{m,N}\}^{\Delta_1}=\{\tilde\Xi_{K-1},\tilde\Omega_{m,N},\{0,1\},\tilde\Lambda_{m,N}\}$ of accuracy $10^{-{K-1}}$. Our construction is similar to the one used in \S \ref{sec:cor_main(iii)(iv)}.
Concretely, the algorithm is constructed as follows. First note that $M_K<M_{K-1}$, since we are assuming $M\notin \SetNine_K$.
Now, writing $\tilde{\iota}=\big(\{y_j^{(n)}\}_{n=0}^{\infty}, \{A_{j,k}^{(n)}\}_{n=0}^{\infty}\big)_{j,k}$ for an input $\tilde\iota\in\tilde\Omega_{m,N}$ corresponding to an $\iota=(y, \matLP)\in \Omega_{m,N}$, we define:

{\bf Algorithm $K-1$ digit Smale's 9th problem} 

\indent Inputs: Dimensions $m$, $N$.\\
\indent Oracles: $\ormat$ providing access to the components $A_{1,k}^{(n)}$ of an input $\tilde\iota\in\tilde\Omega_{m,N}$. \\
\indent Output: Either $1$ or $0$.
\begin{enumerate}[leftmargin=10mm]
	\item We use $\ormat$ to read $A_{1,2}^{(2)}$. If $A_{1,2}^{(2)} \geq 3/4$ then we output $1$ and terminate the algorithm, otherwise we proceed to the next step.\label{step:Smale9thWeakOrStrong}
	\item We execute a loop that proceeds as follows -- at each iteration, we increase $n$, starting with $n=1$. Use oracle $\ormat$ to read $A_{1,1}^{(n)}$. If $A_{1,1}^{(n)} > 2^{-n}$ we output $0$ and terminate the algorithm. If instead $A_{1,1}^{(n)} < -2^{-n}$ we output $1$ and terminate the algorithm.\label{step:Smale9thWeak}
\end{enumerate}
To show the correctness of this algorithm, we consider three cases depending on which $\iota\in\Omega_{m,N}$ the input $\tilde\iota$ corresponds to:

{\bf Case 1:} The input $\iota = (\vecYl(M_K\alpha ,m),\matLPObj(\alpha ,1,m,N)) \in \omlpos$, for some $\alpha[0,1]$. Observing that
$
(M_K \alpha)/\alpha= M_K< M_{K-1}
$
whenever $\alpha\in(0,1]$, Lemma \ref{lemma:ProblemBasicExampleLPObj} yields $\Xi_{K-1}(\iota) = 1$, for all $\alpha[0,1]$. On the other hand, since $\matLPObj(\alpha,1,m,N)_{1,2} = 1$ by the definition of $\matLPObj$, we obtain $A_{1,2}^{(2)} \geq 1 - 1/4 =3/4$ and hence the algorithm also outputs $1$.

{\bf Case 2:} The input $\iota = (\vecYl(M_{K-1} \alpha,m),\matLPObj(\alpha,0,m,N))$ for some $\alpha \in (0,1]$. In this case we have $(M_{K-1}\alpha)/\alpha= M_{K-1}$, and thus by Lemma \ref{lemma:ProblemBasicExampleLPObj} we obtain that $\Xi_{K-1}(\iota) = 0$. To analyse the output of the algorithm, note first that since $\matLPObj(\alpha,1,m,N)_{1,2} = 0$ by the definition of $\matLPObj$ we obtain $ A_{1,2}^{(2)} \leq 0 + 1/4 <3/4$ and hence the algorithm proceeds to step \ref{step:Smale9thWeak}. We have $A_{1,1}^{(n)} \geq  \alpha - 2^{-n} > -2^{-n}$ since $\alpha $ is positive. Hence the loop never terminates by outputting $1$. However, again using that $\alpha$ is positive, for sufficiently large $n$ we have $A_{1,1}^{(n)} \geq \alpha - 2^{-n} > 2^{-n}$ and hence the algorithm eventually outputs $0$.

{\bf Case 3:} The input $\iota = (\vecYl(0,m),\matLPObj(- \alpha,0,m,N))$ for some $\alpha \in (0,1]$. In this case, by Lemma \ref{lemma:ProblemBasicExampleLPObj} we obtain that $\Xi_{K-1}(\iota) = 1$. Our analysis of the output of the algorithm is similar to Case 2 and we once again proceed to step \ref{step:Smale9thWeak}. This time the loop never terminates by outputting $0$ since for all $n$ we have $ A_{1,1}^{(n)} \leq - \alpha + 2^{-n} \leq 2^{-n}$ because $\alpha $ is positive.  Instead for sufficiently large $n$ we have $A_{1,1}^{(n)} \leq - \alpha - 2^{-n} < -2^{-n}$ and hence the algorithm eventually outputs $1$. 

These are the only possible cases for $\iota$ and so the algorithm always returns $\Xi_{K-1}(\iota)$, as desired.

\subsection{Proof of Proposition \ref{prop:SmalesNinthProposition}, $M\notin \SetNine_{K-1} \,\implies\,$(iv)}
The assumption $M\notin \SetNine_{K-1}$ implies that $M_{K-1}<M_{K-2}$ and therefore, as $(M_K\alpha)/\alpha=M_K\leq M_{K-1}<M_{K-2}$ and $(M_{K-1}\alpha)/\alpha=M_{K-1}<M_{K-2}$, for all $\alpha\in(0,1]$, Lemma \ref{lemma:ProblemBasicExampleLPObj} implies that $\Xi_{K-2}(\iota)=1$ for all $\iota\in\Omega_{m,N}$. Our construction of an algorithm to compute $\Xi_{K-2}$ is thus trivial -- we simply create an algorithm that immediately outputs the value $1$.

\section{Geometry of solutions to problems  \eqref{problems3} - \eqref{problems4} -- Part II}

\subsection{A ${(5s-1)\times 5s}$ matrix with the RNP of order $s=2^n$ for non-unique minimisers of $\ell^1$ BP  and $\ell^1$ UL}

We construct a family of robust nullspace matrices which have a line segment of minimisers when performing basis pursuit denoising or unconstrained lasso with $\ell^1$ regularisation.

\begin{proposition}\label{prop:BPDNNFail}
	Fix a natural number $s=2^{n-1}$ and real numbers $\alpha,\gamma > 0$ with $\alpha > \gamma (\alpha^2+1)$. Then there is a matrix $A \in \real^{m \times N}$ with $m = 5s-1 $ and $N = 5s$ such that the following properties hold:
	\begin{enumerate}
		\item $\|AA^*\|_2 \leq \sqrt{\left(1+\frac{\alpha^2\gamma^2}{4}\right)^2 + \frac{2\alpha^2\gamma^4}{16} + \frac{\gamma^4}{16}},\quad  \|(AA^*)^{-1}\|_2 \leq  \sqrt{1+2\alpha^2 + \frac{1}{\gamma^4} \left( 4+\alpha^2\gamma^2 \right)^2}$.  \label{prop:BPDNNCond}
		\item $A$ obeys the robust nullspace property of order $s$ with parameters $\rho = 1/3$ and \\ $\tau =  \sqrt{\frac{4}{3}}\|A\|_2 \|(AA^*)^{-1}\|_2$. \label{prop:BPDNNNSP}
		\item 
		\label{prop:BPDNNMultiSolns} For any $\delta \geq 0$ there is a $y = Ax$ for some $s$-sparse $x$ with $\|y\|_2 \leq \delta \sqrt{1+\alpha^2}$, such that 
		\begin{equation*}
			\xibpdn(y,A)= \{t\left(c_1 \mathbf{1}_{2s} \oplus \mathbf{0}_{2s} \oplus c_2 \mathbf{1}_{s} \right)+(1-t)\left( \mathbf{0}_{2s} \oplus -c_1 \mathbf{1}_{2s} \oplus c_2 \mathbf{1}_{s}\right) \, \vert \, t \in [0,1]\},
		\end{equation*}
		where
		\begin{equation}\label{eq:3.11(i)-c1c2}
			c_1 = \frac{\delta}{\sqrt{s}} \frac{\left(-\gamma - \gamma \alpha^2 + \alpha\right) }{C} ,\quad 
			c_2 = \frac{2\delta}{\gamma\sqrt{s}} \left(\frac{\alpha\gamma  -1}{C} + 1\right), \quad 
			C = \sqrt{\gamma^2 + (1-\alpha\gamma)^2}.
		\end{equation} 
	\end{enumerate}  
\end{proposition}
\begin{proof}
	Let $H_n$ be the $2^n \times 2^n$ dimensioned Hadamard matrix in `natural ordering' (see \ref{sec:CSMatricesAppendix} for a definition). Set $m = 5s-1$ and $N = 5s$, and let $A$ be the $ \real^{m \times N}$ matrix defined by
	\begin{equation*}
		\begin{split}
			A &=  \frac{1}{2\sqrt{s}}\, 
			P \, \left(
			\begin{array}{c c}
				H_1 \otimes H_n & \eta \otimes H_{n-1}\\
				\mathbf{0}_{s \times 4s} & \gamma H_{n-1} 
			\end{array}
			\right), \quad \eta = \left(0,0,\alpha \gamma ,0\right)^T,\\
			Pe_j &= e_{j-1}, \, Pe_1 = 0,
		\end{split}
	\end{equation*}
	where $\{e_j\}$ represents the canonical basis, so that $P$ is the projection from $1,2,\dotsc,N$ to $2,3,\dotsc,N$.
	To show (\ref{prop:BPDNNCond}) we first compute
	\begin{equation*}
		\begin{split}
			&\left(
			\begin{array}{c c}
				H_1 \otimes H_n & \eta \otimes H_{n-1}\\
				\mathbf{0}_{s\times 4s} & \gamma H_{n-1} 
			\end{array}
			\right) 
			\left(
			\begin{array}{c c}
				H_1 \otimes H_n & \eta \otimes H_{n-1}\\
				\mathbf{0}_{s\times 4s} & \gamma H_{n-1} 
			\end{array}
			\right)^* \\
			& \qquad \qquad =
			\left(
			\begin{array}{c c}
				4s\,  I_{4s} + \eta \otimes \eta^* \otimes s I_{s} & \gamma (\eta \otimes H_{n-1})H_{n-1}^*\\
				\gamma ((\eta \otimes H_{n-1}) H_{n-1}^*)^* & \gamma^2  s I_{s}
			\end{array}
			\right)\\
			& \qquad \qquad = \left(
			\begin{array}{c c}
				4s\, I_{4s} + \alpha^2\gamma^2 \xi \otimes \xi^* \otimes sI_{s} & \alpha\gamma^2 \xi \otimes  sI_{s}\\
				\alpha\gamma^2 (\xi \otimes  sI_{s})^*& \gamma^2  s I_{s}
			\end{array}
			\right) =  4s\, (I_{2s} \oplus \left(S
			\otimes  I_{s}\right)), 
		\end{split}		
	\end{equation*}
	where $\xi = (0,0,1,0)^T$ and 
	\begin{equation}\label{S}
		S = \begin{pmatrix}
			(1+\frac{\alpha^2\gamma^2}{4}) & 0 & \frac{\alpha\gamma^2}{4} \\
			0 & 1 & 0 \\
			\frac{\alpha\gamma^2}{4} & 0 & \frac{\gamma^2}{4}
		\end{pmatrix}, 
		\quad S^{-1} = \begin{pmatrix}
			1 & 0 & -\alpha\\
			0 & 1 & 0 \\
			-\alpha & 0 & \frac{1}{\gamma^2}(4+\alpha^2 \gamma^2) 	
		\end{pmatrix}.
	\end{equation}
	Thus, by the definition of $P$ it follows that 
	\begin{equation}\label{AA}
		AA^* = I_{2s-1} \oplus \left(S
		\otimes  I_{s}\right),
		\quad 
		(AA^*)^{-1} = I_{2s-1} \oplus \left(S^{-1} \otimes I_{s} \right),
	\end{equation}
	and so to prove \eqref{prop:BPDNNCond} we only need to estimate $\|S\|$ and $\|S^{-1}\|$ from \eqref{S}. To do that, note that for a symmetric matrix $M\in\real^{2\times 2}$ we have $\|M\|_2\leq \|M\|_F=\left(M_{11}^2 + 2M_{12}^2 + M_{22}^2 \right)^{\frac{1}{2}}$, and we therefore find
	\[
	\|AA^{*}\|_2  \leq 1 \vee \sqrt{\left(1+\alpha^2\gamma^2/4\right)^2 + \alpha^2\gamma^4/8 + \gamma^4/16}, \quad
	\|(AA^*)^{-1}\|_2 \leq 1 \vee \sqrt{1+2\alpha^2 + \frac{1}{\gamma^4} \left( 4+\alpha^2\gamma^2 \right)^2}, 
	\]
	establishing \eqref{prop:BPDNNCond}.
	
	To prove (\ref{prop:BPDNNNSP}), take an arbitrary vector let $v \in \mathbb{R}^N$ and write $v = \xi + A^*w$ where $\xi\in \ker(A)$. Note that $\ker(A)$ consists exactly of vectors of the form $\beta (\ones_{4s} \oplus  \mathbf{0}_{s})$ , for $\beta\in\real$.
	Now let $K = \lbrace 1,2,\dotsc,4s\rbrace$ and consider a set $S\subset\{1,\dots,N\}$ of cardinality $s$. We then have $\|\xi_S\|_2 = \beta\sqrt{s}$ and $\|\xi_{S^c \cap K}\|_1 = \beta|S^c \cap K|$. 
	Consequently
	\begin{align*}
		\|\xi_S\|_2 &=   \frac{\sqrt{s}}{|S^c \cap K|} \|\xi_{S^c \cap K}\|_1 \leq \frac{\sqrt{s}}{|S^c \cap K|} \left(\|v_{S^c\cap K}\|_1 + \|(A^*w)_{S^c \cap K}\|_1\right)\\
		& \leq \frac{s\|v_{S^c}\|_1}{\sqrt{s} |S^c \cap K|} + \sqrt{\frac{s}{|S^c \cap K|}}\;   \frac{\|(A^*w)_{S^c \cap K}\|_1}{\sqrt{|S^c \cap K|}} \leq \frac{\rho\|v_{S^c}\|_1}{\sqrt{s}} +  \sqrt{\frac{s}{|S^c \cap K|}} \|(A^*w)_{S^c \cap K}\|_2,
	\end{align*}
	with $\rho = \frac{1}{3}$, where the last line follows because $|S^c \cap K| \geq 4s-s=3s$. Hence,
	\begin{equation*}\|v_S\|_2  \leq \frac{\rho\|v_{S^c}\|_1}{\sqrt{s}}  + \sqrt{\rho}\|(A^*w)_{S^c \cap K}\|_2 + \|(A^*w)_S\|_2 \leq  \frac{\rho\|v_{S^c}\|_1}{\sqrt{s}}  +  \sqrt{1+\rho} \|A^*w\|_2,
	\end{equation*}
	where the last inequality follows by applying the Cauchy-Schwartz inequality with the vectors $(\sqrt{\rho},1)$ and $(\|(A^*w)_{S^c \cap K}\|_2, \|(A^*w)_S\|_2)$.
	To bound $\|A^{*}w\|_2$, we see that 
	\begin{equation*}\|A^{*}w\|_2 \leq  \|A\|_2\|(AA^*)^{-1}(AA^*)w\|_2 \leq \|A\|_2\|(AA^*)^{-1}\|_2 \|AA^*w\|_2 = \|A\|_2 \|(AA^*)^{-1}\|_2\|Av\|_2
	\end{equation*}
	where the last inequality is due to $Av = AA^*w + A\xi = AA^*w$ (recall that $\xi\in\ker(A)$).
	We conclude that
	\begin{equation*}
		\|v_S\|_2   \leq  \frac{\rho\|v_{S^c}\|_1}{\sqrt{s}}  +  \sqrt{1+\rho} \|A^*w\|_2   \leq  \frac{\rho\|v_{S^c}\|_1}{\sqrt{s}}  +  \tau\|Av\|_2.
	\end{equation*}
	where  $\tau  = \sqrt{\frac{4}{3}} \|A\|_2 \|(AA^*)^{-1}\|_2$. This establishes (2).

	We next prove (\ref{prop:BPDNNMultiSolns}) with the vector $y = \mathbf{0}_{2s-1}\oplus  \, \alpha\delta\,  \oplus\,   \mathbf{0}_{2s-1} \oplus \, \delta \, \oplus \,  \mathbf{0}_{s-1}=Ax$, where
	$
	x = \frac{2 \delta}{\gamma\sqrt{s}}\, \left(\mathbf{0}_{4s}  \oplus   \mathbf{1}_{s} \right)
	$ 
	is $s$-sparse. A simple calculation shows that $\|y\|_2 \leq \delta \sqrt{1+\alpha^2}$. With the aim of proving (\ref{prop:BPDNNMultiSolns}), define $x^{opt}=\frac{c_1}{2} \ones_{2s} \oplus \left(-\frac{c_1}{2} \ones_{2s} \right)\oplus c_2 \ones_{s}$ and note that then
	$
	A x^{opt}-y=\frac{ \delta}{C}\left( \mathbf{0}_{2s-1} \oplus (-\gamma) \oplus \mathbf{0}_{2s-1} \oplus (\alpha\gamma-1)\oplus \mathbf{0}_{s -1} \right).
	$ 
	Thus $\|A x^{opt}-y\|_2=\delta$, and so $x^{opt}$ is a feasible point for the BP denoising problem. 
	Next, we define a dual vector 
	\[
	p=\frac{2C\sqrt{s}}{\delta \gamma}(A x^{opt}-y)= 2\sqrt{s}\big( \mathbf{0}_{2s-1} \oplus (-1) \oplus \mathbf{0}_{2s-1} \oplus (\alpha-\frac{1}{\gamma})\oplus \mathbf{0}_{s -1}\big) \in\real^m,
	\] 
	and note that then by some simple calculations $\langle Ax^{opt} - y,p \rangle=\delta \|p\|_2$ and 
	$-A^* p= \mathbf{1}_{2s}\oplus(-\mathbf{1}_{2s}) \oplus \mathbf{1}_{s}$. It is easily seen that $\alpha\gamma -1 + C >0$ and hence $c_2 > 0$. Moreover $\alpha > \gamma(\alpha^2+1)$ by an assumption in the statement of the proposition so that $c_1 > 0$. Hence $-A^*p \in \partial\|\cdot\|_1(x^{opt})$. Therefore, for arbitrary $z\in \real^N$ satisfying $\|Az-y\|_2\leq \delta$, we have
	\begin{equation}\label{eq:oldnew-dual-calc-BP}
		\begin{aligned}
			\|z\|_1 \;  \stackrel{\mathclap{C-S}}{\geq } &\; \|z\|_1 + \langle Az- y, p \rangle -\delta\|p\|_2= \|z\|_1 + \langle Az- y, p \rangle -\langle Ax^{opt}- y, p \rangle\\
			=&\|z\|_1 - \langle z- x^{opt} , -A^*p \rangle \geq \|x^{opt}\|_1.
		\end{aligned}
	\end{equation}
	As $x^{opt}$ is feasible for the BP denoising problem, we deduce that $z$ is a minimiser if and only if the inequalities in \eqref{eq:oldnew-dual-calc-BP} hold as equalities. This is the case if and only if $Az-y = \delta p/\|p\|_2=Ax^{opt} - y$, $z_j\geq 0$ for $j\in\{1,\dots, 2s\}\cup \{4s+1,\dots, N\}$, as well as $z_j\leq 0$ for $j\in\{2s+1,\dots, 4s\}$. In particular  $A(x^{opt}-z)=0$, and so, recalling that $\ker(A)= \{ \beta (\mathbf{1}_{4s} \oplus  \mathbf{0}_{s})\, \vert\, \beta\in\real \}$, we must have $x^{opt}-z = \beta\,(  \mathbf{1}_{4s} \oplus  \mathbf{0}_{s})$, for some $\beta\in\real$. It now follows that the minimisers are exactly as claimed in the statement of the proposition.
\end{proof}

\subsection{Perturbing $\ell^1$ BP with several minimisers}

Before stating the main result of this section, we need to introduce the concept of the set of minimisers with minimal support. 

\begin{definition}[Minimisers with minimal support]  
	Let $\{\Xi,\Omega\}$ denote the computational problem of either basis pursuit or unconstrained lasso with $\ell^1$ regularisation. For feasible inputs $(y,U)\in\Omega$ we define the set of minimisers with minimal support by
	\begin{equation}\label{eq:def-min-supp}
		\Xi^{\text{ms}}(y,U):= \lbrace x \in \Xi(y,U) \, | \, \forall  x' \in \Xi(y,U), \, \supp(x') \subset \supp(x) \Rightarrow x=x' \rbrace.
	\end{equation}
\end{definition}

\begin{lemma}\label{lem:PeturbToGetEitherMin}
	Let $\{\Xi,\Omega\}$ denote the computational problem of  basis pursuit with $\ell^1$ regularisation with dimensions $m$ and $N$.  Suppose that $(y,U) \in \Omega$ is input data such that $\Xi^{\text{ms}}(y,U)$ contains two distinct points $x^1$ and $x^2$. Then, for every $\epsilon\in(0,1)$, there exist positive semidefinite diagonal matrices $E^1=E^1(\epsilon)$ and $E^2=E^2(\epsilon)$ with $\|E^1\|_\infty\leq \epsilon$ and $\|E^2\|_\infty\leq \epsilon$ such that $\Xi(y,U(I-E^1))=\{x^1\}$ and $\Xi(y,U(I-E^2))=\{x^2\}$.
\end{lemma}

\begin{proof}[Proof of Lemma \ref{lem:PeturbToGetEitherMin}]
	We define the $N \times N$ diagonal positive semidefinite matrices 
	\begin{align*}
		E^j(\epsilon) :=  \mathrm{diag}(\epsilon\,\indic_{\{1 \notin \supp(x^j)\}}, \epsilon\,\indic_{\{2 \notin \supp(x^j)\}} , \dotsc,\epsilon\,\indic_{\{N \notin \supp(x^j)\} }),\quad j=1,2\, 
	\end{align*}
	$\indic_{\{ i \notin \supp(x^j) \} }$ is $1$ if $i \notin \supp(x^j)$ and $0$ otherwise.
	We proceed to show that $x^j$ is the unique vector in $\xi(y,U-UE^j)$, for $j=1,2$. It suffices to argue for $j=1$, as the proof for $j=2$ is analogous. 
	
	Firstly, as $(U -  UE^1)v = Uv$ for $v\in\real^N$ with $\supp(v) = \supp(x^1)$, we have $\|(U - UE^1) x^1 - y\|_2 \leq \delta$. Let us suppose that $\tilde{x}^1\in\real^N$ is a vector such that $\|(U-UE^1) \tilde{x}^1 - y\|_2 \leq \delta$ and $\|\tilde{x}^1\|_1 \leq \|x^1\|_1$. Set $\hat{x}^1 = \tilde{x}^1- E^1\tilde{x}^1$. Then $\|U \hat{x}^1 - y\|_2 \leq \delta$, and, for every $k \in \supp(\hat{x}^1)$ with $k \notin \supp(x^1)$ we have $|\hat{x}^1_k| = (1-{\epsilon}) |\tilde{x}^1_k| < |\tilde{x}^1_k|$, whereas for $k \in \supp(x^1)$ we have $\hat{x}^1_k = \tilde{x}^1_k$. Therefore, unless $\supp(\hat{x}^1)\subset\supp(x^1)$, we have $\|\hat{x}^1\|_1 < \|\tilde{x}^1\|_1 \leq \|x^1\|_1$, contradicting the fact that $x^1 \in \xibpdn (y,U)$. Hence $\supp(\hat{x}^1)\subset\supp(x^1)$ and so $ \hat{x}^1 =x^1$, by definition of $\Xi^{\text{ms}}(y,U)$. We deduce that $ \tilde{x}^1 =x^1$, and so $x^1$ is the unique vector in $\xibpdn (y,U-UE^1)$.
\end{proof}

\section{Proof of Theorem \ref{th:smale_comp_sens}: part (i)}
Parts (i) of Theorem \ref{th:smale_comp_sens} is formally stated in Proposition \ref{prop:CSResult}, which we now prove. We choose the constant $C$ mentioned in Proposition \ref{prop:CSResult} to be $C=(C_5+3)\vee 5$ where $C_5$ is the universal constant in Theorem \ref{thm:NSPExistenceDCT}.

\subsection{Proof of Proposition \ref{prop:CSResult} part (i)}
\begin{proof}[Proof of Proposition \ref{prop:CSResult} part (i)]
	We will establish this part using Proposition \ref{prop:DrivingNegativeProposition}. 
	Concretely, we will construct input sequences $\{\iota^1_n =(y^0,A^{1,n})\}_{n=1}^{\infty}\subset \Omega^{0}_{s,m,N}$, $\{\iota^2= (y^0,A^{2,n})\}_{n=1}^{\infty}\subset \Omega^{0}_{s,m,N}$ and an input $\iota^0=(y^0, A^0) \in \Omega^{0}_{s,m,N}$ such that the following hold: 
	\begin{itemize}
		\item[(i)] there exists $x^1,x^2\in\real^N$ such that $\Xi(\iota^1_n) =\{x^1\}$ and $\Xi(\iota^2_n) =\{x^2\}$, for all $n\in\mathbb{N}$, and $\|x^1 - x^2\|_{2} > \delta$.
		\item[(ii)] the inputs satisfy
		\begin{equation}\label{eq:ShowCloseMatVector-thm3.11} 
			\|A^{j,n}-A^0\|_{\max} \leq 4^{-n},
		\end{equation}
		for all $n\in\mathbb{N}$ and $j\in\{1,2\}$.
	\end{itemize}
	Once we have done this, the result will follow by applying item (iii) of Proposition \ref{prop:DrivingNegativeProposition} to obtain $\epsilon_{\mathbb{P}h\mathrm{B}}^{\mathrm{s}}(\mathrm{p}) > \delta/2 \geq 10^{-K}$. 
	Proposition \ref{prop:BPDNNFail} comes close to constructing a suitable $A^0$. However, the dimensions used in Proposition \ref{prop:BPDNNFail} are fixed as soon as $s$ is chosen. To construct the desired matrix $A^0$ of dimensions $m$ and $N$ satisfying  $N> m$ and $m\geq Cs\log^2(2s)\log(N)$, but otherwise arbitrary, we will concatenate a matrix provided by Proposition \ref{prop:BPDNNFail} and a matrix satisfying the nullspace property provided by Theorem \ref{thm:NSPExistenceDCT}.
	
	We begin by applying Proposition \ref{prop:BPDNNFail} with the constants $\alpha = 1.4$ and $\gamma = 0.37$ to obtain $A\in \real ^{(5s-1)\times 5s}$, $y \in \real^{5s-1}$ and $x \in \real^{5s}$ such that the following hold.
	\begin{enumerate}[label= \alph*)]
		\item By item (1) of Proposition \ref{prop:BPDNNFail} as well as the choices of $\alpha$ and $\gamma$, both $\|AA^*\|_2 < 31.3$ and $\|(AA^*)^{-1}\|_2 < 1.2$. In particular, $\|A\|_2 \leq 6$.\label{item:smaleCompSens:MatrixSize}
		\item By item (2) of Proposition \ref{prop:BPDNNFail} $A$ satisfies the $\ell^2$-RNP of order $s$ with parameters $\rho_A$ and $\tau_A$, where	
		\begin{equation*}
			\rho_A:=1/3<\rho\quad\text{and}\quad \tau_A:=\sqrt{4/3}\|A\|_2 \|(AA^*)^{-1}\|_2\leq 1.2\cdot 6\cdot 1.2<9.
		\end{equation*}
		\item By item (3) of Proposition \ref{prop:BPDNNFail}, we have $y = Ax$ with $x$ an $s$-sparse vector. Furthermore, $\|y\|_2 \leq \delta \sqrt{1+\alpha^2} \leq 2$ where we have used the choice of $\alpha$ and the bound $\delta \leq 1$ assumed in the statement of Theorem \ref{th:smale_comp_sens}. \label{item:smaleCompSens:VectorSize}
		\item Again, using item (3) of Proposition \ref{prop:BPDNNFail}, $
		\Xi^{\text{ms}}(y,A)=\{\hat{x}^1,\hat{x}^2\}
		$
		where
		\begin{equation*}
			\hat{x}^1=c_1 \mathbf{1}_{2s} \oplus \mathbf{0}_{2s} \oplus c_2 \mathbf{1}_{s},\qquad \hat{x}^2=\mathbf{0}_{2s} \oplus -c_1 \mathbf{1}_{2s} \oplus c_2 \mathbf{1}_{s},
		\end{equation*}
		and the constants $c_1$ and $c_2$ are as given in \eqref{eq:3.11(i)-c1c2}. \label{item:smaleCompSens:MinSupp}
	\end{enumerate}

	Next, since $s \geq 2$ and $N > m \geq 5s \geq 10$ we must have 
	\[
	5s-1 \leq 5s \leq \frac{5s \log(N)\log[s\log(N)]\log^2(s)}{\log(10)\log[2\log(10)]\log^2(2)} \leq 3s\log(N)\log[s\log(N)]\log^2(s)
	\]
	so that $m - (5s-1) \geq (C-3)s\log(N)\log[s\log(N)]\log^2(s)$. In particular, since $(C-3)\geq C_5$ and $(C-3) \geq 2$ we have 
	\begin{align*}m - (5s-1) &\geq C_5s\log(N)\log[s\log(N)]\log^2(s) \\ (N-5s) \geq m - (5s-1) &\geq 4\log(10)\log[2\log(10)]\log^2(2) > 3 .\end{align*}
	
	Therefore we can apply Theorem  $\ref{thm:NSPExistenceDCT}$ to conclude there exists a matrix $F\in\real^{(m-(5s-1))\times (N-5s)}$ such that $\|F\|_2 \leq \sqrt{N/m}$ and $F$ obeys the robust nullspace property with parameters $(\rho_F,\tau_F)$ satisfying $\rho_F< 1/3$ and $\tau_F < 2$.
	
	At this stage we can finally define $\iota^0, \iota^1_n$ and $\iota^2_n$: using Lemma \ref{lem:PeturbToGetEitherMin} and \ref{item:smaleCompSens:MinSupp}, we can find sequences of diagonal positive semidefinite matrices $\{E^{1,n}\}_{n=1}^\infty$ and $\{E^{2,n}\}_{n=1}^\infty$, with $\|E^{1,n}\|_{\max},\|AE^{1,n}\|_{2} \leq \gamma_n$ and $\|E^{2,n}\|_{\max},\|AE^{2,n}\|_{2} \leq \gamma_n$ 
	such that $\Xi(y,A(I-E^{1,n}))=\{\tilde{x}^1\}$ and $\Xi(y,A(I-E^{2,n}))=\{\tilde{x}^2\}$ where the positive real numbers $\gamma_n$ are chosen so that 
	\begin{equation*}
		0 < \gamma_n \leq \left(4^{-n} \wedge \frac{2}{3[10 (\sqrt{s}+1)]} \wedge \frac{3\rho -1}{30(\sqrt{s}+\rho)} \wedge \frac{\tau - 10}{10\tau}\right).
	\end{equation*} We define the matrices $A^0=A\oplus F$, $A^{1,n}=A(I-E^{1,n})\oplus F$, $A^{2,n}=A(I-E^{2,n})\oplus F \in \real^{m\times N}$ as well as the vector $y^0=y\oplus \mathbf{0}_{m-(5s-1)}\in\real^m$,
	and set $\iota^0=(y^0,A^0)$, $\iota^j_n=(y^0,A^{j,n})$ for $j\in\{1,2\}$, $n\in\mathbb{N}$.

	Next, we establish that $\iota^0$ is an element of $\Omega^{0}_{s,m,N}$. Firstly, note that $\|A^0\|_2 \leq \|A\| \vee \|F\|_2 \leq 6\sqrt{N/m}\leq b_2 \sqrt{N/m}$ where we have used the bound $\|F\|_2 \leq \sqrt{N/m}$ from Theorem \ref{thm:NSPExistenceDCT} and the bound $\|A\|_2 \leq 6$ from \ref{item:smaleCompSens:MatrixSize}. Set $x^0 = x \oplus 0_{N-5s}$ where $x$ is taken from \ref{item:smaleCompSens:VectorSize}. Note that $A^0 x^0 = y \oplus 0_{m-(5s-1)} = y^0$ so that $y^0 = A^0x^0$ with $x^0$ an $s$-sparse vector. 
	
	The last condition that we will check to ensure that $\iota^0 \in \Omega^{0}_{s,m,N}$ is the robust nullspace property. Let $v$ be an arbitrary vector in $\real^N$ and write $v=v^A\oplus v^F\in\real^N$ with $v^A \in\real ^{5s}$ and $v^F \in\real^{N-5s}$. Next, let $S$ be a subset of $\{1,\dots,N\}$ of cardinality $s$, and write $S_A=S\cap \{1,\dots, 5s\}$, $S_F=S\cap\{5s+1,\dots, N\}$, $S_A'=\{1,\dots, 5s\}\setminus S_A$, and $S_F'=\{5s+1,\dots, N\}\setminus S_F$. Then, as $A$ (respectively $F$) satisfies the $\ell^2$-RNP with constants $\rho_A$ and $\tau_A$ (respectively $\rho_F$ and $\tau_F$), we have
	\begin{equation*}
		\begin{aligned}
			\|v_S\|_2 & \leq \|(v^A)_{S_A}\|_2+\|(v^F)_{S_F}\|_2
			\leq \frac{\rho_A}{\sqrt{{s}}} \|(v^A)_{S_A'}\|_1 + \tau_A \|A v\|_2+  \frac{\rho_F}{\sqrt{{s}}} \|(v^F)_{S_F'}\|_1 + \tau_F \|F v\|_2 \\
			&\leq  \frac{\rho_{A}\vee \rho_F}{\sqrt{{s}}}\left( \|(v^A)_{S_A'}\|_1 +  \|(v^F)_{S_F'}\|_1 \right) +  \sqrt{ \tau_A^2 + \tau_F^2 }\sqrt{ \|A v^A\|_2^2+ \|F v^F\|_2^2} \\
			&\leq  \frac{\rho_{A^0}}{\sqrt{{s}}} \|v_{S^c}\|_1 +\tau_{A^0} \|A^0 v\|_2,
		\end{aligned}
	\end{equation*}
	where $\rho_{A^0}:=\rho_A\vee \rho_F<1/3$ and $\tau_{A^0}:=\sqrt{ \tau_A^2 + \tau_F^2 }<\sqrt{9^2+2^2}<10$, and so $A^0$ satisfies the $\ell^2$-RNP of order $s$ with parameters $1/3$ and $10$ (and in particular, $A^0$ also satisfies the $\ell^2$-RNP with parameters $\rho$ and $\tau$).

	Next, we argue that $\iota^j_n \in \Omega^{0}_{s,m,N}$, where $n \in \mathbb{N}$ and $j=1$ or $j=2$. We have $\|A^{j,n}\|_2 \leq \|A^0\|_2\|(I - E^{j,n})\|_2 < \|A^0\|_2 < b_2 \sqrt{N/m}$ since $E^{j,n}$ is a positive semidefinite diagonal matrix with entries bounded above by $\gamma_n$ and $\gamma_n < 1$. Setting $x^{j,n} = ((I-E^{j,n})^{-1}x) \oplus 0_{N-5s}$ where $x$ is taken from \ref{item:smaleCompSens:VectorSize} gives $y^0 = A^{j,n} x^{j,n}$ for each $n$. Because $I - E^{j,n}$ is a diagonal matrix we conclude that $y^0 = A^{j,n} x^{j,n}$ and that $x^{j,n}$ is an $s$-sparse vector.

	All that remains to prove that $\iota^j_n \in \Omega^{0}_{s,m,N}$ is to check that the robust nullspace property is satisfied with parameters $\rho$ and $\tau$. We have already shown that $A^0$ satisfies the RNP with parameters $(1/3,10)$. We then use Lemma \ref{lemma:RNSPIsStableProperty} as well as the definition of $\gamma_n$ to see that $A^{j,n}$ obeys the robust nullspace property with parameters $(\rho^j_n,\tau^j_n)$ where
	\begin{equation*}
		\rho^j_n = \frac{\frac{1}{3} + 10 \gamma_n \sqrt{s}}{1-10 \gamma_n} \leq \frac{\frac{1}{3} + 10\sqrt{s} \left(\frac{3\rho - 1}{30(\sqrt{s}+\rho)}\right) }{1-10 \left(\frac{3\rho - 1}{30(\sqrt{s}+\rho)}\right) } = \rho,\quad 	\tau^j_n = \frac{10}{1-10\gamma_n} \leq \frac{10}{1-\frac{10(\tau - 10)}{10\tau}} = \tau
	\end{equation*}
	
	We have therefore shown that $\iota^0\in\Omega^0_{s,m,N}$ as well as $\iota^j_n\in\Omega^0_{s,m,N}$, for $j\in\{1,2\}$ and $n \in \mathbb{N}$. Furthermore, by the definition of $\gamma_n$, $A^{j,n}$ and $E^{j,n}$ we have $\|A^{j,n} - A^0\|_{\max} \leq 4^{-n}$. By \ref{item:smaleCompSens:MinSupp}, the block diagonality of $A^{j,n}$ and the fact that the last $m-(5s-1)$ entries of $y^0$ are zero, it follows that the solution to the $\ell^1$ basis pursuit problem with input $\iota^j_n\in\Omega^0_{s,m,N}$ is the unique point $x^j:=\hat{x}^j\oplus \mathbf{0}_{N-5s}$, for $j\in\{1,2\}$ and $n\in\mathbb{N}$. Now, recalling  \eqref{eq:3.11(i)-c1c2} and our choice $\alpha=1.4$ and $\gamma=0.37$, we find
	\begin{equation*}
		\|x^1-x^2\|_2=\|\hat{x}^1-\hat{x}^2\|_2=\|c_1 \mathbf{1}_{2s} \oplus c_1\mathbf{1}_{2s} \oplus \mathbf{0}_{s}\|_2=c_1\sqrt{4s}=2\delta\cdot \frac{-\gamma -\gamma\alpha^2 + \alpha}{\sqrt{\gamma^2+(1-\alpha\gamma)^2}}>\frac{2\delta}{2}=\delta.
	\end{equation*}
	We have thus verified the conditions of Proposition \ref{prop:DrivingNegativeProposition} and so we deduce that $\epsilon^{\mathrm{s}}_{\mathrm{B}} \geq \strbdepsph \geq \delta/2 $ for $\mathrm{p} \in [0,1/2)$, completing the proof of part (i) of Proposition \ref{prop:CSResult}.

\end{proof}

\section{Proof of Theorem \ref{th:smale_comp_sens}: part (ii)}

\subsection{Preliminaries - results from the literature}

The ellipsoid algorithm will be the main driver in the proof of part (ii) of Theorem \ref{th:smale_comp_sens}. Even though it is well known that this algorithm neither has the optimal theoretical complexity nor performs well in practice, it is a powerful and flexible tool to establish the ``in P'' statements rigorously.

We start by presenting the framework necessary to implement the ellipsoid algorithm in the context of $\ell^1$ BP as necessitated in the proof of the theorem. We refer the reader to Appendix \ref{appendix:ellipsoid} for detailed definitions, and here we only list the notation assuming familiarity with the underlying concepts.
For a compact convex set $\CompactK\subset\real^n$ and real $\zeta>0$, write $S(\CompactK,-\zeta)=\{x\in\real^n\, \vert \, \clBall{\zeta}{x}\subset \CompactK\}$ and $S(\CompactK,\zeta)=\{x\in\real^n \, \vert \, \clBall{\zeta}{x} \cap \CompactK\neq \varnothing\}$. For the same $\CompactK$, we write $\mathrm{SEP}_\CompactK$ for the weak separation oracle for $\CompactK$.

The theorems we will cite can only be stated for compact convex sets that are contained within a ball of known radius $R$, so we introduce the following definition for conciseness.

\begin{definition}[Circumscribed class]
	We call a set $\mathscr{K}$ a \emph{circumscribed class} if it is  a subset of 
	\begin{equation*}
		\{(\CompactK,n,R)\,\vert\, n\in\mathbb{N}, R\in\mathbb{Q}\cap (0,\infty), \CompactK\text{ compact and }\CompactK\subset \clBall{R}{0}\subset \real^n \}.
	\end{equation*}
\end{definition}

We state the following result which combines {\cite[Cor. 4.2.7]{Lovasz_book}} on the existence of an `oracle-polynomi\-al' algorithm for the weak optimisation problem with the observation on pages 101 and 102 of \cite{Lovasz_book} that if a separation oracle for a compact convex set $\CompactK$ can be implemented on a Turing machine whose runtime is polynomial in $\DataTur(\CompactK,n,R)$, then `oracle-polynomial' implies that the weak optimisation problem can be solved in time polynomial in $\DataTur(\CompactK,n,R)$ as well.

\begin{theorem}[{\cite[Cor. 4.2.7]{Lovasz_book}} Ellipsoid algorithm , Turing case] \label{thm:Ellipsoid}
	Suppose $\mathscr{K}$ is a circumscribed class equipped with a Turing encoding function $\DataTur:\mathscr{K}\to\Alphabet^*$ (as in Definition \ref{def:EncFun}), and assume that $\mathscr{K}$ is Turing-polynomially separable with respect to $\DataTur$ (according to Definition \ref{def:Tur-poly-sep-class}). Then there exists a Turing machine that takes in $\DataTur(\CompactK,n,R) \in\Alphabet^*$, for $(\CompactK,n,R)\in\mathscr{K}$, a $c\in \mathbb{Q}^n$, and a rational $\zeta>0$ as input and solves the weak optimisation problem $(\CompactK,R,c,\zeta)$ (see Definition \ref{def:WoptProb} for the definition of a weak optimisation problem), such that the runtime of the Turing machine is bounded by a polynomial in \begin{equation*}
		\length\big(\DataTur(\CompactK,n,R)\big),\; \length(R),\;  \length(c),\; \length(\zeta),\; \text{and}\;  n.
	\end{equation*}
\end{theorem}

The other result we will need is an analogue of Theorem \ref{thm:Ellipsoid} in the BSS model for convex bodies that are specified as the intersection of the level sets of several convex functions. This will be sufficient for our needs. 

\begin{theorem}[{\cite[Thm. 1.1]{Nemirovski1995}} Ellipsoid algorithm, BSS case]\label{thm:Ellipsoid-BSS}
	Suppose $\mathscr{K}$ is a circumscribed class so that, for every $(\CompactK,n,R)\in\mathscr{K}$, there exist convex functions $f_1^{\CompactK},\dots, f_{M(\CompactK)}^\CompactK:\real^n \to \real$ such that $ \CompactK=\{ z\in \real^n \, \vert \, f_j^\CompactK(z)\leq 0, j=1,\dots,M(\CompactK)\}$. Furthermore, assume that $\mathscr{K}$ is equipped with a BSS encoding function $\DataBSS:\mathscr{K}\to\InputVecsBSS$ (as in Definition \ref{def:EncFunBSS}) and assume that $\mathscr{K}$ is BSS-polynomially separable with respect to $\DataBSS$ (according to Definition \ref{def:BSS-poly-sep-class}). Then there exists a BSS machine that takes in $\DataBSS(\CompactK,n,R) \in\InputVecsBSS$, for $(\CompactK,n,R)\in\mathscr{K}$, a $c\in \mathbb{R}^n$, and a real $\zeta>0$ as input and solves the weak optimisation problem $(\CompactK,R,c,\zeta)$, such that the runtime of the BSS machine is bounded by a polynomial in 
	\begin{equation*}
		\dim\big(\DataBSS(\CompactK,n,R)\big)\quad\text{and}\quad  \log \big( 2+ V(\CompactK,n,R)/\zeta\big),
	\end{equation*}
	where the scale factor $V(\CompactK,n,R)$ is defined according to
	\begin{equation*}
		V(\CompactK,n,R)=\max_{1\leq j\leq M(\CompactK)} \left[ \max_{\|z\|_2\leq R} f_j^\CompactK(z) - \min_{\|z\|_2\leq R} f_j^\CompactK(z)\right]\vee R\|c\|_2.
	\end{equation*}
\end{theorem}

\subsection{Proof of part (ii) of Proposition \ref{prop:CSResult}}
Part (ii) of Theorem \ref{th:smale_comp_sens} is formally stated as part (ii) of Proposition \ref{prop:CSResult}, which we now prove. To do so, we will make use of are Theorems \ref{thm:Ellipsoid} and \ref{thm:Ellipsoid-BSS} on the complexity of the ellipsoid algorithm. In addition, we will make use of the standard compressed sensing result Theorem \ref{l2RNPerrorest} and the easy Lemma \ref{lemma:NSPImplicationBoundingSparse} stated in Appendix \ref{appendix:sparsity}.

As in \S\ref{sec:UsefulSubroutinesForMainThm}, we fix the notation for an element $\tilde\iota$ of $\tilde\Omega$, writing   $\tilde{\iota}=\big(\{y_j^{(n)}\}_{n=0}^{\infty}, \{A_{j,k}^{(n)}\}_{n=0}^{\infty}\big)_{j,k}$, corresponding to an $\iota=(y,A)\in \Omega^{\epsilon}_{s,m,N}$.

\newcommand{\AllOraOnInp}{\mathcal{I}}
\begin{proof}[Proof of Proposition \ref{prop:CSResult} part (ii).]

	To start with, we mention the following quantities that the algorithm stores in a lookup table
	\begin{flalign} 
		L_1 &= \left \lceil \log_2\left(\frac{16\tau}{1-\rho}\right) \right\rceil,\quad L_2 = \ceil{\log_2(b_2)}, \quad L_3 = 4\tau b_1
	\end{flalign}
	The algorithm computes the following quantities, which (for the convenience of the reader) are listed alongside a brief description of their purpose:{\allowdisplaybreaks
		\begin{flalign} 
			n_1&=\left\lceil\frac{\length(m)}{2}\right\rceil+ 4K+5+L_1&&\text{\parbox{7cm}{The number of digits of precision for the input vector requested from the oracle}}\notag\\[3mm]
			n_2&=\left\lceil\frac{\length(N)+\length(m)}{2}\right\rceil+1 + L_1 &&\text{\parbox{7cm}{The number of digits of precision for the input matrix requested from the oracles.}}\notag\\[3mm]
			n_3&= \lceil \length(N) \rceil + 4K +6 + L_1 +L_2 && \text{\parbox{7cm}{The solution precision for the reformulated optimisation problem \eqref{eq:reform1}.}}\notag\\[3mm]
			n_4&= \left\lceil\frac{\length(N)}{2}\right\rceil + K + 3  && \text{\parbox{7cm}{The precision used when converting algorithm output to a dyadic vector.}}\notag\\[3mm]
			\delta'&=\frac{10^{-K}(1-\rho)}{32\tau}+2^{ -4K-3-L_1} + 4b_2N\,2^{-n_3} &&\text{\parbox{7cm}{The denoising parameter used in \eqref{eq:reform1}}} \notag\\[3mm]
			R&=NL_3 &&\text{\parbox{6cm}{Used as a bound on solutions of the reformulated optimisation problem.}} \label{eq:ConstantsDefinitions}
	\end{flalign}}
	Note that these quantities are functions of $m$, $N$, $K$, however we suppress making this dependence explicit to lighten the notation.

	For $\iota = (y,A)\in\Omega_{s,m,N}^\varepsilon$, instead of directly solving the basis pursuit problem with $\ell^1$ regularisation, our algorithm will be based on the solution of the following reformulated optimisation problem:
	\begin{equation}\label{eq:reform1}
		\begin{aligned}
			\min_{z\in\real^{2N}}{\langle\mathbf{1}_{2N},z\rangle}\quad\text{s.t.}  \quad &z\geq \mathbf{0}_{2N},\,\|\hat{A} z -y'\|_2\leq \delta', \, \|z\|_2\leq R,
		\end{aligned}
	\end{equation}
	where $\hat{A}=(A'\;-\!A')\in\real^{m\times 2N}$, and $y'$ and $A'$ are approximations to $y$, $A$. More formally, for $\iota=(y,A)\in\Omega_{m,N}^\varepsilon$, we write 
	\begin{equation}\label{eq:iota-UnionOverAllOracles}
		\AllOraOnInp(\iota):=\left\{(y',A')\in\real^{m}\times\real^{m\times N}\,\vert\, \|y'-y\|_2 \leq 2^{-n_1}, \|A'-A\|_{\max} \leq 2^{-n_2}   \right\}.
	\end{equation}
	Next, for $\iota\in\Omega_{m,N}^\varepsilon$ and $(y',A') \in \AllOraOnInp(\iota)$, define the compact convex set $\CompactK_{y',A'}=\{z\in\real^{2N}\,\vert\, z\geq \mathbf{0}_{2N}, \|\hat{A} z-y'\|_2\leq \delta', \|z\|_2\leq R \}$, where $\hat{A}=(A'\;-\!A')$, and set 
	\begin{equation*}
		\mathscr{K}=\{(\CompactK_{y',A'},2N,R)\,\vert\, m,N\in \mathbb{N}, m\leq N, \iota\in\Omega_{m,N}^\varepsilon, (y',A') \in \AllOraOnInp(\iota)\}.
	\end{equation*}\
	As each of the $\CompactK_{y',A'}$ are compact convex sets with $\CompactK_{y',A'} \subset \clBall{R}{0} \subset \real^{2N}$, $\mathscr{K}$ is a circumscribed class.
	
	An important observation that will be useful throughout is the following: the bound on $K$ in \eqref{eq:BoundOnKForPositive} and the assumption that $16\tau\delta \leq (1-\rho)$ yields  \begin{equation}\label{eq:CSPosEpsBound}
		10^{-K}\geq \frac{16\tau(\delta+\epsilon)}{1-\rho} \geq \frac{32\tau\epsilon}{1-\rho}. 
	\end{equation}
	We are now in a position to precisely define our algorithm, where the instructions whose labels end with a `T' are only executed by the Turing model and those whose labels end with a `B' are only executed by the BSS model: 
	
	{\it Algorithm for basis pursuit with $ \ell^1$ regularisation:}  
	
	\noindent Inputs: Dimensions $m,N$ and an accuracy parameter $K$ with $(m,N,K)\in\mathbb{N}^3$.\\
	Oracles: $\orvec$ and $\ormat$ providing access to the components $y_j^{(n)}$ and $A_{j,k}^{(n)}$  of an input $\tilde{\iota} \in \tilde \Omega^{\epsilon}_{s,m,N}$.\\
	Output: $x_{\mathrm{out}} \in \mathbb{D}^N$ (in the Turing case) or $x_{\mathrm{out}} \in \mathbb{R}^N$ (in the BSS case) with $\disM(x_{\mathrm{out}},\tilde \Xi_{\mathrm{BPDN}}(\tilde\iota))\leq 10^{-K}$.
	
	\begin{enumerate}[leftmargin=10mm, label= \arabic*.,ref = step \arabic*]
		\item [1T.] Using the inputs $m,N$ and $K$ as well as the stored values $L_1,L_2$ and $L_3$, compute the quantities $n_1$, $n_2$, $n_3$, $n_4$, $\delta'$ and $R$ in \eqref{eq:ConstantsDefinitions}. We encode the integers $n_1,\dotsc n_4$ using their dyadic expansions. The rational numbers $\delta'$ and $R$ are stored as pairs of integers each encoded using their dyadic expansions.
		\item [1B.] Using the inputs $m,N$ and $K$ as well as the stored values $L_1,L_2$ and $L_3$, compute the quantities $n_1$, $n_2$, $n_3$, $n_4$, $\delta'$ and $R$ where we encode $n_1,\dotsc n_4$ as integers and $\delta',R$ as real numbers. \addtocounter{enumi}{1}
		\item Call the oracles to obtain $(y',A')\in \AllOraOnInp(\iota) $ with $y'=y^{(n_1)} \in \real^m$ and $A' = A^{(n_2)} \in \real^{m \times N}$.
		\item Use the ellipsoid algorithm provided by Theorem \ref{thm:Ellipsoid} (in the Turing case) or Theorem \ref{thm:Ellipsoid-BSS} (in the BSS case) to solve the weak optimisation problem $(\CompactK_{y',A'},R,\allowbreak\mathbf{1}_{2N},2^{-n_3})$, yielding the point $z^*$.
		\item Compute $x^*$ according to $x^*_j=z^*_j-z^*_{j+N}$, for $j\in\{1,\dots,N\}$.
		\item[5T.] For the Turing machine, approximate $x^*$ by a vector of dyadic rationals so that every component has precision $n_4$ by means of Newton-Raphson iteration and output the resulting vector as $x_{\mathrm{out}}$.
		\item[5B.] For the BSS machine, output $x_{\mathrm{out}}:=x^*$.
	\end{enumerate}
	\vspace{3mm}
	
	To prove the correctness and polynomial bound on the complexity of the algorithm, we will do each of the following
	\begin{enumerate}[leftmargin=10mm, label = \Roman*)]
		\item We will show that the set $S(\CompactK_{y',A'},-2^{-n_3})$ is non-empty. We will do this by defining \begin{equation}\label{eq:x-tilde-Ellipsoid}
			\hat{x}:=((x)^+,(-x)^+)+2^{-n_3}\mathbf{1}_{2N}
		\end{equation} and showing that $\hat{x} \in S(\CompactK_{y',A'},-2^{-n_3})$, where $(x)^+=\max\{0,x\}$ coordinate-wise.
		\item We will to show that the conditions for Theorem \ref{thm:Ellipsoid} (in the Turing case) or Theorem \ref{thm:Ellipsoid-BSS} (in the BSS case) are met. With the result that $\hat{x}\in  S(\CompactK_{y',A'},-2^{-n_3})$, this will imply that $z^*$ is well defined
		\item We need to show that the algorithm is correct i.e. that $\disM(x_{\mathrm{out}},\Xi(\iota)) \leq 10^{-K}$.
		\item We need to show that the algorithm has bit complexity (in the Turing model) bounded above by a polynomial in $N$ and $K$.
		\item We need to show that the algorithm has arithmetic complexity (in the BSS model) bounded above by a polynomial in $N$ and $K$.
	\end{enumerate}
	
	For the purpose of simplifying the proof, we also introduce the following quantities
	\[\varepsilon_1 = 2^{-4K-5-\ceil {\log_2(\beta)}}, \quad \varepsilon_2 = 2^{-1-\ceil{\log_2(\beta)}}.\] Using these definitions, we note that the $y'$ and $A'$ defined in Instruction 2 of the algorithm will satisfy
	\begin{align}
		\|y-y'\|_2&\leq \sqrt{m}\cdot 2^{-n_1}\leq 2^{\frac{\log_2(m)}{2}}\cdot 2^{-n_1} \leq \varepsilon_1	\label{eq:CSPositiveYSubtractYPBound}\\ \|A-A'\|_2&\leq \|A-A'\|_{\text{F}}\leq \sqrt{mN }\cdot 2^{-n_2}\leq 2^{\frac{\log_2(m) + \log_2(N)}{2}}\cdot 2^{-n_2}\leq \varepsilon_2.\label{eq:CSPositiveASubtractAPBound}
	\end{align}
	where we have used the fact that $\log_2(n) \leq \length(n)$ for any integer $n$. We see from these definitions that $\delta'$ can be written as $\delta'= 10^{-K}(1-\rho)/(32\tau) + 2\epsilon_1 + 4b_2 N 2^{-n_3}$.

	\textbf{Step I - Showing that $\hat{x} \in S(\CompactK_{y',A'},-2^{-n_3})$}
	
	Using \eqref{eq:CSPositiveYSubtractYPBound} and \eqref{eq:CSPositiveASubtractAPBound} we obtain
	\begin{equation}\label{eq:bounds-y'-and-A'}
		\|y'\|_2\leq \varepsilon_1+b_1\sqrt{N/m}\leq 2b_1\sqrt{N/m} \quad\text{and}\quad \|A'\|_2\leq \varepsilon_2+b_2\sqrt{N/m}\leq 2b_2\sqrt{N/m}.
	\end{equation}
	Furthermore, by Lemma \ref{lemma:NSPImplicationBoundingSparse}, the definition of $\Omega_{m,N}^\varepsilon$ and \eqref{eq:CSPosEpsBound} we get
	\begin{equation}\label{eq:bounds-x-and-Axminusy}
		\|x\|_2\leq \tau\left(\varepsilon +  b_1 \sqrt{N/m} \right) \leq 2\tau b_1N \quad\text{and}\quad\|A'x-y'\|_2 \leq \varepsilon+ 2\varepsilon_2 \leq \frac{10^{-K}(1-\rho)}{32\tau}+ 2\varepsilon_2
	\end{equation} Using this, we obtain the following for $u\in \clBall{2^{-n_3}}{0}$,
	\begin{align*}
		\big[\hat{x}+u\big]_j&\geq 2^{-n_\zeta}-u_j\geq 0,\quad\text{for all }j\in\{1,\dots,2N\},\\
		\|\hat{A}(\hat{x}+u)-y'\|_2&=\|A'x-y'+\hat{A}u \|_2\\
		&\leq \| A'x-y'\|_2+2\|{A'}\|_{2}\cdot \|u \|_2 \leq \frac{10^{-K}(1-\rho)}{32\tau}+2\varepsilon_1 + 2b_2\sqrt{N/m}\cdot 2^{-n_3}\leq  \delta'\\
		\|\hat{x}+u\|_2&\leq \|((x)^+,(-x)^+)\|_2+2^{-n_3}\|\mathbf{1}_{2N}\|_2+\|u\|_2\\
		&\leq 2\tau b_1 N+ 2^{-n_3}\sqrt{2N}+2^{-n_3} \leq 2\tau b_1 N + \tau b_1 \sqrt{2N}/2 + \tau b_1/2 \leq R \
	\end{align*} where the penultimate inequality follows because $2^{-n_{3}} \leq 1$ and $\tau,b_1 \geq 1$.
	
	Therefore, we conclude that $\clBall{2^{-n_3}}{\hat{x}}\subset \CompactK_{y',A'}$ and thus  $\hat{x} \in S(\CompactK_{y',A'},-2^{-n_3})$.
	
	\textbf{Step II - Showing that the conditions for Theorem \ref{thm:Ellipsoid} (in the Turing case) or Theorem \ref{thm:Ellipsoid-BSS} (in the BSS case) are met}
	
	We will show that $\mathscr{K}$ is Turing-polynomially separable and BSS-polynomially separable (as specified in Definitions \ref{def:Tur-poly-sep-class} and \ref{def:BSS-poly-sep-class}) with appropriate encoding functions $\DataTur:\mathscr{K}\to \Alphabet^*$ and $\DataBSS:\mathscr{K}\to \InputVecsBSS$. Concretely, in the Turing case we define  
	\begin{equation}\label{eq:DataEncTur}
		\DataTur\big((\CompactK_{y',A'},2N,R)\big)  := \; m\; ;\; 2 N   \; ; \;  \delta'  \; ; \; R \; ; \; y'_1 \; ; \; y'_2  \; \cdots\; ; \; y'_m \; ; \; A'_{1,1} \; ; \; A'_{1,2}   \; \cdots\; ; \; A'_{m,N}  \in \Alphabet^*
	\end{equation}
	where $\Alphabet= \{0, \;1, -\,,\; .\,,\; ; \}$ and all the dyadic rationals are written out in their binary representation, whereas in the BSS case we define
	\begin{equation*}
		\DataBSS\big((\CompactK_{y',A'},2N,R)\big)  =   \left( m,  2N,  \delta', R , y'_1 , \cdots,  y'_m ,  A'_{1,1} , A'_{1,2}   , \cdots ,  A'_{m,N} \right)\in \InputVecsBSS.
	\end{equation*} where $\InputVecsBSS$ is defined in Definition \ref{def:EncFunBSS}.
	
	We now present a polynomial-runtime subroutine that acts as a weak separation oracle for $\CompactK_{y',A'}$ in both the Turing and BSS cases. We term this subroutine `WSS'.
	
	{\it Subroutine WSS}
	
	\noindent Input: $\DataTur\big((\CompactK_{y',A'},2N,R)\big)$ (in the Turing case) or $ \DataBSS\big((\CompactK_{y',A'},2N,R)\big)$ (in the BSS case)  specifying the set $\CompactK_{y',A'}$, a vector $w$ with $2N$ entries, and $\xi>0$ (rational in the Turing case and real in the BSS case).\\
	Output: either a $d\in\real^{2N}$ (rational in the Turing case) such that  $\|d\|_\infty=1$ and $\langle d,z-w \rangle<\xi$, for all $z\in S(\CompactK_{y',A'},-\xi)$, or the assertion that $w\in S(\CompactK_{y',A'},\xi)$.
	
	\begin{enumerate}[leftmargin=10mm, label= \arabic*.,ref = step \arabic*]
		\item Compute $\|\hat{A} w - y'\|_2^2$ and $\|w\|_2^2$ and verify whether or not $w\in \CompactK_{y',A'}$ by testing the inequalities $w\geq \mathbf{0}_{2N}$ coordinate-wise, the inequality $\|\hat{A} w - y'\|_2^2\leq \delta'^2$, and the inequality $\|w\|_2^2 \leq R^2$ in the definition of $\CompactK_{y',A'}$. If  the inequalities all hold then assert $w\in S(\CompactK_{y',A'},\xi)$ and exit this subroutine.
		\item If the subroutine has not terminated, one of the following must be true.
		\begin{enumerate}[label = \alph*.]
			\item $w_j<0$, for some $j\in\{1,\dots, 2N\}$. In this case, we set $d=-e_j$.
			\item $\|\hat{A}w-y'\|_2^2 > \delta'^2$. In this case and assuming that a. does not hold, we set $g=\nabla \big( \|\hat{A}\cdot\,-y'\|_2^2\big) \vert_w=2\hat{A}^T(\hat{A}w-y')$ and $d=g/\|g\|_\infty$.
			\item $\|w\|_2^2>R^2$. In this case and assuming that neither a. nor b. hold, we set $d=w/\|w\|_\infty$.
		\end{enumerate}
		
	\end{enumerate}
	\begin{remark}Strictly speaking, `WSS' is in fact a subroutine that solves the strong separation problem for $\CompactK_{y',A'}$ (see \cite[Def 2.1.4]{Lovasz_book}). However, for the purpose of applying either Theorem \ref{thm:Ellipsoid} or Theorem \ref{thm:Ellipsoid-BSS} it is sufficient to show that `WSS' is a weak separation subroutine.
	\end{remark}
	
	This subroutine is well defined except for one potential issue with the execution of the division $g/\|g\|_{\infty}$ in instruction 2b if $g = 0$. We will first argue that this cannot be the case by showing that if $w$ is such that $\|\hat{A}w-y'\|_2^2 > \delta'^2$ then $g \neq 0$. 
	
	Assume otherwise (i.e. that $g = 0$) for the sake of a contradiction. Then (decomposing $y$ as $y' = y^0 +\hat A v$ where $y^0 \in \ker(\hat A ^*)$) we obtain from $g = 0$ that $\hat A^*\hat A(w-v) = 0$. Hence $\|\hat A(w-v)\|^2 = \langle \hat A^*\hat A(w-v),w-v\rangle =0$ so that $\hat A (w-v) = 0$. Therefore $\delta' > \|\hat Aw - y'\|_2 = \|y^0\|_2$. But this contradicts the already established fact that $\hat x \in S(\CompactK_{y',A'},-2^{-n_3})$ - in particular, $\hat x \in S(\CompactK_{y',A'},-2^{-n_3})$ implies that $\|\hat A \hat x - y'\|_2 \leq \delta'$ and we derive a contradiction from the following argument: \[ \|y^0\|^2_2 \leq \|\hat A (\hat x - v)\|^2_2 + \|y_0\|^2_2 = \|\hat A (\hat x - v)\|^2_2 - 2\langle \hat A(\hat x -v),y_0 \rangle + \|y_0\|^2_2 = \|\hat A \hat x - y'\|^2_2\leq \delta'^2\] where the equality follows from the fact that $y^0 \in \ker(\hat A^*)$.
	
	We have therefore shown that `WSS' is well defined. To show that this is indeed a weak separation subroutine for $\CompactK_{y',A'}$, we first show that `WSS' exits either with the correct assertion that $w \in S(\CompactK_{y',A'},\xi)$ or by constructing $d$ with $\langle d,z-w \rangle<0<\xi$, for all $z\in \CompactK_{y',A'}$. To see this, note first that if `WSS' exits in line 1 then $w \in \CompactK_{y',A'}$ and that $\CompactK_{y',A'} \subseteq S(\CompactK_{y',A'},\xi)$. Hence if `WSS' exits in line 1 then it correctly asserts that $w \in S(\CompactK_{y',A'},\xi)$. 
	
	If `WSS' does not exit in line 1, we examine the three possible cases for line 2, assuming that $z$ is a vector in $S(\CompactK_{y',A'},-\xi) \subset \CompactK_{y',A'}$.
	\begin{enumerate}[label = \alph*.]
		\item If $w_j < 0$ for some $j\in\{1,\dots, 2N\}$ then the output $d$ of `WSS' satisfies  $\langle d, z - w \rangle = -z_j + w_j < 0$ 	where we have used that $z \in \CompactK_{y',A'}$ implies that $z_j \geq 0$.
		\item If $ \|\hat{A}w-y'\|_2^2 > \delta'^2$ then 
		\begin{align*}
			\langle \hat{A}^T(\hat{A}w-y'), z-w \rangle &= \langle \hat{A}w-y', \hat{A}z -y\rangle  +\langle  \hat{A}w-y', y-\hat{A}w \rangle \\&\leq  (\|\hat{A}w-y'\|_2 -\|\hat{A}z -y\|_2 ) \|\hat{A}z -y\|_2 \leq 0 
		\end{align*}
		where we have used that $z \in \CompactK_{y',A'}$ implies that $\|\hat{A}z-y\|_2 \leq \delta$. Thus $\langle d, z-w\rangle = \|g\|_{\infty}^{-1}\langle g,z-w \rangle = 2\|g\|_{\infty}^{-1}\langle \hat{A}^T(\hat{A}w-y'), z-w \rangle \leq 0.$
		\item If $\|w\|_2^2>R^2$ then $
		\langle w, z-w \rangle \leq \|w\|_2(\|z\|_2 - \|w\|_2) < \|w\|_2(R-R)
		$
		where we have used that $z \in \CompactK_{y',A'}$ implies that $\|z\|_2 \leq R$. Hence $\langle d,z-w\rangle < 0$.
	\end{enumerate} 
	In all three cases we have constructed a $d\in\real^{2N}$ (rational in the Turing case) such that  $\|d\|_\infty=1$ and $\langle d,z-w \rangle<0<\xi$, for all $z\in S(\CompactK_{y',A'},-\xi)$.
	
	It is easy to see that the runtime of `WSS' is bounded by a polynomial in $\length\big(\DataTur(\CompactK_{y',A' },2N,\allowbreak R)\big)$, $\length(R)$, $\length(w)$, $\length(\xi)$, and $N$ in the Turing case, and, in the BSS case, by a polynomial of $\dim(\big( \DataBSS\big((\CompactK_{y',A'}, 2N, R)\big) \big)$. We hence deduce that $\mathscr{K}$ is both Turing-polynomially separable and BSS-polynomially separable with respect to the encoding functions $\DataTur$ and $\DataBSS$.

	We are now in a position to apply Theorem \ref{thm:Ellipsoid} to the weak optimisation problem $(\CompactK_{y',A'},R,\allowbreak\mathbf{1}_{2N},2^{-n_\zeta})$, for all $(\CompactK_{y',A'},2N,R)\in \mathscr{K}$ to complete the argument for this step in the Turing model. For the BSS case, in order to apply Theorem \ref{thm:Ellipsoid-BSS} to the same weak optimisation problem, however, we need to establish the existence of convex functions $f_{1}^{y',A'},\dots, f_{M(y',A')}^{y',A'}:\real^{2N}\to \real$ such that $ \CompactK_{y',A'}=\{ z\in \real^{2N}\, \vert \, f_j^{\CompactK_{y',A'}}(z)\leq 0, j=1,\dots,M(y',A')\}$,  for all $(\CompactK_{y',A'},2N,R)\in \mathscr{K}$. Recalling the definition of $\CompactK_{y',A'}$, this is achieved easily by setting
	\begin{equation}\label{eq:EllipsoidBSSFDefinition}
		f^{\CompactK_{y',A'}}_j(z)=-z_j \quad j=1,\dots,2N,\quad f_{2N+1}^{\CompactK_{y',A'}}(z)= \|\hat{A}z - y'\|_2^2- \delta'^2,\quad f_{2N+2}^{\CompactK_{y',A'}}(z)= \|z\|_2^2 - R^2.
	\end{equation}

	\textbf{Step III - Proof of correctness}
	
	We will show that $\disM(x_{\mathrm{out}},\xibpdn(\iota))\leq 10^{-K}$. Since $z^*$ is a solution to the weak optimisation problem $(\CompactK_{y',A'},R',\mathbf{1}_{2N},2^{-n_\zeta})$
	\begin{enumerate}[label =\alph*)]
		\item There exists a $\hat{z}\in \CompactK_{y',A'} $ such that $\|z^*-\hat{z}\|_2\leq2^{-n_3}$. In particular, such a $\hat{z}$ must have non-negative entries and must satisfy $\|\hat{A}\hat{z}-y'\|_2 \leq \delta'$.  \label{item:Ellipsoid:TildeZExistenceProperties}
		\item For all $z\in S(\CompactK,-2^{-n_3})$ we have $\langle \mathbf{1}_{2N},z^*\rangle\leq\langle \mathbf{1}_{2N},z\rangle + 2^{-n_3} $. \label{item:Ellipsoid:ObjectiveFunction}
	\end{enumerate}
	By \ref{item:Ellipsoid:TildeZExistenceProperties} we obtain
	\begin{align}
		\left|\langle \mathbf{1}_{2N},z^*\rangle  - \|z^*\|_1\right| &\leq \left|\langle \mathbf{1}_{2N},z^*\rangle - \langle \mathbf{1}_{2N},\hat z\rangle\right| +  \left|\langle \mathbf{1}_{2N},\hat z\rangle - \|\hat z\|_1 \right| + \left|\,\|\hat z\|_1 - \|z^*\|_1\right|\notag\\
		&\leq\left|\langle \mathbf{1}_{2N},z^* - \hat{z}\rangle\right| + \left|\,\|\hat z-z^*\|_1\right| \leq 2\sqrt{2N}\|z^* - \hat{z}\|_2 \leq 2\sqrt{2N} 2^{-n_3} \label{eq:Ellipsoid:l1to12NComparison}
	\end{align}
	Combining \eqref{eq:Ellipsoid:l1to12NComparison} with \ref{item:Ellipsoid:ObjectiveFunction} as well as the previous result that $\hat x$ defined in \eqref{eq:x-tilde-Ellipsoid} satisfies $\hat x \in S(\CompactK_{y',A'},\allowbreak -2^{-n_3})$ gives the following
	\begin{align}
		\|x^*\|_1&\leq \|\{z_{j}^*\}_{j=1}^N\|_1+\|\{z_j^*\}_{j=N+1}^{2N}\|_1 \underset{\eqref{eq:Ellipsoid:l1to12NComparison}}{\leq} \sqrt{2N} 2^{1-n_3} +\langle \mathbf{1}_{2N}, z^*\rangle\leq \langle \mathbf{1}_{2N}, \hat{x}\rangle + (2\sqrt{2N}+1)2^{-n_3}\notag\\
		&=\|\hat x\|_1 + (2\sqrt{2N}+1)2^{-n_3}\leq \|x\|_1+ (2N+2\sqrt{2N}+1)2^{-n_3} \leq \|x\|_1 + 6Nb_2 2^{-n_3}\label{eq:EllipsoidL1Bounds}
	\end{align}
	where we have used the definition of $\hat x$ in the penultimate inequality and the final inequality uses the fact that $b_2 \geq 1$.
	
	Next, returning to \ref{item:Ellipsoid:TildeZExistenceProperties} and using the definition of $\delta'$ and \eqref{eq:bounds-y'-and-A'} we can estimate
	\begin{align}
		\|A'x^*-y'\|_2&=\|\hat{A}z^*-y'\|_2=\|\hat{A}\hat{z}-y'+\hat{A}(z^*-\hat{z})\|_2\leq \|\hat{A}\hat{z}-y'\|_2+2\|A'\|_{2}\|z^*-\hat{z}\|_2 \notag\\&\leq \delta'+4b_2\sqrt{N/m}\cdot 2^{-n_3}
		\leq \frac{10^{-K}(1-\rho)}{32\tau} +2\varepsilon_1+8b_2N2^{-n_{3}}.\label{eq:EllipsoidBoundAPrimexStar-YPrime}
	\end{align}
	Therefore
	\begin{align}
		\|Ax - Ax^*\|_2 &\leq \|A-A'\|_{2}\|x-x^*\|_2+\|A'x-y'\|_2+\|A'x^*-y'\|_2 \notag\\
		&\leq  \varepsilon_2 \|x-x^*\|_2+\frac{10^{-K}(1-\rho)}{16\tau}+4\varepsilon_1 +8b_AN2^{-n_3} \label{eq:EllipsoidAxAxStarBound}
	\end{align}
	where the bound from $\|A-A'\|_2$ comes from the definition of $A'$ in the algorithm $\Gamma^{\orvec, \ormat}$, the bound for $\|A'x-y'\|_2$ comes from \eqref{eq:bounds-x-and-Axminusy} and the bound for $\|A'x^*-y'\|_2$ from \eqref{eq:EllipsoidBoundAPrimexStar-YPrime}.
	
	We now use Theorem \ref{l2RNPerrorest} applied to $x$ and $x^*$ as well as the fact that $x$ is $s$-sparse to obtain
	\begin{align*}
		\|x-x^*\|_2 &\leq \frac{(1+\rho)^2}{1-\rho}\frac{1}{\sqrt{s}}\left(\|x^*\|_1-\|x\|_1\right)+\frac{(3+\rho)\tau}{1-\rho}\|Ax-Ax^*\|_2 \\ &\leq \frac{4\tau}{1-\rho}\left(\|x^*\|_1-\|x\|_1+\|Ax-Ax^*\|_2\right)
	\end{align*}
	since $\tau,s \geq 1$ and $\rho \in [0,1)$.
	Therefore employing equations \eqref{eq:EllipsoidL1Bounds} and \eqref{eq:EllipsoidAxAxStarBound} gives
	\begin{equation}
		\|x-x^*\|_2\leq \frac{4\tau}{1-\rho} \left(14Nb_22^{-n_3}+ \varepsilon_2 \|x-x^*\|_2+\frac{10^{-K}(1-\rho)}{16\tau}+4\varepsilon_1 \right).\label{eq:3.11iii-long-correctness}
	\end{equation}
	
	From the definitions of $\varepsilon_1,\varepsilon_2$ and $n_{3}$, we also obtain \[2^{-n_3} \leq 2^{-4K - \log_2(N)-\log_2(64) - \log_2(8\tau/(1-\rho)) - \log_2 b_2} \leq \frac{10^{-K}(1-\rho)}{512Nb_2\tau} \] and
	\begin{equation*}\frac{16\tau\varepsilon_1}{1-\rho} = \frac{16\tau}{1-\rho} \cdot 2^{-4K-5-\ceil {\log_2(8\tau/(1-\rho))}} \leq \frac{10^{-K}}{16},\quad \frac{4\tau\varepsilon_2} {1-\rho} = \frac{4\tau}{1-\rho} \cdot 2^{-1-\ceil{\log_2(8\tau/(1-\rho))}}\leq \frac{1}{4}.
	\end{equation*}
	Therefore \eqref{eq:3.11iii-long-correctness} can be written as 
	\begin{equation}\label{x1x*estSD}
		\|x-x^*\|_2 \leq \frac{10^{-K}}{8}  +\frac{\|x-x^*\|_2}{4}+\frac{10^{-K}}{4}+\frac{10^{-K}}{16} = \frac{\|x-x^*\|_2}{4} + \frac{7 \cdot 10^{-K}}{16}
	\end{equation}

	Since $\epsilon \leq \delta$, the set $\xibpdn(\iota)$ is non-empty. Fix a solution $\xi^*\in\xibpdn(\iota)$. Then, as 
	$
	\|Ax-y\|_2\leq\varepsilon\leq \delta
	$, we have that $x$ is feasible for the basis pursuit problem with $\ell^1$ regularisation with parameter $\delta$ and input $(y,A)$. In particular, since $\xi^*$ is optimal for this problem, we have $\|\xi^*\|_1\leq\|x\|_1$, and thus Theorem \ref{l2RNPerrorest} applied to $\xi^*$ and $x$ yields
	\begin{align}
		\|\xi^*-x\|_2 \leq\frac{(3+\rho)\tau}{1-\rho}\left(\|A\xi^*-y\|_2+\|Ax-y\|_2\right)\leq\frac{4\tau(\delta+\varepsilon)}{1-\rho} \leq \frac{10^{-K}}{4}\label{x1z*estSD}
	\end{align} where we have once again made use of \eqref{eq:CSPosEpsBound}.
	
	Next, we note that $x_{\mathrm{out}}=x^*$ in the BSS case, whereas \[\|x_{\mathrm{out}}-x^*\|_2\leq 2^{-n_{4}}\sqrt{N}= 2^{-\left\lceil\frac{\length(N)}{2}\right\rceil - K - 3}\sqrt{N} \leq \frac{10^{-K}}{8}\]  in the Turing case, where we have used the fact that $\length(n) \geq \log_2(n)$ for a natural number $n$. Combining this with \eqref{x1x*estSD} and \eqref{x1z*estSD}  finally yields
	\begin{align*}
		\|x_{\mathrm{out}}-\xi^*\|_2 &\leq \|x_{\mathrm{out}}-x^*\|_2 +  \|x-x^*\|_2+\|\xi^*-x\|_2\leq \frac{10^{-K}}{8} + \frac{4\cdot 7 \cdot 10^{-K}}{3\cdot 16} + \frac{10^{-K}}{4} \leq 10^{-K}
	\end{align*}
	Hence $\disM\left(x_{\mathrm{out}},\xibpdn(\iota)\right)\leq \|x_{\mathrm{out}}-\xi^*\|_2\leq {10^{-K}}$.
	
	\textbf{Step IV - Establishing a polynomial upper bound on the computational complexity of the algorithm in the Turing case}
	
	The initial computation of the integers $n_1,n_2,n_3$ and $n_4$ can be analysed as follows: firstly, computation of both $\length(N)/2$ and $\length(m)/2$ can be established in $\oh(\length(N) \vee \length(m))$ bit operations. After that, we note that we require only finitely many multiplications and additions involving the quantities $\length(N),\length(m),K$ and the fixed quantities $L_1$ and $L_2$. Thus the complexity of computing $n_1,n_2,n_3$ and $n_4$ can be bound by a polynomial in $\length(\length(N)),\length(\length(m))$ and $\length(K)$. This immediately implies that $n_1,n_2,n_{3}$ and $n_{4}$ each have lengths polynomial in $\length(\length(N)),\length(\length(m))$ and $\length(K)$. 
	
	$\delta'$ is the sum of three fractions: the first, $10^{-K}(1-\rho)/(32\tau)$, can be computed as a pair of integers in bit complexity polynomial in $K$ (as $(1-\rho)/(32\tau)$ is fixed. The second, $2^{-4K-3-L_1}$ can also be computed in bit complexity polynomial in $K$ as $L_1$ is fixed. Finally, $4b_AN2^{-n_3}$ can be computed in bit complexity polynomial in $N$, $\length(m)$ and $K$ as $b_A$ is fixed and we have already established that the length of $n_3$ is polynomial in $\length(\length(N)),\length(\length(m))$ and $\length(K)$. Thus the overall cost of computing the rational number $\delta'$ expressed as a pair of two integers is polynomial in $K,N$ and $\length(m)$.
	
	Finally, $R$ involves a multiplication between $N$ and the fixed value $L_3$ and can thus be computed in time polynomial in $N$. Therefore the overall complexity of the first instruction is polynomial in $K,N$ and $\length(m)$.
	
	Since $n_1$ and $n_2$ are both bounded from above by a polynomial in $\length(m)$, $\length(N)$ and $K$, the cost of accessing the oracle to produce each $y'_j$ and $A'_{j,k}$ is bounded above by a polynomial in $j$, $k$, $\length(m)$, $\length(N)$, and $K$ (in both the arithmetic and standard Turing complexity models). As this process is repeated for $j\in \{1,2,\dotsc,m\}$ and $k\in \{1,2,\dotsc,N\}$ to obtain $y'$ and $A'$, the overall complexity of the second instruction for the algorithm is bounded above by a polynomial in $m,N$ and $K$.
	
	Moreover, we have already shown that $n_1$ and $n_2$ are such that both $\length(n_1)$ and $\length(n_2)$ are bounded from above by a polynomial in $\length(\length(N)),\length(\length(m))$ and $\length(K)$. Furthermore, \eqref{eq:bounds-y'-and-A'} tells us that $y'$ and $A'$ are bound above by polynomials in $\sqrt{N/m}$.	Thus the resulting lengths of $y'_k$ and $A'_{j,k}$ is also polynomial in $\length(N),\length(m)$ and $\length(K)$.  
	
	By Theorem \ref{thm:Ellipsoid}, the bit-complexity of the third instruction is polynomial in  \begin{equation*}
		\length\big(\DataTur(\CompactK_{y',A'},R,\mathbf{1}_{2N},2^{-n_3})\big),\; \length(R),\;  \length(1_{2n}),\; \length(2^{-{n_3}})\; \text{and}\;  2N.
	\end{equation*}
	Because of the way $\DataTur$ has been defined in \eqref{eq:DataEncTur} we have \begin{align*}\length\big(\DataTur(\CompactK_{y',A'},R,\mathbf{1}_{2N})\big) &= \length(m)+\length(N)+\length(\delta')\\&\;\;\;+\length(R) + \sum_{j=1}^{m} \left(\length(y'_j) + \sum_{k=1}^{N}  \length(A'_{j,k})\right)
	\end{align*}
	and we have already established that each term on the right hand side is polynomial in $m,N$ and $\length(K)$. Furthermore, $\length(R)$ is polynomial in $\length(N)$, $\length(1_{2N})$ is linear in $N$ and $\length(2^{-{n_3}}) = -n_3$ which is polynomial in $\length(N)$ and $K$. Thus the bit-complexity of the third instruction is polynomial in $m,N$ and $K$. The length of each component of the resulting vector $z^*$ must therefore also have $\length(z^*)$ polynomial in $m,N$ and $K$.
	
	Since the fourth instruction involves a subtraction between two portions of $z^*$, the bit-complexity of executing this instruction is polynomial in $m,N$ and $K$. Finally, the division to convert the rational vector $x^* \in \mathbb{Q}^N$ into the output $x_{out}$ expressed as a vector of dyadic rationals can be done in complexity polynomial in $\length(x^*),\length(n_4)$ and $N$ using Newton-Raphson iterations (see \cite[pp. 92-93]{ElementaryFunctionsImplementation}). As each of these quantities are bound above by a polynomial in $m,N$ and $K$, we see that the bit-complexity of executing this final step is polynomial in $m,N$ and $K$. 
	
	We have therefore shown that each instruction can be executed in bit complexity bound above by a polynomial in $m,N$ and $K$ and thus as there are only five instructions executed in the algorithm the overall complexity can be bound above by a polynomial in $m,N$ and $K$, completing the proof in the Turing case.

	\textbf{Step V - Establishing a polynomial upper bound on the computational complexity of the algorithm in the BSS case}
	
	For a natural number $n$, it is possible to compute $\length(n)$ on a BSS machine by incrementing a value $k$ (starting at $k=1$) and stopping when $2^k > n$. This can be done in $\length(n)$ arithmetic operations and thus the BSS complexity of computing $n_1$ through $n_{4}$ is bounded by a polynomial in $\length(N)$ and $\length(m)$. The computation of $\delta'$ can be done with BSS complexity bounded above by a polynomial in $K$ and $n_{3}$. Since $n_{3}$ is itself linear in $K$ and $\length(N)$, the BSS	 complexity of computing $\delta'$ is bounded above by a polynomial in $K$ and $\length(N)$. Finally, computing $R$ requires a single arithmetic operation. Thus the computational cost of executing the first instruction on a BSS machine is polynomial in $\length(N),\length(m)$ and $K$.
	
	In the same way as in the arithmetic Turing case, the cost of calling the oracles in the second instruction of the algorithm is polynomial in $m,N$ and $K$. 
	In the context of our application of Theorem \ref{thm:Ellipsoid-BSS} in the third instruction of the algorithm, the parameter $V$ is equal to
	\begin{equation*}
		V=R\|\ones_{2N}\|_2 \vee\max_{1\leq j\leq 2N+2} g_j  ,\quad \text{where}\quad g_j:= \max_{\|z\|_2\leq R} f^{\CompactK_{y',A'}}_j(z) - \min_{\|z\|_2\leq R} f^{\CompactK_{y',A'}} _j(z)
	\end{equation*}
	and $f^{\CompactK_{y',A'}}$ are defined as in \eqref{eq:EllipsoidBSSFDefinition}. From these definitions, it follows that $g_j = \max_{\|z\|_2\leq R} z_j - \min_{\|z\|_2\leq R}z_j = 2R$ for $j=1,2,\dotsc,2N$, that 
	\begin{align*}
		g_{2N+1} &= \max_{\|z\|_2\leq R} \|Az-y'\|_2^2 - \min_{\|z\|_2\leq R}\|Az-y'\|_2^2 \\&\leq \left(\|A'\|_2 \cdot 2\|z\|_2 + \|y'\|_2 \right)^2 \leq \left(4b_2 \sqrt{\frac{N}{m}}R + 2b_1 \sqrt{\frac{N}{m}}\right)^2\leq 16NR^2(b_1+b_2)^2
	\end{align*}
	and that $g_{2N+2} = \max_{\|z\|_2\leq R} \|z\|^2_2 - \min_{\|z\|_2\leq R}\|z\|^2_2 = R^2$.
	Therefore $V \leq R\sqrt{2N}\vee 2R \vee  16NR^2(b_1+b_2)^2 \vee R^2 = 16NR^2(b_1+b_2)^2$. Thus $V \leq C_{\text{BSS}}N^3$ where the real number $C_{\text{BSS}}$ exceeding $16$ depends only on $b_1,b_2,\tau$ (crucially, $C_{\text{BSS}}$ is independent of $m,N$ and $K$).
	
	Thus $\log(2+V/2^{-n_{3}}) \leq \log(2C_{\text{BSS}}N^3 2^{-n_{3}}) = n_{3} \log(2) + 3\log(N) + \log(2C_{\text{BSS}})$. In particular, using the definition of $n_3$, we note that $\log(2+V/2^{-n_3})$ can be written as a polynomial in $\length(N)$ and $K$. The dimension of $\DataBSS \big((\CompactK_{y',A'},2N,R)\big)$ is $4+m+mN$ and thus we conclude via Theorem \ref{thm:Ellipsoid-BSS} that the BSS complexity of the third instruction of the algorithm is polynomial in $m,N$ and $K$.
	
	The fourth instruction takes $N$ arithmetic operations and the final instruction is constant time on a BSS machine. Thus the overall complexity of the algorithm is polynomial in $m,N$ and $K$, as claimed.
\end{proof}

\section{Proof of Theorem \ref{th:smale_comp_sens}: part (iii)}\label{sec:smale_comp_sens_proof_iv}
\subsection{Preliminaries - distance to several minimisers for the $\ell^1$ BP  problem}
We begin by proving the following lemma that allows us to create examples of problems with infinite condition number for the basis pursuit problem with $\delta=0$.

\begin{lemma}\label{lemma:ProveInfiniteCondition}
	Let $\{\Xi,\Omega\}$ be the computational problem of basis pursuit with $\delta=0$. Suppose that $A \in \real^{m \times N}$ is a matrix such that for every $y' \in \real^m$, $(y',A) \in \Omega^{\text{act}}$. Furthermore,  suppose that there is a non-empty set $S \subseteq \lbrace 1,2,\dotsc,N \rbrace$ and a vector $v$ in the row span of $A$ such that the following conditions hold:
	\begin{enumerate}
		\item there is a non-zero $\hat{\xi}\in\real^N $ such that $\supp(\hat{\xi}) \subseteq S$ and $A\hat{\xi}= 0$, and  \label{property:NonTrivialNullspace}
		\item $\|v_{S^c}\|_{\infty} \leq 1$, and $|v_i| = 1$ for all $i\in S$. \label{property:Constantvi}
	\end{enumerate}
	If $T \subseteq S$ is a set such that
	\begin{enumerate}[resume]
		\item if $A\rho = 0$ and $\supp(\rho) \cap T \neq \varnothing$ for some vector $\rho \in \real^{N}$, then there is a non-zero $\rho'$ with 
		$
		\supp(\rho) \cap T \subseteq \supp(\rho') \subseteq S 
		$ 
		and $A\rho' = 0$, \label{property:SupportBoundsCondition}
	\end{enumerate}
	then there exists an $x\in\real^N$ with $\supp(x) = T$ so that $\dist_2[(Ax,A),\MultiMinInputs \cup \Omega^{\text{act}}] = 0$ and that $\|Ax\|_2 \leq 1$. Note that condition \eqref{property:SupportBoundsCondition} is trivially satisfied if $S = \left\lbrace 1,2,\dotsc,N \right\rbrace$ (with $\rho' = \rho$) or $T = \varnothing$ (where there is nothing to show). 
\end{lemma}
\begin{proof}
	Let $\xi\in\real^N$ be a vector such that $A\xi = 0$, $\supp(\xi) \subseteq S$, and $\xi$ has maximal support in $S$, in the sense that
	\begin{equation}
		\text{If }\xi'\in\real^N\text{ is such that, if } A\xi' = 0\text{ and } \supp(\xi) \subseteq \supp(\xi') \subseteq S, \text{ then } \supp(\xi')  = \supp(\xi). \label{property:maximalSupportNullspace}
	\end{equation}
	The existence of $\xi$ is guaranteed by condition (\ref{property:NonTrivialNullspace}) in the statement of the lemma: indeed, the set of supports of all vectors in the nullspace of $A$ is a finite, non-empty (here we use \ref{property:NonTrivialNullspace}), partially ordered set under the set inclusion ordering and thus such a set of supports contains at least one maximal element (one of which we choose arbitrarily and call $R$). Thus there must exist a vector $\xi$ in the nullspace of $A$ with support $R$ satisfying \eqref{property:maximalSupportNullspace}.  
	
	Next, set $c = (\sqrt{N}\|A\|_2)^{-1} \wedge (\sqrt{N}\|\xi\|_{\infty}\|A\|_2)^{-1}$ (note that $c$ is well defined as both $\xi$ and $A$ are non-zero) and partition $R = \supp(\xi)$ into sets $J_{+}$ and $J_{-}$ according to
	\begin{equation*}
		J_{+} = \lbrace i \in \supp(\xi) \, \vert \, v_i \xi_i > 0\rbrace, \qquad J_{-} = \lbrace i \in \supp(\xi) \, \vert \, v_i \xi_i < 0 \rbrace,
	\end{equation*}
	and let
	\begin{equation*}
		x_i = \begin{cases}
			\sgn(v_i\xi_i)\,c\xi_i & \text{if } i \in T\cap \supp(\xi)\\[-1mm]
			c & \text{if } i \in  T \setminus \supp(\xi) \\[-1mm]
			0 & \text{otherwise}
		\end{cases}\;.
	\end{equation*}
	Then clearly $\supp(x) = T$ and $\|Ax\|_2 \leq \|A\|_2 \sqrt{N}\|x\|_{\infty} \leq 1$. We now proceed to establish the following for every sufficiently small $\epsilon>0$:
	\begin{enumerate}[label=(\Roman*)]
		\item $A(x+\epsilon\xi_{J_+}) = A(x-\epsilon\xi_{J_-})$, \label{eq:CondUx+xi = Ux-xi}
		\item $\|x + \epsilon\xi_{J_+}\|_1 = \|x -\epsilon\xi_{J_-}\|_1$, and\label{eq:Condx+xi=x-xi}
		\item if $\hat{\xi} \in \real^{N}$ is such that $A(\hat{\xi}+x+\epsilon\xi_{J_+}) = A(x+\epsilon\xi_{J_+})$, then $\|x+\epsilon\xi_{J_+}+\hat{\xi}\|_1 \geq \|x+\epsilon\xi_{J_+}\|_1$.\label{eq:CondUx+xiMin}
	\end{enumerate} 
	This will imply that 
	$
	\{x+\epsilon\xi_{J_+},x+\epsilon\xi_{J_-}\} \subset \Xi\big(A(x+\epsilon\xi_{J_+}),A\big)
	$ and thus that $\big(A(x+\epsilon\xi_{J_+}),A\big) \in \Sigma^{\mathrm{RCC}}$.
	Indeed, from \ref{eq:CondUx+xiMin}, $x+\epsilon\xi_{J_+}$ is a minimiser, whereas \ref{eq:CondUx+xi = Ux-xi} implies that $x-\epsilon\xi_{J_-}$ is feasible and \ref{eq:Condx+xi=x-xi} implies that $x-\epsilon\xi_{J_-}$ is also a minimiser.
	Consequently, using the assumption that $\big(A(x+\epsilon\xi_{J_+}),A\big) \in \Omega^{\text{act}}$, we obtain
	\begin{equation*}
		\dist_2[(Ax,A),\MultiMinInputs \cap \Omega^{\text{act}}] \leq \|A(x+\epsilon\xi_{J_+}) - Ax\|_2 \leq \epsilon \|\xi\|_2 \|U\|_2
	\end{equation*}
	and, as $\epsilon$ was arbitrary and $\xi$ and $A$ are independent of $\epsilon$, we must have $C_{\mathrm{RCC}}(Ax,A) = \infty$.
	
	It hence suffices to establish \ref{eq:CondUx+xi = Ux-xi}, \ref{eq:Condx+xi=x-xi}, and \ref{eq:Condx+xi=x-xi} in order to complete the proof.
	To this end, note that $0=A\xi=A(\xi_{J_+} + \xi_{J_-})=A\xi_{J_+} + A\xi_{J_-}$, and so \ref{eq:CondUx+xi = Ux-xi} follows immediately. Next, since $\xi$ is in the nullspace of $A$ and $v$ is in the row span of $A$, we have $\langle v,\xi \rangle = 0$ and, as $v_i=\sgn(v_i)$ for $i\in \supp(\xi)\subset S$, we obtain
	\[
	0=\sum_{i\in\supp(\xi)}v_i\xi_i=\sum_{i\in\supp(\xi)} \sgn(v_i)\sgn(\xi_i)|\xi_i|=\sum_{i\in\supp(\xi)} \sgn(v_i\xi_i)|\xi_i|= \|\xi_{J_+}\|_1-\|\xi_{J_-}\|_1
	\]
	i.e. that $\|\xi_{J_+}\|_1=\|\xi_{J_-}\|_1$. Next, we have
	\begin{equation}\label{eq:x+xi, x-xi explicit}
		\big(x + \epsilon \xi_{J_+}\big)_i= \begin{cases}
			(c+\epsilon)\xi_i & \text{if } i \in T\cap J_+\\[-1mm]
			-c\xi_i & \text{if } i \in T\cap J_-\\[-1mm]
			c & \text{if } i \in  T \setminus \supp(\xi) \\[-1mm]
			\epsilon\xi_i & \text{if }i\in J_+ \setminus T\\[-1mm]
			0 & \text{otherwise}
		\end{cases},\quad 
		\big(x_i- \epsilon \xi_{J_-}\big)_i= \begin{cases}
			c\xi_i & \text{if } i \in T\cap J_+\\[-1mm]
			-(c+\epsilon)\xi_i & \text{if } i \in T\cap J_-\\[-1mm]
			c & \text{if } i \in  T \setminus \supp(\xi) \\[-1mm]
			-\epsilon\xi_i & \text{if }i\in J_- \setminus T\\[-1mm]
			0 & \text{otherwise}
		\end{cases}
	\end{equation}
	and therefore $\|x + \epsilon \xi_{J_+}\|_1=\|x\|_1+\epsilon \| \xi_{J_+}\|_1=\|x\|_1+\epsilon \| \xi_{J_-}\|_1=\|x - \epsilon \xi_{J_-}\|_1 $, establishing \ref{eq:Condx+xi=x-xi}.
	
	Now consider an arbitrary $\hat{\xi}$ such that $A\hat{\xi}=0$. We claim that $\supp(\hat{\xi})\subset (T\setminus \supp(\xi))^c$. To prove this, suppose by way of contradiction that there exists a $t\in \supp(\hat{\xi})\cap(T\setminus \supp(\xi))$. Then property \eqref{property:SupportBoundsCondition} implies that we can find a $\rho'\neq 0$ such that $A\rho'=0$ and $\supp(\hat{\xi})\cap T\subset \supp(\rho')\subset S$. But then $A(\xi+\kappa \rho')=0$ and $\supp(\xi+\kappa \rho')\supseteq \supp(\xi)\cup\{t\}\supsetneq \supp(\xi)$, for sufficiently small $\kappa>0$, contradicting \eqref{property:maximalSupportNullspace}. Therefore $\supp(\hat{\xi})\subset (T\setminus \supp(\xi))^c$. 
	Next, by inspecting the explicit expression for $x + \epsilon \xi_{J_+}$ in \eqref{eq:x+xi, x-xi explicit} and recalling that $\|v\|_\infty\leq 1$, we see that $(x+\epsilon\xi_{J_+}+ \hat{\xi})_i=(x+\epsilon\xi_{J_+})_i=c$, for $i\in T\setminus \supp(\xi)$, and $\sgn\big((x + \epsilon \xi_{J_+})_i\big)\supset \{v_i\}$, for all $i\in (T\setminus \supp(\xi))^c$, and therefore
	\begin{equation*}
		\begin{aligned}
			\|x+\epsilon\xi_{J_+}+ \hat{\xi}\|_1 - \|x+\epsilon\xi_{J_+}\|_1&= \sum_{i\in (T\setminus \supp(\xi))^c} |(x+\epsilon\xi_{J_+}+ \hat{\xi}\,)_i|- |(x+\epsilon\xi_{J_+})_i|\\
			&\geq \sum_{i\in (T\setminus \supp(\xi))^c} \hat{\xi}_i v_i= \langle \hat{\xi},v \rangle=0,
		\end{aligned}
	\end{equation*}
	where the last equality is due to the fact that $v$ is in the row span of $A$ and $A\hat{\xi}=0$. This establishes \ref{eq:CondUx+xiMin} and concludes the proof.
\end{proof}

For the purpose of proving part (iv) of Theorem \ref{th:smale_comp_sens}, we will apply Lemma \ref{lemma:ProveInfiniteCondition} to both the case where $A$ is a subsampled Hadamard matrix and the case where $A$ is a subsampled Hadamard to Haar matrix, as in \S\ref{sec:CSMatricesAppendix}. We do this in the following result:

\begin{lemma}\label{lemma:HighConditionHadSampling}
	Let $\{\Xi,\Omega\}$ be the computational problem of basis pursuit with $\delta=0$, and let $A\in\real^{m\times N}$ with $1\leq m<N$ be either
	\begin{itemize}
		\item[(i)] a row-subsampled Hadamard matrix, or 
		\item[(ii)] a row-subsampled Hadamard-to-Haar matrix as given in \eqref{eq:HadToHaarDecomposition}, with $S_i \neq \varnothing$, for $i=1,2\dots,n-1$.
	\end{itemize}
	Additionally, assume that $(y',A) \in \Omega^{\text{act}}$ for all $y' \in \real^m$. 	Then, for every set $T \subseteq \lbrace 1,2,\dotsc, N\rbrace$, there is a vector $x\in\real^N$ with support $T$ so that $\dist_2[(Ax,A),\MultiMinInputs \cap \Omega^{\text{act}}] = 0$ and $\|Ax\|_2 \leq 1$. 
\end{lemma}
\begin{proof}[Proof of Lemma \ref{lemma:HighConditionHadSampling}]
	Our aim will be to verify that $A$ satisfies the desired conditions of Lemma \ref{lemma:ProveInfiniteCondition}. 
	
	It suffices to verify the conditions (\ref{property:NonTrivialNullspace}) and to construct a suitable $v$ satisfying (\ref{property:Constantvi}) with $S=\{1,\dots, N\}$, noting the observation that (\ref{property:SupportBoundsCondition}) holds trivially for such an $S$ and any $T$ as defined in this lemma. To this end, we first observe that (\ref{property:NonTrivialNullspace}) holds simply as $m<N$. 
	
	Next, in order to establish (\ref{property:Constantvi}), we construct a vector $v\in\real^N$ in the row span of $A$ such that $|v_i|=1$, for all $i=1,\dots,N$. If  $A$ is a row-subsampled Hadamard  matrix, we can simply take $v$ to be the first row of $A$. Suppose now that $A$ is a row-subsampled Hadamard-to-Haar matrix as given in \eqref{eq:HadToHaarDecomposition}, with $S_i \neq \varnothing$, for $i=1,2\dots,n-1$.  In this case we select an arbitrary $j_i\in S_i$, for every  $i=1,2\dots,n-1$, and let $v=(1)\oplus \bigoplus_{i=0}^{n-1} \sqrt{2^i}(X^i)^*_{j_i}$, where $(X^i)^*_{j_i}$ denotes the $j_i$-th row of $X^i$. Since $|(X^i)_{j,k}| = 1/\sqrt{2^i}$, for $i=0,1,2,\dotsc,n-1$ we have $|v_i| = 1$ for $i=1,2,\dotsc,N$.
	
	We can now apply Lemma \ref{lemma:ProveInfiniteCondition} to complete the proof, where we use the assumption that $(y',A) \in \Omega^{\text{act}}$ for all $y' \in \real^m$.
\end{proof}

We are now ready to prove part (iii) of Theorem \ref{th:smale_comp_sens}. The argument proceeds as follows - first, we use existing results from the literature (Theorem \ref{thm:RIPExistenceBOS} and Theorem \ref{thm:RIP->RNP}) to argue that if $N$ is sufficiently large and provided $m$ is sufficiently large relative to $N$, then, with strictly positive probability, a randomly subsampled Hadamard or Hadamard-to-Haar matrix will satisfy the RIP and hence the robust nullspace property.
This will, in particular, imply the existence of at least one such matrix $A$, to which Lemma \ref{lemma:HighConditionHadSampling} will be applied to conclude the proof.

\subsection{Proof of part (iii) of Proposition \ref{prop:CSResult}}

In the language of the SCI hierarchy, Theorem \ref{th:smale_comp_sens} part (iii) is written as Proposition \ref{prop:CSResult} part (iii). We will prove this proposition. 
\begin{proof}[Proof of Proposition \ref{prop:CSResult} part (iii)]
	
	Fix an $s$ as in the theorem statement and set $K=1$. Let $C_4$ be the constant in Theorem \ref{thm:RIPExistenceBOS} corresponding to $K$. We proceed assuming that $N = 2^n$ for some natural number $n \geq 2$ satisfying $N-2 \geq 2C_4s\log(N)\log(s\log(N))\log(s)^2$ and we set $m = N-1$. Clearly, there are infinitely many such pairs $(m,N)$.

	Next we let $U \in \real^{N \times N} = \real^{2^n \times 2^n}$ be either a Hadamard matrix or a Hadamard to Haar matrix. We will produce a subsampled matrix $A$, a vector $x \in \real^N$ and a vector $y \in \real^m$ with the following properties:
	\begin{enumerate}
		\item $y = Ax$ with $|\supp(x)| = s$
		\item $A$ satisfies the $\ell^2$-RNP of order $s$ with parameters $(\rho,\tau)$.
		\item $\|y\|_2 \leq 1$ and $\|A\|_2 \leq \sqrt{\frac{N}{m}}$
		\item $\dist_2[(y,A),\MultiMinInputs \cap \actv{\Omega^{0}_{s,m,N}}] = 0$.
	\end{enumerate}
	This will complete the proof. Indeed, if we set $\iota = (y,A)$ then (1)-(3) proves that $\iota \in \Omega^{0}_{s,m,N}$ and (4) proves that $C_{\mathrm{RCC}}(\iota) = \infty$.
	
	We divide into two cases depending on whether $U$ is a Hadamard matrix or a Hadamard to Haar matrix.
	\begin{itemize}
		\item If $U$ is a Hadamard matrix,  we start by noting that \[m = N-1 \geq N-2 \geq C_4s\log(N)\log(s\log(N))\log(s)^2\] and thus we can employ Theorem \ref{thm:RIPExistenceBOS} to conclude that there exists a set $S$ with $|S| = m$ such that $A:=P_S \sqrt{\frac{N}{m}} U$ satisfies the RIP of order $2s$ with constant $\delta_{2s} \leq 1/5$. Moreover, note that since $U$ is unitary, $\|A\|_2 = \sqrt{N/m} \leq b_2 \sqrt{N/m}$. 
		\item If $U$ is a Hadamard-to-Haar matrix then $U$ can be decomposed as $
		U = (1) \oplus \bigoplus_{i=0}^{n-1} X^{i}$ by the argument in \eqref{eq:HadToHaarDecomposition}. By observing that 
		\[
		2^{n-1} - 1 = N/2 - 1 \geq C_4s\log(N)\log(s\log(N))\log(s)^2
		\] we can employ 
		Theorem \ref{thm:RIPExistenceBOS}  to see that there exists a set $S_{n-1}$ of size $2^{n-1} - 1$ such that $c_{n-1}P_{S_{n-1}} X^{n-1}$ with $c_{n-1} = \sqrt{\frac{2^{n-1}}{2^{n-1}-1}}$ satisfies the RIP of order $2s$ with constant $\delta_{2s} \leq 1/5$. We choose \[A = \big(\hspace{1pt} 1\,\big) \oplus \bigoplus_{i=0}^{n-2} X^{i} \oplus c_{n-1}P_{S_{n-1}} X^{n-1}.\] Since $X^i$ are unitary for $i=0,1,\dotsc,n-2$ and $c_{n-1}P_{S_{n-1}}X^{n-1}$ obeys the RIP of order $2s$ with constant $\delta_{2s}\leq 1/5$ the matrix $A$ also obeys the RIP of order $2s$ with constant $\delta_{2s} \leq 1/5$. Furthermore, \begin{equation*}\|A\|_2 = c_{n-1} = \sqrt{\frac{2^{n-1}}{2^{n-1}-1}} \leq \sqrt{2}\sqrt{\frac{2^n}{2^{n}-1}} \leq b_2 \sqrt{\frac{N}{m}}
		\end{equation*}
	\end{itemize}

	In either case, we deduce that there is a matrix $A \in \real^{m \times N}$ subsampled from $U$ such that $A$ obeys the RIP of order $2s$ with constant $\delta_{2s}\leq 1/5$ and $\|A\|_2 \leq b_2 \sqrt{N/m}$. By Theorem \ref{thm:RIP->RNP}, $A$ satisfies the $\ell^2$-RNP of order $s$ with parameters 
	\begin{equation}\label{eq:RCCInfinityARhoTau}
		\frac{1/4}{\sqrt{1-(1/4)^2}-(1/4)/4}<\frac{1}{3}<\rho\quad\text{and}\quad \frac{\sqrt{1+1/4}}{\sqrt{1-(1/4)^2}-(1/4)/4}<10<\tau.
	\end{equation}

	Furthermore, the construction of $A$ is such that $A$ contains a non-zero entry in both the Hadamard case or the Hadamard to Haar case (here we have used the fact that $S_{n-1}$ is non-empty and that every element of $X^i$ is non-zero). For a given $j \in \{1,2,\dotsc,m\}$, let $k \in \{1,2,\dotsc,N\}$ be such that $A_{j,k} \neq 0$. Using \eqref{eq:RCCInfinityARhoTau} and the bound $\|A\|_2 \leq b_2 \sqrt{N/m}$, we note that for sufficiently small $\epsilon > 0$ both $( \epsilon A e_k,A),(-\epsilon A e_k,A) \in \Omega^0_{s,m,N}$. 
	
	Since the $j$th entry of the vectors $A(\epsilon e_k)$ and $-A(\epsilon e_k)$ are not equal, we have shown that the $j$th coordinate of the input vector is an active coordinate of $\Omega^0_{s,m,N}$, i.e. the $j$th coordinate of the input vector  is in $\mathcal{A}(\Omega^0_{s,m,N})$ (see \eqref{eq:def-active-set}). We have thus shown that for every $y' \in \real^m$, we have $(y',A) \in \actv{\Omega^{0}_{s,m,N}}$.
	Thus the conditions of Lemma \ref{lemma:HighConditionHadSampling} apply and we obtain $x$ such that $\|Ax\|_2 \leq 1$, $\dist_2[(Ax,A),\MultiMinInputs\allowbreak \cap \actv{\Omega^{0}_{s,m,N}}] = 0$ and $|\supp(x)|=s$. Setting $y = Ax$ completes the proof.
\end{proof}

\appendix

\section{Auxiliary results for Theorem \ref{Cor:main} -- Solutions to convex problems}\label{appendix:geometry-lemmas}

In this section we derive the exact solution to the introductory Lasso example \eqref{eq:lasso43} and prove the lemmas given in  \S\ref{sec:BasicProblemsSolutions} on the solutions of the convex problems used in Theorems \ref{Cor:main}, \ref{thm:ExitFlag}, and \ref{thm:Smales9}. We use standard tools from convex analysis, most extensively the subdifferential $\partial f (x) :=\{v\in\real^N \,\vert\, f(y)-f(x)\geq \langle y-x,v\rangle ,\, \forall y\in D_f \}$, for $x\in D_f$, where $f:D_F\to(-\infty,+\infty]$ is a lower semi-continuous convex function on a closed convex domain $D_f\subset\real^N$. For a more detailed account of convex analysis see \cite{Rockafellar1970}.

\subsection{Solution to \eqref{eq:lasso43}}
\begin{lemma}
	Let $m =3, N=2$ and $\lambda \in (0, 1/\sqrt{3}\, ]$, as well as \begin{equation*}
		A = \begin{pmatrix} \frac{1}{\sqrt{2}} - a & \frac{1}{\sqrt{2}} \\ -\frac{1}{\sqrt{2}} - a & -\frac{1}{\sqrt{2}}\\ 2a & 0 \end{pmatrix} \in \real^{3 \times 2}, \quad y = \begin{pmatrix} 1/\sqrt{2} & -1/\sqrt{2} & 0 \end{pmatrix}^T \in \real^3.
	\end{equation*}
	where $a> 0$.
	Then if $w^0 \in \real$ and $w^1 \in \real^2$ are such that
	\begin{equation*}
		(w^0,w^1) \in \argmin_{(x^0, x^1) \in \real \times \real^2} \frac{1}{2m}\|ADx^1- x^0\ones_{m} - y\|_2^2+ \lambda \|x^1\|_1
	\end{equation*}
	where $D$ is the unique diagonal matrix such that every column of $AD$ has norm $\sqrt{m}$, we have $w^0 = 0$, $Dw^1 = (1-\sqrt{3}\lambda)e_2$.
\end{lemma}
\begin{proof}
	Note that 
	\begin{equation*}
		(w^0,w^1) \in \argmin_{(x^0, x^1) \in \real \times \real^2} \frac{1}{2m}\|ADx^1- x^0\ones_{m} - y\|_2^2+ \lambda \|x^1\|_1
	\end{equation*}
	if and only if 
	\begin{equation*}
		(w^0,\hat{w}^1) \in \argmin_{(x^0, x^1) \in \real \times \real^2} \frac{1}{2m}\|Ax^1- x^0\ones_{m} - y\|_2^2+ \lambda \| D^{-1}x^1\|_1,
	\end{equation*}
	where $\hat w^1 = Dw^1$. Now, for all $(x^0, x^1) \in \real \times \real^2$, we have
	\begin{align*}
		\|Ax^1- x^0\ones_{m} - y\|_2^2 &= \left[ \left(\frac{1}{\sqrt{2}} - a\right)x^1_1 + \frac{x^1_2}{\sqrt{2}} - x^0 - \frac{1}{\sqrt{2}}\right]^2 + \left[ \left(\frac{1}{\sqrt{2}} + a\right)x^1_1 + \frac{x^1_2}{\sqrt{2}} +  x^0 - \frac{1}{\sqrt{2}}\right]^2 \\&\quad+ (2ax^1_1 - x^0)^2 = (x_1^1 + x_2^1 -1)^2 + 6a^2 (x_1^1)^2 + 3 (x^0)^2
	\end{align*}
	Note also that by definition of $D=\mathrm{diag}(d_{11}, d_{22})$ we have
	\begin{equation*}
		d_{11} = \sqrt{3} \left[\left(\frac{1}{\sqrt{2}}-a\right)^2 +\left(\frac{1}{\sqrt{2}}+a\right)^2 + 4a^2 \right]^{-\frac{1}{2}} = \sqrt{3}(1+6a^2)^{-\frac{1}{2}} < \sqrt{3}
	\end{equation*}
	and $d_{22} = \sqrt{3}$.
	Hence
	\begin{align*}
		& \frac{1}{2m}  \|Ax^1- x^0\ones_{m} - y\|_2^2 + \lambda \|D^{-1}x^1\|_1 \geq \frac{1}{6} {\|Ax^1- x^0\ones_{m} - y\|_2^2} + \lambda \frac{\|x^1\|_1}{\sqrt{3}} \\
		 =\;&  \frac{1}{6}(x^1_1+x^1_2 - 1)^2 + a^2(x^1_1)^2 + \frac{1}{2} (x^0)^2 + \lambda \frac{\|x^1\|_1}{\sqrt{3}}\geq   \frac{1}{6} (x^1_1+x^1_2 - 1)^2 + \frac{\lambda}{\sqrt{3}}  (x_1^1+x^1_2) \\
		=\;&  \frac{1}{6} \left[\left(x^1_1+x^1_2 - (1-\sqrt{3}\lambda ) \right)^2 - (1-\sqrt{3}\lambda)^2 +1 \right] \geq  1 - (1-\sqrt{3}\lambda )^2 
	\end{align*}
	These inequalities hold as equalities if and only if $x_1^1=x^0=0$, $x_2^1\geq 0$, and $x_1^1+x_2^1=1-\sqrt{3}\lambda$. 
	Therefore, we must have $w^0=0$ and $Dw^1=\hat{w}^1=(1-\sqrt{3}\lambda)e_2$, as desired.
	\end{proof}

\subsection{Proofs of results in \S\ref{sec:BasicProblemsSolutions}}

\begin{proof}[Proof of Lemma \ref{lemma:ProblemBasicExampleLP}]
	Since both $\alpha$ and $\beta$ are assumed to be positive, we have that for any feasible $x$ (that is, $x$ with $x \geq 0$ and $\matLP x = \vecYLP$),
	\begin{equation}\label{eq:InequalitiesForLPBasic}
	\langle c,x \rangle \geq x_1 + x_2 + 1 \geq  \frac{\alpha x_1 + \beta x_2}{\alpha \vee \beta} + 1= \frac{y_1}{\alpha \vee \beta} + 1.
	\end{equation}
	Thus $\min \{\langle c,x \rangle  \, \vert \, x\geq 0, \matLP x = \vecYLP \} \geq y_1/(\alpha \vee \beta)\,+ 1$.
	Moreover, all claimed minimisers in the statement of the lemma obey \eqref{eq:InequalitiesForLPBasic} as an equality and furthermore such $x$ are feasible for the LP problem. Therefore 
	$\min \{\langle c,x \rangle  \, \vert \, x\geq 0, \matLP x = \vecYLP \} = y_1/(\alpha \vee \beta)$, and thus the solutions to $\xilp(\vecYLP,\matLP)$ are precisely the vectors $x$ for which every inequality in \eqref{eq:InequalitiesForLPBasic} is obeyed as an equality. More specifically, the following must occur:
	\begin{itemize}
		\item For the first inequality to be an equality we must have $x_3=0$, $x_4=1$, and  $x_5 = \dotsb = x_N = 0$.
		\item For the second inequality to be an equality we must have $x_2 = 0$ in the case $\alpha > \beta$ and $x_1 = 0$ in the case $\alpha < \beta$. In the case $\alpha = \beta$ this is always an equality.
	\end{itemize}
	It is easy to verify that the $x$ satisfying these properties as well as the conditions $x \geq 0$ and $\matLP x = \vecYLP$ are exactly the claimed minimisers in the statement of the lemma.
\end{proof}

\begin{proof}[Proof of Lemma \ref{lemma:ProblemBasicExampleULASSO}]
Let $x^{opt}=\frac{2\alpha y_1-\lambda}{2\alpha^2} e_1$ if $\alpha\geq \beta$, and $x^{opt}=\frac{2\beta y_1-\lambda}{2\beta^2} e_2$ if $\beta>\alpha$. Define a dual vector $p=\matl x^{opt}-\vecYl=-\frac{\lambda}{2(\alpha\vee \beta)}e_1\in \real^m$, and note that then
\begin{equation*}
 -\frac{2}{\lambda}\matl^* p= \left(1\wedge \frac{\alpha}{\beta}\right)e_1 + \left(1\wedge \frac{\beta}{\alpha}\right)e_2\in \partial \|\cdot\|_1(x^{opt}).
\end{equation*}
Therefore, for every $x\in\real^N$ we have
\begin{align}
\frac{1}{2}\|\matl x -\vecYl\|_2^2 + \frac{1}{2}\lambda\|x\|_1&\geq \langle \matl x -\vecYl,p\rangle  - \frac{1}{2}\|p\|_2^2 + \frac{\lambda}{2} \|x\|_1\notag\\
&=   \langle \matl x^{opt} - \vecYl,p \rangle  - \frac{1}{2}\|p\|_2^2+   \frac{\lambda}{2}\left(\|x\|_1 - \langle  x -x^{opt} , -\frac{2}{\lambda}\matl^*p\rangle   \right)\notag\\
&\geq \frac{1}{2}\|\matl x^{opt} -\vecYl\|_2^2+  \frac{\lambda}{2}\|x^{opt}\|_1\notag.
\end{align}
It follows that $x$ is a minimiser if and only if this string of inequalities holds with equality. This is the case if and only if $\matl x - \vecYl = p =\matl x^{opt}-\vecYl$ and $\|x\|_1 - \langle  x -x^{opt},-\frac{2}{\lambda}\matl^*p\rangle = \|x^{opt}\|_1  $, or equivalently $\alpha x_1 + \beta x_2 = \alpha x^{opt}_1 + \beta x^{opt}_2 = y_1-\frac{\lambda}{2(\alpha\vee \beta )}$, $\{x_j\}_{j=3}^N=\{x^{opt}_j\}_{j=3}^N$,
\begin{equation*}
 |x_1| - (x_1-x_1^{opt})\left(1\wedge \frac{\alpha}{\beta}\right)=|x^{opt}_1|,\quad \text{and}\quad |x_2| - (x_2-x_2^{opt})\left(1\wedge \frac{\beta}{\alpha}\right)=|x^{opt}_2|.
\end{equation*}
It is now straightforward to verify that the minimisers are precisely those as in the statement of the lemma.
\end{proof}

\begin{proof}[Proof of Lemma \ref{lemma:ProblemBasicExampleBPDNL1}]
	From the definition of $\matl$ and $\vecYl$, if $\|\matl x-\vecYl\|_2 \leq \delta$ then $|\alpha x_1 + \beta x_2 - y_1| \leq \delta$. Thus
	\begin{equation}\label{eq:InequalitiesForBPDNBasic}
	(\alpha \vee \beta)\|x\|_1 \geq (\alpha \vee \beta) (|x_1| + |x_2|) \geq \alpha|x_1| + \beta |x_2| \geq \alpha x_1 + \beta x_2 \geq y_1 - \delta.
	\end{equation}
	Thus $\min \{\|x\|_1 \, \vert \, \|\matl x-\vecYl\|_2 \leq \delta \} \geq ( y_1 - \delta)/(\alpha \vee \beta)$. 
	Note that all claimed minimisers $x$ in Lemma \ref{eq:bpdnsolnsbasic} obey \eqref{eq:InequalitiesForBPDNBasic} as an equality and furthermore such $x$ are feasible for the BPDN problem. Therefore 
	$\min \{\|x\|_1 \, \vert \, \|\matl x-\vecYl \|_2 \leq \delta \} = ( y_1 - \delta)/(\alpha \vee \beta)$.
	Thus the solutions to $\xibpdn(\vecYl,\matl)$ are precisely the vectors $x$ for which every inequality in \eqref{eq:InequalitiesForBPDNBasic} is obeyed as an equality. More specifically, the following must occur:
	\begin{itemize}
		\item For the first inequality to be an equality we require $x_3 = x_4 = \dotsb = x_N = 0$.
		\item For the second inequality to be an equality we require $x_2 = 0$ in the case $\alpha > \beta$ and $x_1 = 0$ in the case $\alpha < \beta$. In the case $\alpha = \beta$ this is always an equality.
		\item For the third inequality to be an equality we require $x_1$ and $x_2$ to be non-negative.
		\item For the final inequality to be an equality we need $\alpha x_1 + \beta x_2 = y_1 - \delta$.
	\end{itemize}
	It is easy to verify that the $x$ satisfying these properties are exactly the claimed minimisers in the statement of the lemma.
\end{proof}

\begin{proof}[Proof of Lemma \ref{lemma:ProblemBasicExampleCLASSO}]
	Suppose that $x$ is such that $\|x\|_1 \leq \tau$. Then \begin{equation}\tau - x_3 \geq \sum\limits_{\substack{1 \leq i \leq N\\i \neq 3}} |x_i| \geq |x_1|+|x_2| \geq \frac{\alpha |x_1| + \beta |x_2|}{\alpha \vee \beta} \geq \frac{|\alpha x_1 + \beta x_2|}{\alpha \vee \beta}\label{eq:tau-x3Inequality}\end{equation} and in particular since $|\alpha x_1 + \beta x_2| \geq 0$ we conclude that $(\tau - x_3)^2 \geq (\alpha x_1 + \beta x_2)^2/(\alpha \vee \beta)^2$. Thus for such $x$ we have 
	\begin{align}
	\|\matl x - \vecYCL\|_2^2 &= (\alpha x_1 +\beta x_2 - y_1)^2 + (\tau - x_3)^2 + \sum_{i=4}^{m} x_i^2\notag \geq [\alpha x_1 + \beta x_2 - y_1]^2 + \frac{(\alpha x_1 + \beta x_2)^2}{(\alpha \vee \beta)^2}  \\&=\left[\frac{(\alpha \vee \beta)^2 + 1}{(\alpha \vee \beta)^2}\right]\left[ \alpha x_1 + \beta x_2 - \frac{y_1 (\alpha \vee \beta)^2}{(\alpha \vee \beta)^2+1}\right]^2 + \frac{y_1^2}{(\alpha \vee \beta)^2+1} \geq \frac{y_1^2}{(\alpha \vee \beta)^2+1} \label{eq:InequalitiesForCLBasic}
	\end{align}
	Therefore $\min \{\|\matl x - \vecYCL\|_2 \, \vert \, \|x\|_1 \leq \tau \} \geq y_1 [(\alpha \vee \beta)^2 + 1]^{-1/2}$. 
	Note that all claimed minimisers $x$ in Lemma \ref{lemma:ProblemBasicExampleCLASSO} obey \eqref{eq:InequalitiesForCLBasic} as an equality and furthermore such $x$ have $\|x\|_1 \leq \tau$. Therefore $\min \{\|\matl x - \vecYCL\|_2 \, \vert \, \|x\|_1 \leq \tau \} \geq y_1 [(\alpha \vee \beta)^2 + 1]^{-1/2}$.
	Thus the solutions to $\xicl(\vecYCL,\matl)$ are precisely the vectors $x$ for which every inequality in \eqref{eq:InequalitiesForCLBasic} is obeyed as an equality. More specifically, the following must occur:
	\begin{enumerate}[label= \arabic*.,ref = step \arabic*]
		\item We require each of the equalities in \eqref{eq:tau-x3Inequality} to hold as equality. Thus
		
		\begin{enumerate}[label = \alph*.]
			\item The first inequality in \eqref{eq:tau-x3Inequality} is an equality if and only if $x_3 = \tau - \sum\limits_{\substack{1 \leq i \leq N\\i \neq 3}} |x_i|$.
			\item The second inequality in \eqref{eq:tau-x3Inequality} is an equality if and only if $x_4 = x_5 = \dotsb = x_N = 0$.
			\item The third inequality in \eqref{eq:tau-x3Inequality} is an equality if and only if $x_2 = 0$ if $\alpha > \beta$ or $x_1 = 0 $ if $\alpha < \beta$. If $\alpha = \beta$ then the third inequality in \eqref{eq:tau-x3Inequality} is always an equality.
			\item The final inequality is an equality provided $x_1$ and $x_2$ both have the same sign.
		\end{enumerate}
		\item We require $\alpha x_1 + \beta x_2 = y_1 (\alpha \vee \beta)^2 [(\alpha \vee \beta)^2 + 1]^{-1}$.
	\end{enumerate}
	It is now simple to check that the claimed minimisers in \eqref{eq:unequalAlphaBetaCLASSO} are the only vectors which satisfy these conditions.
	
\end{proof}

\begin{proof}[Proof of Lemma \ref{lemma:BPTVSolutions}]

	 Let $x^{opt}=\eta(y_1,\alpha,\beta)$, if $\beta\geq \alpha$, and $x^{opt}=\flipOp \eta(y_1,\alpha,\beta)$, if $\alpha>\beta$. In both cases we find $\matTV x^{opt} -\vecYTV= -\frac{\delta(m-1)}{\theta}e_1 +\sum_{j=2}^{m}\frac{\delta(\alpha+\beta)}{\theta} e_{j}$, and so $\|\matTV x_{opt} -\vecYTV\|_2 =\delta$, by definition of $\theta$. Next, define a dual vector according to 
	$
	p=-\frac{1}{\alpha\vee \beta}\, e_1+\frac{\alpha+\beta}{(\alpha\vee\beta)(m-1)} \sum_{j=2}^m e_j \in \real^{m}
	$ so that $p= \frac{\theta}{(\alpha \vee \beta) \delta(m-1)}(\matTV x^{opt} -\vecYTV )$ whenever $\delta>0$.
	This in particular implies $\langle \matTV x^{opt} -\vecYTV, p \rangle=\delta\|p\|_2$ and $\delta p/\|p\|_2 = \matTV x^{opt}  - \vecYTV $, for all $\delta\in[0,1]$.
	
	Next, write $D\in\real^{(N-1)\times N}$ for the matrix corresponding to the pairwise differences operator: $Dx=(x_1-x_2,x_2-x_3,\dots, x_{N-1}-x_N)$, so that $\tv{x}=\|Dx\|_{1}$, for $x\in\real^{N}$.
	Defining an auxiliary vector $q=\frac{1}{\alpha\vee \beta }\left( \sum_{j=1}^{N-2} \left(\alpha  - \frac{(\alpha+\beta)\left(j\wedge(m-1)-1\right)}{(m-1)}\right)e_j\; -\beta e_{N-1}\right)\in \real^{N-1}$, we have 
	\begin{equation}\label{eq:BPTV-dual-q}
	1\geq  \frac{\alpha}{\beta} \wedge 1=q_1>q_2>\dots>q_{m-1}=\dots=q_{N-2}= q_{N-1}+\frac{\alpha+\beta}{(\alpha\vee\beta)(m-1)}>q_{N-1}=- \frac{\beta}{\alpha \vee \beta} \geq -1.
	\end{equation}
	Note that $D x^{opt}=- re_{N-1}$ in the case $\beta\geq \alpha$, and $D x^{opt} = r e_1$ in the case $\alpha>\beta$, where $r:=\frac{y_1-\delta\theta/(m-1)}{\alpha\vee \beta}>0$ by the assumption in the statement of the lemma. We therefore deduce $q_j\in {\partial |\cdot|_1\big((Dx^{opt})_j\big)}$, for all $j=1,\dots, N-1$.
	Furthermore,
	\begin{equation*}
	-{\matTV}^* p = \frac{1}{\alpha\vee \beta} \Bigg( \alpha  e_1 - \frac{\alpha+\beta}{m-1}\Big(\displaystyle\sum_{j=2}^{m-1}e_j + e_{N-1}\Big) + \beta e_{N}\Bigg)=D^* q
	\end{equation*}
	and we thus have, for arbitrary $x\in\real^N$ such that $\|\matTV x -\vecYTV \|_2\leq \delta$, the following inequalities
	\begin{align}
	\tv{x} &\stackrel{\mathclap{C-S}}{\geq} \tv{x} +\langle \matTV x  - \vecYTV , p \rangle -\delta\|p\|_2 \label{eq:BPTV-dual-calc-1}\\
	&=\tv{x}  - \langle x-x^{opt} ,  - \matTV^* p \rangle + \underbrace{ \langle \matTV x^{opt}  - \vecYTV , p \rangle -\delta\|p\|_2}_{\mathclap{=0}}\notag \\
	&=\sum_{j=1}^{N-1}\left(|(Dx)_j| -  (Dx -Dx^{opt})_j q_j\right) \geq \sum_{j=1}^{N-1}|(Dx^{opt})_j| =\tv{x^{opt}}.\label{eq:BPTV-dual-calc-2}
	\end{align}
	A feasible $x$ is therefore a minimiser if and only if the inequality in \eqref{eq:BPTV-dual-calc-1} as well as the inequalities in \eqref{eq:BPTV-dual-calc-2} (for each $j$) all hold with equality, that is to say 
	\begin{itemize}
		\item By the Cauchy-Schwarz inequality, \eqref{eq:BPTV-dual-calc-1} holds with equality iff $\matTV x  - \vecYTV =\delta p/\|p\|_2 = \matTV x^{opt}  - \vecYTV  $ i.e. if $\matTV x=\matTV x^{opt}$. By the definition of $\matTV$ this occurs when each of $\alpha x^1 + \beta x_N = \alpha x_1 +\beta x_ N= \alpha x_1^{opt} +\beta x_N^{opt}= y_1 -\frac{\delta(m-1)}{\theta}$, $\{x_j\}_{j=2}^{m-1} = \{x^{opt}_j\}_{j=2}^{m-1}$ and $x_{N-1} = x^{opt}_{N-1}$.
		\item The inequalities in \eqref{eq:BPTV-dual-calc-2} hold with equality iff for each $j=1,2,\dotsc,N-1$ we have $ |(Dx)_j| - (Dx - Dx^{opt})_j q_j = |D(x^{opt})_j|$. 
	\end{itemize}
	Let $x$ be a minimiser. Suppose for now that $\beta\geq \alpha$. Since $D\{x^{opt}_j\}_{j=2}^{N-2} = 0$ and $|q_j|<1$, for all $j\in\{2,\dots,N-2\}$, the requirements for \eqref{eq:BPTV-dual-calc-2} to be an equality imply that $D\{x_j\}_{j=2}^{N-2} = 0$ and hence $x_2 = x_3 = x_4 = \dotsb = x_{N-1}$. It thus follows that $ \{x_j\}_{j=2}^{N-1}= \{x^{opt}_j\}_{j=2}^{N-1}$. Furthermore, the inequalities in \eqref{eq:BPTV-dual-calc-2} hold with equality only if
$
	|x_1-x_2|-\left(x_1-x_2\right)q_1= 0,
$
and
$
|x_{N-1}-x_{N}|-\left(x_{N-1} - x_N +r \right)q_{N-1}=r.
$
As $q_1=\frac{\alpha}{\beta}$ and $q_{N-1}=-1$, it follows that the second of these inequalities is satisfied provided $x_{N}\geq x_{N-1}$, whereas the first is satisfied provided $x_1=x_2$ in the case $\alpha>\beta$, or $x_1\geq x_2$ in the case $\alpha=\beta$. As $x_{N-1}=x_2=\eta(y_1,\alpha,\beta)_1<\eta(y_1,\alpha,\beta)_N $, and appreciating that $x$ must additionally satisfy $\alpha x_1 +\beta x_ N= y_1 -\frac{\delta(m-1)}{\theta}$, we find that $x$ must have the form claimed in the statement of the lemma. 
	
	Similarly, if $\alpha>\beta$, we once again  $ \{x_j\}_{j=2}^{N-1}= \{x^{opt}_j\}_{j=2}^{N-1}$ but now the inequalities in \eqref{eq:BPTV-dual-calc-2} hold with equality only if 
	\begin{equation*}
	\begin{aligned}
	|x_1-x_2|-\left(x_1-x_2-r \right)q_1&= r,\text{ and}\\
	|x_{N-1}-x_{N}|-\left(x_{N-1} - x_N \right)q_{N-1}&=0.
	\end{aligned}
	\end{equation*}
	Then, as $q_1=1$ and $q_{N-1}=-\beta/\alpha$, these are satisfied only if $x_1>x_2$ and $x_{N}=x_{N-1}$. Again appreciating $x_{N-1}=x_2=\eta(y_1,\alpha,\beta)_1$ and $\alpha x_1 +\beta x_ N= y_1 -\frac{\delta(m-1)}{\theta}$ establishes $x=x^{opt}$.
	Conversely, every $x$ of the claimed form satisfies \eqref{eq:BPTV-dual-calc-1} and \eqref{eq:BPTV-dual-calc-2} as equalities, and is therefore a minimiser.
	
\end{proof}
\begin{corollary}\label{Cor:BPTVObjFunction}
	Write $f(\delta) = \min \{\tv{x} \, \vert \,  \|\matTV x-\vecYTV\|_2 \leq \delta\}$. Then 
	\begin{equation*}
	f(\delta) = \begin{cases} \frac{y_1 - \delta\theta/(m-1)}{\alpha \vee \beta}  & \text{if } y_1 > \frac{\delta\theta}{m-1}\\
	0 & \text{if } y_1 \leq \frac{\delta\theta}{m-1}
	\end{cases}.
	\end{equation*}
\end{corollary}
\begin{proof}[Proof of Corollary \ref{Cor:BPTVObjFunction}]
	Assume first that $y_1 > \delta\theta/(m-1)$. Then, using Lemma \ref{lemma:BPTVSolutions}, we see that $f(\delta)= \eta(y_1,\alpha,\beta)_N - \eta(y_1,\alpha,\beta)_1=  \frac{y_1 - \delta\theta/(m-1)}{\alpha \vee \beta}  $. To prove the result in the case $y_1 \leq  \delta\theta/(m-1)$, we note that by the definition of $f$, $f$ is a non-negative decreasing function, and hence we have
	$0 \leq f(\delta) \leq f(y_1(m-1)/\theta) \leq \inf\{ f(\delta')\, \vert\,  \delta'< y_1(m-1)/\theta \}=0$ and thus $f(\delta) = 0$. 
\end{proof}
\begin{proof}[Proof of Lemma \ref{lemma:ULTVSolutions}]
	We can write the unconstrained lasso problem as 
	\begin{equation}\label{eq:ULTVConverted}
	\argmin\limits_{x \in \real^N} \min_{\delta \geq 0} \{\delta^2 + \lambda \tv{x} \, \vert \, \|\matTV x - \vecYTV\|_2 \leq \delta\}.
	\end{equation}
	Let us analyse $\argmin_{\delta \geq 0} \min_{x \in \real^N} \{\delta^2 + \lambda \tv{x} \, \vert \, \|\matTV x - \vecYTV\|_2 \leq \delta\}$. Note that this is equivalent to 
	$\argmin_{\delta \geq 0}\left( \delta^2 +  \min_{x \in \real^N} \{\lambda \tv{x} \, \vert \, \|\matTV x - \vecYTV\|_2 \leq \delta\}\right) = \argmin_{\delta \geq 0}\{\delta^2 + \lambda f(\delta)\}$ where $f(\delta)$ is defined in Corollary \ref{Cor:BPTVObjFunction}. Using Corollary \ref{Cor:BPTVObjFunction}, for $y_1 \leq \delta \theta/(m-1)$ the function $\delta \to \delta^2 + \lambda f(\delta)$ is strictly increasing since $f(\delta) = 0$. Thus we can write  $\argmin_{\delta \geq 0} \min_{x \in \real^N} \{\delta^2 + \lambda \tv{x} \, \vert \, \|\matTV x - \vecYTV\|_2 \leq \delta\}$ as 
	$\argmin\limits_{\delta \geq 0, y_1 \geq \delta\theta/(m-1)} \delta^2 + \lambda f(\delta)$. Noting that $f(\delta) = (\alpha \vee \beta)^{-1} (y_1 - \delta\theta/(m-1))$, we get
	\[\argmin\limits_{\delta \geq 0, y_1 \geq \delta\theta/(m-1)} \delta^2 + \lambda f(\delta)=\argmin_{\delta \geq 0, y_1 \geq \delta\theta/(m-1)} \delta^2 -\lambda \delta\theta (m-1)^{-1}(\alpha \vee \beta)^{-1} + \lambda y_1 (\alpha \vee \beta)^{-1}\] which is minimised when $\delta = \hat{\delta}$, where $\hat{\delta} := \lambda\theta (m-1)^{-1}(\alpha \vee \beta)^{-1}/2$ (we have used here that $\hat{\delta}$ is feasible for the problem since $y_1> \lambda\theta^2[2 (m-1)^2(\alpha \vee \beta)]^{-1} $ and so $y_1 > \hat\delta \theta / (m-1)$). 
	
	We can use this result in \eqref{eq:ULTVConverted} to get
	\begin{align*} 
	\xiultv(\matTV,\vecYTV) &= \argmin\limits_{x \in \real^N}\left\{\hat\delta^2+\lambda \tv{x} \, \middle\vert \, \|\matTV x - \vecYTV\|_2 \leq \hat{\delta}\right\} \\&=\argmin\limits_{x \in \real^N} \left\{ \tv{x} \, \middle\vert \, \|\matTV x - \vecYTV\|_2 \leq \frac{\lambda\theta}{ 2(m-1)(\alpha \vee \beta)}\right\} 
	\end{align*} 
	Noting that $\hat{\delta}\leq \frac{\lambda m }{2(m-1)\frac{1}{4}}\leq 3\lambda \leq 1$, and so Lemma \ref{lemma:BPTVSolutions} can be applied in each of the cases $\alpha < \beta$, $\alpha = \beta$ and $\alpha > \beta$ to reach the desired conclusion.
	
\end{proof}

	\begin{proof}[Proof of Lemma \ref{lemma:ProblemBasicExampleLPObj}]
	Suppose first that  $y_1>0$, $\alpha > 0$, and $\beta \geq 0$.
	Simple gaussian elimination shows that the set of $x$ for which $x \geq 0$ and $\matLPObj x = \vecYl$ is exactly
	\[
	\{x \in \real^N \, \vert \, \alpha x_1 + \beta x_2 = y_1, x_3 = x_4 = \dotsb = x_{m} = x_{m+1} = 0, x_{m+2},x_{m+3},\dotsc ,x_N \geq 0\}
	\]
	 so that $\langle x,c \rangle  = \sum_{i=1}^{N} x_i \geq x_1 \geq (\alpha x_1+\beta x_2)/\alpha \geq y_1/\alpha$.
	Therefore as $\langle x,c \rangle_k$ is increasing as a function of $\langle x,c \rangle$ we have that any $x \geq 0$ such that $\matLPObj x = \vecYl$ must satisfy $\langle x, c \rangle_k \geq \lfloor  10^{k} y_1/\alpha \rfloor  10^{-k}$ and this equality is attained when $x = y_1 e_1 /\alpha$. The proof of \eqref{eq:LPObjBasic} is complete by noting that $ 10^{-k} \lfloor  10^{k} y_1/\alpha \rfloor  \leq 10^{-k} \floor{10^{k} M} \leq M$ if $y_1/\alpha < 10^{-k} (\floor{10^{k} M}+1)$ and $10^{-k} \lfloor  10^{k} y_1/\alpha \rfloor \geq 10^{-k} (\floor{10^{k} M}+1)>M$ if $y_1/\alpha \geq  10^{-k} (\floor{10^{k} M}+1 )$.
	
	Now suppose that $y_1 = 0$. Then $x = 0$ satisfies $\matLPObj x = \vecYl$ and $x\geq 0$. Moreover, $\langle 0, c \rangle_k  = 0\leq M$. This concludes the proof of the lemma.
\end{proof}

\section{Auxiliary results for Theorem \ref{th:smale_comp_sens} -- Sparse recovery}\label{appendix:sparsity}
The concept of sparsity has been dominating modern signal and image processing over the last few decades. Sparsity is a crucial element to this paper and we use this appendix to cover some standard concepts and tools that we will make use of.

\subsection{The basics of sparsity}
Recall that we say that a vector $x \in \mathbb{C}^N$ is $s$-sparse if $x$ has at most $s$ nonzero entries. One of the key mainstays in the theory of sparse recovery is the $\ell^2$-robust nullspace property (RNP).
\begin{definition}[Robust Nullspace Property]\label{def:l2NSPROBUSTSM}
	A matrix $U \in \real^{m \times N}$ satisfies the \emph{$\ell^2$-robust nullspace property of order $s$} with parameters $\rho \in (0,1)$ and $\tau>0$ if
	\begin{equation}\label{eq:NSPDefinitionOrderSM}
		\|v_S\|_2 \leq \frac{\rho}{\sqrt{{s}}} \|v_{S^c}\|_1 + \tau \|Uv\|_2
	\end{equation}
	for all sets $S\subset \{1,\dots, N\}$ of cardinality $s$ and all vectors $v \in \real^n$.
\end{definition} 

Directly showing that a matrix has the RNP is often difficult. Thus, one usually attempts to instead establish the restricted isometry property (RIP) that implies the RNP, recalled next.
\begin{definition}[RIP]A matrix $U \in \complex^{m \times N}$ is said to satisfy the \emph{restricted isometry property (RIP) or order $s$} if there is a $\delta \in [0,1)$ such that
	\begin{equation} \label{eq:RIPDefinition}
		(1-\delta)\|x\|_2^2 \leq \|Ux\|^2_2 \leq (1+\delta) \|x\|^2_2,
	\end{equation}
	for all $s$-sparse $x \in \real^{N}$.
 The \emph{restricted isometry constant of order $s$}, denoted by $\delta_s$, is defined to be the smallest such $\delta$ so that \eqref{eq:RIPDefinition} is satisfied.
	Equivalently, $\delta_s$ is given by 
	\[
	\delta_s = \sup_{ \substack{ T\subset \{1,2,\dotsc,N\} \\ |T| \leq s}} \| P_TU^*UP_T - I_T\|_{2} 
	\]
	where $P_T$ denotes the projection onto the coordinates indexed by $T$.
\end{definition}
We thus have the following result linking the RIP and RNP.
\begin{theorem}[\!{\cite[Theorem 6.13]{foucartBook}}]\label{thm:RIP->RNP}
	If the restricted isometry constant $\delta_{2s}$ of order $2s$ of a matrix $A\in\real^{m\times N}$ obeys $\delta_{2s}<4/\sqrt{41}$, then $A$ satisfies the $\ell^2$-robust nullspace property of order $s$ with parameters
	\begin{equation*}
		\rho=\frac{\delta_{2s}}{\sqrt{1-\delta_{2s}^2}-\delta_{2s}/4}\quad\text{and}\quad \tau=\frac{\sqrt{1+\delta_{2s}}}{\sqrt{1-\delta_{2s}^2}-\delta_{2s}/4}.
	\end{equation*}
\end{theorem} 

It is generally difficult to construct deterministic matrices with the RIP. Instead, one often resorts to random matrices, such as provided by the following result taken from \cite{RudelsonVershyninRIP} and adapted to fit the notation of this paper.
\begin{theorem}[\!{\cite[Theorem 3.9]{RudelsonVershyninRIP}}]\label{thm:RVRandomRIP}
	For every $K> 0$, there exist constants $C_1, C_2$ and $C_3$ depending only on $K$ such that the following occurs. 
	Let $x^1,x^2,\dotsc,x^N$ be $N$-dimensional complex vectors with $N \geq 3$ and $\|x^i\| \leq K$, for all $i$. Assume that $\sum_{i=1}^{N} x^i \otimes x^i = N I_N$ and that $m$ is a natural number with \[N \geq m \geq C_1s\log(N)\epsilon^{-2}\log[s\log(N)\epsilon^{-2}]\log^2{s}\] where $\epsilon \in (0,1)$ and $s$ is an integer with $s \geq 3$. Let $R$ be a random subset of $\{1,2,\dotsc,N\}$ chosen according to a Bernoulli model with probability $m/N$, that is, each entry of $R$ is chosen independently and $i \in R$ with probability $m/N$. Set $X := \sup_{ |T| \leq s } \|I_T - m^{-1}\sum_{i \in R} x^i_T \otimes x^i_T \|_2$, where $x^i_T$ is the vector $x^i$ restricted to a set $T$. Then
	\begin{equation}\label{eq:RVRandomRIPConclusion}
		\probab( X > C_{2} \alpha \epsilon) \leq 3 \exp\left(- C_3\alpha \epsilon m s^{-1}\right) + 2 \exp(-\alpha^2), \quad\text{for all }\alpha > 1.
	\end{equation}
\end{theorem}

We can use Theorem \ref{thm:RVRandomRIP} to show the existence of matrices that satisfy the RIP, as follows:
\begin{theorem}\label{thm:RIPExistenceBOS}
	For every $K> 0$, there exists a constant $C_4$ depending only on $K$ so that for all natural numbers $m$, $N$ and $s$, each at least $3$ and satisfying 
	\begin{equation}\label{eq:RIPThm:mLB} 
	N \geq m \geq C_4s\log(N)\log[s\log(N)]\log^2{s},
	\end{equation}
	as well as any unitary matrix $U$ with $\|U\|_{\max} \leq {K}/{\sqrt{N}}$, there exists a set $S$ of size exactly $m$ such that $\sqrt{{N}/{m}}\, P_{S} U$ has the RIP of order $s$ with $\delta_s \leq 1/5$. 
\end{theorem} 
\begin{proof}
	Fix $K > 0$ and let $C_1, C_2$ and $C_3$ be the constants from Theorem \ref{thm:RVRandomRIP}. We set
	\begin{equation*} 
		\epsilon =  \frac{1}{10C_2}\wedge \frac{1}{2}, \quad \alpha = \frac{1}{5C_2\epsilon} ,\quad C_4 =\left[ C_1 \epsilon^{-2}\left(1 + 2\log \epsilon^{-1}\right)\right] \vee \frac{5C_2s}{\epsilon^2C_3}
	\end{equation*}
	These choices of parameters ensure that $0 < \epsilon \leq 1/2 < 1$, that $\alpha \geq \frac{10C_2}{5C_2} = 2$ and, using \ref{eq:RIPThm:mLB}, that \begin{equation}\label{eq:RIPThm:constantDef}
		m \geq C_1 s\log(N) \epsilon^{-2}\log[s\log(N)]\left[1 + 2\log \epsilon^{-1}\right]\log^2{s} \geq C_1s\log(N)\epsilon^{-2}\log[s\log(N)\epsilon^{-2}]\log^2{s}.
	\end{equation} since both $s$ and $N$ are at least $3$. Furthermore, our choice of $C_4,\alpha$ and $\epsilon$ depend only on $K$.
	
	Next, we choose the vectors $x^1,x^2,\dotsc,x^N$ so that $x^i$ is the $i$th column of $U$ multiplied by $\sqrt{N}$. Such a choice of $x^i$ gives $\|x^i\|_{\infty} \leq K$. Moreover, $\sum_{i=1}^{N} x^i \otimes x^i = NU^*U = I_N$ by the assumption that $U$ is unitary. Therefore, Theorem \ref{thm:RVRandomRIP} applies and we conclude that for a random set $R$ chosen according to the Bernoulli model equation \eqref{eq:RVRandomRIPConclusion} holds.
	
	The choice of the parameters made in \eqref{eq:RIPThm:constantDef} ensures that $C_2 \alpha \epsilon \leq 1/5$. Because $\log(s),\log(N)$ and $\log(s\log(N))$ are each at least $1$, condition \eqref{eq:RIPThm:mLB} implies that $m \geq C_4 \geq \frac{5C_2s}{\epsilon^2 C_3}$. Therefore, again using \eqref{eq:RIPThm:constantDef}, we see that $C_3 \alpha \epsilon m/s \geq \epsilon^{-2}$. Finally, our choice of $x^i$ ensure that $X$ can be written as $X = \|I_T - m^{-1} N P_T U^*P_S U P_T\|_2$. By the definition of the restricted isometry constant this is in fact $\delta_s$ for the matrix $\sqrt{N/m} \, P_S U$. Thus \eqref{eq:RVRandomRIPConclusion} implies that the random variable $\delta_s$ corresponding to the random matrix $\sqrt{N/m} \, P_S U$ satisfies
	$\probab(\delta_s > 1/5) \leq 3 \exp(-\epsilon^{-2})  + 2 \exp(-\alpha^2) \leq 5 \exp(-4)$.
	
	At this point, the proof would be done if the cardinality of $S$ were fixed since $5\exp(-4) < 1$ and hence there must exist at least one such $R$ with $\delta_s \leq 1/5$. Unfortunately, the cardinality of $R$ is a random variable since $S$ is selected according to the Bernoulli model. We will therefore consider random sets $S'$ chosen as follows: let $X$ be the set of all subsets of $\{1,2,\dotsc,N\}$ with cardinality $m$. We take $S$ to be an element of $X$ chosen uniformly at random. Using an argument similar to the one presented in \cite[p. 468]{foucartBook}
	we will bound the probability that the random variable $\delta_s$ corresponding to the random matrix $\sqrt{N/m} \, P_{S'} U$ exceeds $1/5$.

	To this end, let $E^u_U$ be the event that there exists an $s$-sparse unit vector $x$ such that $\|\sqrt{N/m} \, P_{S'} U x\|^2_2 > 6/5$ and let $E^l_U$ be the event that there exists an $s$-sparse unit vector $x$ such that $\| \sqrt{N/m} \, P_{S'} U x\|^2_2 < 4/5$. Similarly, for the random sets $S$ chosen according to the Bernoulli model, let $E^u_B$ be the event that there exists an $s$-sparse unit vector $x$ such that $\|\sqrt{N/m} \, P_{S} U x\|^2_2 > 6/5$ and let $E^l_B$ be the event that there exists an $s$-sparse unit vector $x$ such that $\|\sqrt{N/m} \, P_{S} U x\|^2_2 < 4/5$. We have shown already that $\probab(E^l_B \cup E^u_B) \leq 5\exp(-4)$.

	For a given $i$, let $B_i$ denote the collection of sets $S$ with $S \subseteq \{1,2,\dotsc,N\}$,  $|S| = i$ and such that there exists an $s$-sparse unit vector $x$ such that $\|\sqrt{N/m}\,P_{S} U x\|^2_2 > 6/5$. Because the Bernoulli model selects elements independently and with equal probability, $\probab\left(E^u_B\, \big\vert \, |S| = i\right)  = {|B_i|} /{{n \choose i}}$. By an argument originally used by Sperner to prove Sperner's theorem \cite[p.3]{SpernerTheory}, the shade $\nabla B_i$ of $B_i$ defined by 
	\[\nabla B_i = \{ \hat S \subseteq \{1,2,\dotsc,N\} |, \vert \, |\hat S| = i+1 \text{ and } \exists S \in B_i \text{ with } S \subseteq \hat S\}  
	\]  
	satisfies $|\nabla B_i| \geq (n-i)|B_i|/(i+1) $. Moreover, if $S$ is in $B_i$ and $\hat S$ is such that $|\hat S| = i+1$ and $S \subset \hat S$ then $\|\sqrt\frac{N}{m}P_{\hat S } Ax\|^2_2 \geq \|\sqrt\frac{N}{m}P_S Ax\|^2$ and hence $\hat S \in B_{i+1}$. Thus $\nabla B_i \subset B_{i+1}$. Therefore for $i \leq N-1$ we have
	\begin{equation*}
		\probab\left(E^u_B\, \big\vert \, |S| = i\right)  = \frac{|B_i|} {{n \choose i}} \leq \frac{(i+1)|\nabla B_i|}{{n \choose i} (n-i)} = \frac{|\nabla B_i|}{{n \choose i+1} } \leq \frac{|B_{i+1}|}{{n \choose i+1} }  = \probab\left(E^u_B\, \big\vert \, |S| = (i+1)\right).
	\end{equation*}  
	
	Let $A_i$ denote the collection of sets $S$ with $S \subseteq \{1,2,\dotsc,N\}$,  $|S| = i$ and such that there exists an $s$-sparse unit vector $x$ such that $\|\sqrt{N/m}\,P_{S} U x\|^2_2 < 4/5$. A similar argument to the preceding one for $B_i$ this time using the fact (also proven by Sperner) that the shadow $\Delta A_i$ of $A_i$ defined by 
	\[\Delta A_i = \{ \hat S \subseteq \{1,2,\dotsc,N\} |, \vert \, |\hat S| = i-1 \text{ and } \exists S \in A_i \text{ with } \hat S \subseteq S\}  
	\] 
	satisfies $|\Delta A_i| \geq (i |A_i|)/(n-i+1)$ gives
	$	\probab\left(E^l_B\, \big\vert \, |S| = i\right) \leq \probab\left(E^l_B\, \big\vert \, |S| = (i-1)\right)$.
	Thus
	\begin{align} \probab(E^u_B) = \sum_{i=0}^{N} \probab\left(E^u_B\, \big\vert \, |S| = i\right) \probab(|S| = i) &\geq \sum_{i=m}^{N}\probab\left(E^u_B\, \big\vert \, |S| = i\right) \probab(|S| = i)\notag \\&\geq \probab\left(E^u_B\, \big\vert \, |S| = m\right) \sum_{i=m}^{N} \probab(|S| = i) \geq \frac{\probab(E^u_U)}{2}. \label{eq:RIPThm:UBBernoulliUniform}
	\end{align}
	where the final inequality follows because $m$ is the median of the random variable $|S|$. Similarly we obtain 
	\begin{equation}
		\probab(E^l_B) \geq \sum_{i=0}^{m} \probab\left(E^u_B\, \big\vert \, |S| = i\right) \probab(|S| = i) \geq \frac{\probab(E^l_U)}{2}.\label{eq:RIPThm:LBBernoulliUniform}
	\end{equation}
	
	Combining \eqref{eq:RIPThm:UBBernoulliUniform}, \eqref{eq:RIPThm:LBBernoulliUniform} and the already established result $\probab(E^l_B \cup E^u_B) \leq 5\exp(-4)$ gives 
	\begin{equation*}
		\probab(E^l_U \cup E^u_U) \leq \probab(E^l_U) + \probab(E^u_U) \leq 2 \left[\probab(E^l_B) + \probab(E^u_B)\right] \leq 4 \probab(E^l_B \cup E^u_B) \leq 20 \exp(-4) < 1
	\end{equation*}
	Hence $\probab(\delta_s >1/5) <1$ where $\delta_s$ is the random RIP constant associated to the random matrix $P_{S'}U$ such that $S'$ a uniformly randomly chosen subset of $\{1,2,\dotsc,N\}$ with $|S'| = m$. 
	
	Thus there exists a set $S$ with $|S| = m$ such that the matrix $P_{S} U$ has $\delta_{s} \leq 1/5$, as otherwise we would have $\probab(\delta_s \leq 1/5) =0$ and this would be a contradiction. 
\end{proof}
We can apply this result to the discrete cosine matrix to derive the following:
\begin{theorem}\label{thm:NSPExistenceDCT}
	There exists a constant $C_5$ with the following property: for any natural numbers $m$,$N$, and $s$ with $N \geq m \geq 3$ and $s\geq 2$ obeying 
	\begin{equation}\label{eq:NSPThm:mLB}
	N \geq m \geq C_5s\log(N)\log[s\log(N)]\log^2{s},
	\end{equation}  
	there exists a matrix $F \in \real^{m \times N}$ with the RNP of order $s$ with parameters $\rho < 1/3$ and $\tau < 2$ that also satisfies $\|F\|_2 \leq \sqrt{N/m}$. \end{theorem}
\begin{proof}
	
	Let $U\in\real^{N\times N}$ be the matrix corresponding to the second variant of the discrete cosine transform, i.e., its entries are given by \begin{equation*} U_{jk}=\begin{cases} \sqrt{\frac{2}{N}}\cos\left(\frac{\pi}{2N}(j-1)(2k-1)\right) &\text{ for } j,k\in\{1,\dots,N\}, j\neq 1 \\\sqrt{\frac{1}{N}} & \text{ for } j = 1 \end{cases} \end{equation*} 
	Note that $U$ is an orthonormal matrix with $\|U\|_{\max} \leq \sqrt{{2/N}}$. Take $C_5 = 16C_4$ where $C_4$ is taken from the statement of Theorem \ref{thm:RIPExistenceBOS} with $K = \sqrt{2}$. Then condition \ref{eq:NSPThm:mLB} implies that
	\[
	N \geq m \geq 2C_4s\log(N) (2\log(s\log(N)))\left(2\log(s)\right)^2 \geq  2C_4s\log(N) (\log(2s\log(N)))\left(\log(2s)\right)^2
	\]
	since $s \geq 2$. Thus, we can apply Theorem \ref{thm:RIPExistenceBOS} to obtain a set $S$ so that $F:=\sqrt{N/m}\, P_S U$ obeys the RIP of order $2s$ with $\delta_{2s}\leq 1/5$. Hence, by Theorem \ref{thm:RIP->RNP}, the matrix $F$ satisfies the RNP with parameters
	\begin{equation*}
		\rho:=\frac{1/5}{\sqrt{1-(1/5)^2}-(1/5)/4}<\frac{1}{3}\quad\text{and}\quad \tau:= \frac{\sqrt{1+1/5}}{\sqrt{1-(1/5)^2}-(1/5)/4}<2\,.
	\end{equation*}
	The proof is complete by noting that $\|F\|_2 \leq \sqrt{N/m}\, \|U\|_2 \leq \sqrt{N/m}$.
	
\end{proof}

This paper also makes use of the following results:
\begin{theorem}[\!{\cite[Theorem 4.25]{foucartBook}}] \label{l2RNPerrorest}
	Suppose that $A\in\mathbb{C}^{m\times N}$ satisfies the RNP of order $s$ with parameters $\rho\in(0,1)$ and $\tau>0$.
	Then, for all $x,z\in\mathbb{C}^N$, we have
	\begin{equation}
		\|x-z\|_2\leq \frac{(1+\rho)^2}{1-\rho}\frac{1}{\sqrt{s}}\left(\|z\|_1-\|x\|_1+2\sigma_s(x)_1\right)+\frac{(3+\rho)\tau}{1-\rho}\|Ax-Az\|_2,
	\end{equation}
	where $\sigma_s(x)_1=\min\{\|x-y\|_1 \,\vert\,\text{$s$-sparse } y\in\real^N\}$.
\end{theorem}

\begin{lemma}\label{lemma:NSPImplicationBoundingSparse}
	Let $b_y,\varepsilon_y>0$ and suppose $A\in\mathbb{R}^{m\times N}$ satisfies the RNP of order $s$ with parameters $\rho\in(0,1)$ and $\tau>0$. Suppose we are given $y\in\real^m$ such that $\|y\|_2\leq b_y\sqrt{N/m}$ and $\|Ax-y\|_2\leq \varepsilon$, for some $s$-sparse $x \in \real^{N}$. Let $A'$ and $y'$ be such that
	\begin{equation*}
		\|y-y'\|_2\leq \varepsilon_y,\qquad\|A-A'\|_{2}\leq\frac{\varepsilon_y}{ \tau\left( \varepsilon+ b_y\sqrt{N/m}\right)}.
	\end{equation*}
	Then
	\begin{equation*}
		\|x\|_2\leq \tau\left(\varepsilon +  b_y \sqrt{N/m} \right)\quad\text{and}\quad\|A'x-y'\|_2 \leq \varepsilon+ 2\varepsilon_y.
	\end{equation*}
\end{lemma}
\begin{proof}[Proof of Lemma \ref{lemma:NSPImplicationBoundingSparse}]
	Let $S=\mathrm{supp}(x)$. Now, by applying \eqref{eq:NSPDefinitionOrderSM} to $x$ we have
	\begin{equation*}
		\begin{aligned}
			\|x\|_2&\leq\tau\|Ax\|_2\leq\tau\left(\|Ax-y\|_2+\|y\|_2\right)\leq \tau\left( \varepsilon+ b_y\sqrt{N/m}\right),
		\end{aligned}
	\end{equation*}
	and thus 
	\begin{equation*}
		\begin{aligned}
			\|A'x-y'\|_2&\leq\|A-A'\|_{2}\|x\|_2+\|Ax-y\|_2+\|y-y'\|_2\\
			&\leq \|A-A'\|_{2}\cdot \tau\left( \varepsilon+ b_y\sqrt{N/m}\right)+\varepsilon+\varepsilon_y\leq \varepsilon + 2\varepsilon_y.
		\end{aligned}
	\end{equation*}
\end{proof}
\begin{lemma}[\!{\cite[Lemma 8.5]{AdcockHansenBook}}] \label{lemma:RNSPIsStableProperty}
	 Suppose $A \in \real^{m \times N}$ has the RNP of order $s$ with constants $\rho \in (0,1)$ and $\tau > 0$. If $A' \in \real^{m \times N}$ satisfies $\|A'-A\|_2 \leq \epsilon$ where $\epsilon$ is a non-negative real number with $\epsilon \leq \frac{1-\rho}{\tau(\sqrt{s}+1)}$ then $A'$ satisfies the RNP of order $s$ with constants $\rho'$ and $\tau'$ satisfying
	\begin{equation*}
		\rho' = \frac{\rho + \tau \epsilon \sqrt{s}}{1-\tau \epsilon},\quad \tau' = \frac{\tau}{1-\tau\epsilon}
	\end{equation*}
\end{lemma}

\subsection{Standard matrices used in the theory of sparsity}\label{sec:CSMatricesAppendix}
We recall here two types of matrices that are frequently used in compressive sensing and sparse regularisation applications:

\subsubsection*{Hadamard matrices}
In this paper we will only consider Hadamard matrices of size $2^n \times 2^n$, for $n \in \mathbb{N}$. We define a `naturally ordered' Hadamard matrix $H_n$ with entries $\pm 1$ of dimension $2^n \times 2^n$ by the recurrence relation 
\begin{equation*}
	H_j = H_1 \otimes H_{j-1}, \quad H_1 = \begin{pmatrix} 1 & 1 \\ 1 &- 1 \end{pmatrix} \quad H_0 = 1, \quad j \in \mathbb{N}.
\end{equation*} 
where $\otimes$ denotes the Kronecker product. $H_n$ has orthogonal rows, columns and the property that $H_n^*H_n = 2^{n}I_{2^n}$ where $I_{2^n} \in \real^{2^n \times 2^n}$ is the identity matrix. We also consider Hadamard matrices in the `sequency ordering': that is, the rows of $H_n$ are ordered so that the number of sign changes in a given row is increasing. We shall make the distinction clear where important.

\subsubsection*{Hadamard-to-Haar matrices}

Let $H_{n} \in \mathbb{R}^{2^n \times 2^n}$ be the Hadamard matrix in the sequency ordering and let $W_n \in \mathbb{R}^{2^n \times 2^n}$ be the 1D discrete Haar wavelet transform matrix. Then, by \cite[Lemma 1 \& Lemma 2]{LauraWalshHaar} we have $H_{n}W_n^{-1} = \big(\hspace{1pt} 1\,\big) \oplus \bigoplus_{i=0}^{n-1} X^i$ where the matrix $X^{i} \in \mathbb{R}^{2^{i} \times 2^{i}}$ is unitary and satisfies $|(X^i)_{jk}|=1/\sqrt{2^i}$, for all $i=0,\dots, n-1$ and $j,k=1,\dots, 2^{i}$. Now, a row-subsampled Hadamard-to-Haar matrix is any matrix of the form
\begin{equation}\label{eq:HadToHaarDecomposition}
	A = \big(\hspace{1pt} 1\,\big) \oplus \bigoplus_{i=0}^{n-1} c_i P_{S_i}X^{i},
\end{equation} 
where the $S_i$ are subsets of $\{1,\dots, 2^{i-1}\}$,$P_{S_i}:\real^{2^{i-1}}\to \real^{|S_{i}|}$ are the corresponding projection operators selecting the coordinates in $S_i$ and $c_i = \sqrt{\frac{2^{i}}{|S_i|}}$ for nonempty $S_i$ and $c_i = 0$ if $S_i = \varnothing$.
Matrices formed by taking the product of a Hadamard transform with an inverse wavelet transform, like the Hadamard-To-Haar matrices defined above, have proven to be very effective in compressive sensing, particularly on imaging applications \cite{FloHadHaar}.

\section{Separation oracles and the ellipsoid algorithm}\label{appendix:ellipsoid}

In this section we define the concepts needed in the statements of Theorem \ref{thm:Ellipsoid} and Theorem \ref{thm:Ellipsoid-BSS}, namely those of the weak optimisation problem, weak separation oracle, encoding functions and polynomially separable classes. We start with the weak optimisation problem, which we define as a minimisation problem (rather than a maximisation problem as in \cite{Lovasz_book}) for convenience.

\begin{definition}[{\cite[Def. 2.1.10]{Lovasz_book}}, Weak optimisation problem]\label{def:WoptProb}
	Let $\CompactK\subset\real^n$ be a compact convex set, and suppose that $R>0$ is a rational  such that $\CompactK\subset \clBall{R}{0}$. Furthermore, let $c\in\mathbb{Q}^n$ and  $\zeta\in\mathbb{Q}$. The weak optimisation problem $(\CompactK,R,c,\zeta)$ is the task to either
	\begin{itemize}
		\item[(a)] find a $z^*\in\mathbb{Q}^n$ such that $z^*\in S(\CompactK,\zeta)$ and $\langle c,z^*\rangle\leq\langle c,z\rangle +\zeta $, for all $z\in S(\CompactK,-\zeta)$, or
		\item[(b)] assert that $S(\CompactK,-\zeta)=\varnothing$.
	\end{itemize}
\end{definition}
The following is the definition of a weak separation oracle, which we present as a synthesis of  \cite[Assump. 1.2.1]{Lovasz_book},  \cite[Def. 2.1.13]{Lovasz_book}, and the discussion on pages 54 and 55 of \cite{Lovasz_book} on the description of compact convex sets by means of a separation oracle.

\begin{definition}[Weak separation oracle]\label{def:WsepOr}
	Let $\CompactK\subset\real^n$ be a compact convex set. We say that a procedure $\mathrm{SEP}_\CompactK$ is a weak separation oracle for $\CompactK$ if, given a vector $w\in\mathbb{Q}^n$ and a rational $\xi>0$,
	\begin{itemize}
		\item[(a)] $\mathrm{SEP}_\CompactK$ either outputs a $d\in\mathbb{Q}^n$ with $\|d\|_\infty=1$ and  such that $\langle d,z\rangle\leq\langle d,w\rangle +\xi $, for all $z\in S(\CompactK,-\xi)$, or 
		\item[(b)] $\mathrm{SEP}_\CompactK$ asserts that $w\in S(\CompactK,\xi)$.
	\end{itemize}
	Moreover, in the Turing case,  we insist that there exist a polynomial $P_{\CompactK}:\real\to\real$ such that, whenever $\mathrm{SEP}_\CompactK$ outputs a vector $d$ as in item (i), we have $\length(d) \leq P_{\CompactK}( \length( w )  +\length( \xi )  )$.
\end{definition}

\noindent Note that by the separating hyperplane theorem \cite{Boyd} applied to the convex sets $K$ and $\{w\}$, at least one of (a) and (b) in Definition \ref{def:WsepOr} are satisfied whenever $K$ is non-empty. Moreover, (a) is trivially satisfied for any $d \in \mathbb{Q}^n$ with $\|d\|_{\infty}=1$ if $K$ is empty.  

\begin{remark}
	When considering the BSS instead of the Turing model, all quantities in Definitions \ref{def:WoptProb} and \ref{def:WsepOr} specified to be rationals are allowed to be (not necessarily rational) real numbers.
\end{remark}

Next, we introduce Turing encoding functions which represents the encoding of the various sets $\CompactK$ in a form that can be presented to a Turing machine as input.
\begin{definition}[Turing encoding function]\label{def:EncFun}
	Let $\Alphabet^*$ denote the set of finite-length strings in a finite alphabet  $\Alphabet$. A Turing encoding function for $\mathscr{K}$ is an injective function $\DataTur:\mathscr{K}\to\Alphabet^*$.
\end{definition}
\noindent Note that the explicit form of $\DataTur$ depends on the particular class $\mathscr{K}$ under consideration. The condition on $\DataTur$ in Definition \ref{def:EncFun} simply states that the sets are encoded uniquely. For example, for basis pursuit denoising where our compact convex sets will be of the form $\{z\in\real^N \, \vert \, \|A'z-y'\|_2\leq \delta', \|z\|_2\leq R' \}$, where $\delta',R'\in\mathbb{D}$, $y'\in\mathbb{D}^m$, and $A'\in\mathbb{D}^{m\times N}$, one possible encoding in the alphabet $\Alphabet= \{0, \;1, -\,,\; .\,,\; ; \}$ is
\begin{equation*}
	m\; ;\;  N   \; ; \;  \delta'  \; ; \; R' \; ; \; y'_1 \; ; \; y'_2  \; \cdots\; ; \; y'_m \; ; \; A'_{1,1} \; ; \; A'_{1,2}   \; \cdots\; ; \; A'_{m,N}  \quad\in \Alphabet^*
\end{equation*}
where all the dyadic rationals are written out in their binary representation. 

We are now ready to define polynomially separable classes in the Turing model.

\begin{definition}[Polynomially separable class -- Turing case]\label{def:Tur-poly-sep-class}
	Suppose $\mathscr{K}$ is a circumscribed class equipped with a Turing encoding function $\DataTur:\mathscr{K}\to\Alphabet^*$. We say that $\mathscr{K}$ is \emph{Turing-polynomially separable with respect to $\DataTur$} if there exist 
	a Turing machine that takes in $\DataTur(\CompactK,n,R) \in\Alphabet^*$, a $w\in\mathbb{Q}^n$, and a rational $\xi>0$ as its input and acts as a weak separation oracle for $\CompactK$, i.e., it either
	\begin{itemize}
		\item[(a)] outputs a $d\in\mathbb{Q}^n$ with $\|d\|_\infty=1$ and  such that $\langle d,z\rangle\leq\langle d,w\rangle +\xi $, for all $z\in \CompactK$, or
		\item[(b)] asserts that $w\in S(\CompactK,\xi)$,
	\end{itemize}
	such that the runtime of the Turing machine is bounded by a polynomial of $\length\big(\DataTur(\CompactK,n,R)\big)$, $ \length(R)$,  $\length(w)$, $\length(\xi)$, and $n$.
\end{definition}

Finally, we present analogues of the concepts above for the BSS model of computation, following the ideas in \cite[Sec. 1.1, Sec. 1.2]{Nemirovski1995}.

\begin{definition}[BSS encoding function]\label{def:EncFunBSS}
	We define $\InputVecsBSS=\bigcup_{k=1}^\infty \real^k$, i.e., the set of real vectors of arbitrary length. A BSS encoding function for $\mathscr{K}$ is a function $\DataBSS:\mathscr{K}\to\InputVecsBSS$ so that, for all $(\CompactK_1,n,R),(\CompactK_2,n,R)\in\mathscr{K}$,  $\CompactK_1\neq\CompactK_2$ implies $\DataBSS(\CompactK_1,n,R)\neq\DataBSS (\CompactK_2,n,R)$.
\end{definition}

In analogy to the Turing encoding function, $\DataBSS$ serves to encode $\CompactK$ as a vector of reals, which can be accepted as input by a BSS machine.  For basis pursuit denoising, the convex set $\{z\in\real^N\,\vert\, \|A'z-y'\|_2\leq \delta', \|z\|_2\leq R \}$, where $y'\in\mathbb{R}^m$ and $A'\in\mathbb{R}^{m\times N}$, can be encoded as
\begin{equation*}
	\left( m,  N,  \delta', R' , y'_1 , \cdots,  y'_m ,  A'_{1,1} , A'_{1,2}   , \cdots ,  A'_{m,N} \right)\in\real^{4+m+mN}\subset\InputVecsBSS.
\end{equation*}

\begin{definition}[Polynomially separable class -- BSS case]\label{def:BSS-poly-sep-class}
	Suppose $\mathscr{K}$ is a circumscribed class equipped with a BSS encoding function $\DataBSS:\mathscr{K}\to \InputVecsBSS$. We say that $\mathscr{K}$ is \emph{BSS-polynomially separable with respect to $\DataBSS$} if there exists a BSS machine that takes in $\DataBSS(\CompactK,n,R) \in\InputVecsBSS$, a $w\in\mathbb{\real}^n$, and a real $\xi>0$ as its input and acts as a weak separation oracle for $\CompactK$, i.e., it either
	\begin{itemize}
		\item[(a)] outputs a $d\in\mathbb{\real}^n$ with $\|d\|_\infty=1$ and  such that $\langle d,z\rangle\leq\langle d,w\rangle +\xi $, for all $z\in \CompactK$, or
		\item[(b)] asserts that $w\in S(\CompactK,\xi)$,
	\end{itemize}
	such that the runtime of the BSS machine is bounded by a polynomial of $\dim(\DataBSS(\CompactK,n,R))$, i.e., the dimension of the real vector $\DataBSS(\CompactK,n,R)$.
\end{definition}

\noindent Note that, unlike in the Turing case, we now have no concept of the length of the encoding of a convex set, but instead the separation oracle must be executable in runtime which is polynomial only in the dimension of the data vector encoding the convex set.

\bibliographystyle{abbrv}
\bibliography{FoundOpt}

\begin{thebibliography}{100}

\bibitem{Adcock2016}
B.~Adcock and A.~C. Hansen.
\newblock Generalized sampling and infinite-dimensional compressed sensing.
\newblock {\em Foundations of Computational Mathematics}, 16(5):1263--1323,
  2016.

\bibitem{AdcockHansenBook}
B.~Adcock and A.~C. Hansen.
\newblock {\em Compressive Imaging: Structure, Sampling, Learning}.
\newblock Cambridge University Press, 2021.

\bibitem{Breaking}
B.~Adcock, A.~C. Hansen, C.~Poon, and B.~Roman.
\newblock Breaking the coherence barrier: A new theory for compressed sensing.
\newblock {\em Forum of Mathematics, Sigma}, 5:1--84, 001 2017.

\bibitem{Arora2007}
S.~Arora and B.~Barak.
\newblock {\em Computational Complexity - A Modern Approach}.
\newblock Princeton University Press, 2009.

\bibitem{Arora_JACM_98_2}
S.~Arora, C.~Lund, R.~Motwani, M.~Sudan, and M.~Szegedy.
\newblock Proof verification and the hardness of approximation problems.
\newblock {\em J. ACM}, 45(3):501--555, May 1998.

\bibitem{Arora_JACM_98}
S.~Arora and S.~Safra.
\newblock Probabilistic checking of proofs: A new characterization of np.
\newblock {\em J. ACM}, 45(1):70--122, Jan. 1998.

\bibitem{NNInstab}
A.~Bastounis, A.~C. Hansen, and V.~{Vla\v{c}i\'{c}}.
\newblock The mathematics of adversarial attacks in ai -- why deep learning is
  unstable despite the existence of stable neural networks.
\newblock {\em arXiv:2109.06098}, 2021.

\bibitem{Becker_2011}
S.~Becker, J.~Bobin, and E.~J. Cand{\`{e}}s.
\newblock {NESTA}: A fast and accurate first-order method for sparse recovery.
\newblock {\em {SIAM} Journal on Imaging Sciences}, 4(1):1--39, jan 2011.

\bibitem{Sudan}
M.~Bellare, O.~Goldreich, and M.~Sudan.
\newblock Free bits, pcps, and nonapproximability -- towards tight results.
\newblock {\em SIAM J. Comput.}, 27(3):804--915, 1998.

\bibitem{SCI}
J.~Ben-Artzi, M.~J. Colbrook, A.~C. Hansen, O.~Nevanlinna, and M.~Seidel.
\newblock Computing spectra -- {O}n the solvability complexity index hierarchy
  and towers of algorithms.
\newblock {\em arXiv:1508.03280}, 2020.

\bibitem{CRAS}
J.~Ben-Artzi, A.~C. Hansen, O.~Nevanlinna, and M.~Seidel.
\newblock New barriers in complexity theory: On the solvability complexity
  index and the towers of algorithms.
\newblock {\em Comptes Rendus Mathematique}, 353(10):931 -- 936, 2015.

\bibitem{ben2021computing}
J.~Ben-Artzi, M.~Marletta, and F.~R{\"o}sler.
\newblock Computing the sound of the sea in a seashell.
\newblock {\em Found. Comput. Math.}, pages 1--35, 2021.

\bibitem{Nemirovski_robust}
A.~Ben-Tal, L.~El~Ghaoui, and A.~Nemirovski.
\newblock {\em Robust Optimization}.
\newblock Princeton Series in Applied Mathematics. Princeton University Press,
  October 2009.

\bibitem{NemirovskiLMCO}
A.~Ben-Tal and A.~Nemirovski.
\newblock {Lectures on Modern Convex Optimization: Analysis, Algorithms, and
  Engineering Applications}.
\newblock Available online at \url{https://www2.isye.gatech.edu/~nemirovs/},
  2000.

\bibitem{Nemirovski_robust2}
A.~Ben-Tal and A.~Nemirovski.
\newblock Robust solutions of linear programming problems contaminated with
  uncertain data.
\newblock {\em Mathematical Programming}, 88(3):411--424, 2000.

\bibitem{Nemirovski_Book2001}
A.~Ben-Tal and A.~S. Nemirovski.
\newblock {\em Lectures on Modern Convex Optimization: Analysis, Algorithms,
  and Engineering Applications}.
\newblock Society for Industrial and Applied Mathematics, Philadelphia, PA,
  USA, 2001.

\bibitem{bishop1967foundations}
E.~Bishop.
\newblock {\em Foundations of Constructive Analysis}.
\newblock McGraw-Hill Series in higher mathematics. McGraw-Hill, 1967.

\bibitem{BCSS}
L.~Blum, F.~Cucker, M.~Shub, and S.~Smale.
\newblock {\em Complexity and Real Computation}.
\newblock Springer-Verlag New York, Inc., Secaucus, NJ, USA, 1998.

\bibitem{BSS_machine}
L.~Blum, M.~Shub, and S.~Smale.
\newblock {On a theory of computation and complexity over the real numbers:
  $NP$- completeness, recursive functions and universal machines}.
\newblock {\em Bulletin of the American Mathematical Society}, 21(1):1 -- 46,
  1989.

\bibitem{Boyd}
S.~Boyd and L.~Vandenberghe.
\newblock {\em Convex Optimization}.
\newblock Cambridge University Press, New York, NY, USA, 2004.

\bibitem{bravermancook2006computing}
M.~Braverman and S.~Cook.
\newblock {Computing over the reals: Foundations for scientific computing}.
\newblock {\em Notices of the American Mathematical Society}, 53(3):318--329,
  2006.

\bibitem{Condition}
P.~B{\"u}rgisser and F.~Cucker.
\newblock {\em Condition : the geometry of numerical algorithms}.
\newblock Grundlehren der mathematischen Wissenschaften. Springer, Berlin,
  Heidelberg, New York, 2013.

\bibitem{Osher_JAMS}
J.-F. Cai, B.~Dong, S.~Osher, and Z.~Shen.
\newblock Image restoration: Total variation, wavelet frames, and beyond.
\newblock {\em J. Amer. Math. Soc.}, 25(4):1033--1089, 2012.

\bibitem{candesCSMag}
E.~J. Cand{\`e}s.
\newblock An introduction to compressive sensing.
\newblock {\em IEEE Signal Process. Mag.}, 25(2):21--30, 2008.

\bibitem{CandesRombergTao}
E.~J. Cand{\`e}s, J.~Romberg, and T.~Tao.
\newblock Robust uncertainty principles: exact signal reconstruction from
  highly incomplete frequency information.
\newblock {\em IEEE Trans. Inform. Theory}, 52(2):489--509, 2006.

\bibitem{Chambolle_Alg}
A.~Chambolle.
\newblock An algorithm for total variation minimization and applications.
\newblock {\em Journal of Mathematical Imaging and Vision}, 20(1):89--97, 2004.

\bibitem{Chambolle_Lions}
A.~Chambolle and P.-L. Lions.
\newblock Image recovery via total variation minimization and related problems.
\newblock {\em Numerische Mathematik}, 76(2):167--188, 1997.

\bibitem{Chambolle_2011}
A.~Chambolle and T.~Pock.
\newblock A first-order primal-dual algorithm for convex problems with
  applications to imaging.
\newblock {\em J. Math. Imaging Vis.}, 40(1):120--145, May 2011.

\bibitem{Chan}
T.~F. Chan, G.~H. Golub, and P.~Mulet.
\newblock A nonlinear primal-dual method for total variation-based image
  restoration.
\newblock {\em SIAM J. Sci. Comput.}, 20(6):1964--1977, May 1999.

\bibitem{Donoho_BP}
S.~S. Chen, D.~L. Donoho, and M.~A. Saunders.
\newblock Atomic decomposition by basis pursuit.
\newblock {\em SIAM Rev.}, 43(1):129--159, Jan. 2001.

\bibitem{CohenDahmenDeVore}
A.~Cohen, W.~Dahmen, and R.~DeVore.
\newblock Compressed sensing and best {$k$}-term approximation.
\newblock {\em J. Amer. Math. Soc.}, 22(1):211--231, 2009.

\bibitem{Colbrook_2019}
M.~Colbrook.
\newblock On the computation of geometric features of spectra of linear
  operators on hilbert spaces.
\newblock {\em Found. Comp. Math.}, (to appear).

\bibitem{colbrook2019foundations}
M.~Colbrook and A.~C. Hansen.
\newblock The foundations of spectral computations via the solvability
  complexity index hierarchy.
\newblock {\em J. Eur. Math. Soc.}, (to appear).

\bibitem{Colbrook_2021}
M.~J. Colbrook.
\newblock Computing spectral measures and spectral types.
\newblock {\em Communications in Mathematical Physics}, 384(1):433--501, 2021.

\bibitem{NNInstab2}
M.~J. Colbrook, V.~Antun, and A.~C. Hansen.
\newblock The difficulty of computing stable and accurate neural networks: {O}n
  the barriers of deep learning and {S}male's 18th problem.
\newblock {\em Proc.\ Natl.\ Acad.\ Sci.}, 119(12):e2107151119, 2022.

\bibitem{Cucker_Smale97}
F.~Cucker and S.~Smale.
\newblock Complexity estimates depending on condition and round-off error.
\newblock {\em J. ACM}, 46(1):113--184, 1999.

\bibitem{donohoCS}
D.~L. Donoho.
\newblock Compressed sensing.
\newblock {\em IEEE Trans. Inform. Theory}, 52(4):1289--1306, 2006.

\bibitem{Doyle_McMullen}
P.~Doyle and C.~McMullen.
\newblock Solving the quintic by iteration.
\newblock {\em Acta Math.}, 163(3-4):151--180, 1989.

\bibitem{SpernerTheory}
K.~Engel.
\newblock {\em Sperner Theory}.
\newblock Encyclopedia of Mathematics and its Applications. Cambridge
  University Press, 1997.

\bibitem{Gabriel1}
J.~M. Fadili and G.~Peyre.
\newblock Total variation projection with first order schemes.
\newblock {\em IEEE Transactions on Image Processing}, 20(3):657--669, March
  2011.

\bibitem{Fefferman_Klartag2}
C.~Fefferman and B.~Klartag.
\newblock {Fitting a $C^m$-Smooth Function to Data II}.
\newblock {\em Revista Matemática Iberoamericana}, 25(1):49 -- 273, 2009.

\bibitem{fefferman1990}
C.~Fefferman and L.~Seco.
\newblock On the energy of a large atom.
\newblock {\em Bull. Amer. Math. Soc. (N.S.)}, 23(2):525--530, 10 1990.

\bibitem{fefferman1992}
C.~Fefferman and L.~Seco.
\newblock Eigenvalues and eigenfunctions of ordinary differential operators.
\newblock {\em Adv. Math.}, 95(2):145 -- 305, 1992.

\bibitem{fefferman1993aperiodicity}
C.~Fefferman and L.~Seco.
\newblock Aperiodicity of the {H}amiltonian flow in the {T}homas-{F}ermi
  potential.
\newblock {\em Revista Matem{\'a}tica Iberoamericana}, 9(3):409--551, 1993.

\bibitem{fefferman1994}
C.~Fefferman and L.~Seco.
\newblock The eigenvalue sum for a one-dimensional potential.
\newblock {\em Adv. Math.}, 108(2):263--335, 1994.

\bibitem{fefferman1994_2}
C.~Fefferman and L.~Seco.
\newblock On the {D}irac and {S}chwinger corrections to the ground-state energy
  of an atom.
\newblock {\em Adv. Math.}, 107(1):1--185, 1994.

\bibitem{fefferman1995}
C.~Fefferman and L.~Seco.
\newblock The density in a three-dimensional radial potential.
\newblock {\em Adv. Math.}, 111(1):88 -- 161, 1995.

\bibitem{fefferman1996}
C.~Fefferman and L.~Seco.
\newblock The eigenvalue sum for a three-dimensional radial potential.
\newblock {\em Adv. Math.}, 119(1):26 -- 116, 1996.

\bibitem{fefferman1996interval}
C.~Fefferman and L.~Seco.
\newblock Interval arithmetic in quantum mechanics.
\newblock In {\em Applications of interval computations}, pages 145--167.
  Springer, 1996.

\bibitem{fefferman1997}
C.~Fefferman and L.~Seco.
\newblock The density in a one-dimensional potential.
\newblock {\em Adv. Math}, 107, 05 1997.

\bibitem{Fefferman_Klartag}
C.~L. {Fefferman} and B.~{Klartag}.
\newblock {Fitting a \(C^m\)-smooth function to data. I}.
\newblock {\em {Ann. Math. (2)}}, 169(1):315--346, 2009.

\bibitem{Lovasz_JACM_96}
U.~Feige, S.~Goldwasser, L.~Lov\'{a}sz, S.~Safra, and M.~Szegedy.
\newblock Interactive proofs and the hardness of approximating cliques.
\newblock {\em J. ACM}, 43(2):268--292, 1996.

\bibitem{Mario}
M.~A.~T. Figueiredo, R.~D. Nowak, and S.~J. Wright.
\newblock Gradient projection for sparse reconstruction: Application to
  compressed sensing and other inverse problems.
\newblock {\em IEEE Journal of Selected Topics in Signal Processing},
  1(4):586--597, 2007.

\bibitem{foucartBook}
S.~Foucart and H.~Rauhut.
\newblock {\em A mathematical introduction to compressive sensing}.
\newblock Springer, 2013.

\bibitem{gacs1981khachiyan}
P.~G{\'a}cs and L.~Lov{\'a}sz.
\newblock Khachiyan's algorithm for linear programming.
\newblock In {\em Mathematical Programming at Oberwolfach}, pages 61--68.
  Springer, 1981.

\bibitem{Lovasz_book}
M.~Gr{\"o}tschel, L.~Lov{\'a}sz, and A.~Schrijver.
\newblock {\em Geometric algorithms and combinatorial optimization}.
\newblock Springer, 1988.

\bibitem{Hales1}
T.~C. Hales.
\newblock A proof of the {K}epler conjecture.
\newblock {\em Ann. of Math. (2)}, 162(3):1065--1185, 2005.

\bibitem{Hales2}
T.~C. Hales and et~al.
\newblock A formal proof of the kepler conjecture.
\newblock {\em Forum of Mathematics, Pi}, 5:e2, 2017.

\bibitem{Hansen_JAMS}
A.~C. Hansen.
\newblock On the solvability complexity index, the {$n$}-pseudospectrum and
  approximations of spectra of operators.
\newblock {\em J. Amer. Math. Soc.}, 24(1):81--124, 2011.

\bibitem{Hastad_Acta}
J.~H{\aa}stad.
\newblock Clique is hard to approximate within $n^{1-\epsilon}$.
\newblock {\em Acta Mathematica}, 182(1):105--142, 1999.

\bibitem{Hastad_JACM}
J.~H{\aa}stad.
\newblock Some optimal inapproximability results.
\newblock {\em J. ACM}, 48(4):798--859, 2001.

\bibitem{Tibshirani_Book}
T.~Hastie, R.~Tibshirani, and J.~Friedman.
\newblock {\em The Elements of Statistical Learning}.
\newblock Springer Series in Statistics. Springer New York Inc., New York, NY,
  USA, 2001.

\bibitem{Juditsky_2012}
A.~Juditsky, F.~Kilin{\c{c}}{-}Karzan, A.~Nemirovski, and B.~Polyak.
\newblock {Accuracy guaranties for $\ell_{1}$ recovery of block-sparse
  signals}.
\newblock {\em The Annals of Statistics}, 40(6):3077 -- 3107, 2012.

\bibitem{Juditsky_2011}
A.~B. Juditsky, F.~Kilin{\c{c}}{-}Karzan, and A.~Nemirovski.
\newblock Verifiable conditions of $\ell_{1}$-recovery for sparse signals with
  sign restrictions.
\newblock {\em Math. Program.}, 127(1):89--122, 2011.

\bibitem{karmarkar1984new}
N.~Karmarkar.
\newblock A new polynomial-time algorithm for linear programming.
\newblock In {\em Proceedings of the sixteenth annual ACM symposium on Theory
  of computing}, pages 302--311. ACM, 1984.

\bibitem{kepler1966strena}
J.~Kepler.
\newblock {\em Strena seu de nive sexangula}.
\newblock 1611.

\bibitem{khachiyan1980polynomial}
L.~G. Khachiyan.
\newblock Polynomial algorithms in linear programming.
\newblock {\em Zhurnal Vychislitel'noi Matematiki i Matematicheskoi Fiziki},
  20(1):51--68, 1980.

\bibitem{Ko1991ComplexityTO}
K.~Ko.
\newblock {\em Complexity Theory of Real Functions}.
\newblock 1991.

\bibitem{Lagarias2011}
J.~C. Lagarias.
\newblock {\em Bounds for Local Density of Sphere Packings and the Kepler
  Conjecture}, pages 27--57.
\newblock Springer New York, New York, NY, 2011.

\bibitem{Lagarias2011_intro}
J.~C. Lagarias.
\newblock {\em The Kepler Conjecture and Its Proof}, pages 3--26.
\newblock Springer New York, New York, NY, 2011.

\bibitem{Lagarias_Wright}
J.~C. Lagarias, B.~Poonen, and M.~H. Wright.
\newblock Convergence of the restricted nelder--mead algorithm in two
  dimensions.
\newblock {\em SIAM Journal on Optimization}, 22(2):501--532, 2012.

\bibitem{lawler1980great}
E.~L. Lawler.
\newblock The great mathematical sputnik of 1979.
\newblock {\em The Mathematical Intelligencer}, 2(4):191--198, 1980.

\bibitem{lovasz1987algorithmic}
L.~Lovasz.
\newblock {\em An Algorithmic Theory of Numbers, Graphs and Convexity}.
\newblock CBMS-NSF Regional Conference Series in Applied Mathematics. Society
  for Industrial and Applied Mathematics, 1987.

\bibitem{Lovasz2010}
L.~Lov{\'a}sz.
\newblock Discrete and continuous: Two sides of the same?
\newblock In N.~Alon, J.~Bourgain, A.~Connes, M.~Gromov, and V.~Milman,
  editors, {\em Visions in Mathematics: GAFA 2000 Special Volume, Part I},
  pages 359--382. Birkh{\"a}user Basel, Basel, 2010.

\bibitem{McMullen1}
C.~McMullen.
\newblock Families of rational maps and iterative root-finding algorithms.
\newblock {\em Ann. of Math. (2)}, 125(3):467--493, 1987.

\bibitem{McMullen2}
C.~McMullen.
\newblock Braiding of the attractor and the failure of iterative algorithms.
\newblock {\em Invent. Math.}, 91(2):259--272, 1988.

\bibitem{ElementaryFunctionsImplementation}
J.-M. Muller.
\newblock {\em Elementary Functions: Algorithms and Implementation}.
\newblock Birkhauser, 2005.

\bibitem{Nemirovski1995}
A.~Nemirovski.
\newblock Polynomial time methods in convex programming.
\newblock In J.~Renegar, M.~Shub, and S.~Smale, editors, {\em The Mathematics
  of Numerical Analysis}, volume~32, pages 543--589. AMS-SIAM Summer Seminar on
  Applied Mathematics, 1995.

\bibitem{Nemirovski07}
A.~Nemirovski.
\newblock Advances in convex optimization: Conic programming.
\newblock In {\em In Proceedings of International Congress of Mathematicians},
  pages 413--444, 2007.

\bibitem{Nesterov2}
Y.~Nesterov.
\newblock {\em Introductory lectures on convex optimization : a basic course}.
\newblock Applied optimization. Kluwer Academic Publ., Boston, Dordrecht,
  London, 2004.

\bibitem{Nesterov2018}
Y.~Nesterov.
\newblock {\em Lectures on Convex Optimization}.
\newblock Springer Publishing Company, Incorporated, 2nd edition, 2018.

\bibitem{Nesterov1}
Y.~Nesterov and A.~Nemirovskii.
\newblock {\em Interior-point polynomial algorithms in convex programming}.
\newblock SIAM studies in applied mathematics. Society for Industrial and
  Applied Mathematics, Philadelphia, 1994.

\bibitem{Nesterov_Nemirovski_Acta}
Y.~E. Nesterov and A.~Nemirovski.
\newblock On first-order algorithms for l1/nuclear norm minimization.
\newblock {\em Acta Numer.}, 22:509--575, 2013.

\bibitem{NoceWrig06}
J.~Nocedal and S.~J. Wright.
\newblock {\em Numerical Optimization}.
\newblock Springer, New York, NY, USA, second edition, 2006.

\bibitem{odifreddi1992classical}
P.~Odifreddi.
\newblock {\em Classical Recursion Theory: The Theory of Functions and Sets of
  Natural Numbers}.
\newblock ISSN. Elsevier Science, 1992.

\bibitem{Papadimitriou1994}
C.~H. Papadimitriou.
\newblock {\em Computational complexity.}
\newblock Addison-Wesley, 1994.

\bibitem{Gabriel2}
G.~Peyr{\'e}, S.~Bougleux, and L.~D. Cohen.
\newblock Non-local regularization of inverse problems.
\newblock In D.~A. Forsyth, P.~H.~S. Torr, and A.~Zisserman, editors, {\em
  Proc. of ECCV'08}, volume 5304 of {\em Lecture Notes in Computer Science},
  pages 57--68. Springer, 2008.

\bibitem{realRAM}
F.~Preparata and M.~Shamos.
\newblock {\em Computational Geometry: An Introduction}.
\newblock Monographs in Computer Science. Springer New York, 2012.

\bibitem{renegar1988polynomial}
J.~Renegar.
\newblock A polynomial-time algorithm, based on newton's method, for linear
  programming.
\newblock {\em Mathematical Programming}, 40(1-3):59--93, 1988.

\bibitem{Renegar2}
J.~Renegar.
\newblock Incorporating condition measures into the complexity theory of linear
  programming.
\newblock {\em SIAM Journal on Optimization}, 5(3):506--524, 1995.

\bibitem{Renegar1}
J.~Renegar.
\newblock Linear programming, complexity theory and elementary functional
  analysis.
\newblock {\em Mathematical Programming}, 70(1):279--351, 1995.

\bibitem{Rockafellar1970}
R.~T. Rockafellar.
\newblock {\em Convex Analysis}.
\newblock Princeton University Press, 1970.

\bibitem{RudelsonVershyninRIP}
M.~Rudelson and R.~Vershynin.
\newblock On sparse reconstruction from fourier and gaussian measurements.
\newblock {\em Communications on Pure and Applied Mathematics},
  61(8):1025--1045, 2008.

\bibitem{Osher_ROF}
L.~I. Rudin, S.~Osher, and E.~Fatemi.
\newblock Nonlinear total variation based noise removal algorithms.
\newblock {\em Physica D: Nonlinear Phenomena}, 60(1):259 -- 268, 1992.

\bibitem{MarkovModel}
A.~Salomaa, C.~U. Press, G.~Rota, B.~Doran, T.~Lam, P.~Flajolet, M.~Ismail, and
  E.~Lutwak.
\newblock {\em Computation and Automata}.
\newblock EBL-Schweitzer. Cambridge University Press, 1985.

\bibitem{Smale81}
S.~Smale.
\newblock {The fundamental theorem of algebra and complexity theory}.
\newblock {\em Bulletin of the American Mathematical Society}, 4(1):1 -- 36,
  1981.

\bibitem{Smale85}
S.~Smale.
\newblock {On the efficiency of algorithms of analysis}.
\newblock {\em Bulletin of the American Mathematical Society}, 13(2):87 -- 121,
  1985.

\bibitem{21century_Smale}
S.~Smale.
\newblock Mathematical problems for the next century.
\newblock {\em Mathematical Intelligencer}, 20:7--15, 1998.

\bibitem{Smale_McMullen}
S.~Smale.
\newblock {The work of {C}urtis {T} {M}c{M}ullen}.
\newblock In {\em Proceedings of the International Congress of Mathematicians
  I, Berlin}, Doc. Math. J. DMV, pages 127--132. 1998.

\bibitem{FloHadHaar}
V.~Studer, J.~Bobin, M.~Chahid, H.~S. Mousavi, E.~Candes, and M.~Dahan.
\newblock Compressive fluorescence microscopy for biological and hyperspectral
  imaging.
\newblock {\em Proceedings of the National Academy of Sciences},
  109(26):E1679--E1687, 2012.

\bibitem{Sudan_Overview_2009}
M.~Sudan.
\newblock Probabilistically checkable proofs.
\newblock {\em Commun. ACM}, 52(3):76--84, 2009.

\bibitem{LauraWalshHaar}
L.~Thesing and A.~C. Hansen.
\newblock Linear reconstructions and the analysis of the stable sampling rate.
\newblock {\em Sampling Theory in Signal and Image Processing}, 2018.

\bibitem{TibshiraniLasso}
R.~Tibshirani.
\newblock Regression shrinkage and selection via the lasso.
\newblock {\em Journal of the Royal Statistical Society, Series B},
  58:267--288, 1996.

\bibitem{Turing_Machine}
A.~M. Turing.
\newblock On {C}omputable {N}umbers, with an {A}pplication to the
  {E}ntscheidungsproblem.
\newblock {\em Proc. London Math. Soc.}, S2-42(1):230, 1936.

\bibitem{Valiant_book}
L.~Valiant.
\newblock {\em Probably Approximately Correct: Nature's Algorithms for Learning
  and Prospering in a Complex World}.
\newblock Basic Books, Inc., New York, NY, USA, 2013.

\bibitem{spgl1}
E.~van~den Berg and M.~P. Friedlander.
\newblock Probing the pareto frontier for basis pursuit solutions.
\newblock {\em SIAM Journal on Scientific Computing}, 31(2):890--912, 2008.

\bibitem{von_Neumann}
J.~von Neumann.
\newblock First draft of a report on the edvac.
\newblock {\em IEEE Ann. Hist. Comput.}, 15(4):27--75, 1993.

\bibitem{Wright}
M.~H. Wright.
\newblock Ill-conditioning and computational error in interior methods for
  nonlinear programming.
\newblock {\em SIAM Journal on Optimization}, 9(1):84--111, 1998.

\bibitem{Wright2}
M.~H. Wright.
\newblock The interior-point revolution in optimization: history, recent
  developments, and lasting consequences.
\newblock {\em Bull. Amer. Math. Soc. (N.S.)}, 42(1):39--56, 2005.

\bibitem{Wrig97}
S.~J. Wright.
\newblock {\em Primal-Dual Interior-Points Methods}.
\newblock SIAM, Philadelphia, Pa, USA, 1997.

\end{thebibliography}
\end{document}